\PassOptionsToPackage{pdfauthor={},pdftitle={}}{hyperref}
\documentclass[aos]{imsart}

\setpkgattr{journal}{name}{\empty}
\setpkgattr{author}{prefix}{}

\RequirePackage{amsthm,amsmath,amsfonts,amssymb, yhmath}
\RequirePackage[numbers,sort]{natbib}
\RequirePackage[colorlinks,citecolor=blue,urlcolor=blue]{hyperref}%% uncomment this for coloring bibliography citations and linked URLs
\RequirePackage{graphicx}%% uncomment this for including figures

\usepackage{comment}
\usepackage{natbib} 
\usepackage[usenames,dvipsnames]{color}
\definecolor{RefColor}{rgb}{0,0,.85}  

\makeatletter
\def\acks{\gdef\@thefnmark{}\@footnotetext}
\makeatother  

\usepackage{mathtools}
\usepackage[capitalize]{cleveref}
\usepackage{dsfont}
\usepackage{mathdots}
\usepackage{nccmath} 
\usepackage{scalerel}
\usepackage[hang,flushmargin,para]{footmisc}
\usepackage{caption}
\usepackage{titlesec}
\titleformat*{\section}{\large\bfseries}
\titleformat*{\subsection}{\bfseries}

\RequirePackage[OT1]{fontenc}

\usepackage[english]{babel}

\usepackage{enumitem}
\setlist[itemize]{leftmargin=1.5em}

\usepackage{booktabs}
\usepackage{tikz}
\usetikzlibrary{calc}
\usetikzlibrary{decorations.markings,decorations.pathreplacing}
\tikzstyle{mybraces}=[mirrorbrace/.style={
          decoration={brace, mirror},
          decorate},brace/.style={
          decoration={brace},
          decorate}]

\usepackage{amsmath,amssymb,amscd,amsfonts,amsthm,mathtools}
\usepackage{thmtools}

\startlocaldefs

\theoremstyle{plain}
\declaretheoremstyle[postheadspace=.4em,headfont=\bfseries,bodyfont=\itshape,spaceabove=8pt,
spacebelow=10pt]{basic}
\theoremstyle{basic}
\declaretheorem[style=basic,name={Theorem}]{theorem}
\declaretheorem[style=basic,sibling=theorem,name={Lemma}]{lemma}
\declaretheorem[style=basic,sibling=theorem,name={Proposition}]{proposition}
\declaretheorem[style=basic,sibling=theorem,name={Corollary}]{corollary}
\theoremstyle{definition}

\declaretheorem[style=definition,name={Remark}]{remark}
\declaretheorem[style=definition,name={Remark},numbered=no]{remark*}
\declaretheorem[style=definition,name={Example}]{example}
\declaretheorem[style=definition,name={Assumption}]{assumption}
\usepackage{times}

\newenvironment{proplist}{\begin{enumerate}[
    label=(\roman{enumi}),
    ]}
{\end{enumerate}}

\newcommand{\argdot}{{\,\vcenter{\hbox{\tiny$\bullet$}}\,}}%{\bullet}

\newcommand{\tagaligneq}{\refstepcounter{equation}\tag{\theequation}}

\newcommand{\msup}{\sup\nolimits}
\newcommand{\minf}{\inf\nolimits}
\newcommand{\mmax}{\max\nolimits}
\newcommand{\mmin}{\min\nolimits}
\newcommand{\Tr}{\text{\rm Tr}}
\newcommand{\ind}{\mathbb{I}}

\def\bV{\mathbf{V}}
\def\bW{\mathbf{W}}
\def\bX{\mathbf{X}}
\def\bY{\mathbf{Y}}
\def\bZ{\mathbf{Z}}

\def\bx{\mathbf{x}}
\def\by{\mathbf{y}}

\def\bzero{\mathbf{0}}

\def\N{\mathbb{N}}

\def\P{\mathbb{P}}

\def\R{\mathbb{R}}

\def\cB{\mathcal{B}}

\def\cE{\mathcal{E}}

\def\cG{\mathcal{G}}

\def\cI{\mathcal{I}}

\def\cN{\mathcal{N}}

\newcommand{\mean}{\mathbb{E}}
\newcommand{\Var}{\text{\rm Var}}
\newcommand{\Cov}{\text{\rm Cov}}
\newcommand{\Kurt}{\text{\rm Kurt}}

\DeclareMathOperator{\msum}{\scaleobj{.8}{\sum}}
\DeclareMathOperator{\mprod}{\scaleobj{.8}{\prod}}

\DeclareMathOperator{\mint}{\scaleobj{.8}{\int}}

\DeclareMathOperator{\condind}{{\perp\!\!\!\perp}}

\newcommand{\tvec}{\text{\rm vec}}

\endlocaldefs

\begin{document}

\begin{frontmatter}
    \title{Gaussian universality for approximately polynomial functions of high-dimensional data}
    %\title{A sample article title with some addzitional note\thanksref{T1}}
    % \runtitle{???}
    %\thankstext{T1}{A sample of additional note to the title.}
    
\begin{aug}
    \author{\fnms{Kevin Han}~\snm{Huang}},
    \author{\fnms{Morgane}~\snm{Austern}}
    \and    
    \author{\fnms{Peter}~\snm{Orbanz}}
    \\[1em]\normalfont\small University of Warwick, Harvard University and University College London
\end{aug}

    % \begin{aug}
    %     %%%%%%%%%%%%%%%%%%%%%%%%%%%%%%%%%%%%%%%%%%%%%%%
    %     %% Only one address is permitted per author. %%
    %     %% Only division, organisation and e-mail is %%
    %     %% included in the address.                  %%
    %     %% Additional information can be included in %%
    %     %% the Acknowledgments section if necessary. %%
    %     %% ORCID can be inserted by command:         %%
    %     %% \orcid{0000-0000-0000-0000}               %%
    %     %%%%%%%%%%%%%%%%%%%%%%%%%%%%%%%%%%%%%%%%%%%%%%%
    %     % \author{\fnms{Kevin Han}~\snm{Huang}},
    %     % \author{\fnms{Morgane}~\snm{Austern}}
    %     % \and    
    %     % \author{\fnms{Peter}~\snm{Orbanz}}
    %     % \\[1em]\normalfont\small University College London, Harvard University and University College London
    
    %     \author[ucl]{\fnms{Kevin Han}~\snm{Huang}\ead[label=khh]{han.huang.20@ucl.ac.uk}},
    %     \author[harvard]{\fnms{Morgane}~\snm{Austern}\ead[label=ma]{maustern@fas.harvard.edu}}
    %     \and 
    %     \author[ucl]{\fnms{Peter}~\snm{Orbanz}\ead[label=po]{p.orbanz@ucl.ac.uk}}
    %     % %%%%%%%%%%%%%%%%%%%%%%%%%%%%%%%%%%%%%%%%%%%%%%
    %     % %% Addresses                                %%
    %     % %%%%%%%%%%%%%%%%%%%%%%%%%%%%%%%%%%%%%%%%%%%%%%
    %     \address[ucl]{University College London\printead[presep={,\ }]{khh,po}}
    %     \address[harvard]{Harvard University\printead[presep={,\ }]{ma}}
    % \end{aug}
    
    \begin{abstract}
        Gaussian universality results assert that the properties of many estimators remain unchanged when the input data are replaced by Gaussians. Such results have gained popularity in high-dimensional statistics and machine learning, as Gaussianity often substantially simplifies downstream analyses. Yet, an open question remains on when universality may cease to hold. To address this, we establish nearly optimal upper and lower bounds for Gaussian universality approximation, measured in Kolmogorov distance, over the class of  approximately polynomial functions of high-dimensional random vectors. The upper bounds adapt the invariance principle of Mossel, O'Donnell and Oleszkiewicz (2010) for high-dimensional vectors and functions beyond multilinear forms. As applications, we obtain a delta method for high-dimensional data with non-Gaussian limits, a necessary and sufficient condition for asymptotic normality, and simple estimators that are asymptotically normal but for which bootstrap fails to be consistent. We also extend recent results on the high-dimensional degeneracy of non-degenerate U-statistics, phase transition of MMD in two-sample tests with imbalanced data, and confidence spheres for high-dimensional averages. Our lower bound is constructive and shows that, for polynomials of even degree $m$, universality holds up to $m=o(\log n)$. As a corollary, the Gaussian polynomial approximation error of $\Omega(n^{-1/6m})$ is not improvable for even-degree U-statistics and V-statistics. Our results also explain how universality results for U-statistics and V-statistics differ significantly in their dependence on dimensions.

        % We establish Gaussian universality results for polynomial functions of $n$ independent, high-dimensional random vectors, and also show that the obtained rates are nearly optimal. Both the dimension of the vectors and the degree of the polynomial are permitted to grow with $n$. Specifically, we obtain a finite sample upper bound for the error of approximation by a polynomial of Gaussians, measured in Kolmogorov distance, and extend it to functions that are approximately polynomial in a mean squared error sense. We give a corresponding lower bound that shows the invariance principle holds up to polynomial degree $o(\log n)$. The proof is constructive and adapts an asymmetrisation argument due to V. V. Senatov. We also give a necessary and sufficient condition for asymptotic normality via the fourth moment phenomenon of Nualart and Peccati. As applications, we obtain a higher-order delta method with possibly non-Gaussian limits, and generalise a number of known results on high-dimensional and infinite-order U-statistics, and on fluctuations of subgraph counts.
    \end{abstract}

        % \begin{keyword}[class=MSC]
        % \kwd[Primary ]{62E17} % Statistics - Statistical distribution theory - Approximations to statistical distributions (non-asymptotic)
        % \kwd[; secondary ]{60F05} % Central limit and other weak theorems
        % \end{keyword}
        
        
    \end{frontmatter}

\section{Introduction} % and main results}
\label{sec:introduction}

Gaussian universality results approximate functions of $n$ random vectors in distribution by substituting each argument by a
Gaussian surrogate, typically as $n$ grows \citep{chatterjee2006generalization,mossel2010noise,korada2011applications}.
Informally, such results are of the form
\begin{equation}
  \label{gaussian:universality}
  f_n(X_1,\ldots,X_n)\;\approx\;f_n(Z_1,\ldots,Z_n) \qquad\text{ in distribution for large }n\;.
\end{equation}
% \begin{equation}
%   \label{gaussian:universality}
%   f_n(X_{n1},\ldots,X_{nn})\;\approx\;f_n(Z_{n1},\ldots,Z_{nn}) \qquad\text{ in distribution for large }n\;,
% \end{equation}
% where the arguments $X_{ni}$ are independent random vectors in $\mathbb{R}^{d(n)}$ for each ${n\in\mathbb{N}}$.
% The approximation substitutes independent Gaussian vectors ${Z_{ni}}$, where each ${Z_{ni}}$ has the same mean and
% covariance as $X_{ni}$. 
% The case where the dimension $d(n)$ grows with $n$ at a specified rate is commonly referred to
% as the high-dimensional regime. We shall measure the approximation error of \eqref{gaussian:universality} in Kolmogorov distance, which characterizes weak convergence and is also practically relevant for statistical applications such as confidence intervals and hypothesis testing.
% Applications of Gaussian universality range from the spectra of large random matrices \citep{chatterjee2006generalization, tao2011random,tao2015random, wood2016universality, basak2018circular} to the risk of neural networks and other high-dimensional regressors \citep{korada2011applications, han2023universality, montanari2022universality, gerace2024gaussian, hu2022universality}.
% \vspace{.5em}
$(X_i)_{i \leq n}$ is a set of independent random vectors taking values in $\R^d$.
The approximation substitutes the data by 
independent Gaussian vectors $(Z_i)_{i \leq n}$, where each $Z_i$ has the same mean and covariance as $X_i$. For $d=1$, the central limit theorem and Wigner's semicircle law are both special cases: In the former case, $f_n$ is an empirical average, and in the latter, $f_n$ is the empirical spectral distribution of a random matrix. We shall measure the approximation error of \eqref{gaussian:universality} in Kolmogorov distance, which characterises weak convergence and is also practically relevant for statistical applications such as confidence intervals and hypothesis testing.
% One of the first efforts to develop universality results beyond averages are the works of Rotar \cite{ rotar1979limit}, Mossel, O'Donnel and Oleszkiewicz \cite*{mossel2005noise,mossel2010noise} and Chatterjee \cite{chatterjee2006generalization}. The core proof technique is Lindeberg's swapping method \cite{lindeberg1922neue,trotter1959elementary}, which 
% is known to yield sub-optimal rates in several cases \cite{chen2011normal,brailovskaya2022universality}. 
% but has nevertheless led to fruitful results in random matrix theory \citep{chatterjee2006generalization, tao2011random,tao2015random}.

In the high-dimensional regime, where $d \equiv d(n)$ is allowed to grow in $n$, there has been a surge of interest in establishing universality results for many complicated and non-linear estimators in statistics and machine learning, all on a case-by-case basis. A non-exhaustive list includes properties of high-dimensional covariance matrices \cite{tao2011random, basak2018circular}, the risks of various high-dimensional regressors and classifiers \citep{hu2022universality,han2023universality, dandi2023universality, gerace2024gaussian, montanari2022universality, montanari2023universality}, and representations of data generated by generative adversarial networks \citep{seddik2020random} and student-teacher models \citep{loureiro2021learning,pesce2023gaussian}. 
% In recent years, Gaussian universality has garnered significant interest within the high-dimensional statistics and machine learning community. 
The practical motivation is that, if \eqref{gaussian:universality} holds, one may study properties of the estimators either through direct simulation on Gaussian data or through many theoretical techniques developed specifically for Gaussians, e.g.~approximate message passing \cite{donoho2009message}, the cavity method \cite{opper2001naive} or the convex Gaussian min-max theorem \cite{gordon1985some,thrampoulidis2014gaussian}.
% The general recipe for establishing universality results has been developed in 
% \cite{rotar1979limit,mossel2005noise,chatterjee2006generalization,mossel2010noise} based on Lindeberg's principle.
% In the case of a growing dimension $d(n)$, existing works have established universality for the risk of neural networks and other high-dimensional regressors \citep{korada2011applications, han2023universality, montanari2022universality, gerace2024gaussian, hu2022universality} as well as properties of high-dimensional covariance matrices \cite{tao2011random,tao2015random, wood2016universality, basak2018circular}, all on a case-by-case basis. 
Despite the popularity and utility of universality results, few results are available on understanding the following question:
\begin{center}
  \emph{When may universality fail for high-dimensional estimators?}
\end{center}
To address this question, one needs to establish tight upper and lower bounds for a class of general functions and understand how the bounds behave in the high-dimensional regime. In this paper, we do so by focusing on the class of 
% To understand the extent of applicability of universality, we aim to characterize the class of functions for which Gaussian universality holds in the high-dimensional regime. 
% In this paper, we focus on establishing these results for 
approximately polynomial functions, where the polynomial degree $m \equiv m(n)$ is also allowed to depend on $n$. This, for example, includes any estimator that can be approximated by a suitable $m$-th order Taylor expansion. 

One aspect of the question is whether known universality approximation rates are optimal. The general recipe for establishing universality results has been developed in \cite{rotar1979limit,mossel2005noise,chatterjee2006generalization,mossel2010noise} based on Lindeberg's principle. Despite the method's generality, 
% it is known to 
a naive application
suffers from sub-optimal rates in the cases of empirical averages \citep{chen2011normal} and random matrices \citep{brailovskaya2022universality}, and lead to upper bounds for degree-$m$ multilinear polynomials that require $m \log m = o(\log n)$ \cite{mossel2010noise}. This raises the question whether the rates from Lindeberg's principle are improvable in a high-dimensional setup, especially in terms of how fast $m$ is allowed to grow in $n$. 
 
Another related and crucial question is how fast the dimension $d$ may grow in $n$ and how the bounds may depend on $d$. 
This is well-understood in universality works that either explicitly or implicitly involve the spectral statistics of high-dimensional random matrices in $\R^{n \times d}$ \citep{hastie2022surprises,montanari2022universality}: There, the natural asymptotic is the proportional regime where $d, n \rightarrow \infty$ and $d/n \rightarrow $ some constant,  and we do not focus on this setting. Meanwhile, the condition on the growth of $d$ can become much more subtle, once $f$ is allowed to be a more general statistic.
This has been extensively observed in the case of high-dimensional empirical averages: When $f(X_1, \ldots, X_n)$ is the maximal coordinate of an $d$-dimensional i.i.d.~average, Gaussian universality is known to hold with the dimension growing as fast as $d = O(e^{n^c})$ for some constant $c > 0$ \citep{chernozhukov2013gaussian,fang2021high}. When $f(X_1, \ldots, X_n)$ is the $l_2$-norm of an i.i.d.~average, however, Fang and Koike \citep{fang2024large} show that the condition $d=o(n)$ is not improvable.
% Zhilova \cite{zhilova2020nonclassical} shows that the condition $d=o(n^{1/3})$ is not improvable.
On the other hand, if $f(X_1, \ldots, X_n)$ are degree-two U-statistics, a recent line of work \citep{kim2024dimension,huang2023high,gao2025dimension} has shown that distributional approximations may be obtained regardless of how fast $d$ grows. These results form part of a growing body of statistical work on dimension-agnostic inference \citep{chen2025randomized, takatsu2025bridging, kim2025locally} and have important implications for handling data beyond the proportional regime. These results show that much is to be understood about how dimension dependence changes depending on $f$. Crucially, notice that the squared $l_2$-norm of an i.i.d.~average --- considered in the literature on Gaussian approximation in high-dimensional Euclidean balls \citep{sazonov1972bound,bentkus2003dependence,zhilova2020nonclassical,fang2024large} --- is a simple degree-two V-statistic, which classically admits a similar distributional approximation as a degree-two U-statistic and yet seems to have a different dimension dependence from that of U-statistics as $d$ grows. Our results will explain how this discrepancy arises due to how the growth of $d$ interacts with Gaussian universality properties of these statistics as well as the assumptions used in the different lines of work.
\color{black}

\vspace{.5em}

Motivated by these questions, the main results of this paper are: 
\begin{proplist}
  \item A universality result for degree-$m$ multilinear polynomials that upper-bounds the approximation error in \eqref{gaussian:universality} (\Cref{thm:main}). This extends the invariance principle of Mossel et al. \cite{mossel2010noise} to high-dimensional vectors and accommodates dependence across coordinates.
  \item An extension of the upper bound to approximately multilinear polynomials (\Cref{thm:VD}), which covers non-multilinear polynomials such as V-statistics (\Cref{prop:simpler:V}). 
  \item A nearly matching lower bound (\Cref{thm:lower:bound}), which is established for the first time.
\end{proplist}

These results enable us to provide conditions for the validity of Gaussian universality for approximately polynomial estimators in the high-dimensional regime. When universality does hold, our results also reveal new insights on how the behaviours of various high-dimensional estimators may depart from classical theory. Specifically, our main findings are:

\begin{itemize}
  \item \textbf{Polynomial degree dependence}. Both upper and lower bounds on the universality approximation deteriorate as degree $m$ grows, and universality fails when $m = o(\log n)$. This confirms that the requirement $m \log m = o(\log n)$, known from existing results based on Lindeberg's technique, is optimal up to a $\log m$ factor. As special cases, we obtain even-degree U-statistics and V-statistics that admit a Gaussian polynomial approximation error of $\Omega(n^{-1/6m})$ (\Cref{sec:lower:bound}), which answers an open conjecture by \cite{yanushkevichiene2012bounds}.
  \item \textbf{Dimension-agnostic results for U-statistics}.  Our bounds for multilinear polynomials, which include U-statistics as special cases, do not have explicit dependence on $d$ under a mild moment condition. These bounds extend dimension-agnostic distributional approximations for degree-two U-statistics to degree-$m$ U-statistics (\Cref{sec:degree:m:U:stats}). As a by-product, we show that degree-$m$ U-statistics can undergo a phase transition effect as hyperparameters of the U-statistic kernel change, similar to what has been observed for degree-two U-statistics \citep{bhattacharya2022asymptotic,huang2023high,bhattacharya2023fluctuations}.
  % also generalize the phase transition effect 
  % \citep{bhattacharya2022asymptotic,huang2023high,bhattacharya2023fluctuations} of degree-two U statistics to degree-$m$ ones. 
  This notably implies that a non-degenerate U-statistic can exhibit non-Gaussianity and a degenerate one can exhibit Gaussianity in high dimensions. This has been shown to have important implications on high-dimensional distribution testing \citep{yan2021kernel,huang2023high}, and in the case of a two-sample test, we provide an extension of existing phase transition results for the Maximum Mean Discrepancy (MMD) statistic (\Cref{sec:MMD:imbalance}). On the other hand, we also clarify when the universality approximation can become dimension-sensitive due to violation of the mild moment condition (\Cref{sec:non:multilinear}).
  \item \textbf{Dimension sensitive results for V-statistics}. 
  % Our dimension-agnostic bounds for multilinear polynomials are directly applicable to non-multilinear ones by a reparameterisation of the data vectors and the surrogates (\Cref{cor:non:multilinear}). However,
  The universality approximation \eqref{gaussian:universality} with a non-multilinear polynomial $f_n$ on the original data vectors, e.g.~a V-statistic, involves an additional dimension-sensitive term, which arises due to dependence across coordinates (\Cref{sec:non:multilinear}). As an application, these dimension-sensitive bounds extend 
  % $l_2$-confidence sphere results
  the high-dimensional confidence sphere results for empirical averages in \citep{sazonov1972bound,bentkus2003dependence,zhilova2020nonclassical,fang2024large} to 
  % $l_m$-confidence spheres.
  allow for non-invertible covariance matrices, which are useful for data vectors that are possibly low-ranked or sparse (\Cref{sec:high:d:averages}). 
  Our bounds are dimension-agnostic in special cases, and give the non-improvable condition
  $d=o(n)$ of \cite{fang2024large} 
  % $d=o(n^{1/3})$ of \cite{zhilova2020nonclassical} 
  in others. 
  \item \textbf{Delta method for high-dimensional data.} As an extension of the U-statistics and V-statistics results, we develop a non-classical delta method for handling high-dimensional data (\Cref{sec:delta:method}). This holds for estimators that can be approximated by an order-$m$ Taylor approximation. Additionally, it also provides distributional approximations for strictly monotonic transformations of these estimators.
  \item \textbf{Asymptotic normality and bootstrap inconsistency}. As our approximating distribution for $f_n(X)$ is a Gaussian polynomial, the fourth moment phenomenon of Nualart and Peccati \citep{nualart2005central} applies, which says that a necessary and sufficient condition for the asymptotic normality of $f_n(X)$ is a vanishing excess kurtosis. This is made precise in \Cref{sec:normality}. As an application, we obtain some simple statistics that demonstrate asymptotic normality but suffer from bootstrap inconsistency (\Cref{sec:bootstrap}). 
\end{itemize} 
We also discuss applicability of our results to non-multilinear polynomials, non-polynomial functions and dependent data in \Cref{appendix:additional:results}.

\subsection{Related work}
\label{sec:related} 

A substantial body of work exists on Gaussian universality and its applications. This section briefly surveys results directly relevant to ours.

\vspace{.5em}

\emph{Lindeberg's principle. } Gaussian universality results are typically proved using Lindeberg's principle. This method was generalised beyond the central limit theorem by Rotar \citep{rotar1979limit},  by Mossel, O’Donnell, and Oleszkiewicz \citep{mossel2005noise,mossel2010noise}, and by Chatterjee 
\citep{chatterjee2006generalization}.   
Compared to distribution-specific techniques, such as Fourier analysis or Stein's method, 
it has the advantage that 
it adapts easily to nonlinear estimators with unknown asymptotic distributions. On the other hand, it is known to result in sub-optimal error rates for a range of specific problems (see the discussion after Theorem 3.3 in \cite{chen2011normal} for empirical averages, and the remark in Section 4 of \citep{brailovskaya2022universality} for random matrices). This sub-optimality is due to too much smoothing: Lindeberg's principle approximates the Kolmogorov metric by an integral probability metric, using test functions with a second derivative that is ${(\nu-2)}$-H\"older for some ${\nu > 2}$. Our lower bound (\cref{thm:lower:bound}) shows this sub-optimality is a necessary price of generality: For the  general class of $n$-dependent probability measures, the bounds obtained by Lindeberg's principle are, in fact, near-optimal. 

\vspace{.5em}

\emph{Invariance principle for general multilinear forms. } 
The use of Lindeberg's principle for multilinear forms goes back to Rotar \citep{rotar1979limit}, who
obtained asymptotic approximations by polynomials of Gaussian for fixed $m$ and $d$. By additionally employing an inequality of Carbery and Wright \citep{carbery2001distributional}, Mossel et al. \citep{mossel2010noise} obtain a finite-sample Kolmogorov-distance bound for univariate data distributions. Recent work also considers convergence in total variation, in both the asymptotic \citep{nourdin2015invariance} and the finite-sample  regime \citep{bally2019total,  kosov2024improved}. 
Our work focuses on the regime with large $m$ and $d$, and our lower bound in Kolmogorov distance directly implies a lower bound in total variation.
Another line of work adapts the multivariate Stein's method to general functions of asymptotically normal random vectors, where the error bounds explicitly sum over the $d$ different coordinates of the random vectors \cite{gaunt2020stein, gaunt2023improved}. These results either explicitly assume $m$ and $d$ to be fixed, or impose implicit requirements on $d$ that may be too strong for applications. Indeed, the main technical work behind our upper bound (\cref{thm:main}) and its applications is to obtain tight dependence on $d$ and $m$.

\vspace{.5em}

\emph{Error bounds for Gaussian universality. } Our lower bound (\cref{thm:lower:bound}) helps to explain similar rates that have been observed in the upper bound for a number of problems. For the problem of non-Gaussian approximation of a degree-two U-statistic with kernel $u$, the known rates depend on the number of non-zero eigenvalues of the Hilbert-Schmidt operator associated with $u$: $O(n^{-1})$ for five non-zero eigenvalues \cite{gotze2014explicit},  $O(n^{-1/12})$ for one non-zero eigenvalue \cite{yanushkevichiene2012bounds}, and $O(n^{-1/14})$ without eigenvalue assumptions \cite{huang2023high}. It is also conjectured in \cite{yanushkevichiene2012bounds} that $O(n^{-1/12})$ is optimal, based on a result by Senatov \cite{senatov1998normal} that the Gaussian approximation of a heavy-tailed multivariate average in a carefully chosen sequence of Euclidean balls is $\Omega(n^{-1/12})$. The proof of our lower bound relies on adapting an asymmetry argument by \cite{senatov1998normal} (see \Cref{sec:construct:lower:bound}
% \cref{sec:construct:lower:bound} 
for a sketch of intuition), and implies that for every even degree $m \in \N$ and every $\gamma \in (0, 1]$, there exists a degree-$m$ U-statistic and a degree-$m$ V-statistic of random vectors, whose Gaussian polynomial approximation errors are between $n^{-\gamma / 6m}$ and $ m n^{-\gamma / (6m+2)}$. This proves the conjecture by \cite{yanushkevichiene2012bounds}  in the case $m=2$ and $\gamma=1$.

% \cref{thm:lower:bound} (with $\nu=3$) does not directly address the conjecture by \cite{yanushkevichiene2012bounds}: We provide a rate between $n^{-1/6m}$ and $m n^{-1/(6m+2)}$, but for a mixture of a heavy-tailed average and a degree-$m$ V-statistic of Gaussians. We do not know whether a similar lower bound holds for degree-$m$ U-statistics. 

\vspace{.5em}

\emph{Distribution approximations for high-dimensional statistics. } 
The approximation of specific estimators in high-dimensional regimes is the subject of an extensive literature. Such regimes differ by 
the parameters that grow with $n$. Examples include degree-two U-statistics with $d=d(n)$ \cite{chen2010two, wang2015high, yan2021kernel, gao2023two, huang2023high}, quadratic forms with $n$-dependent weights 
\cite{bhattacharya2022asymptotic}, infinite-order U-statistics with $m=m(n)$ \citep{van1988elementary,song2019approximating}, and statistics with high-dimensional outputs \citep{chernozhukov2013gaussian, chen2018gaussian,song2019approximating}. Our main results assume univariate outputs, but allow general 
dependence of $d$, $m$ and weights of the polynomial on $n$. One key application is therefore degree-$m$ U-statistics of high-dimensional inputs and univariate outputs, for which  
% The extension to functions that can be approximated by polynomials in an $L_2$ sense (\cref{thm:VD}) plays a key role in our applications.
additional references 
% on U-statistics, 
% and subgraph count statistics, 
% and more detailed comparisons to our results, see 
are included in \cref{sec:degree:m:U:stats}. Meanwhile, our results do apply to certain cases with high-dimensional outputs, such as the Gaussian approximation of a high-dimensional empirical average (\Cref{sec:high:d:averages}). The key argument is to rewrite the approximation of a multivariate statistic in a given metric as the approximation of a univariate statistic.

\vspace{.5em}

\emph{Terminology: Gaussian universality.} We use the term Gaussian universality to refer to the approximation of a function of $(X_i)_{i \leq n}$ by the same function of Gaussian vectors $(Z_i)_{i \leq n}$; this
matches the nomenclature of most of the literature surveyed above.
Some texts---for example, in the context of Gaussian approximation of Wiener chaos---use the term instead to indicate that
the overall function is asymptotically normal, see e.g.~Chapter 11 of \cite{nourdin2012normal}. Since our approximation $f_n(Z_1, \ldots, Z_n)$ still involves $n$-dependent quantities, it may have Gaussian or non-Gaussian limits (see \Cref{sec:normality}).
%  (\cref{prop:gaussian}). 
 We also note that our results take the form of a penultimate approximation. Other examples of penultimate approximations include chi-squared and shifted-gamma distributions as approximations for the Gaussian CLT limit \cite{hall1983chi,boutsikas2015penultimate} and the Type-III extreme value distribution for extreme value theorem \cite{cohen1982penultimate}. 
 While those penultimate approximations typically yield sharper approximation rates than the final limits (Gaussian or extreme value distributions), our approximations do not.

\section{Main results} \label{sec:main:results}

Throughout, we suppress dependence on $n$, and use the abbreviations
\begin{equation*}
  d\;=\;d(n)
  \qquad\quad
  m\;=\;m(n)
  % \qquad
  % X_{i}\;=\;X_{ni}
  \qquad\quad
  X\;\coloneqq\;(X_1,\ldots,X_n)\,.
\end{equation*}
The variables ${X_1,\ldots,X_n}$ are always independent (but not necessarily identically distributed) random elements of $\mathbb{R}^d$,
and $p_m$ is a polynomial ${\mathbb{R}^{nd}\rightarrow\mathbb{R}}$ of degree $\leq m$. By an absolute constant, we mean a constant that does
not depend on $n$, $d$, 
% or 
$m$, 
% on 
$p_m$, or 
% on 
the law of $X$. We also denote the collection of Gaussian vectors
\begin{equation*}
  Z\;\coloneqq\;(Z_1,\ldots,Z_n)\qquad\text{ where }\quad Z_1\condind\ldots\condind Z_n\quad\text{ and }\quad Z_i\sim\mathcal{N}(\mean[X_i],\Var[X_i])\;.
\end{equation*}
We occasionally use the tensor notation, which is defined in \Cref{sec:proof:main}. We also use $\| \argdot \|_{l_p}$ to denote the vector $l_p$ norm and use ${\| \argdot \|_{L_\nu} = ( \mean | \argdot |^\nu)^{1/\nu}}$ for the $L_\nu$ norm.
Moreover, $\overset{d}{\rightarrow}$ denotes convergence in distribution and $\overset{\P}{\rightarrow}$ denotes convergence in probability.

\subsection{Universality of multilinear polynomials}  \label{sec:multilinear}

We first consider a multilinear polynomial function $q_m: \R^{nd} \rightarrow \R$ with degree $\leq m$. By multilinearity, we mean that
\begin{align*}
    \partial_i \, q_m(x_1, \ldots, x_n) \text{ does not depend on } x_i 
    \quad 
    \text{ for all } 1 \leq i \leq n 
    \text{ \;\; and \;\; for }
    x_1, \ldots, x_n \in \R^d
    \;.
\end{align*}
Here and throughout, $\partial_i$ is the partial differential operator with respect to $x_i$ or the $i$-th $\R^d$ vector input.
 When viewed as a function of $nd$ univariate inputs, this is exactly the class of functions considered by the universality results of Mossel et al.~\citep{mossel2005noise,mossel2010noise}. A crucial difference is that while they consider i.i.d.~univariate inputs, we study independent high-dimensional inputs, which allows for arbitrary dependence across the coordinates of each $\R^d$ input vector. This dependence plays a crucial role in our subsequent analyses.

Our first result concerns the approximation of $q_m(X)$ by $q_m(Z)$ in Kolmogorov distance, and can be viewed as a high-dimensional generalisation of \citep{mossel2010noise}. Similar to \citep{mossel2010noise}, the bound is given in terms of the moment terms
\begin{align*}
  \sigma \;\coloneqq&\; \sqrt{\Var\,q_m(X)}\;
  \;&\text{ and }&&\;
  M_{\nu;i} \;\coloneqq&\;  \| \partial_i q_m(W_i)^\top ( X_i - \mean[X_i]) \|_{L_\nu}\;,
  \tagaligneq \label{eq:defn:moments} 
\end{align*}
where $W_i \coloneqq (X_1, \ldots, X_{i-1}, 0, Z_{i+1}, \ldots, Z_n) \in \R^{nd}$. $M_{\nu;i}$ measures the influence of the $i$-th centred data vector on $f$ in the $L_\nu$ norm, which extends the scalar version used in \citep{mossel2010noise}.

\begin{theorem}[Universality for multilinear polynomials]
  \label{thm:main}
  Fix $\nu \in (2,3]$. Then there exists some absolute constant $C > 0$ such that 
\begin{align*}
    \Big|
        \P\Big(  \mfrac{q_m(X) - \mean[q_m(X)]}{\sigma} \leq t\Big) 
        - 
        \P\Big(  \mfrac{q_m(Z) - \mean[q_m(X)]}{\sigma} \leq t\Big) 
    \Big|
    \;\leq\;
    C
    m
    \,
    \Big( \mfrac{ \sum_{i=1}^n M_{\nu;i}^\nu }
    { (1+t^2)^{\nu/2} \, \sigma^\nu
    }
    \Big)^{\frac{1}{\nu m +1}}
    \;. 
\end{align*}
for every $n, m, d \in \N$, every $t \in \R$, and every $\sigma > 0$.
Moreover, we have $ \mean[q_m(X)] = \mean[q_m(Z)]$ and $\Var[q_m(X)]= \Var[q_m(Z)]$. 
\end{theorem}

\Cref{thm:main} is proved by adapting the Lindeberg's method in \citep{mossel2010noise} for multivariate vectors and applying the Carbery-Wright anti-concentration inequality of \citep{carbery2001distributional}. The next example illustrates the upper bound in the simple case of an empirical average.

\begin{example}[Empirical average] \label{example:average} Let $m=d=1$, $q_m(x_1,\ldots, x_n) = \frac{1}{\sqrt{n}}\sum_{i \leq n} x_i$ and $X_1, \ldots, X_n$ be i.i.d.~univariate variables with mean $0$ and variance $1$. Each influence term can be computed as $M_{\nu,i} = \frac{1}{\sqrt{n}} \| X_i \|_{L_\nu}$, and \Cref{thm:main} becomes a Berry-Ess\'een type bound
\begin{align*}
    \Big| \P\Big( \mfrac{1}{\sqrt{n}} \msum_{i \leq n} X_i \leq t\Big) - \P\Big(  \mfrac{1}{\sqrt{n}} \msum_{i \leq n} Z_i \leq t \Big) \Big|
    \;\leq\;
    C \,
    \Big( \mfrac{  \mean | X_1 |^\nu }{ n^{(\nu - 2) / 2}}
    \Big)^{\frac{1}{\nu +1}}
    \;. 
\end{align*}
Assuming $\nu=3$, i.e.~the existence of a third moment, the error becomes $( \mean | X_i |^3 / \sqrt{n} \, )^{1/4}$, which yields a sub-optimal rate. As discussed in \Cref{sec:related}, this sub-optimality is a known issue with Lindeberg's method when applied to empirical averages. We will show in \Cref{sec:lower:bound} that for general polynomials of high-dimensional vectors, our \Cref{thm:main} and therefore Lindeberg's method can in fact yield nearly optimal rates. On the other hand, \Cref{thm:main} applies also to a weighted average of non-i.i.d.~data. In that case, requiring the upper bound to vanish is equivalent to a sufficient condition for Lindeberg's condition for CLT.
\end{example}

For degree-$m$ multilinear polynomials, we make the following observations:
\begin{proplist}
    \item {\em A heuristic for the rate. } To illustrate the upper bound further, it is instructive to have a heuristic for the error rate. Following \Cref{example:average}, a useful heuristic is to consider polynomials for which $\sigma$ is normalised to be $\Theta(1)$, assume third moment existence ($\nu=3$), and suppose a ``typical'' influence term to be on the order $M_{3,i} = \Theta(n^{-1/2})$. In this case, the upper bound of \Cref{thm:main} is on the order
    \begin{align*}
         O\big( \, m \, \big( n^{-1/2} \big)^{1/(3m+1)} \, \big) \;=\; O\big( m n^{ - 1/(6m+2)} \big) \;.
         \tagaligneq \label{eq:main:heuristic}
    \end{align*}
    In \Cref{sec:degree:m:U:stats}, we show that under mild conditions, \eqref{eq:main:heuristic} holds for all degree-$m$ U-statistics, which covers all degree-$m$ symmetric multilinear polynomials. 
    \item {\em Degree dependence. } Under the heuristic \eqref{eq:main:heuristic}, the upper bound of \Cref{thm:main} is meaningful only when $m \log m = o( \log n)$. More generally, we require the condition $m \log m = o( \log n)$ as long as $\sum_{i \leq n} M^\nu_{\nu;i} / \sigma^\nu $ decays as $O(n^{-\alpha})$ for some $\alpha > 0$.
    \item {\em Dimension dependence. } The bound in \Cref{thm:main} depends on the dimension $d$ only implicitly via the moment terms defined in \eqref{eq:defn:moments}. For example, if we replace $X_i = v^\top X'_i$ in \Cref{example:average} where $v \in \R^d$ is fixed and $X'_i$'s are $\R^d$ random vectors, then $d$-dependence only shows up implicitly through moments of $v^\top X'_i$. This is key to obtaining dimension-agnostic results for U-statistics in \citep{kim2024dimension,huang2023high} and in our \Cref{sec:degree:m:U:stats}.
\end{proplist}

\noindent
\Cref{thm:main} also applies to a not-necessarily-multilinear polynomial function $p_m: \R^{nd} \rightarrow \R$, by reparameterising the estimator $p_m(X)$ as a multilinear polynomial function of some larger data vectors, which includes information about $X_i$ up to its $m$-th power; see \Cref{appendix:non:multilinear}. \Cref{appendix:extension} also discusses extensions to non-polynomial functions and dependent data.

\subsection{Approximately polynomial functions} \label{sec:approx:poly}

Most statistical estimators are only polynomial in some approximate sense, such as U-statistics, V-statistics and any estimator that is well-approximated by some Taylor series. The next result extends \Cref{thm:main} to these statistics. Below, $f: \R^{nd} \rightarrow \R$ is a generic function to be approximated by the multilinear polynomial function $q_m$, and we inherit the notations $\sigma = \sqrt{\Var \, q_m(X)}$ and $M_{\nu;i}$ from \Cref{thm:main}. 

\begin{theorem}[Approximate polynomials]
  \label{thm:VD}
  Fix $\nu \in (2,3]$. There is an absolute constant $C > 0$ such that, for every $n, m, d \in \N$ and $\sigma > 0$, and for every measurable function $f: \R^{nd} \rightarrow \R$,
    \begin{align*}
        \msup_{t \in \R}
        \big| \P\big( \sigma^{-1} f(X) \leq t\big) - \P\big(  &  \sigma^{-1} q_m(Z) \leq t\big) \big|
        \\
        \;\leq&\;
        C 
        m 
        \Big( 
        \Big(  \mfrac{\| f(X) - q_m(X) \|_{L_2}}{\sigma} \Big)^{\frac{2}{2m+1}}
        +
        \Big( \mfrac{ \sum_{i=1}^n M_{\nu;i}^\nu }
        { \sigma^\nu
        }
        \Big)^{\frac{1}{\nu m +1}}
        \Big)\;.
    \end{align*}
\end{theorem}

\Cref{thm:VD} shows that to obtain a degree-$m$ Gaussian polynomial approximation, it suffices for $f(X)$ to be well-approximated by $q_m(X)$ in the $L_2$ norm. This can be useful since $f$ needs not be a differentiable function. Moreover, provided that $\mean[f(X)]=0$, the additional approximation error takes the form of a variance ratio 
\begin{align*}
    \big( \Var[f(X) - q_m(X)] \big) \,\big/ \, \Var[q_m(X)] \;.
    \tagaligneq \label{eq:VD:simple}
\end{align*}
In other words, $f(X)$ can be approximated by a Gaussian polynomial if we can extract a polynomial component of $f(X)$ that ``dominates in variance''. We refer to this property as \emph{variance domination}. Note that \eqref{eq:VD:simple} is with respect to the original data $X$ and not $Z$.

\vspace{.5em}

The next example illustrates how variance domination already plays a role in understanding the classical and high-dimensional limits of a simple U-statistic. We only sketch the arguments here, and defer a formal treatment of general degree-$m$ U-statistics to \Cref{sec:degree:m:U:stats}.

\begin{example}[A simple U-statistic] \label{example:simple:U} Consider $f(X) = u_2(X) = \frac{1}{n(n-1)} \sum_{i \neq j} X_i^\top X_j$ and let $X_i$'s be i.i.d.~$\R^d$-valued random vectors. This is exactly a multilinear polynomial, and \Cref{thm:main} allows us to take $X_i$ as Gaussians under mild moment conditions. Now let $\mu \coloneqq \mean[X_1]$, $\Sigma \coloneqq \Var[X_1]$, $\bar X_i \coloneqq X_i - \mu$ and consider the decomposition 
\begin{align*}
    u_2(X)
    \;=&\; \mfrac{1}{n(n-1)} \msum_{i \neq j} (\bar X_i + \mu)^\top (\bar X_j + \mu) 
    \\
    \;=&\;
    \mu^\top \mu  
    \,+\, 
    \mfrac{2}{n} \msum_{i=1}^n \mu^\top \bar X_i \,+\, \mfrac{1}{n(n-1)} \msum_{i \neq j} \bar X_i^\top \bar X_j
    \;\eqqcolon\; \mu^\top \mu + q_1(X) + q_2(X)
    \;.
\end{align*}
$\mu^\top \mu$ is the population quantity that $u_2(X)$ estimates, whereas $q_1(X)$ and $q_2(X)$ are two stochastic polynomial components of $u_2(X)$. Their variances can be computed as 
\begin{align*}
    \Var[q_1(X)] \;=&\; O\big( n^{-1} \mu^\top \Sigma \mu \big)
    &\text{ and }&&
    \Var[q_2(X)] \;=&\; O\big( n^{-2} \Tr( \Sigma^2 ) \big) \;.
\end{align*}
For a fixed $d$, the classical asymptotic theory for U-statistics says that a non-degenerate U-statistic admits a Gaussian limit, whereas a degenerate one admits a sum-of-centred-and-weighted chi-squares limit \citep{serfling1980approximation}. In the simple case of $u_2(X)$ above, non-degeneracy reduces to $\mu \neq 0$, and the two possible limits of $u_2(X)$ are described exactly by variance domination of \Cref{thm:VD}: When $\mu \neq 0$, $\Var[q_1(X)] \gg \Var[q_2(X)] = \Var[u_2(X) - q_1(X)]$, and the approximation by the degree-one polynomial $q_1(Z)$ prevails; when $\mu = 0$, $\Var[q_1(X)] = 0 \ll \Var[q_2(X)]$, and the approximation by the degree-two polynomial $q_2(Z)$ prevails.
\end{example}
When $d$ is allowed to grow in $n$, 
% a previous work by the first author 
a line of work \citep{bhattacharya2022asymptotic,huang2023high} has shown theoretically and empirically that $\mu=0$ is no longer the correct condition for degeneracy. Indeed if $X_i \sim \cN(\mu, I_d)$ with $\| \mu \|_{l_2}=1$, the variances compute as 
\begin{align*}
    \Var[q_1(X)] \;=&\; O(n^{-1})
    &\text{ and }&&
    \Var[q_2(X)] \;=&\; O(d n^{-2})\;.
    \tagaligneq \label{eq:simple:U:var:high:d}
\end{align*}
Even when $\mu \neq 0$, the variance of $q_2(X)$ may still be non-negligible if $d = \Omega(n)$. A practical consequence is that non-degenerate U-statistics can exhibit degeneracy in high dimensions, and vice versa. We call this a \emph{phase transition} effect, since for general U-statistics, the relative scaling of the two variances is affected by how hyperparameters of the U-statistic are chosen to scale with $d$, as demonstrated by \citep{huang2023high} for some one-sample and two-sample tests. However, those works are restricted to degree-two U-statistics, whereas our \Cref{thm:VD} can be applied to show that this is a more general phenomenon for approximately polynomial estimators.

The validity of applying variance domination to compare \eqref{eq:simple:U:var:high:d} hinges on the requirement that the bounds in \Cref{thm:VD} are dimension-agnostic, which allows for studying regimes where $d \ll n$, $d \asymp n$ and $d \gg n$. Indeed, the dimension-agnostic result of \citep{huang2023high} is a special case of our \Cref{thm:VD}. In \Cref{sec:degree:m:U:stats}, we show that the dimension-agnostic aspect of our bounds extends to degree-$m$ U-statistics whose degree does not grow too fast in $n$ and under a mild moment condition. On the other hand, dimension-dependence does arise when the moment condition is violated. This is discussed next in  \Cref{sec:non:multilinear}, which also gives another example of distributional simplification by variance domination.

% \begin{remark}[Variance domination]
%   \cref{thm:VD} reduces the three examples above to a single property that we refer to as \emph{variance domination}:
% Provided that $\mean[f(X)]= 0$, the behaviour of $f(X)$ matches that of a polynomial 
% up to an error that depends only on the variance ratio $\Var[f(X) - q_m(\bX)] / \Var[q_m(\bX)]$. In other words,
% $f(X)$ can be approximated by a polynomial if we can extract a polynomial component of $f(X)$ that ``dominates in variance''.
% For the subgraph densities in \cref{sec:subgraph:u}, for example, 
% variance domination corresponds exactly to known geometric conditions that determine which fluctuation is dominant.
% \end{remark}

\subsection{Universality of symmetric non-multilinear polynomials
% without reparameterisation, 
and a discussion on dimension dependence
% Universality of symmetric non-multilinear polynomials  and $d$-dependence
} \label{sec:non:multilinear}

In this section, we consider universality results for degree-$m$ V-statistics with polynomial summands: For $x_1, \ldots, x_n \in \R^d$, we define 
\begin{align*}
    v_m(x_1, \ldots, x_n) 
    \;\coloneqq\; 
    \mfrac{1}{n^m} \msum_{i_1, \ldots, i_m \leq n} \langle S ,\, x_{i_1} \otimes \ldots \otimes x_{i_m} \rangle\;.
    \tagaligneq \label{eq:simple:v:defn}
\end{align*}
where $S \in \R^{d^m}$ is a deterministic tensor, $\otimes$ indicates the tensor product (see \Cref{sec:proof:main} for definition), and $\langle \argdot, \argdot \rangle$ denotes the Euclidean inner product in $\R^{d^m}$. $v_m$ is a prototypical example of a symmetric non-multilinear polynomial. We focus on $v_m$ for simplicity, although we expect the same argument to apply to more general classes of nonlinear polynomials.
% The motivation for studying the universality of $v_m$ is as follows. As discussed in \Cref{sec:multilinear} (formalised in \Cref{appendix:non:multilinear}), to obtain a distributional approximation for $p_m(X)$ where $p_m$ is non-multilinear, we may reparameterise $p_m(X)$ into a multilinear polynomial of augmented vectors $\bX_i = \tvec( (X_i, \ldots, X_i^{\otimes m}))$, and establish universality for the augmented vectors. In this section, we investigate universality with respect to the original data vectors $X_i$'s instead. One motivation is that universality results with respect to $\bX_i$'s retain the influence of $2m$ moments of $X_i$, which can be undesirable in practice. The other motivation is that, in the case of $p_m(X)=g(n^{-1/2} \sum_{i=1}^n X_i)$ for some polynomial function $g: \R^d \rightarrow \R$, the continuous mapping theorem suggests the universality approximation by $p_m(Z)$ is valid when $g$ and $d$ are both fixed in $n$. This raises the question whether and when this remains valid in the high-dimensional regime. 
We impose the following simplifying assumptions:

\begin{assumption}[Simplification of the V-statistic] \label{assumption:V:stat} (i) $S$ is a symmetric tensor in the sense that $S_{l_1 \ldots l_m} = S_{l_{\pi(1)} \ldots l_{\pi(m)}}$ for all $l_1, \ldots, l_m \leq d$ and for all permutations $\pi$ acting on $\{1, \ldots, m\}$.  (ii) $X_1, \ldots, X_n$ are i.i.d.~zero-mean random vectors.
\end{assumption}

\begin{assumption}[Uniformly sub-Weibull entries] \label{assumption:sub:Weibull} There exist some absolute constants $\theta > 0$ and $K_1 > 0$ such that, for all $l \leq d$ and $t \geq 0$,
    \begin{align*}
        \P( | X_{1l} - \mean[X_{1l}] | \geq t ) \;\leq\; 2 \exp( - (t/K_1)^{1/\theta})
        \;.
    \end{align*}
\end{assumption}

\begin{remark} \Cref{assumption:sub:Weibull} gives a sub-Gaussian tail when $\theta=\frac{1}{2}$ and a sub-exponential tail when $\theta = 1$, but also accommodates slower tails. It is used such that $\mean[X_{1l}^m]$ grows no faster than $O(m^{\theta m})$ \cite{vladimirova2020sub} and therefore $o(n^{-\alpha})$ for any $\alpha > 0$ under the condition $m \log m = o(\log n)$. 
    % , and is therefore negligible. 
    \Cref{assumption:sub:Weibull} can be relaxed at the cost of more tedious notation (\Cref{sec:proof:v:additional}). 
\vspace{1em}

The next result gives the universality approximation bound for $v_m(X)$. 
\end{remark}

\begin{proposition} \label{prop:simpler:V} Let $2m^2 < n$. Under Assumptions \ref{assumption:V:stat} and \ref{assumption:sub:Weibull} and the additional assumptions that $\| S \|_{l_1} = \Theta(1)$ and $m \log m = o(n)$, there is some absolute constant $C > 0$ such that 
    \begin{align*}
        &\;
        \sup_{t \in \R} \big| \P( v_m(X) \leq t) - \P( v_m(Z) \leq t) \big|
        \\
        &\,\;\leq\;\,
        C m 
            n^{-\frac{\nu-2}{2\nu m + 2}}
            \,
            \bigg( 
                \mfrac{
                    \big\| \langle S, X_1 \otimes \ldots \otimes X_m \rangle \big\|_{L_\nu}^2 
                    + O(m^{m(4+2\theta)} n^{-1} )
                }
                {
                        \| \langle S, X_1 \otimes \ldots \otimes X_m \rangle \|_{L_2}^2 
                        + O(m^{m(4+2\theta)} n^{-1} )
                }
            \bigg)^{\frac{\nu}{2 \nu m + 2}}
            \tagaligneq \label{eq:simpler:V:intermediate:one}
        \\
        &\hspace{2em} 
                +
                C m^{3 + \theta} \,  n^{ - \frac{1}{2 m + 1 } } \,
                \Big( 
                \mfrac{ 
                    \| S \|_{l_1}^2
                }{
                        \| \langle S, X_1 \otimes \ldots \otimes X_m \rangle \|_{L_2}^2 
                        + O(m^{m(4+2\theta)} n^{-1} )
                } 
                \Big)^{\frac{1}{2m+1}}
            % \;.
            \tagaligneq \label{eq:simpler:V:intermediate:two}
        \\
        & \,\;\leq\; \, 
        2 C m^{3 + \theta} n^{- \frac{\nu-2}{2\nu m + 2}}
        \,
        \bigg( 
                \mfrac{
                    \| S \|_{l_1}^2 
                    +
                    O(n^{-1})
                }
                {
                    \| \langle S, X_1 \otimes \ldots \otimes X_m \rangle \|_{L_2}^2 
                        + O(m^{m(4+2\theta)} n^{-1} )
                }
            \bigg)^{\frac{\nu}{2 \nu m + 2}}
        \;,
    \end{align*}
    where the bounding constants in $\Theta(\argdot)$, $O(\argdot)$ and $o(\argdot)$ are independent of $n$, $m$ and $d$.
\end{proposition}

\Cref{prop:simpler:V} provides two bounds. The second is stated for simplicity, whereas the first demonstrates how the errors arise from applying \Cref{thm:main} and \Cref{thm:VD}: \vspace{-.5em}
\begin{itemize}
    \item \eqref{eq:simpler:V:intermediate:one} is the Gaussian polynomial approximation term from applying \Cref{thm:main} to the reparameterised polynomial of augmented vectors (\Cref{appendix:non:multilinear}), which replaces each augmented vector $\bX_i=\tvec( (X_i - \mean[X_i], \ldots, X_i^{\otimes m}-\mean[X_i^{\otimes m}]))$ by a Gaussian $\xi_i$. For a multilinear polynomial, e.g.~a U-statistic, which does not involve $X_i^{\otimes 2}, \ldots, X_i^{\otimes m}$, the replacement by $\xi_i$ suffices for establishing universality with respect to $X_i$, and \eqref{eq:simpler:V:intermediate:one} is the only approximation error required. Indeed, if the 2nd and 3rd absolute moments of $\langle S, X_1 \otimes \ldots \otimes X_m \rangle$ are of the same order, the error in \eqref{eq:simpler:V:intermediate:one} scales as $O(m n^{-\frac{1}{6m+2}})$, which matches the heuristic \eqref{eq:main:heuristic}.
    \vspace{.5em}
    \item \eqref{eq:simpler:V:intermediate:two} is the error term from applying \Cref{thm:VD} to approximate $\xi_i$'s by $\tvec((Z_i, \ldots, Z_i^{\otimes m}))$. This requires replacing e.g.~$\cN(\mean[X_i^{\otimes m}], \Var[X_i^{\otimes m}])$ by $\cN(\mean[X_i], \Var[X_i])^{\otimes m}$, and in the proof, we have invoked \Cref{assumption:sub:Weibull} to ensure the error incurred is small. Under the condition $\| S \|_{l_1} = \Theta(1)$, an additional dimension dependence can arise if $\| \langle S, X_1 \otimes \ldots \otimes X_m \rangle \|_{L_2}^2 = o(1)$ as $d$ grows. We illustrate this in the next example. 
\end{itemize}
\vspace{-1em}

\begin{example}[A simple V-statistic] \label{example:simple:V} Consider the V-statistic analogue of the U-statistic $u_2(X)$ considered in \Cref{example:simple:U}, namely $v_2(X) = \frac{1}{n^2} \sum_{i,j \leq n} X_i^\top X_j$, and again let $X_i$'s be i.i.d.~$\R^d$-valued random vectors. 
    % Classically, $v_2(X)$ is known to have the same asymptotic distribution as $u_2(X)$ \citep{serfling1980approximation}. If this still holds in the high-dimensional regime, we can conclude that $v(X)$ exhibits the same phase transition effect as the one observed for U-statistic in \eqref{eq:simple:U:var:high:d}. However, a subtlety arises from that the additional error term 
We investigate how the additional error term \eqref{eq:simpler:V:intermediate:two} behaves as a function of dimension $d$. 
%  may no longer be dimension-agnostic.
 To see this, we first renormalise the V-statistic as 
\begin{align*} 
    \mfrac{1}{d} \, v_2(X)
    \;=\;
    \mfrac{1}{n^2}  \msum_{i,j \leq n} (X_i^\top X_j)/d 
    \;=\;
    \mfrac{1}{n^2}  \msum_{i,j \leq n}
    \, \langle S, X_i \otimes X_j \rangle\;,
\end{align*}
where $S = \frac{1}{d} I_d$ satisfies that $\| S \|_{l_1} = 1$. Take $\nu=3$, $\mean[X_1] = 0$ and $\Var[X_1] = \Sigma$ for simplicity, in which case we are focusing on a degenerate V-statistic. The error term \eqref{eq:simpler:V:intermediate:one} is the same as that for the U-statistic in \Cref{example:simple:U}, and is dimension-agnostic if $\| X_1^\top X_2 \|_{L_3} / \| X_1^\top X_2 \|_{L_2} = \Theta(1)$. Meanwhile, since $m=2$, \eqref{eq:simpler:V:intermediate:two} reduces to
\begin{align*}
    \Theta\Big( n^{-\frac{1}{5}} \Big( \mfrac{1}{ \| \frac{1}{d} X_1^\top X_2 \|_{L_2} + o(1) } \Big)^{\frac{1}{5}} \Big)
    \;=&\;
    \Theta \Big( n^{-\frac{1}{5}} \Big( \mfrac{1}{d^{-2}  \, \Tr( \Sigma^2 )  + o(1) } \Big)^{\frac{1}{5}}  \Big)
    \;.
    \tagaligneq \label{eq:V:L2:calc}
\end{align*}
Assume the first coordinate of $X_1$ satisfies $\Var[X_{11}] = \Theta(1)$ for simplicity. The dimension dependence of \eqref{eq:simpler:V:intermediate:two} is now determined by the correlations of the different coordinates of $X_1$. To illustrate this, consider the following cases:
\begin{proplist}
    \item \textbf{Identical coordinates $\Rightarrow$ dimension-agnostic bounds.} When $X_1$ has identical coordinates, $\Tr(\Sigma^2) = \Theta(d^2)$, \eqref{eq:simpler:V:intermediate:two} is dimension-agnostic, and by the second bound of \Cref{prop:simpler:V}, the universality approximation holds regardless of how fast $d$ grows.
    \item \textbf{Shared signal variable across all coordinates $\Rightarrow$ dimension-agnostic bounds. } Suppose each coordinate of $X_1$ is generated by the model $X_{ij}= \bar X_i + \epsilon_{ij}$, where $\bar X_i$'s are i.i.d.~real-valued random variables with $\Var[\bar X_i] = \Theta(1)$ and $\epsilon_{ij}$'s are some noise variables independent of $\bar X_i$'s. In this case, we have  $\Tr(\Sigma^2) = \Omega(d^2)$ and again obtain the same conclusion. That is to say, if there is some common signal variable that persists across all coordinates of the data, we expect the universality approximation to be dimension-agnostic.
    \item \textbf{Uncorrelated coordinates $\Rightarrow$ require $d=o(n)$.} Assume the coordinates of $X_1$ are identically distributed for simplicity. If the coordinates are uncorrelated, then $\Sigma$ is a diagonal matrix and $\Tr(\Sigma^2) = \Theta(d)$. This introduces a dimension dependence in \eqref{eq:simpler:V:intermediate:two}, and for  \eqref{eq:simpler:V:intermediate:two} to decay to zero, we require $d = o(n)$.
    \item  \textbf{Sparse coordinates $\Rightarrow$ possibly worse requirement on $d$.} Suppose that the $k$ coordinates of $X_1$ are i.i.d.~non-degenerate, where $k=\Theta(1)$, and the remaining coordinates of $X_1$ are vanishingly small such that $\Tr(\Sigma^2) = \Theta(1)$. Then $d=o(n^{1/2})$ is required for \eqref{eq:simpler:V:intermediate:two} to vanish. Note that in this case, one can find simple distributions under which $\| X_1^\top X_2 \|_{L_3} / \| X_1^\top X_3 \|_{L_2} = \Theta(1)$, and thus the error \eqref{eq:simpler:V:intermediate:one} remains dimension-agnostic. Note also that if the $d-k$ small coordinates of $X_1$ are exactly zero, we can explicitly incorporate this into the coefficient tensor $S$ in \Cref{prop:simpler:V} by setting appropriate entries to zero (while renormalising $S$ s.t. $\| S \|_{l_1} = \Theta(1)$), which allows \eqref{eq:simpler:V:intermediate:two} to remain dimension-agnostic. Therefore, a ``bad'' case for \eqref{eq:simpler:V:intermediate:two} is when the sparsity is unknown or approximate. \vspace{-1em}
    \item \textbf{Independent coordinates $\Rightarrow$ dimension-agnostic bounds. } 
    If the coordinates of $X_1$ are additionally also independent, we can instead view $v_2(X)$ as a polynomial of $nd$ i.i.d.~univariate random variables, and apply \Cref{thm:main} and \Cref{thm:VD} to these $nd$ variables. This allows us to obtain dimension-agnostic bounds (\Cref{appendix:v:iid:coord}).
\end{proplist}
\end{example}

\vspace{-.5em}

\Cref{example:simple:V} suggests that one ``bad'' case in which \eqref{eq:simpler:V:intermediate:two} becomes dimension-sensitive is when the coordinates of $X_1$ are uncorrelated but still dependent, and we do not have access to additional information about the dependence structure such as (i) or (ii) to obtain a better control. It turns out that this is also a case where the multilinear universality error \eqref{eq:simpler:V:intermediate:one} can become dimension-sensitive, as illustrated in the next example.

\begin{example}[\textbf{Uncorrelated but arbitrarily dependent coordinates $\Rightarrow$ dimension-sensitive bounds for both U-statistics and V-statistics}] \label{example:break:moment:cond} Consider the U-statistic $u_2(X)$ in \Cref{example:simple:U} and the V-statistic $v_2(X)$ in \Cref{example:simple:V}. Taking $\nu = 3$, the universality approximation error \eqref{eq:simpler:V:intermediate:one} of both statistics is on the order
\begin{align*}
    &\;
    n^{-1/14} \bigg( \mfrac{\| \frac{1}{d} X_1^\top X_2 \|_{L_3}^2 + O(n^{-1})}{\| \frac{1}{d} X_1^\top X_2 \|_{L_2}^2 + O(n^{-1})} \bigg)^{3/14}
    \;\leq\;
    n^{-\frac{1}{14}} \bigg( \mfrac{\| \frac{1}{d} X_1^\top X_2 \|_{L_4}^2 + O(n^{-1})}{\| \frac{1}{d} X_1^\top X_2 \|_{L_2}^2 + O(n^{-1})} \bigg)^{3/14}
    \\
    &\hspace{5em}\;=\;
    n^{-1/14}
    \bigg(
    \mfrac{
        \big( 
            \,
            d^{-4}
            \,
            \sum_{l_1, l_2, l_3, l_4 \leq d}
            \,
            \mean\big[  
                X_{1l_1} X_{1l_2} X_{1l_3}  X_{1l_4}
            \big]^2 \, 
        \big)^{1/2}
        +
        O(n^{-1}) 
    }
    { 
            d^{-2}
            \,
           \sum_{l_1, l_2 \leq d}
           \,
            \mean\big[  
                X_{1l_1} X_{1l_2} 
            \big]^2 \,
        +
        O(n^{-1}) 
    }
    \bigg)^{3/14}
    \;.
    \tagaligneq \label{eq:break:moment:cond}
\end{align*}
Under case (iii) in \Cref{example:simple:V}, where the coordinates of $X_1$ are identically distributed but uncorrelated, the above bound is on the order 
\begin{align*}
    \bigg( 
        n^{-\frac{1}{3}}  
        \,
        \times 
        \,
    \mfrac{
        \,
        \mean[X_{11} X_{12} X_{13} X_{14}]
        +
        O(d^{-1/2} + n^{-1})
        \,
    }
    { 
        d^{-1}
        \,
        \mean\big[  
            X_{11}^2
        \big]^2 \,
        +
        O( n^{-1} )
    }
    \bigg)^{3/14}
    \;.
\end{align*}
Therefore, if the dependence across coordinates is such that $\mean[X_{11} X_{12} X_{13} X_{14}] = \Theta(1)$, this introduces the requirement $d = o(n^{1/3})$ in order for our universality approximation to be valid. 
% For the V-statistic $v(X)$, a result of \cite{zhilova2020nonclassical} implies that this condition is tight, and we include a detailed discussion in \Cref{sec:high:d:averages}. 
Meanwhile, if the dependence is such that
\begin{align*}
    \mean[X_{11}X_{12}X_{13}X_{14}] 
    \;=\; 
    \mean[X_{11}X_{11}X_{12}X_{13}] 
    \;=\; 
    0\;,
    \tagaligneq \label{eq:unbreak:moment:cond}
\end{align*}
then the bound \eqref{eq:break:moment:cond} and therefore \eqref{eq:simpler:V:intermediate:one} remains dimension-free. Note however that in this case, the additional error \eqref{eq:simpler:V:intermediate:two} for V-statistics still requires $d=o(n)$.
\end{example}

\vspace{-1.5em}

\begin{remark}[Comparison to known works and optimality of dimension dependence] \label{remark:compare:high:d:CLT} Since $v_2(X)=\| \frac{1}{n} \sum_{i \leq n} X_i \|^2$, the universality result for $v_2(X)$ is exactly the Gaussian approximation of a high-dimensional average with respect to the centred Euclidean balls.
    \\
(i) Existing bounds on high-d CLT in Euclidean balls \citep{sazonov1972bound,bentkus2003dependence,zhilova2020nonclassical,fang2024large} typically require the covariance matrix $\Sigma$ of $X_1$ to be invertible and involve terms of the form $\mean | \Sigma^{-1/2} X_1 |^3$. While those results may accommodate the non-invertible $\Sigma$ case by a linear transformation and replacing $\Sigma^{-1}$ by the pseudo-inverse $\Sigma^\dagger$, $\mean | (\Sigma^\dagger)^{1/2} X_1 |^3$ can still be cumbersome to compute. In contrast, we do not require $\Sigma$ to be invertible and our bounds do not involve $\Sigma^{-1}$. This can have useful practical implications, and we defer a discussion to \Cref{sec:high:d:averages}.
\\
(ii) The $d=o(n^{1/3})$ condition in \Cref{example:break:moment:cond} is sub-optimal compared to the non-improvable $d=o(n)$ condition in \citep{fang2024large}, proved in the case of an invertible $\Sigma$. This is because we have used a crude bound to illustrate the role of coordinate dependence. If we take $\Sigma=I_d$ and assume that $X_1$ has uniformly bounded coordinates, then
\begin{align*}
    n^{-\frac{1}{3}}
    \mfrac{\| \frac{1}{d} X_1^\top X_2 \|_{L_3}^2}{\| \frac{1}{d} X_1^\top X_2 \|_{L_2}^2}
    =
    n^{-\frac{1}{3}}
    \mfrac{ \big( \mean \big[ \big| \frac{1}{d} X_1^\top X_2  \big| \times \big| \frac{1}{d} X_1^\top X_2  \big|^2  \big] \big)^{\frac{2}{3}} }{
        \mean \big| \frac{1}{d} X_1^\top X_2  \big|^2 
        }
    =
    O\Big( 
        \mfrac{n^{-\frac{1}{3}}}{
        \big( \mean \big| \frac{1}{d} X_1^\top X_2  \big|^2  \big)^{\frac{1}{3}}
        }    
    \Big)
    =
    O\Big( \mfrac{d^{1/3}}{n^{1/3}}\Big)
    \,.
\end{align*}
Substituting this into the bound in \Cref{example:break:moment:cond} gives $d=o(n)$ as a requirement for \eqref{eq:simpler:V:intermediate:one} to decay to zero. Meanwhile under case (iii) in \Cref{example:simple:V}, $d=o(n)$ also allows \eqref{eq:simpler:V:intermediate:two} to decay to zero. 
\end{remark}
\color{black}

\Cref{example:simple:V,example:break:moment:cond} show that for both V and U-statistics, the universality approximation error can be dimension-sensitive, if the different data coordinates possess arbitrary dependence that do not manifest as correlations. 
However, the 
requirements for V-statistics 
are more stringent than for U-statistics, and arise differently in the literature:
\begin{itemize}
    \item For the degenerate U-statistic $u_2(X)$, the universality approximation error only involves \eqref{eq:simpler:V:intermediate:one}, which does not have dimension dependence if the second and third moment terms, $\|\frac{1}{d} X_1^\top X_2\|_{L_2}$ and $\|\frac{1}{d} X_1^\top X_2\|_{L_3}$, are asymptotically on the same scale. As shown in \cite{huang2023high} and more generally in our \Cref{sec:degree:m:U:stats}, for general U-statistics, this is the requirement that the U-statistic kernel function $u: \R^d \times \R^d \rightarrow \R$ satisfies that  
    $\|u(X_1, X_2)\|_{L_2}$ and $\|u(X_1, X_2)\|_{L_3}$ are on the same order and that  $\| \mean[u(X_1, X_2) | X_1]\|_{L_2}$ and $\| \mean[ u(X_1, X_2) | X_1]\|_{L_3}$ are on the same order. 
    Since $\mean[u(X_1,X_2)]$ is the target of estimation, this is a natural assumption for many practical applications, and therefore Gaussian polynomial approximation is viewed as dimension-agnostic for U-statistics.
    \item For the V-statistic $v_2(X)$, due to the additional term \eqref{eq:simpler:V:intermediate:two} in the bound, the dimension dependence additionally arises directly from $\| \frac{1}{d} X_1^\top X_2 \|_{L_2}^{-1}$ and not just from the ratio $\| \frac{1}{d} X_1^\top X_2 \|_{L_3} / \| \frac{1}{d} X_1^\top X_2 \|_{L_2}$. A similar observation extends to degree-$m$ V-statistics. At a high level, this is because the additional error term \eqref{eq:simpler:V:intermediate:two} for V-statistics arises from approximating $\cN(\mean[X_i^{\otimes m}], \Var[X_i^{\otimes m}])$ by $Z_i^{\otimes m}$, which involves comparisons of  $\| \frac{1}{d} X_1^\top X_2 \|_{L_2}$ against higher-order moments. Therefore even if \eqref{eq:simpler:V:intermediate:one} stays bounded, e.g.~in the case when coordinates are uncorrelated and dependent such that \eqref{eq:unbreak:moment:cond} holds, \eqref{eq:simpler:V:intermediate:two} can still introduce dimension sensitivity. To relate this discussion to the works on Gaussian approximation in high-dimensional Euclidean balls \citep{sazonov1972bound,bentkus2003dependence,zhilova2020nonclassical,fang2024large}, the dimension sensitivity in their approximation arises because they consider e.g.~our $v_2(X)$ from \Cref{example:simple:V}.
    % , in the context of confidence sphere construction of high-dimensional averages.
    In that context, the different moments of $\frac{1}{d} X_1^\top X_2$ are indeed not guaranteed to be on the same order, which causes both \eqref{eq:simpler:V:intermediate:one} and \eqref{eq:simpler:V:intermediate:two} to become dimension-sensitive; see \Cref{sec:high:d:averages}.
\end{itemize}
In short, the seeming disagreement between the dimension-agnostic literature on U-statistics \citep{kim2024dimension,huang2023high,gao2025dimension} and the dimension-sensitive results for V-statistics \citep{zhilova2020nonclassical} can be attributed to two factors: (i) our additional universality approximation error term for non-multilinear statistics, and (ii) the different applications contexts that lead to different moment assumptions.

\subsection{Nearly matching lower bound and dependence on the degree $m$ } \label{sec:lower:bound} The universality approximation bounds for all toy statistics considered so far (\Cref{example:average,example:simple:U,example:simple:V}) are sub-optimal in $n$.  As discussed in \Cref{sec:related}, this sub-optimality is a known issue with Lindeberg's method for these statistics, and raises the question whether the error bound from \Cref{thm:main} can be tight. Our next result shows that there exists a degree-$m$ polynomial for which the upper bound from \Cref{thm:main} is essentially tight. We also state an approximation with respect to $\Xi=(\xi_i)_{i \leq n}$, the Gaussian surrogates for the augmented variables $\bX_i = \tvec( (X_i-\mean[X_i], \ldots, X_i^{\otimes m}-\mean[X_i^{\otimes m}]))$ (see \Cref{appendix:non:multilinear}  for a precise formulation).

\begin{theorem}[Lower bound]
  \label{thm:lower:bound} Fix $\nu \in (2,3]$, and assume that $m$ is even with $m = o(\log n)$.
    Then there exist a sequence of probability measures $\mu^{(n)}_{\nu}$\,, a sequence of polynomials ${p^*_m=p^*_{m(n)}}$, and
    absolute constants ${c, C > 0}$ and ${N \in \N}$, such that
    \begin{align*}
        c n^{-\frac{\nu-2}{2\nu m}} 
        \;\leq&\; 
        \msup_{t \in \R} \big| \P( p^*_m(X) \leq t) - \P( p^*_m(Z) \leq t) \big|
        \;\leq\;
        C m n^{-\frac{\nu-2}{2\nu m +2 }} \;,
    \end{align*}
    for i.i.d.\ variables ${X_1, \ldots, X_n\sim\mu^{(n)}_{\nu}}$ and all $n \geq N$. Moreover, there exists a sequence of multilinear polynomials $q^*_m$ such that
    \begin{align*}
        c n^{-\frac{\nu-2}{2\nu m}} 
        \;\leq&\; 
        \msup_{t \in \R} \big| \P\big( p^*_m(X) - \mean[p^*_m(X)] \leq t  \big) - \P\big( q^*_m(\Xi) \leq t \big) \big|
        \;\leq\;
        C m n^{-\frac{\nu-2}{2\nu m +2 }}\;.
    \end{align*}
\end{theorem}

\begin{proof}
  See %\cref{sec:proof:lower:bound}
  Appendix G. 
  The proof is inspired by an asymmetric construction by Senatov \cite{senatov1998normal} that controls the probability of a centred, multivariate average lying in a sequence of uncentred Euclidean balls. For our problem, a different asymmetrisation is required to obtain a comparable lower bound to the upper bound from \cref{thm:main}.
  % \cref{sec:construct:lower:bound} 
  \Cref{sec:construct:lower:bound} motivates the proof and constructs $(\mu^{(n)}_{\nu, \sigma_0})$, $(p_m^*)$ and $(q_m^*)$ explicitly. $q^*_m$ is essentially a reparameterisation of $p^*_m$ by viewing the augmented vectors $(\bX_i)_{i \leq n}$ as input variables.
\end{proof}

\Cref{thm:lower:bound} says that, if the only knowledge we have about our estimator is that it is a polynomial of degree ${\leq m}$, Lindeberg's principle yields a nearly tight rate in $n$. The upper bounds of \Cref{thm:lower:bound} are obtained directly from applying \Cref{thm:main} to the multilinear reparameterisation of $p_m$ (see Corollary A.1 in the appendix) and \Cref{thm:VD}, and under $\nu=3$, match the $O(mn^{-\frac{1}{6m+2}})$ heuristic rate \eqref{eq:main:heuristic} for \Cref{thm:main}. Meanwhile, the lower bounds confirm that the requirement $m=o(\log n)$ is necessary. 

To reconcile this result with the sub-optimality of \Cref{thm:main} for classical statistics (\Cref{example:average}), we note that \Cref{thm:lower:bound} only applies to the non-classical regime, where the probability measure of the data $\mu_\nu^{(n)}$ is allowed to depend on $n$. This includes the high-dimensional regime, where the data vectors are $\R^{d(n)}$-valued and therefore the data distribution depends on $n$ via $d(n)$. Moreover, the bounds become closer to each other as the degree $m$ grows.

\Cref{thm:lower:bound} also has concrete implications on the error of approximating general degree-$m$ U-statistics and degree-$m$ V-statistics by Gaussian polynomials. Denote
\begin{align*}
    \tilde u^*_m(Y) 
    \;\coloneqq&\; 
    \mfrac{1}{n(n-1) \ldots (n-m+1)} \msum_{i_1, \ldots, i_m \in [n] \text{ \rm distinct }} k^*_u(Y_{i_1}, \ldots, Y_{i_m}) 
    \;,
    \\
    \tilde v^*_m(Y) 
    \;\coloneqq&\; 
    \mfrac{1}{n^m} \msum_{i_1, \ldots, i_m \in [n]} k^*_v(Y_{i_1}, \ldots, Y_{i_m})
    \;,
    \qquad 
    Y \;\coloneqq\; (Y_1, \ldots, Y_n)\;.
\end{align*}
In \Cref{sec:construct:lower:bound}, we demonstrate how the constructions $p^*_m(X)$ and $q^*_m(X)$ used in \Cref{thm:lower:bound} can be reformulated respectively as a V-statistic $v^*_m(Y)$ and a U-statistic $u^*_m(Y)$. As a consequence, \Cref{thm:lower:bound} immediately implies the following:

\begin{corollary} \label{cor:matching:bound:V} Assume the setup of \Cref{thm:lower:bound} and additionally $m < d$. There exist a sequence of probability measures $\mu^{(n)}_{\nu}$, a symmetric function $k^*_v: (\R^d)^m \rightarrow \R$ and absolute constants ${c, C > 0}$ and ${N \in \N}$, such that
\begin{align*}
    c n^{-\frac{\nu-2}{2\nu m}} 
    \;\leq&\; 
    \msup_{t \in \R} \big| \P( \tilde v^*_m(Y) \leq t) - \P( p^*_m(Z) \leq t) \big|
    \;\leq\;
    C m n^{-\frac{\nu-2}{2\nu m + 2}} 
\end{align*}
for i.i.d.\ variables ${Y_1, \ldots, Y_n\sim\mu^{(n)}_{\nu}}$ and all $n \geq N$. Moreover, there exists another sequence of probability measures $\tilde \mu^{(n)}_{\nu}$ and a symmetric function $k^*_u: (\R^d)^m \rightarrow \R$ such that 
\begin{align*}
    c n^{-\frac{\nu-2}{2\nu m}} 
    \;\leq&\; 
    \msup_{t \in \R} \big| \P( \tilde u^*_m(Y) \leq t) - \P( q^*_m(\Xi) \leq t) \big|
    \;\leq\;
    C m n^{-\frac{\nu-2}{2\nu m + 2}} 
\end{align*}
for i.i.d.\ variables ${Y_1, \ldots, Y_n\sim \tilde \mu^{(n)}_{\nu}}$ and all $n \geq N$. 
\end{corollary}

\Cref{cor:matching:bound:V} says that the Gaussian polynomial approximations are valid for both $\tilde v^*_m(Y)$ and $\tilde u^*_m(Y)$, but the approximation error decays at a slow rate. For $\nu=3$, the lower bound for degree-two U-statistics is on the order $n^{-1/12}$, which closes an open conjecture by \cite{yanushkevichiene2012bounds}. More generally, \Cref{cor:matching:bound:V} here says that for every even degree $m \in \N$ and every $\gamma \in (0, 1]$, there exists a degree-$m$ U-statistic and a degree-$m$ V-statistic of random vectors, whose Gaussian polynomial approximation errors are between $n^{-\gamma / 6m}$ and $ m n^{-\gamma / (6m+2)}$. 

\section{Applications} \label{sec:applications}

This section provides statistical applications of our theorems, some of which are new and some of which generalise existing results. \Cref{sec:normality} presents a necessary and sufficient condition for asymptotic normality of an approximately polynomial estimator $f(X)$ in the high dimensional regime. Using the insights gained, \Cref{sec:bootstrap} constructs a U-statistic that is asymptotic normal but for which bootstrap fails to be consistent. \Cref{sec:degree:m:U:stats} extends the phase transition in \Cref{example:simple:U} and provides distributional approximations of degree-$m$ complete U-statistics of high-dimensional data. In a similar vein, \Cref{sec:MMD:imbalance} establishes the phase transition of the Maximum Mean Discrepancy (MMD) test statistic for a high-dimensional two-sample test with imbalanced sample sizes. \Cref{sec:high:d:averages} gives a non-classical Berry-Ess\'een bound used for confidence sphere constructions for high-dimensional averages, where the data may have non-invertible covariance matrices. \Cref{sec:delta:method} discusses how a non-classical delta method can be obtained from our results (formal statements included in \Cref{appendix:delta:method}) and discusses how consistency and asymptotic normality can fail in 
high dimensions.

For simplicity, several results are presented as asymptotic statements, but finite-sample bounds can be found in their proofs in Appendices I--M. We use the phrase \emph{non-classical asymptotic as $n \rightarrow \infty$} to refer to the asymptotic where $n \rightarrow \infty$ and  $d=d(n)$, $m=m(n)$, the distribution of the random vectors and the statistic are all allowed to vary in $n$. This includes the high-dimensional regime where $d$ and $m$ grow in $n$, but also extends to other setups. 

% In this section, we obtain generalisations of several existing results as applications of our theorems. For notational simplicity, we mostly consider i.i.d.~data and symmetric functions. As our general results also hold in the non-i.i.d.~and asymmetric case, the interesting phenomena we observe can be readily extended, and \cref{sec:subgraph:u} provides one such example. Notably, since any degree-$m$, symmetric polynomial of $n$ variables can be written as a weighted sum of U-statistics and V-statistics with degree $\leq m$, all our applications will be reduced to studying the asymptotic distributions of U-statistics and V-statistics. The proofs for all results in this section are included in \Cref{appendix:senatov}.
% \cref{sec:proof:applications}.

\subsection{Necessary and sufficient condition for asymptotic normality} \label{sec:normality} All results in \Cref{sec:main:results} concern the approximations of a statistic $f(X)$ by some Gaussian polynomial $p_m(Z) = p_m(Z_1, \ldots, Z_n)$. By Gaussianity, for each $1 \leq i \leq n$, we can rewrite 
\begin{align*}
    Z_i \;=\; \mean[Z_i] + \Var[Z_i]^{1/2} (\eta_{i1}, \ldots, \eta_{id})^\top 
    \;,
\end{align*}
where $(\eta_{il})_{i \leq n, l \leq d}$ are i.i.d.~univariate standard normals. This allows us to view $p_m(Z)$ as a degree-$m$ polynomial of $nd$ i.i.d.~standard normal variables.

Suppose for now that $\mean[p_m(Z)] = 0$ for simplicity. This polynomial lives in the span of products of Hermite polynomials with total degree $\leq m$, and therefore $p_m(Z)$ lives in the $m$-th order Wiener chaos of $(\eta_{il})_{i \leq n, l \leq d}$ (see \cite{nourdin2013lectures} for an introduction). If $m$ is fixed, the fourth moment theorem by Nualart and Peccati \cite{nualart2005central} applies and shows that $p_m(Z)$ is asymptotically Gaussian as $n \rightarrow \infty$ if and only if its excess kurtosis, defined as
\begin{align*}
    \Kurt[p_m(Z)] \;\coloneqq\; 
    \mean[ (p_m(Z) - \mean[p_m(Z)])^4 ] \,/\, \Var[p_m(Z)]^2 - 3 
    \;,
\end{align*}
is asymptotically zero. 

To establish a similar result for the case where degree $m$ may grow, we leverage on the finite-sample bound developed by Nourdin and Peccati \cite{nourdin2015optimal} measured by the total variation distance $d_{\rm TV}$. The next result says that a vanishing excess kurtosis is sufficient even as $m$ grows, and is necessary under a uniform integrability condition.

\begin{proposition}[Fourth moment phenomenon] \label{prop:gaussian} Let $p_m: (\R^d)^n \rightarrow \R$ be a generic degree-$m$ polynomial function, $Z=(Z_1, \ldots, Z_n)$ be a collection of independent Gaussian vectors, and $\eta$ be a univariate standard normal variable independent of all other quantities. For every $n, m, d \in \N$ and $\sigma = \sqrt{\Var[p_m(Z)]} > 0$\,, 
\begin{align*}
    d_{\rm TV}( \sigma^{-1} (p_m(Z) - \mean[p_m(Z)]) \,,\, \eta )
    \;\leq&\;
    \big(\mfrac{4m-4}{3m} \, \big| \Kurt[p_m(Z)] \big| \big)^{1/2}
    \;.
\end{align*}
Moreover, under the non-classical asymptotic as $n \rightarrow \infty$, 
\begin{proplist}
    \item $\Kurt[ p_m(Z)] \rightarrow 0$ is a sufficient condition for $\sigma^{-1} (p_m(Z) - \mean[p_m(Z)]) \overset{d}{\rightarrow} \eta$;
    \item if $\sigma^{-4} (p_m(Z) - \mean[p_m(Z)])^4$ is uniformly integrable,  $\Kurt[p_m(Z)] \rightarrow 0$ is also necessary.
\end{proplist}
\end{proposition}

\begin{proof}[Proof of \cref{prop:gaussian}] The finite-sample bound is a restatement of Theorem 1.1 of Nourdin and Peccati \cite{nourdin2015optimal} for $p_m(Z)$, and directly implies (i). (ii) is proved by noting that, when $\sigma^{-1} p_m(Z) \overset{d}{\rightarrow} \eta$, by continuous mapping theorem, $\sigma^{-4} p_m(Z)^4 \overset{d}{\rightarrow} \eta^4$. Uniform integrability then implies the desired moment convergence (see Theorem 25.12 of \cite{billingsley1995probability}). 
\end{proof}

By identifying $p_m(Z)$ as the multilinear Gaussian polynomial $q_m(Z)$ in \Cref{thm:main} (e.g.~up to a reparameterisation, see \Cref{appendix:non:multilinear}), we can replace $\mean[p_m(Z)]$ and $\sigma^2$ in \Cref{prop:gaussian} by the mean and variance of the original statistic $p_m(X)$. Combining \Cref{prop:gaussian} with any Gaussian polynomial approximation result in \Cref{sec:main:results} by the triangle inequality, we immediately obtain a normal approximation bound for the statistic of interest. Indeed, for any statistic for which a Gaussian polynomial approximation is valid, \Cref{prop:gaussian} provides a necessary and sufficient condition for its asymptotic normality. A practical consequence is that, even if $p_m(X)$ is not asymptotically Gaussian when $d$ and $m$ are fixed, it can become Gaussian in the high-dimensional regime. Similar phenomenon has already been observed for U-statistics \cite{janson1991asymptotic,yan2021kernel,bhattacharya2022asymptotic}, and we provide a toy example below.

\begin{example}[Degenerate V-statistics can be asymptotically normal] \label{example:degen:V:normal} Consider the V-statistic $v_2(X) = \frac{1}{n^2} \sum_{i,j \leq n} X_i^\top X_j$ in \Cref{example:simple:V}. Under universality, we can replace $X_i$'s by Gaussians. Suppose for simplicity that $X_i \sim \cN(0, \Sigma)$. This allows us to express the rescaled statistic via some i.i.d.~standard normal vectors $(\eta_i)_{i \leq n}$ in $\R^d$:
\begin{align*}
    n \,  v_2(X) \;=\; \Big( \mfrac{1}{\sqrt{n}} \msum_{i \leq n} \eta_i\Big)^\top \Sigma \, \Big( \mfrac{1}{\sqrt{n}} \msum_{i \leq n} \eta_i\Big) \;.
\end{align*}
For $d=1$, this is a chi-squared variable with $1$ degree of freedom. For $d \rightarrow \infty$ and $\Sigma = \frac{1}{d} I_d$, this becomes asymptotically normal; in general, whether $n \, v_2(X)$ is asymptotically normal depends on the limiting spectrum of $\Sigma$. For comparison, Bhattarcharya et al.~\cite{bhattacharya2022asymptotic} characterise the asymptotic distribution of a large family of quadratic forms, which includes $n \, v_2(X)$ here as a special case. Their condition for asymptotic normality is precisely described by the fourth moment phenomenon, which they relate to the spectral behaviour of $\Sigma$. \cite{bhattacharya2022asymptotic} also demonstrates further statistical applications in subgraph count statistics. Our results can be viewed as a degree-$m$ extension with finite-sample bounds; see \Cref{sec:degree:m:U:stats} for a discussion.
\end{example}

\subsection{An asymptotically normal statistic for which bootstrap is inconsistent} \label{sec:bootstrap} 
A large of literature has shown how bootstrap can fail to be consistent for many statistics of interest \citep{hall1993inconsistency,andrews2000inconsistency,sen2010inconsistency,lin2024failure}, with many focusing on ranking-based statistics or statistics under irregularity conditions. Our main results can be combined with the observation in \Cref{sec:normality} to showcase a particular example of failure of bootstrap, despite asymptotic normality of the statistic of concern and due to the high-dimensional regime. We seek to approximate the simple U-statistic $u_2(X) = \frac{1}{n(n-1)} \sum_{i \neq j} X_i^\top X_j$ in \Cref{example:simple:U}, by using the bootstrap samples 
\begin{align*}
  X^{(b)}_1, \ldots, X^{(b)}_n  \,|\, (X_1, \ldots, X_n) \;\sim\; \textrm{Uniform}\{X_1, \ldots, X_n\}\;.  
\end{align*}
For $d=1$, the classical work of Bickel and Freedman \citep{bickel1981some} shows that in the non-degenerate case where $\mean[X_i] \neq 0$, the standard bootstrap estimate 
\begin{align*}
    U^{(b)}_1 \;\coloneqq\; u_2\big( X^{(b)}_1, \ldots, X^{(b)}_n \big)
\end{align*}
is consistent. However, for the degenerate case where $\mean[X_i] = 0$, Arcones and Gin\'e \cite{arcones1992bootstrap} shows that $U^{(b)}_1$ is inconsistent. Their proposal is exploit the degeneracy condition $\mean[X_i] = 0$ and approximate each $X_i$ by the conditionally centred bootstrap variable $X^{(b)}_i - \bar X$, where $\bar X = \frac{1}{n} \sum_{i \leq n} X_i$. The proposed bootstrap estimate is then 
\begin{align*}
    U^{(b)}_2 
    \;\coloneqq\;
    u_2\big( X^{(b)}_1 - \bar X, \ldots, X^{(b)}_n - \bar X \big)
    \;=\;
    \mfrac{1}{n(n-1)} \msum_{i \neq j } ( X^{(b)}_i - \bar X)^\top ( X^{(b)}_j - \bar X) 
    \;.
\end{align*}
This corresponds to performing a bootstrap approximation directly on the degenerate component of $u_2(X)$ under the Hoeffding decomposition (see \Cref{sec:degree:m:U:stats} for a precise definition of the Hoeffding decomposition). For fixed $d$, \cite{arcones1992bootstrap} proves that $U^{(b)}_2$ is consistent for the degenerate case and that, for more general degenerate U-statistics, the same idea works by performing a bootstrap procedure on the correct Hoeffding component.

In the notation of \Cref{example:simple:U}, the above results say that $U^{(b)}_1$ is consistent when the first-order term $q_1(X)$ dominates in $u_2(X)$, and $U^{(b)}_2$ is consistent when the second-order term $q_2(X)$ dominates in $u_2(X)$. However, in the high-dimensional regime, either component can dominate even if $u_2(X)$ is non-degenerate with $\mean[X_i] \neq 0$ (\Cref{example:simple:U}). Consequently, the choice of which bootstrap to use may depend on information that is unavailable to the user.

We formalise this intuition in the next result. For $l=1,2$, we say that $U^{(b)}_l$ is \emph{consistent} if 
\begin{align*}
    &\;
    \msup_{t \in \R} \,\big|\, \P\big( U^{(b)}_l - \mean\big[ U^{(b)}_l  \,\big|\, X \big] \leq t \,\big| X \big) - \P( u_2(X) - \mean[ u_2(X) ] \leq t) \big| \;\xrightarrow{\P}\; 0
    \;.
\end{align*}
Otherwise, we say that $U^{(b)}_l$ is not consistent. 
% We also say that the conditional mean $\mean[U^{(b)}_l | X]$ is \emph{consistent for} $\mean[u_2(X)]$ if 
% \begin{align*}
%         \mean[U^{(b)}_l | X] - \mean[u_2(X)]
%     \;\xrightarrow{\P}\; 
%     0
%     \;.
% \end{align*}
We shall focus on the simplifying setup:

\begin{assumption}  \label{assumption:bootstrap:model} (i) $X_1, \ldots, X_n$ are generated from the model 
$ X_i \coloneqq \mu + \tau U_i$, where $\mu \in \R^d$ and $\tau \in \R$ are deterministic (and may depend on $n$), and $U_i$'s are i.i.d.~random vectors; \\
 (ii) $U_1$ has i.i.d.~$1$-sub-Gaussian entries, each with zero mean  and unit variance;\\ (iii) $d^{1/4}  \| \mu \|_{l_4}$,  $d^{\frac{\nu-2}{2\nu}}  \| \mu \|_{l_\nu}$, and $d^{\frac{\nu-1}{2\nu}}  \| \mu \|_{l_{2\nu}}$ are all $\Theta( \| \mu \|_{l_2})$, where $\nu \in (2,3]$ is fixed. 
\end{assumption}

\Cref{assumption:bootstrap:model}(iii) says that the vector norm inequalities $ \|\mu \|_{l_2} \leq d^{1/4} \, \| \mu \|_{l_4}$ and so on are asymptotically tight up to a constant. This can be verified for $\mu$ with the same entries, as well as for $\mu$'s whose bulk of the entries are approximately on the same scale. In the asymptotic normality result below, this condition plays the role of a Lindeberg condition for CLT for weighted averages of the form $d^{-1/2} \mu^\top X_i =  d^{-1/2} \sum_{l \leq d} \mu_l X_{il}$.

\begin{proposition}[Bootstrap inconsistency for an asymptotically normal U-statistic] \label{prop:bootstrap:inconsistency} Let $u_2(X) =  \frac{1}{n(n-1)} \sum_{i \neq j} X_i^\top X_j$. Under \Cref{assumption:bootstrap:model}, as $n, d \rightarrow \infty$, we have 
\begin{align*}
    (\Var[u_2(X)])^{-1/2} ( u_2(X) - \mean[u_2(X)]) \;\xrightarrow{d}\; \cN(0,1)\;,
\end{align*}
where the rate of convergence is described by 
\begin{align*}
    (\Var[u_2(X)])^{1/2} \;=\; \Big( 
         \mfrac{4 \tau^2 \| \mu \|_{l_2}^2}{n} + \mfrac{2 \tau^4 d}{n(n-1)}
    \Big)^{1/2}
    \;.
\end{align*}
Let $\nu \in (2,3]$ be given as in \Cref{assumption:bootstrap:model}(iii). If additionally $d=o(n^{2/\nu})$, then 
\begin{proplist}
    \item $ U^{(b)}_1$ is consistent if and only if  $\tau d^{1/2} = o( n^{1/2} \| \mu \|_{l_2} )$;
    \item $ U^{(b)}_2$ is consistent if and only if $\| \mu \|_{l_2} = o(1)$ and $\tau d^{1/2} = \omega( n^{1/2} \| \mu \|_{l_2} )$.
\end{proplist}   
\end{proposition}

\begin{remark}[Comparing \Cref{prop:bootstrap:inconsistency} to the classical case] \Cref{prop:bootstrap:inconsistency} requires $d \rightarrow \infty$. Nevertheless, it is instructive to consider the case where $d$ grows at a negligible rate and perform an informal comparison of the consistency conditions to those in the the classical case. Set $\tau$ and $\mu$ to be $O(1)$. Then, the consistency condition for $U^{(b)}_1$ is equivalent to saying that either $\mu \neq 0$, i.e.~$u_2(X)$ is non-degenerate, or $\mu = 0$ and $\tau = 0$, i.e.~all data are zero almost surely. Meanwhile, the consistency condition for $U^{(b)}_2$ requires that $\mu=0$, i.e.~$u_2(X)$ is degenerate. These agree with the classical bootstrap results surveyed above.
\end{remark}

\Cref{prop:bootstrap:inconsistency} shows that both bootstrap estimates  can exhibit nuanced behaviours, despite (i) the simple setup of \Cref{assumption:bootstrap:model}, (ii) asymptotic normality of $u_2(X)$, and (iii) the fact that $d$ is only growing at a moderate rate $d=o(n^{2/\nu})$. The consistency conditions can be viewed as conditions on the signal-to-noise ratio $\| \mu \|_{l_2} / \tau$, and in the intermediate regime where $\| \mu \|_{l_2} / \tau$ scales as $\sqrt{d / n}$, both methods fail. Importantly in many statistical applications, we typically use $u_2(X)$ to estimate $\mean[u_2(X)] = \| \mu \|_{l_2}^2$ and employ bootstrapped confidence intervals due to the lack of knowledge of the noise level $\tau$. Since the consistency conditions depend on the unknown quantities $\mu$ and $\tau$, this points to a substantial issue of using either bootstrap methods for U-statistics, even when the dimension grows moderately in $n$.

While \Cref{prop:bootstrap:inconsistency} only applies to the toy U-statistic $u_2(X)$, we conjecture that a similar issue can be present for general U-statistics of high-dimensional data, in view of known ways of generalising bootstrap results on $u_2(X)$ to general U-statistics \cite{arcones1992bootstrap}. For instance, if $\tau U_i$ is replaced by some $\tilde U_i$ with non-i.i.d.~coordinates, we expect the condition to depend more intricately on the coordinate-wise dependence of $\tilde U_1$, as we have observed in \Cref{sec:multilinear}. For more general U-statistics, $d$, $\mu$ and $\tau$ need to be replaced by quantities that depend on the Hoeffding decompositions of the U-statistics and on the kernel of the U-statistic.

We conclude by some comments on the proof. Bootstrap validity results are typically proved as applications of distributional approximation results \citep{bickel1981some,arcones1992bootstrap,zhilova2020nonclassical,austern2020asymptotics}, and so is \Cref{prop:bootstrap:inconsistency} here. Indeed, our proof proceeds by applying the universality result (\Cref{thm:main}) both for approximating $X_i$'s by Gaussians and for approximating $X^{(b)}_i$'s by conditional Gaussians given $X$. The result then follows by analysing the corresponding Gaussian polynomials. 

\begin{remark}[Potential remedy and further issues] If $\mu$ has equal entries, one potential remedy is to view the i.i.d.~entries of $X_1, \ldots, X_n$ as $nd$ i.i.d.~variables and perform bootstrap over them. In view of classical bootstrap results for non-identically distributed data \cite{liu1988bootstrap}, this may remain valid even when the entries of $\mu$ are only approximately equal, similar to our \Cref{assumption:bootstrap:model}(iii). However, we do not expect this method to perform well if the entries of $\mu$ are substantially different. When one has no information about the structure of $\mu$, it remains unclear whether there is a version of bootstrap that is guaranteed to be valid.
\end{remark}

\subsection{Distribution approximations of degree-$m$ U-statistics of high-dimensional data} \label{sec:degree:m:U:stats} Let $Y_1, \ldots, Y_n$ be i.i.d.~random variables taking values in some general (not necessarily Euclidean), possibly $n$-dependent measurable space $\cE \equiv \cE(n)$. Given a symmetric function $u: \cE^m \rightarrow \R$, we study the degree-$m$ complete U-statistic given by 
\begin{align*}
    u_m(Y) 
    \;\coloneqq&\; 
    \mfrac{1}{n(n-1) \ldots (n-m+1)} \msum_{i_1, \ldots, i_m \in [n] \text{ \rm distinct }} u(Y_{i_1}, \ldots, Y_{i_m}) 
    \tagaligneq \label{eq:defn:u}
    \;.
\end{align*}
Classically, under the H\'ajek projection technique, the asymptotic distribution of $u_m(Y)$ can be approximated by that of a degree-$M$ polynomial of Gaussians, where
\begin{align*} 
    M \;\coloneqq\; 
    \min \{ j \in [m] \,:\, \sigma_j > 0 \}\;,
    \qquad \text{and } \quad
    \sigma_j^2 \coloneqq \Var\, \mean[ u(Y_1, \ldots, Y_m) \,|\, Y_1, \ldots, Y_j ] \;.
\end{align*}
See \cite{filippova1962mises}, \cite{rubin1980asymptotic} and Chapter 5.4 of \cite{serfling1980approximation} for references on the classical theory for U-statistics.

In this section, we show how the high-dimensional asymptotics of $u_m(Y)$, where $\cE=\R^d$ with $m$ and $d$ allowed to be large relative to $n$, can again be obtained by universality. A particular focus is to generalise the phase transition observed for the simple degree-two U-statistic in \Cref{example:simple:U} to general degree-$m$ U-statistics. The main result is a bound on the universality approximation, which 
generalises a number of results in the literature and provides finite-sample bounds for some known asymptotic statements.

% We first present the ingredients required for the result.

\vspace{.5em}

\noindent
\textbf{Decomposition of $u$. } We use a functional decomposition assumption to obtain finite-sample bounds. For $K \in \N$, we approximate $u(y_1, \ldots, y_m)$ by a polynomial
\begin{align*}
    \msum_{k_1, \ldots, k_m=1}^K \lambda^{(K)}_{k_1 \ldots k_m} \phi^{(K)}_{k_1}(y_1) \times \ldots \times \phi^{(K)}_{k_m}(y_m)\;,
\end{align*}
where for each $K$, the coordinates of the $m$ vectors are defined by a triangular array of $\cE \rightarrow \R$ functions $\{\phi^{(K)}_k\}_{k \leq K, K \in \N}$, 
% evaluated at $y_1, \ldots, y_m$, 
and the polynomial weights are given by a triangular array of real values $\{\lambda^{(K)}_{k_1 \ldots k_m}\}_{k_1, \ldots, k_m \leq K, K \in \N}$. We also denote the $L_\nu$ approximation error as
\begin{align*}
    \varepsilon_{K;\nu} \;\coloneqq\; \big\| \msum_{k_1, \ldots, k_m=1}^K \lambda^{(K)}_{k_1 \ldots k_m} \phi^{(K)}_{k_1}(Y_1) \times \ldots \times \phi^{(K)}_{k_m}(Y_m)
    -
    u(Y_1, \ldots, Y_m) \big\|_{L_\nu}
    \;.
\end{align*}

% \vspace{.5em}

\begin{assumption} \label{assumption:L_nu} There exists some $\nu \in (2,3]$ such that, for every fixed $n$, $m$, $d$ and $\cE$, as $K \rightarrow \infty$, the error $\varepsilon_{K;\nu} \rightarrow 0$ for some choice of $\{\phi^{(K)}_k\}_{k \leq K, K \in \N}$ and $\{\lambda^{(K)}_{k_1 \ldots k_m}\}_{k_1, \ldots, k_m \leq K, K \in \N}$.
\end{assumption}

\begin{remark} \label{remark:decomposition} (i) \cref{assumption:L_nu} is mild: In the appendix, we show that it holds with $\nu=2$ under mild assumptions 
    % on the $L_2$ space equipped with the $m$-fold product measure of $Y_1$ 
(\Cref{lem:L2:decomposition}), 
% \cref{lem:L2:decomposition}, 
and that it holds for $\cE=\R^d$ and any $u$ well-approximated by a Taylor expansion (\Cref{lem:assumption:taylor}). For $m=2$, \cite{huang2023high} has verified \cref{assumption:L_nu} for general kernels under Mercer decomposition and for specific kernels via the Taylor expansion technique described by our \Cref{lem:assumption:taylor}.
% \cref{lem:assumption:taylor}
(ii) Unlike asymptotic U-statistics results that use the $L_2$ decomposition from Hilbert-Schmidt operator theory (see e.g.~Chapter 5.4 of \cite{serfling1980approximation}), \cref{assumption:L_nu} does not require orthogonality or boundedness conditions on $\phi^{(K)}_k$ or $\lambda^{(K)}_{k_1 \ldots k_m}$ and also allow them to vary with $K$ or $n$. In particular, they are generally not unique and can be chosen at convenience for verification, 
% of the assumption, 
e.g.~by a suitable Taylor expansion.
\end{remark}

% \vspace{.5em}

\noindent
\textbf{Penultimate approximation. } To identify the dominating component, we first recall that by Hoeffding's decomposition theorem (see e.g.~Theorem 1.2.1 of \cite{denker1985asymptotic}), 
\begin{align*}
    u_m(y_1, \ldots, y_n) 
    \;=\; 
    \mean[ u(Y_1, \ldots, Y_m) ] \, + \, \msum_{j=1}^m \mbinom{m}{j} \, U^{\rm H}_j(y_1, \ldots, y_n) \;, \tagaligneq \label{eq:hoeffding:decompose}
\end{align*}
where each Hoeffding's decomposition $U^{\rm H}_j(Y)$ is a degenerate degree-$j$ U-statistic
\begin{align*}
    U^{\rm H}_j(y_1, \ldots, y_n) \;\coloneqq&\; 
    \mfrac{1}{n(n-1) \ldots (n-j+1)} \msum_{i_1, \ldots, i_j \in [n] \text{ \rm distinct }} u^{\rm H}_j(y_{i_1}, \ldots, y_{i_j})\;,
    \\
    u^{\rm H}_j(y_1, \ldots, y_j)
    \;\coloneqq&\;
    \msum_{r=0}^j (-1)^{j-r} \msum_{1 \leq l_1 < \ldots < l_r \leq j }  
    \mean\big[ u(y_{l_1}, \ldots, y_{l_r}, Y_1, \ldots, Y_{m-r}) \big]\;.
\end{align*}
Under \cref{assumption:L_nu}, each Hoeffding's decomposition can be approximated by a polynomial of $\R^K$ vectors, to which our universality results can be applied. Let $\Xi^{(K)} \coloneqq \{ \xi^{(K)}_1, \ldots, \xi^{(K)}_n \}$ be a collection of i.i.d.~zero-mean Gaussian random vectors in $\R^{mK}$ with the same variance as $(\phi^{(K)}_1(Y_1), \ldots, \phi^{(K)}_K(Y_1))^\top$. For $j \in [m]$ and $K \in \N$, the penultimate approximation for $U^{\rm H}_j(Y)$ is therefore given as a polynomial of $\R^K$ Gaussians,
\begin{align*}
    & U^{(K)}_j ( \Xi^{(K)})
    \;\coloneqq\;   
    \mfrac{1}{n(n-1) \ldots (n-j+1)} \msum_{i_1, \ldots, i_j \in [n] \text{ \rm distinct }} \tilde u_j^{(K)}(\xi^{(K)}_{i_1}, \ldots, \xi^{(K)}_{i_j})\;, \tagaligneq \label{eq:defn:u:approximating}
    \\
    & u^{(K)}_j(v_1, \ldots, v_j)
    \;\coloneqq\;
    \msum_{k_1, \ldots, k_m=1}^K 
    \, \lambda^{(K)}_{k_1 \ldots k_m} 
    v_{1k_1}  \cdots v_{jk_j}
    \, \mean\big[\phi^{(K)}_{k_{j+1}}(Y_1)\big] \cdots \mean\big[\phi^{(K)}_{k_m}(Y_1)\big] 
    \;.
\end{align*}

\vspace{.5em}

\noindent
\textbf{Moments. } By variance domination, the asymptotic distribution of $u_m(Y)$ is determined by the relative size of the rescaled variances, given for $M \in [m]$ as
\begin{align*}
    \rho_{m,n;M}
    \;&\coloneqq\; 
    \big( \msum_{r \in [m] \setminus \{M\}}  \sigma_{m,n;r}^2 \big) \,\big/ \, \sigma_{m,n;M}^2
    \;,
    \\
    &\text{where }
    \; \sigma_{m,n;j}^2
    \;\coloneqq\; 
    \mbinom{m}{j}^2 
    \mbinom{n}{j}^{-1} 
    \Var\, \mean[ u(Y_1, \ldots, Y_m) \, |\, Y_1, \ldots, Y_j ]
    \qquad\text{ for } j \in [m]
    \;.
\end{align*}
$\sigma^2_{m,n;j}$ describes the contribution of $U^{\rm H}_j$ to the overall variance of $u_m(Y)$ (\Cref{lem:u:variance} in the appendix). 
    % \cref{lem:u:variance}
The bound also depends on a Berry-Ess\'een moment ratio
\begin{align*}
    \tilde \beta_{M,\nu}
    \coloneqq
        \big\| 
            u_M^{\rm H} \big( Y_1,\ldots, Y_M \big)
        \big\|_{L_\nu} 
      \,\big/\,
        \big\| 
            u_M^{\rm H} \big( Y_1,\ldots, Y_M \big) 
        \big\|_{L_2}
      \; 
      \text{ defined for } \nu \in (2,3] \text{ and } M \in [m]
      \;. 
\end{align*}

\begin{proposition} \label{prop:higher:u} Suppose \cref{assumption:L_nu} holds for a fixed $\nu \in (2,3]$, and that for some $M \in [m]$ that is allowed to depend on $n$,
\begin{align*}
    m n^{-\frac{\nu-2}{2(\nu M + 1)}} \tilde \beta_{M,\nu}^{\frac{\nu}{\nu M +1}} 
    \;=&\; o(1)
    &\text{ and }&&
    m \rho_{m,n;M}^{\frac{1}{2M+1}}
    \;=&\;
    o(1)
    \;.
\end{align*}    
Then under the non-classical asymptotic as $n \rightarrow \infty$,
\begin{align*}
    \msup_{t \in \R} \Big| \P\big( u_m(Y) - \mean[ u(Y_1, \ldots, Y_m) ]  \leq t\big) 
    - 
     \lim_{K \rightarrow \infty} 
    \P\Big( \mbinom{m}{M}  \,  U_M^{(K)}(\Xi^{(K)})  \leq t\Big) \Big|
    \;\rightarrow\; 0
    \;.
\end{align*}
If additionally $ \lim_{K \rightarrow \infty}  \Kurt\big[ U_M^{(K)}(\Xi^{(K)})\big] \xrightarrow{n \rightarrow \infty} 0$,  then
\begin{align*}
    \textstyle\lim_{K \rightarrow \infty}^{\rm w}  
    \,\mfrac{ u_m(Y) - \mean[ u(Y_1, \ldots, Y_m) ] }{\binom{m}{M}  \Var\big[   U_M^{(K)}(\Xi^{(K)}) \big]^{1/2}}
    \;
    \;\underset{n \rightarrow \infty}{\overset{d}{\longrightarrow}}\;
    \;
    \cN(0,1)\;.
\end{align*}
$\textstyle\lim^{\rm w}$ denotes the limit in distribution, which we assume to exist for the statement above.
\end{proposition}

\cref{prop:higher:u} allows the degree of the U-statistic $m$ to be much larger than $\log n$: Since $M$ appears in the exponent instead of $m$, we only require the degree of the Gaussian polynomial to satisfy $M \log M = o(\log n)$, 
%  In the above bound, when $\tilde \beta_{M,\nu} = O(1)$ and $\rho_{m,n;M} = o(1)$, we only require the degree of the polynomial of Gaussians to satisfy $M \log M = o(\log n)$, 
with no explicit requirement on $m$. Moreover, under the moment conditions on $\rho_{m,n;M}$ and $\tilde \beta_{M,\nu}$, the bound is dimension-agnostic. This generalises the discussion in \Cref{sec:non:multilinear} and allows \cref{prop:higher:u} to recover many known high-dimensional results on U-statistics. For instance, there are two cases in which a normal approximation is valid, but with different variances:
\begin{itemize}
    \item When $M=1$, the distributional limit $\lim_{K \rightarrow \infty}  U_1^{(K)}(\Xi^{(K)})$, when exists, is Gaussian. In this case, \cref{prop:higher:u} provides a finite-sample bound for the Gaussian approximation of $u_m(Y)$, a potentially infinite-order U-statistic (IOUS). Notably if $\tilde \beta_{M,\nu}$ is bounded, the only dependence on $m$ in the bound is through $\rho_{m,n;M}$; when $d$ is fixed and $ \Var\, \mean[ u(Y_1, \ldots, Y_m) \, |\, Y_1, \ldots, Y_j ] \in (0, \infty)$ is fixed for all $j \in [m]$, we can compute 
    \begin{align*}
        \rho_{m,n;1} 
        \;\leq\; 
        \mfrac{n}{m^2} \msum_{r=2}^m \Big( \mfrac{m^2 e^2}{n} \Big)^r
        \;=\; O\Big( \mfrac{m^2}{n} \Big)
        \hspace{5em} 
        \text{ if }\; 
        m^2 < \mfrac{n}{2e^2}\;.
    \end{align*}
    Therefore when $d$ is fixed and $M=1$, Gaussianity holds if $m=o(n^{-1/2})$. This agrees with known optimal growth condition on $m$ for Gaussianity of 1d IOUS \citep{van1988elementary}. On the other hand, when $d$ and the variances are allowed to change, $\rho_{m,n;1}$ has more complex dependence on $n$; this can lead to a more stringent or relaxed condition on the size of $m$.

    \vspace{.5em}

    \item When $M > 1$, $U_M^{(K)}(\Xi^{(K)})$ is a degree-$M$ polynomial of Gaussians. 
    % Since the approximating polynomial of Gaussians is allowed to vary in $n$, the resultant approximation may still be Gaussian in view of the fourth moment phenomenon \cite{nualart2005central, nourdin2015optimal,nourdin2016multidimensional} (see \cref{remark:delta:fourth:moment}).
     By 
    %  the fourth moment phenomenon (
        \cref{prop:gaussian}, the asymptotic limit can still be Gaussian in high dimensions. This  behaviour is specific to the high-dimensional regime, where the Gaussian polynomial approximation is allowed to vary in $n$. See \Cref{example:degen:V:normal} for an analogous example on V-statistics.
\end{itemize}
The above two cases recover many existing works on the normal approximation of a high-dimensional degree-two U-statistic: Some of those results come without a fourth-moment condition and require assumptions similar to those for (i) \citep{chen2010two, harchaoui2020asymptotics, wang2015high, yan2021kernel}, whereas others consider a fourth moment condition and fall under (ii) \citep{gao2023two,bhattacharya2022asymptotic}. A practical consequence of (ii) is that a degenerate U-statistic -- often found in hypothesis testing and which requires numerical approximation due to its classical non-Gaussianity \cite{leucht2013dependent} -- may be Gaussian in high dimensions. In view of the necessity statement in \cref{prop:gaussian}, one may argue that the two cases cover most, if not all of the situations of Gaussianity.

Meanwhile, \cref{prop:higher:u} highlights when $u_m(Y)$ may be asymptotically non-Gaussian. Here, non-Gaussianity happens not according to the degeneracy of the U-statistic, but depending on the relative sizes of the rescaled variances $\sigma_{m,n;j}^2$, each corresponding to the variance of the degree-$j$ Hoeffding decomposition. If the ratio of the variances change as $d$ grows, the limiting distribution of $u_m(Y)$ can transition from one low-degree polynomial of Gaussians to a higher one, or vice versa. This phase transition effect has been highlighted in the literature for degree-two U-statistics \citep{bhattacharya2022asymptotic, huang2023high}. 
% We also refer readers to 
\cite{huang2023high} also includes empirical evidence of the transition for a U-statistic used in high-dimensional kernel-based distribution tests:
% of a $U$-statistic between different limits in high dimensions: 
There, the ratio of variances reduces to a comparison between problem-specific hyperparameter choices, which determine whether the phase transition occurs.
% variance domination reduces to problem-specific hyperparameter choices.

\begin{remark}[Further applications] (i) The same argument in this section can be extended to a U-statistic of non-identically distributed data, although a more elaborate decomposition than \eqref{eq:hoeffding:decompose} is required to accommodate for the fact that $\mean[u(y_1, \ldots, y_r, Y_1, \ldots, Y_{m-r})]$ may not be equal to $\mean[u(y_1, \ldots, y_r, Y_2, \ldots, Y_{m-r+1})]$. (ii) One notable statistical application of degree-$m$ U-statistics is the study of subgraph count statistics. In a line of recent work  \cite{hladky2021limit,bhattacharya2023fluctuations,kaur2021higher}, substantial efforts are spent on characterising the same phase transition effect we discussed above. There, asymptotic degeneracy arises not because of dimension dependence, but because of geometric properties of the graph models. Those degeneracies can again be interpreted through our variance ratios $\rho_{m,n;M}$, and our \Cref{prop:higher:u} can be applied to obtain a finite-sample bound for those setups; see \Cref{appendix:subgraph:u}. In a concurrent work, Chatterjee, Dan, and Bhattacharya \citep{chatterjee2024higher}
obtain asymptotic limits for graphon models with higher-order degeneracies and for joint subgraph counts, with additional geometric insights.
    % \cref{sec:subgraph:u} illustrates such an example.
\end{remark}

\subsection{Phase transitions of MMD in two-sample tests with imbalanced data} \label{sec:MMD:imbalance} Let $X=(X_i)_{i \leq n_1}$ and $Y=(Y_i)_{i \leq n_2}$ be two independent datasets of $\R^d$ vectors, with $X_i \overset{i.i.d.}{\sim} P_X$ and $Y_i \overset{i.i.d.}{\sim} P_Y$. Given a symmetric kernel function $\kappa: \R^d \times \R^d \rightarrow \R$ (see \cite{steinwart2012mercer} for a definition), we study the maximum mean discrepancy (MMD) \cite{gretton2012kernel},
\begin{align*}
    q_{\rm MMD}
    \;\coloneqq\;
    \mean[\kappa(X_1,X_2)]  - 2 \mean[\kappa(X_1,Y_1)] + \mean[\kappa(Y_1, Y_2)] 
    \;,
\end{align*}
which is used for the two-sample distribution test $H_0: P_X = P_Y$ v.s. $H_1: P_X \neq P_Y$. A popular unbiased estimator of $q_{\rm MMD}$ takes the form of an incomplete U-statistic, namely 
\begin{align*}
    U_{\rm MMD} 
    \;\coloneqq&\; 
    \mfrac{1}{n_1(n_1-1)} \sum_{i \neq j}^{n_1} \kappa(X_i, X_j)
    -
    \mfrac{2}{n_1 n_2} \sum_{i \leq n_1, j \leq n_2} \kappa(X_i, Y_j)
    % \\
    % &\;
    +
    \mfrac{1}{n_2 (n_2-1)} \sum_{i \neq j}^{n_2} \kappa(Y_i, Y_j)
    \;.
\end{align*}
With a suitable kernel choice $\kappa$, the population quantity $q_{\rm MMD}$ measures the difference between $P_X$ and $P_Y$, whereas the distribution of the centred statistic $U_{\rm MMD} - q_{\rm MMD}$ determines the size and power of tests designed based on the test statistic $U_{\rm MMD}$; see \cite{gretton2012kernel} for details. 

The behaviour of $U_{\rm MMD}$ and its variants in the high-dimensional regime has been a subject of great interest in statistics and machine learning \citep{reddi2015high, ramdas2015decreasing, yan2021kernel, gao2021two, kim2024dimension, shekhar2022permutation}. When the sample sizes are balanced, i.e.~$n_1=n_2=n$, $U_{\rm MMD}$ can be re-expressed as a degree-two complete U-statistic, i.e.~of the form \eqref{eq:defn:u} with $m=2$, by identifying 
\begin{align*}
    \tagaligneq \label{eq:MMD:balanced}
    &\;
    U_{\rm MMD} \;=\; \mfrac{1}{n(n-1)} \msum_{i \neq j}^n \tilde \kappa( V_i, V_j)\;,
    \quad 
    V_i \;\coloneqq\; ( X_i, Y_i)\;,
    \\
    &\;
    \tilde \kappa( (x_1,y_1), (x_2,y_2)) 
    \;\coloneqq\;
    \kappa(x_1, x_2)
    +
    \kappa(y_1, y_2)
    -
    \mfrac{n-1}{n}
    \kappa(x_1, y_2)
    -
    \mfrac{n-1}{n}
    \kappa(x_2, x_1)
    \;.
\end{align*}
This allows  $U_{\rm MMD}$ to be analysed by directly applying results on general complete U-statistics: In the high-dimensional case, \citep{huang2023high} established a phase transition across two possible asymptotic limits as $\kappa$ varies and related it to the test performance of $U_{\rm MMD}$. 

When the sample sizes are imbalanced $n_1 \neq n_2$, however, the analysis of $U_{\rm MMD}$ becomes more complex, as each U-statistic in $U_{\rm MMD}$ involve linear and quadratic components, each of which may dominate or vanish in the limit. Notably if $n_1 = \gamma n_2$ with $\gamma \neq 1$, i.e.~the imbalance ratio is not negligible, $U_{\rm MMD}$ cannot be expressed as \eqref{eq:MMD:balanced} and new results are required to understand its high-dimensional asymptotics. This is again enabled by our universality result. One possible phase transition again arises from comparing the variances of the linear and quadratic components of $U_{\rm MMD}$, namely
\begin{align*}
    \sigma_1^2 
    \;\coloneqq&\;
    \mfrac{4 \,
        \Var \big[ 
        \mean[ \kappa(X_1, X_2) | X_1] - \mean[ \kappa(X_1, Y_1) | X_1] 
        \big]
    }{n_1}
    +
    \mfrac{4 \,
        \Var \big[ 
        \mean[ \kappa(Y_1, Y_2) | Y_1] - \mean[ \kappa(X_1, Y_1) | Y_1] 
        \big]
    }{n_2}\;,
    \\
     \sigma_2^2 
    \;\coloneqq&\;
        \mfrac{2 \, \Var[\kappa(X_1, X_2)]}{n_1 (n_1 -1 )}
        +
        \mfrac{2 \, \Var[\kappa(Y_1, Y_2)]}{n_2 (n_2 -1 )}
        +
        \mfrac{4 \, \Var[\kappa(X_1, Y_1)]}{n_1 n_2}
    %     \\
    % &\;
    %     - 
    %     \mfrac{8 \, \Cov[   \kappa(X_1, X_2),  \kappa(X_1, Y_1)]}{n_1} 
    %     - 
    %     \mfrac{8 \, \Cov[   \kappa(Y_1, Y_2),  \kappa(X_1, Y_1)]}{n_2} 
        \;.
\end{align*}
Write $\| \argdot \|_{L_\nu;{\rm ctr}}^\nu \coloneqq \mean | \argdot - \mean[\argdot] |^\nu$ as the $\nu$-th centred moment of a random variable. We  denote the corresponding $\nu$-th moments by 
\begin{align*}
    M_{\nu;1}^\nu
    \;\coloneqq&\;
    \mfrac{
        \big\|  
            \mean[ \kappa(X_1, X_2) | X_1] - \mean[ \kappa(X_1, Y_1) | X_1] 
        \big\|_{L_\nu;{\rm ctr}}^\nu
    }{n_1}
    +
    \mfrac{
        \big\|  
        \mean[ \kappa(Y_1, Y_2) | Y_1] - \mean[ \kappa(X_1, Y_1) | Y_1] 
        \big\|_{L_\nu;{\rm ctr}}^\nu
    }{n_2}\;,
    \\
     M_{\nu;2}^\nu
    \;\coloneqq&\;
        \mfrac{ \big\| \kappa(X_1, X_2) \big\|_{L_\nu;{\rm ctr}}^\nu}{n_1 (n_1 -1 )}
        +
        \mfrac{ \big\|   \kappa(Y_1, Y_2) \big\|_{L_\nu;{\rm ctr}}^\nu}{n_2 (n_2 -1 )}
        +
        \mfrac{  \big\|  \kappa(X_1, Y_1)\big\|_{L_\nu;{\rm ctr}}^\nu}{n_1 n_2}
        \;.
\end{align*}
We also assume a functional decomposition for $\kappa$ similar to that of \Cref{sec:degree:m:U:stats}.
% , which was verified in \cite{huang2023high} both via Mercer decompositions \cite{steinwart2012mercer} for general kernels and via Taylor expansions for specific kernels.

\begin{assumption} \label{assumption:L_nu:MMD} For some fixed $\nu \in (2,3]$, \Cref{assumption:L_nu} holds for $u =\kappa$ with $m=2$ under the joint probability measures $P \times P$, $P \times Q$ as well as $Q \times Q$.
\end{assumption}

The next result is a high-dimensional analogue of the classical limit theorems for MMD \cite{gretton2012kernel}, and is again proved by  universality (\Cref{thm:main}) and variance domination (\Cref{thm:VD}). For the degenerate limit, we recall that the classical degenerate limit of a degree-two U-statistic is expressed through an infinite sum of weighted and centred chi-squares \cite{serfling1980approximation,gretton2012kernel}. To formalise this, we use $U_{K, n_1, n_2}$ to denote a quadratic form of $4K$ mean-zero univariate normal variables, whose explicit expression is included in the proof in \Cref{appendix:MMD}, and assume the existence of the distributional limit $\lim_{K \rightarrow \infty}^{\rm w} U_{K, n_1, n_2}$. This exactly corresponds to the infinite mixture of chi-squares limit in the classical regime. The non-classical regime introduces additional $n_1$ and $n_2$ dependence, which requires the result to be presented in terms of a sequential distributional limit.

\begin{proposition} \label{prop:MMD:imbalanced}  Suppose \Cref{assumption:L_nu:MMD} holds and that $\frac{M_{\nu;1}}{\sigma_1}$ and $\frac{M_{\nu;2}}{\sigma_2}$ are both $O(1)$. Under the non-classical asymptotic with $n_1, n_2 \rightarrow \infty$, we have
\begin{proplist}
    \item if $\sigma_1 / \sigma_2 = \omega(1)$, then
    $
        \sigma_1^{-1} (U_{\rm MMD} - q_{\rm MMD}) \;\xrightarrow{d}\; \cN(0,1)
    $;
    \item if $\sigma_1 / \sigma_2 = o(1)$, then there exists some sequence of Gaussian quadratic forms $U_{K, n_1, n_2}$ indexed by $K, n_1, n_2 \in \N$ such that
    $
        \sigma_2^{-1} (U_{\rm MMD} - q_{\rm MMD}) 
        \;\xrightarrow{d}\; 
       U_* 
    $,  assuming that the sequential limit $U_* \coloneqq \lim_{n_1, n_2 \rightarrow \infty}^{\rm w} \,
       \lim_{K \rightarrow \infty}^{\rm w}  U_{K, n_1, n_2}$ exists.
\end{proplist}
\end{proposition}

 As with the complete U-statistics in \Cref{sec:degree:m:U:stats}, the classical degeneracy condition is now replaced by a condition on the ratio $\sigma_1 / \sigma_2$. The limit for the degenerate case $\sigma_1/\sigma_2=o(1)$ may simplify to a Gaussian by \Cref{prop:gaussian}; in this case, a computable confidence interval can be formed even under degeneracy of $U_{\rm MMD}$. This effect was examined in detail by \cite{yan2021kernel} for MMD with an isotropic kernel, whereas our \Cref{prop:MMD:imbalanced} holds more generally.

 % A key difference of the imbalanced case from the balanced case is that, in the intermediate regime where $\sigma_1 / \sigma_2 = \Theta(1)$, different mixture limits can arise depending on how $n_1$ scales with $n_2$. For instance, consider the case where the first term of $\sigma_1^2$ vanishes. If the remaining $\Var[\argdot]$ terms are $\Theta(1)$ and if the sample sizes scale as $n_1 = \Theta(n_2^{1/2})$, the second term of $\sigma_1^2$ is on the same scale as $\sigma_2^2$. The resultant limit is a mixture of the non-degenerate Gaussian component of the U-statistic $\frac{1}{n_2 (n_2-1)} \msum_{i \neq j}^{n_2} \kappa(Y_i, Y_j)$, which evaluates the kernel on the dataset $(Y_i)_{i \leq n_2}$, and the degenerate second-order component of the U-statistic $\frac{1}{n_1 (n_1-1)} \msum_{i \neq j}^{n_2} \kappa(X_i, X_j)$, which evaluates the kernel on the dataset $(X_i)_{i \leq n_1}$. When the $\Var[\argdot]$ terms are no longer $\Theta(1)$, e.g.~when they are affected by the growth of $d$, this mixture limit can arise additionally due to the relative scaling between $d$, $n_1$ and $n_2$. In \Cref{appendix:MMD}, we demonstrate this effect by considering a high-dimensional distribution test with mean and variance shifts: For a classically non-degenerate MMD test statistic, we show that this mixture of non-degenerate and degenerate limits can arise as a joint consequence of the high-dimensional regime and the sample size imbalance.

As mentioned, a difference of the imbalanced case from the balanced case is how a combination of different components in each of the three sums in $U_{\rm MMD}$ can dominate in the limit. (i) of \Cref{prop:MMD:imbalanced} corresponds to when the linear components from all three sums dominate, and (ii) corresponds to when the quadratic components dominate. \color{black} In the intermediate regime where $\sigma_1 / \sigma_2 = \Theta(1)$, different mixture limits can arise depending on how $n_1$ scales with $n_2$. For instance, consider the case where the first term of $\sigma_1^2$ vanishes. If the remaining $\Var[\argdot]$ terms are $\Theta(1)$ and if the sample sizes scale as $n_1 = \Theta(n_2^{1/2})$, the second term of $\sigma_1^2$ is on the same scale as $\sigma_2^2$. The resultant limit is a mixture of the non-degenerate Gaussian component of the U-statistic $\frac{1}{n_2 (n_2-1)} \msum_{i \neq j}^{n_2} \kappa(Y_i, Y_j)$, which evaluates the kernel on the dataset $(Y_i)_{i \leq n_2}$, and the degenerate second-order component of the U-statistic $\frac{1}{n_1 (n_1-1)} \msum_{i \neq j}^{n_2} \kappa(X_i, X_j)$, which evaluates the kernel on the dataset $(X_i)_{i \leq n_1}$. When the $\Var[\argdot]$ terms are no longer $\Theta(1)$, e.g.~when they are affected by the growth of $d$, this mixture limit can arise additionally due to the relative scaling between $d$, $n_1$ and $n_2$. In \Cref{appendix:MMD}, we examine this for a high-dimensional distribution test with mean and variance shifts: For a classically non-degenerate MMD test statistic, we show how mixture limits can arise depending on how the high-dimensional regime interacts with the sample size imbalance.

\subsection{Non-classical Berry-Ess\'een bounds for high-dimensional averages and data with non-invertible covariance matrices}  \label{sec:high:d:averages}

As pointed out in \Cref{remark:compare:high:d:CLT}, the universality approximation for the degree-$2$ V-statistic $v_2(X)$ considered in \Cref{example:simple:V} is exactly the Gaussian approximation of a high-dimensional average in centred Euclidean balls. Compared to existing works \citep{sazonov1972bound,bentkus2003dependence,zhilova2020nonclassical,fang2024large}, our results accommodate data vectors with non-invertible covariance matrices. We formalise the statement here.

Since our \Cref{prop:simpler:V} handles degree-$m$ V-statistics up to $m \log m = o(\log n)$, for $m$ even, it also yields Gaussian approximation of an average with respect to centred $l_m$ balls:
\begin{align*}
    \cB_m \;\coloneqq&\; \{ B_r \,|\, r \geq 0 \} 
    &\text{ where }&&
    B_r \;\coloneqq&\; \big\{ x \in \R^d \,\big|\, \| x \|_{l_m} \leq r \big\} \;.
\end{align*}
These are the statistical objects of interest for constructing $l_m$ confidence spheres for centred empirical averages. Write $\Sigma_{jj'}$ as the $(j,j)$-th entry of $\Sigma=\Var[X_1]$. \Cref{prop:simpler:V} implies:

\begin{corollary} \label{cor:highd:CLT} Fix $\nu \in (2,3]$. Suppose $m$ is even, $2m^2 \leq n$ and $m \log m = o(n)$. Under \Cref{assumption:sub:Weibull}, there exists some absolute constant $C > 0$ such that
\begin{align*}
    \sup_{B \in \cB_m}
        \big| \P \big( S_n \in B\big) - \P\big( S^Z_n \in B \big) \big| 
    \;\leq\; 
    C m^{3 + \theta} n^{- \frac{\nu-2}{2\nu m + 2}}
    \,
    \bigg( 
            \mfrac{
                1
                +
                o(1)
            }
            {
                \frac{1}{d^2} \sum_{1 \leq j,j' \leq d}  \Sigma_{jj'}^m  
                + 
                o(1)
            }
        \bigg)^{\nu/(2 \nu m + 2)}
        \;.
\end{align*}
If additionally the entries of $X_1$'s are uniformly bounded and $\Var[X_1]=I_d$, we have
\begin{align*}
        \msup_{B \in \cB_m}
            \big| \P \big( S_n \in B\big) - \P\big( S^Z_n \in B \big) \big| 
        \;\leq\; 
        C m^{3+\theta} \,
            \Big( 
                \mfrac{
                    d (1+ o(1) )
                }
                {
                      n
                }
            \Big)^{(\nu-2)/(2 \nu m + 2)}
            \;.
\end{align*}
\end{corollary}

The first statement follows directly from \Cref{prop:simpler:V}, and the second follows from an analogous calculation as \Cref{remark:compare:high:d:CLT}(ii). Note that for $\Sigma = I_d$, the condition on dimension does not depend on $m$, whereas $m$ can interact with $d$ and $n$ for more general $\Sigma$. As our main universality result applies to a more flexible polynomial, we expect it to be applicable to constructing confidence spheres for more complicated high-dimensional statistics.

A notable feature of  \Cref{prop:simpler:V}, and subsequently \Cref{cor:highd:CLT}, is that these bounds doe not require $\Sigma$ to be invertible (see also \Cref{remark:compare:high:d:CLT}(i)). This enables us to compute the bounds easily for data vectors that may be sparse or low-ranked, which are commonly considered in statistical applications with high-dimensionality \citep{fan2011sparse,basu2015regularized,wang2016statistical}. Recall that in \Cref{example:simple:V}(iv), we have considered how sparsity can affect the dimension dependence of our bound in the case $m=2$. The same discussion applies to a general $m$: When the sparsity structure is known and exact, one can apply \Cref{prop:simpler:V} with a different coefficient tensor $S$ to improve dimension dependence; otherwise, sparsity generally hurts dimension dependence of our bound.

\subsection{Non-classical delta method for high dimensional data} \label{sec:delta:method} Our results can also be applied to develop a non-classical delta method. Let $(X_i)_{i \leq n}$ be i.i.d.~and, for a smooth function $g: \R^d \rightarrow \R$, define
\begin{align*}
    \hat g(X) 
    \;\coloneqq\; 
    g \Big( \mfrac{1}{n} \msum_{i=1}^n X_i \Big)
    \;=\; 
    g \Big( \mfrac{1}{n} \msum_{i=1}^n \bar X_i  + \mean [X_1] \Big)\;, \qquad \text{ where } \bar X_i = X_i - \mean[X_1]\;.
\end{align*}
$\hat g(X)$ is called a plug-in estimator of $g(\mean[X_1])$. A guiding example is $g_{\rm toy}(x) \coloneqq x^\top x$, which yields a simple V-statistic
\begin{align*}
    \hat g_{\rm toy}(X) 
    \;\coloneqq\; 
    g_{\rm toy} \Big( \mfrac{1}{n} \msum_{i=1}^n X_i \Big)
    \;=\;
    \mfrac{1}{n^2} \msum_{i,j=1}^n X_i^\top X_j \;. 
\end{align*}
To apply the delta method, 
% To see how the distribution of the plug-in estimator behaves through
% variance domination,
consider a Taylor expansion
\begin{align*}
    \hat g_{\rm toy}(X) 
    -
    g_{\rm toy} (\mean[X_1])
    \;=\; 
    \mfrac{2}{n} \msum_{i=1}^n \bar X_i^\top \mean[X_1]
    \,+\,
    \mfrac{1}{n^2} \msum_{i,j=1}^n \bar X_i^\top \bar X_j\;.
\end{align*}
% The variances of the first and second-order expansions are respectively on the order
% \begin{align*}
%     \Theta\Big( & \mfrac{\mean[X_1]^\top \Var[X_1] \mean[X_1]}{n} \Big)
%     \;&\text{ and }&&\;
%     \Theta\Big( &
%         \mfrac{\Var[ \bar X_1^\top \bar X_2]}{n^2} 
%         +
%         \mfrac{\Var[\bar X_1^\top \bar X_1]}{n^3}
%     \Big)
%     \;. 
%     \tagaligneq \label{eq:plug:in:toy:var}
% \end{align*}
Classically when $d$ is fixed, normality of the plug-in estimator follows from the first-order delta method whenever $\partial g(\mean[X_1]) \neq 0$. For $\hat g_{\rm toy}(X)$, this condition is equivalent to $\mean[X_1] \neq 0$. 
% and exactly corresponds to when the first-order Taylor expansion dominates in variance.
 In the high-dimensional regime, however, the moment terms above can vary with $n$ through their $d$-dependence:
 % Take $X_i \sim \cN( \mu, I_d)$ with $\|\mu\|_2=1$, in which case \eqref{eq:plug:in:toy:var} becomes
% \begin{align*}
%     \Theta\Big( & \mfrac{1}{n} \Big)
%     &\text{ and }&&
%     \Theta\Big( & \mfrac{d}{n^2} + \mfrac{d}{n^3} \Big)\;.
% \end{align*}
Even if $\partial g(\mean[X_1])$ is non-zero, the first-order term may still be asymptotically small compared to the second-order term. This is analogous to the phase transition effect observed for the U-statistics in \Cref{example:simple:U} and \Cref{sec:degree:m:U:stats}, and arises from variance domination (\Cref{thm:VD}). For delta methods, the implication is that the first-order delta method fails to hold, and the plug-in estimator is no longer asymptotically normal. The results are notationally cumbersome to state, although the intuitions resemble the discussion in \Cref{sec:degree:m:U:stats}. 
We include a formal statement and discussion of the results in \Cref{appendix:delta:method}.

 \subsection*{Acknowledgements}
 KHH acknowledges funding from the Gatsby Charitable Foundation and the EPSRC grant EP/Y028783/1 (Prob\_AI). MA was partially supported by ONR grant N000142112664. PO is supported by the Gatsby Charitable Foundation. 

% \begin{supplement}
%     \stitle{Supplementary Material to ``Gaussian universality for approximately polynomial functions of high-dimensional data"}
%     \sdescription{The supplements contain additional results and proofs.}
% \end{supplement}

\bibliographystyle{imsart-number}

\bibliography{ref}

\newpage
\setcounter{page}{1}
\thispagestyle{empty}
\appendix

\begin{center}
    % \textbf{Supplementary Material to ``Gaussian universality for approximately polynomial functions of high-dimensional data"}
    \textbf{Appendices}
\end{center}

\vspace{.5em}

\noindent
The appendix is organised as follows. \Cref{appendix:additional:results} includes additional results:
\begin{itemize}
    \item \Cref{appendix:non:multilinear} states a direct application of \Cref{thm:main} to obtain a distributional approximation for a non-multilinear polynomial by reparameterisation;
    \item \Cref{appendix:extension} discusses extensions of our results in the main text to non-polynomial functions and dependent data;
    \item \Cref{appendix:v:iid:coord} states a result related to
    the discussion in \Cref{sec:non:multilinear} that, if the coordinates of each $X_i$ are i.i.d., the obtained universality bound is dimension-agnostic.
    \item \Cref{appendix:delta:method} states and discusses the formal result of our non-classical delta method;
    \item \Cref{appendix:assumption:L_nu} provides sufficient conditions for the functional decomposition assumption of U-statistics, i.e.~Assumptions \ref{assumption:L_nu} and \ref{assumption:L_nu:MMD}, to hold.
    \item \Cref{appendix:subgraph:u} includes a further application of our degree-$m$ U-statistics results to studying subgraph count statistics.
\end{itemize}
The next set of appendices concerns our main results in \Cref{sec:main:results}:
\begin{itemize}
    \item \Cref{sec:proof:main} proves \Cref{thm:main} and \Cref{cor:non:multilinear};
    \item \Cref{sec:proof:VD} proves the variance domination result of \Cref{thm:VD};
    \item \Cref{sec:auxiliary:u} includes several results from  \cite{denker1985asymptotic}, useful for computing the moments of higher-order U-statistics;
    \item \Cref{appendix:proof:v:var:approx} proves the results in \Cref{sec:main:results} as well as \Cref{appendix:v:iid:coord}, which concern the universality of a V-statistic;
    \item \Cref{sec:construct:lower:bound} provides the explicit construction used in the lower bound of \Cref{thm:lower:bound}, and shows how it can be expressed as the degree-$m$ V-statistic and  degree-$m$ U-statistic in \Cref{cor:matching:bound:V};
    \item \Cref{sec:proof:lower:bound} proves the nearly matching bounds of \Cref{thm:lower:bound};
    \item \Cref{appendix:senatov} proves the properties of several univariate distributions considered in \Cref{sec:construct:lower:bound}. 
\end{itemize}
The remaining appendices concern proofs for the applications in \Cref{sec:applications}, \Cref{appendix:delta:method} and \Cref{appendix:subgraph:u}:
\begin{itemize}
    \item \Cref{appendix:bootstrap} proves \Cref{prop:bootstrap:inconsistency} for bootstrap inconsistency;
    \item \Cref{sec:proof:higher:u} proves \Cref{prop:higher:u} for degree-$m$ complete U-statistics; 
    \item \Cref{appendix:MMD} proves \Cref{prop:MMD:imbalanced} concerning MMD statistic under imbalanced sample sizes. We also present a simple high-dimensional distribution test setup, where a mixture limit arises as a result of both the imbalanced data and the high-dimensional regime;
    \item \Cref{appendix:high:d:average} proves \Cref{cor:highd:CLT}, the non-classical Berry-Ess\'een bound for high-dimensional averages in $l_m$ balls. This follows from the V-statistic universality result of \Cref{prop:simpler:V};
    \item \Cref{sec:proof:delta} proves \Cref{prop:delta} in \Cref{appendix:delta:method} for the non-classical delta method;
    \item \Cref{sec:proof:subgraph:u} proves the results in \Cref{appendix:subgraph:u} concerning subgraph count statistics.
\end{itemize}

\vspace{.5em}

% \appendix

% The appendices contain additional examples, a discussion on \cref{assumption:L_nu} and additional details on Senatov's results \cite{senatov1998normal} used in our proof of \cref{thm:lower:bound}.

\section{Additional results} \label{appendix:additional:results}

\subsection{Applying \Cref{thm:main} to non-multilinear polynomials} \label{appendix:non:multilinear} 

The goal of this section is to apply \Cref{thm:main} to obtain distributional approximations for a not-necessarily-multilinear polynomial function $p_m: \R^{nd} \rightarrow \R$. The idea is to rewrite $p_m(X)$ as a multilinear polynomial of some larger random vector $\bX_i$ that include information about the original data $X_i$ up to its $m$-th power, by making use of the fact that $p_m(X)$ depends on each $X_i$ through its $k$-fold tensor products $X_i^{\otimes k}$ for $k \leq m$. Here, each $X_i^{\otimes k}$ is a $\R^{d^k}$-valued tensor, whose $(j_1, \ldots, j_k)$-th entry is given by 
\begin{align*}
    (X_i)_{j_1} \times \ldots \times (X_i)_{j_k} \;.
\end{align*} 
Let $\tvec(\argdot)$ be a vectorisation operator that flattens a tensor into a vector. The dependence of $p_m(X)$ on each $X_i$ is therefore completely captured by an augmented random vector 
\begin{align*}
    \bX_i \;\coloneqq\; \tvec\big( 
        \begin{pmatrix} \, X_i \;,\;  X_i^{\otimes 2} \;,\;  \ldots \;,\;  X_i^{\otimes m} \, \end{pmatrix} 
    \big) 
    \tagaligneq \label{eq:defn:aug}
\end{align*}
with dimension ${D=d+d^2+\ldots+d^m}$. Since $p_m(X)$ is of degree at most $m$ in each $X_i$, we can parameterise $p_m(X)$ such that it is of degree at most $1$ in each $\bX_i$. More concretely, there exists a multilinear polynomial $\tilde q_m: \R^{nD} \rightarrow \R$ such that
\begin{align*}
    p_m(X) \;=\; \tilde q_m(\bX) \; \text{ almost surely, } \;\qquad\; \text{ where } \bX \;\coloneqq\; (\bX_1, \ldots, \bX_n) \;.
\end{align*}
\Cref{thm:main} now directly applies to $\tilde q_m(\bX)$, although the surrogate Gaussian variables live in a larger space $\R^D$: 
\begin{equation*}
    \Xi\;\coloneqq\;(\xi_1,\ldots,\xi_n)\qquad\text{ where }\quad \xi_1\condind\ldots\condind \xi_n\quad\text{ and }\quad \xi_i\sim\mathcal{N}(\mean[\bX_i],\Var[\bX_i])\;.
\end{equation*}
To state the next result, we write $\sigma = \sqrt{\Var[p_m(X)]} =\sqrt{\Var[\tilde q_m(\bX)]}$ as before and denote 
\begin{align*}
    \tilde M_{\nu;i} \;\coloneqq\;  \| \partial_i \tilde q_m(\bW_i)^\top ( \bX_i - \mean[\bX_i]) \|_{L_\nu}\;,
\end{align*}
where $\bW_i \coloneqq (\bX_1, \ldots, \bX_{i-1}, \bzero, \xi_{i+1}, \ldots, \xi_n) \in \R^{nD}$. $\tilde M_{\nu;i}$ is the analogue of $M_{\nu;i}$ from (2) in the main text.

\begin{corollary}[Distributional approximation for generic polynomials] \label{cor:non:multilinear} Fix $\nu \in (2,3]$. Let $p_m: \R^{nd} \rightarrow \R$ be a polynomial function. Then there exist some absolute constant $C > 0$ and a multilinear polynomial function $\tilde q_m: \R^{nD} \rightarrow \R$ such that 
    \begin{align*}
        \Big|
            \P\Big(  \mfrac{p_m(X) - \mean[p_m(X)]}{\sigma} \leq t\Big) 
            - 
            \P\Big(  \mfrac{\tilde q_m(\Xi) - \mean[p_m(X)]}{\sigma} \leq t\Big) 
        \Big|
        \;\leq\;
        C
        m
        \,
        \Big( \mfrac{ \sum_{i=1}^n \tilde  M_{\nu;i}^\nu }
        { (1+t^2)^{\nu/2} \, \tilde \sigma^\nu
        }
        \Big)^{\frac{1}{\nu m +1}}
        \;. 
    \end{align*}
    for every $n, m, d \in \N$, every $t \in \R$, and every $\sigma > 0$.
    Moreover, we have $ \mean[p_m(X)] = \mean[\tilde q_m(\Xi)]$ and $\Var[p_m(X)]= \Var[\tilde q_m(\Xi)]$. 
\end{corollary}

The precise form of $\tilde q_m$ is specified in \Cref{appendix:tensor:notation}, which involves a slightly tedious tensor notation but is nonetheless simple to identify from $p_m$. Notice that \Cref{cor:non:multilinear} is \emph{not} a universality approximation of the form
\begin{equation*}
  f_n(X_1,\ldots,X_n)\;\approx\;f_n(Z_1,\ldots,Z_n) \qquad\text{ in distribution for large }n\;,
\end{equation*}
and the surrogates $\bX_i$'s require information about $2m$ moments of $X_i$. Nevertheless, it is a valid Gaussian polynomial approximation that may be of independent interest, since it enables $p_m(X)$ to be analysed by many theoretical tools built for Gaussian polynomials. The fourth moment theorem by \cite{nualart2005central,nourdin2015optimal} is such a tool, which we have exploited in \Cref{sec:normality} to establish asymptotic normality of $p_m(X)$.

\subsection{Extensions to non-polynomial functions and dependent data}  \label{appendix:extension}

We discuss some extensions of our results in the main text.

\vspace{.5em}

\emph{Monotonic transforms of polynomials. } The simplest toy example that violates the degree condition $m=o(\log n)$ is a product of many i.i.d.~averages:
    \begin{align*}
        \Big( \Big( \mfrac{1}{n} \msum_{i \leq n} X_{i1} \Big) \times \cdots \times \Big( \mfrac{1}{n} \msum_{i \leq n} X_{im} \Big) \Big)^{1/m}
        \;,
    \end{align*}
    where $X_{11}, \ldots, X_{nm}$ are i.i.d.~univariate non-negative random variables and $m = \Omega(\log n)$. Indeed, if $m \gg n$, one should rewrite the above statistic as 
    \begin{align*}
        \exp\Big( \mfrac{1}{m} \msum_{j \leq m} \log \Big( \mfrac{1}{n} \msum_{i \leq n} X_{ij} \Big)  \Big)
        \;,
    \end{align*}
    and the correct distributional approximation is to approximate each $\log \big( \frac{1}{n} \msum_{i \leq n} X_{ij} \big)$ by a Gaussian and therefore the overall statistic by a log-normal distribution.

    It turns out that this example can be addressed by \Cref{thm:main}: Our results control a Kolmogorov distance, which is invariant under strictly monotonic transformations e.g.~$x \mapsto \exp(x)$. More generally, our results applies to any strictly monotonic transformation of a low-degree polynomial, provided that the data vector are reparameterised appropriately (e.g.~replacing $\log \big( \frac{1}{n} \msum_{i \leq n} X_{ij} \big)$ above by Gaussians instead of replacing $\frac{1}{n} \msum_{i \leq n} X_{ij}$).
    
%     \cref{thm:main} also applies to strictly monotone functions of polynomials (since
%   it controls a Kolmogorov distance, which is invariant under strictly monotone transforms).
%   For example,
%   an exponential of $q_m(\bX)$ can be approximated by an exponential of $q_m(\Xi)$. In other words, the result also applies to simple functions such as $x \mapsto \exp(x)$ that are not themselves approximatable by any degree-$m$ polynomial;

  \vspace{1em}
  
  \emph{Beyond independence}. All results so far concern independent random vectors. However, since we accommodate coordinate dependence within each data vector, the results extend trivially to block-dependent random vectors. To see this, let the block size parameter $k$ be a divisor of $n$ with $n = n' k$, and consider a block-dependent dataset $(X_1, \ldots, X_n)$ such that
  \begin{align*}
     (X_{(i-1)k + 1}, \ldots, X_{ik} ) 
     \;\condind\;
     (X_{(i'-1)k + 1}, \ldots, X_{i'k} ) 
     \quad 
     \text{ for all } i \neq i' \;.
  \end{align*}
  We can rewrite any degree-$m$ polynomial in $X_1, \ldots, X_n$ as a polynomial of the $n'$ vectors $X'_1, \ldots, X'_{n'}$, where each $X'_i$ is a flattened block vector of dimension $\R^{kd}$:
  \begin{align*}
        X'_i \;\coloneqq\; (X_{(i-1)k + 1}, \ldots, X_{ik} )\;.
  \end{align*}
  All our results are directly applicable for approximating $X'_i$ by Gaussians. The effects of coordinate-wise dependence on the speed at which $d$ may grow in $n$, discussed in \Cref{sec:non:multilinear}, now extend to understanding the effects of data-wise dependence on how fast $k$ may grow in $n$. For example, when $d=1$, we can rewrite
    \begin{align*}
        \mfrac{1}{\sqrt{n}} \msum_{i=1}^n X_i \;=\; \mfrac{1}{\sqrt{n}} \msum_{i=1}^{n'}  (1, \ldots, 1) X'_i\;,   
    \end{align*}
    and applying a similar argument to that of \Cref{example:break:moment:cond} gives the condition $k=o(n^{1/3})$. This matches the condition that appears in locally dependent CLT bounds established by Stein's method \citep{chen2004normal,ross2011fundamentals,liu2023local}, and reaffirms our observation that the $d=o(n^{1/3})$ condition in \cite{zhilova2020nonclassical} is closely connected to the knowledge of coordinate-wise dependency structure. Moreover, in view of the many works that use results with growing block sizes to establish results under mixing conditions \citep{bernstein1927extension,ibragimov1975independent,davidson1992central}, we expect our result to hold also under certain mixing conditions. In particular, notice that the distributional approximation of a degree-$m$ polynomial of $X$ implicitly controls the $m$-th moment of $X$. In view of recent results on weakly dependent CLT bounds in Wasserstein-$p$ distance \cite{liu2023mixing}, where the mixing condition depends on $p$, we expect that a generalisation of our result to the mixing case will require the rate of mixing to depend on $m$.

\subsection{Dimension-agnostic universality bound for i.i.d.~coordinates} \label{appendix:v:iid:coord}

The next result formalises the discussion in \Cref{sec:non:multilinear} that, if the coordinates of each $X_i$ are i.i.d., the obtained universality bound is dimension-agnostic.

\begin{lemma} \label{lem:simple:V:independent:coords} Assume the setup of \Cref{example:simple:V} and additionally that $X_1$ consists of i.i.d.~coordinates. Then there exists some absolute constant $C > 0$ such that
\begin{align*}
    &\;
        \msup_{t \in \R} \big| \P( v_2(X) \leq t) - \P(  v_2(Z) \leq t) \big|
        \\
        &\,\;\leq\;\,
        C  (nd)^{-\frac{\nu-2}{4\nu+ 2}}
            \,
            \bigg( 
                \mfrac{
                    \| X_{11} \|_{L_\nu}^4 
                    + O(n^{-1} )
                }
                {
                        \| X_{11} \|_{L_2}^4 
                        + O(n^{-1} )
                }
            \bigg)^{\frac{\nu}{4 \nu + 2}}
        +
        C  n^{ - \frac{1}{5} } \,
        \Big( 
        \mfrac{ 
            1
        }{
                \| X_{11} \|_{L_2}^2 
                + O(n^{-1} )
        } 
        \Big)^{\frac{1}{5}}
        \;.
\end{align*} 
\end{lemma}

We include the proof in \Cref{appendix:proof:simple:V:iid}.

\begin{remark*} The second term arises due to the use of variance domination (\Cref{thm:VD}) to obtain a universality approximation. As with the discussion after \Cref{prop:simpler:V}, this additional term can be omitted if we only need a distributional approximation. In this case, as $d$ grows, the observations of more and more i.i.d.~entries helps with the convergence of the bound.
\end{remark*}

\subsection{Results on the non-classical delta method} \label{appendix:delta:method}

 Our V-statistic result (\Cref{prop:simpler:V}) 
%  \Cref{prop:simpler:V}
 can be combined with variance domination (\Cref{thm:VD}) 
%  (\Cref{thm:VD}) 
 to develop a non-classical delta method. Let $(X_i)_{i \leq n}$ be i.i.d.~and, for a smooth function $g: \R^d \rightarrow \R$, define
\begin{align*}
    \hat g(X) 
    \;\coloneqq\; 
    g \Big( \mfrac{1}{n} \msum_{i=1}^n X_i \Big)
    \;=\; 
    g \Big( \mfrac{1}{n} \msum_{i=1}^n \bar X_i  + \mean [X_1] \Big)\;, \qquad \text{ where } \bar X_i = X_i - \mean[X_1]\;.
\end{align*}
$\hat g(X)$ is called a plug-in estimator of $g(\mean[X_1])$. A guiding example is $g_{\rm toy}(x) \coloneqq x^\top x$, which yields a simple V-statistic
\begin{align*}
    \hat g_{\rm toy}(X) 
    \;\coloneqq\; 
    g_{\rm toy} \Big( \mfrac{1}{n} \msum_{i=1}^n X_i \Big)
    \;=\;
    \mfrac{1}{n^2} \msum_{i,j=1}^n X_i^\top X_j \;. 
\end{align*}
To apply the delta method, 
% To see how the distribution of the plug-in estimator behaves through
% variance domination,
consider a Taylor expansion
\begin{align*}
    \hat g_{\rm toy}(X) 
    -
    g_{\rm toy} (\mean[X_1])
    \;=\; 
    \mfrac{2}{n} \msum_{i=1}^n \bar X_i^\top \mean[X_1]
    \,+\,
    \mfrac{1}{n^2} \msum_{i,j=1}^n \bar X_i^\top \bar X_j\;.
\end{align*}
% The variances of the first and second-order expansions are respectively on the order
% \begin{align*}
%     \Theta\Big( & \mfrac{\mean[X_1]^\top \Var[X_1] \mean[X_1]}{n} \Big)
%     \;&\text{ and }&&\;
%     \Theta\Big( &
%         \mfrac{\Var[ \bar X_1^\top \bar X_2]}{n^2} 
%         +
%         \mfrac{\Var[\bar X_1^\top \bar X_1]}{n^3}
%     \Big)
%     \;. 
%     \tagaligneq \label{eq:plug:in:toy:var}
% \end{align*}
Classically when $d$ is fixed, normality of the plug-in estimator follows from the first-order delta method whenever $\partial g(\mean[X_1]) \neq 0$. For $\hat g_{\rm toy}(X)$, this condition is equivalent to $\mean[X_1] \neq 0$. 
% and exactly corresponds to when the first-order Taylor expansion dominates in variance.
 In the high-dimensional regime, however, the moment terms above can vary with $n$ through their $d$-dependence:
 % Take $X_i \sim \cN( \mu, I_d)$ with $\|\mu\|_2=1$, in which case \eqref{eq:plug:in:toy:var} becomes
% \begin{align*}
%     \Theta\Big( & \mfrac{1}{n} \Big)
%     &\text{ and }&&
%     \Theta\Big( & \mfrac{d}{n^2} + \mfrac{d}{n^3} \Big)\;.
% \end{align*}
Even if $\partial g(\mean[X_1])$ is non-zero, the first-order term may still be asymptotically small compared to the second-order term. This is exactly analogous to the phase transition effect observed for the U-statistic considered in 
\Cref{example:simple:U} and \Cref{sec:degree:m:U:stats},
% \Cref{example:simple:U} and \Cref{sec:degree:m:U:stats}, 
and is a result of variance domination (\Cref{thm:VD}).
% (\Cref{thm:VD}). 
For delta methods, the implication is that the first-order delta method fails to hold, and the plug-in estimator is no longer asymptotically normal. We now formalise this in a result concerning an $m$-th order delta method.

\vspace{.5em}

\noindent 
\textbf{Derivatives.} A delta method relies on the derivatives of $g$, and some shorthands are introduced next. 
For
$x, x_1, \ldots, x_m \in \R^d$ 
and 
$m \in \N \cup \{0\}$, we define 
\begin{align*}
    g_m^{(x)} (x_1, \ldots, x_m)
    \;\coloneqq&\; 
    \mfrac{1}{m!}
    \big\langle 
        \partial_m\, g(x) 
        \,,\,
       (x_1 - \mean X_1) \otimes \ldots \otimes (x_m - \mean X_1)
    \big\rangle
    \\
    \hat g_m^{(\mean X_1)}(x_1, \ldots, x_n)
    \;\coloneqq&\;
     n^{-m} \,
    \msum_{i_1, \ldots, i_m \in [n]} \,
    g_m^{(\mean X_1)}\big( x_{i_1} , \ldots ,  x_{i_m} \big)
    \;.
\end{align*}
which is a degree-$m$ polynomial function. If $g$ is $(m+1)$-times continuously differentiable everywhere, we can perform an $m$-th order Taylor expansion with the integral remainder. Let $S_n \coloneqq n^{-1} \sum_{i \leq n} \bar X_i$ and draw $\Theta \sim \textrm{Uniform}[0,1]$ independently of $X$. Then almost surely
\begin{align*}
   &\hat g(X) 
   \;=\; 
   g( \mean X_1 ) + \msum_{j=1}^{m} \hat g_j^{(\mean X_1)}(X) 
   \,+\,
   (m+1)
   \mean\big[ 
   (1-\Theta)^m
   \,
   \hat g_{m+1}^{(\mean X_1 + \Theta S_n )}(X)
   \,\big|\, X \big]
   \;.
\end{align*}
Note that when $\hat g_{m+1}^{(\mean X_1 + \Theta S_n )}$ is independent of $\Theta$, the $(m+1)$ factor is cancelled out by $\mean[(1-\Theta)^m] = (m+1)^{-1}$. 

The penultimate approximation will be identified by comparing which term of the Taylor expansion dominates in variance. If we suppose that the $k$-th order term dominates and that the $(m+1)$-th term onwards is negligible, the error incurred by variance domination (\Cref{thm:VD})
% (\Cref{thm:VD}) 
is the sum of ratios 
\begin{align*}
    \epsilon_m
    \;\coloneqq\;
    \mfrac
    {\Var\Big[ \sum_{l=1}^{m-1} \hat g^{(\mean X_1)}_{l}(X) \Big]}
    {\Var\Big[ \hat g_m^{(\mean X_1)}(X) \Big]}
    +
    \mfrac
    {
    (m+1)^2
    \Big\|
     \mean\Big[ 
    (1-\Theta)^m
       \,
       \hat g_{m+1}^{\big(\mean X_1 + \Theta n^{-1} \sum_{i \leq n} \bar X_i \big)}(X)
    \Big| X \Big]
    \Big\|_{L_2}^2
    }
    {\Var\Big[ \hat g_m^{(\mean X_1)}(X) \Big]}
    \;.
\end{align*}
% \vspace{.5em}
% For simplicity of notation, we will assume that $\hat g^{(\mean X_1)}_m$ is $\delta$-regular with respect to $(\bar X_i)_{i \leq n}$ (see \cref{defn:delta:regular} and \cref{remark:delta:reg:crude} for a sufficient condition), so that we can apply universality directly to $X$ and not just the augmented variables. In the case of $g_{\rm toy}$ with $\bar X_i \sim \cN(0,I_d)$, this corresponds to requiring the variance term involving $\bar X_1^\top \bar X_1$ to be negligible, which always holds in view of \eqref{eq:plug:in:toy:var}.
% \noindent 
% \textbf{Moments.} 
Each $m$-th order term also carries a bias, measured by 
% The result involves several moment terms. For each $m \in \N \cup \{0\}$, define
\begin{align*}
    \mu_m \;\coloneqq\; \mean  \, \hat g_m^{(\mean X_1)}(X)
    \;.
\end{align*}
Note that $\mu_0 = g(\mean X_1)$, which is the target of estimation under the plug-in principle. 

\vspace{.5em}

\begin{proposition} \label{prop:delta}  Let $g$ be $(m+1)$-times continuously differentiable. Suppose \Cref{assumption:sub:Weibull}
    % \Cref{assumption:sub:Weibull} 
holds and assume that $m \log m  = o(n)$ and $2m^2 < n$. If $ m \epsilon_m^{1/(2 m+1)} = o(1)$ and 
\begin{align*}
    \max\Big\{ \mfrac{m^{m (4 + 2 \theta)}}{n}, 1\Big\} \times \| \partial_m g(\mean[X_1]) \|_{l_1}^2
    \;=\; 
    o\big( \big\| \langle \partial_m g(\mean[X_1]), \bar X_1 \otimes \ldots \otimes \bar X_m \rangle \big\|_{L_2}^2 \big)
    \;,
\end{align*}
then under the non-classical asymptotic as $n \rightarrow \infty$,
\begin{align*}
    \msup_{t \in \R}  
        \big| \P\big(  \hat g(X) -  \msum_{l=0}^{m} \mu_l \leq t\big)
        -  
        \P\big( \,   \hat g_m^{(\mean X_1)}(Z) - \mean\big[ \hat g_m^{(\mean X_1)}(Z)\big]  \leq t \, \big) 
        \big| 
    \;\rightarrow\; 0\;.
\end{align*}
If additionally $\Kurt \big[\hat g_{m}^{(\mean X_1)}(Z)\big] \rightarrow 0$, then we also have 
\begin{align*}  
        \Var\big[ \hat g_{m}^{(\mean X_1)}(&X)\big]^{-1/2}  \big(\hat g(X) -  \msum_{l=0}^{m} \mu_l \big) 
        \;\xrightarrow{d}\;
        \cN(0,1)\;.
\end{align*}
\end{proposition}

% \vspace{.5em}

\begin{remark} (i) The condition on $\epsilon_m$ requires that the variance domination error term is small. (ii) The condition on $\partial_m g(\mean[X_1])$ is such that the approximation error in the V-statistic result (\Cref{prop:simpler:V})
    % (\Cref{prop:simpler:V}) 
is small, where we can identify $S=\partial_m g(\mean[X_1])$. See 
\Cref{sec:non:multilinear}
% \Cref{sec:non:multilinear} 
for a discussion on what this implies for the structure of $S$. (iii) \cref{prop:delta} considers only the case when one of the derivatives dominates, but can be extended to the case when several terms dominate at the same time. This is done by applying variance domination 
(\Cref{thm:VD})
% (\cref{thm:VD})
to the appropriate dominating quantity. 
\end{remark}

% \vspace{.5em}

% \vspace{1em}

When $g$ is smooth, \cref{prop:delta} implies that the asymptotic distribution of $\hat g(X)$ is determined by its $m$-th order Taylor expansion for the smallest $m$ such that $\epsilon_m \rightarrow 0$, i.e.~the $m$-th order term that dominates in variance. 
% As such, 
\cref{prop:delta} strictly generalises the $m$-th order delta method:
 The classical requirement for all lower derivatives to be zero 
 is now replaced with the condition on $\epsilon_m$. The condition on $\partial_m g(\mean[X_1])$ holds automatically for $d$ and $m$ fixed, and only plays a role in the high-dimensional regime.
%  , whereas the finite-sample bound introduces a second condition $\beta_{m,\nu} = o(n^{(\nu-2)/2\nu})$ analogous to a Berry-Ess\'een moment ratio.
%  In particular, under these two conditions, the bound does not have any additional dependence on $d$.
%  In the high-dimensional regime, 
Similar to U-statistics 
% (\Cref{sec:degree:m:U:stats}), 
(\Cref{sec:degree:m:U:stats}), 
\cref{prop:delta} 
%  generalises the observations in $g_{\rm toy}$ to 
 shows two aspects of delta method that depart from classical behaviours:

\begin{itemize}
    \item \emph{Non-Gaussianity despite a non-zero first derivative. } The approximation by a degree-$m$ Gaussian polynomial can dominate the first-order Gaussian term, e.g.~when the first derivative is negligible compared to some $m$-th derivative. 
    If the excess kurtosis of this degree-$m$ Gaussian polynomial does not vanish, 
    \Cref{prop:gaussian}
    % \cref{prop:gaussian} 
    implies that $\hat g(X)$ is not  asymptotically normal under a further uniform integrability condition. In particular, this may hold despite a non-zero first derivative term. On the other hand, 
    \Cref{example:degen:V:normal} 
    demonstrates how the higher degree polynomial may also be asymptotically normal.
    % We do note that this polynomial of Gaussians may still be asymptotically Gaussian, and the condition for Gaussianity can be characterised by the fourth moment phenomenon \cite{nualart2005central, nourdin2015optimal,nourdin2016multidimensional} when $X_i$'s have isotropic covariance matrices 
    % -- see \cref{remark:delta:fourth:moment}.
    \vspace{.5em}
    \item \emph{Non-consistency (under the plug-in principle). } Provided that 
    % By applying a further Markov's inequality on $\P\big( \,   \hat g_m^{(\mean X_1)}(Z) - \mean\big[ \hat g_m^{(\mean X_1)}(Z)\big]  \leq t \, \big)$, we see that if $ \Delta_{\hat g, m} = o(1)$ and 
    $\Var\big[ \hat g_{m}^{(\mean X_1)}(Z)\big]= o (1)$, \Cref{prop:delta} and the Markov inequality imply that
    \begin{align*}
        \hat g(X) - \msum_{l=0}^m \mu_l \;\xrightarrow{\P}\; 0\;.
    \end{align*}
    Under the plug-in principle, $\hat g(X)$ is intended as an estimator for $g(\mean X_1) = \mu_0$. Since
    \begin{align*}
        \mu_1 \;=\; \mean\big[\langle \partial g(\mean X_1), X_1 - \mean X_1 \rangle \big] \;= \;0 \;,
    \end{align*}
    $\hat g(X)$ is indeed consistent when \cref{prop:delta} holds with $m=1$ (e.g.~in the classical case with non-zero first derivative). When $m > 1$, however, $\hat g(X)$ may be non-consistent, and the bias is exactly given by the limit of $\sum_{l=2}^m \mu_l$.
\end{itemize}

% \vspace{.5em}

% \begin{remark}[Gaussianity under a zero first derivative] \label{remark:delta:fourth:moment} \cref{prop:gaussian} also implies that where $m$ and $d$ are allowed to grow, a degree-$m$ polynomial of $d$-dimensional Gaussians may itself be asymptotically Gaussian: This can be seen from $\hat g_{\rm toy}(X)$ with $X_i \sim \cN(\mu, I_d)$, whose second-order term satisfies
% \begin{align*}
%     \mfrac{1}{n^2} \msum_{i,j=1}^n \bar X_i^\top \bar X_j
%     \;\overset{d}{=}\;
%     \mfrac{d}{n} \times \mfrac{1}{d} \msum_{j=1}^d \eta_{j}^2
%     \qquad 
%     \text{ with } \eta_j \overset{i.i.d.}{\sim} \cN(0,1)\;.
% \end{align*}
% This is an average of i.i.d.~variables, and becomes asymptotically normal when $d$ and $n$ both grow. A consequence is that in the high-dimensional regime, Gaussianity may also occur even when the first derivative is zero. A phase transtion from the Gaussian limit to the chi-squared limit can happen, if the second-order term dominates as $d$ becomes large.
% \end{remark}
% \vspace{-1em}

\Cref{thm:main} and \Cref{thm:VD}
% \cref{thm:main} and \cref{thm:VD} 
in fact apply to a more general version of delta method on $g_n(T_n)$, where we consider a sequence of random variables $\{T_n\}_{n=1}^\infty$, with each $T_n$ determined by $X_1, \ldots, X_n$, and a sequence of well-behaved functions $\{g_n\}_{n=1}^\infty$. The only requirement is a suitable Taylor approximation of the estimator, which yields the polynomial structure. The degree-$m$ U-statistics in \Cref{sec:degree:m:U:stats} 
% \Cref{sec:degree:m:U:stats} 
are such an example, and the Taylor approximation is obtained from the functional decomposition in \Cref{assumption:L_nu}.
% \Cref{assumption:L_nu}. 
Moreover, by the discussion in \Cref{appendix:extension}, \Cref{prop:delta} also applies to strictly monotonic functions of $\hat g(X)$, although the distributional approximation may no longer be Gaussian polynomials.

\begin{remark}[Upper bound on $\epsilon_m$] \label{remark:delta:epsilon} 
(i) The first term of $\epsilon_m$ can be controlled by bounding the variance ratios $\Var\big[ \hat g_l^{(\mean X_1)}(X) \big] \, /\, \Var\big[ \hat g_m^{(\mean X_1)}(X) \big]$ for $l \leq m-1$. Since each $\hat g_l^{(\mean X_1)}(X)$ involves a Euclidean product with the flattened vector $\tvec(\partial_l g(\mean[X_1]))$, a crude bound may be obtained from the maximum eigenvalue of $\tvec(\partial_l g(\mean[X_1]))\tvec(\partial_l g(\mean[X_1]))^\top$ and the minimum eigenvalue of  $\tvec(\partial_m g(\mean[X_1]))\tvec(\partial_m g(\mean[X_1]))^\top$. (ii) An analogous argument 
% Using a triangle inequality and a Jensen's inequality, we can obtain a further upper bound on the first term of $\epsilon_m$ as
% \begin{align*}
%     (m-1) \, \sum_{l=1}^{m-1} 
%     \mfrac{\Var\Big[ \hat g_l^{(\mean X_1)}(X) \Big]}{\Var\Big[ \hat g_m^{(\mean X_1)}(X) \Big]}\;.
% \end{align*}
% Each summand involves variances of V-statistics, whose upper and lower bounds are given in \cref{lem:universality:v:stat:aug}. As discussed in \cref{remark:delta:reg:crude,remark:universality:v:stat:berry:esseen}, these variance ratios can either be controlled by a further spectral bound or an explicit computation. An analogous bound 
applies to the second term, provided that the $(m+1)$-th derivative satisfies some stability condition: One example of such a condition is the existence of some $\nu > 2$ such that
\begin{align*}
    \Big\|
    \, \hat g_{m+1}^{(\mean X_1 + \Theta S_n )}(X)
    -
    \hat g_{m+1}^{(\mean X_1)}(X)
    \,
    \Big\|_{L_\nu}
    \;=\; o(1) \;.
\end{align*}
\end{remark}

\color{black}

\subsection{Validity of Assumptions \ref{assumption:L_nu} and \ref{assumption:L_nu:MMD}: Functional decomposition of U-statistics} \label{appendix:assumption:L_nu} 

Denote $L_2(\cE, \mu)$ as the $L_2$ space of $\cE \rightarrow \R$ functions under $\mu$.  We first show that \Cref{assumption:L_nu} in \Cref{sec:degree:m:U:stats}, used for U-statistics, 
% \cref{assumption:L_nu} 
holds under mild conditions when $\nu=2$.

\begin{lemma} \label{lem:L2:decomposition} Fix $n,m,d \in \N$, $\nu \geq 1$ and $(\cE,\mu)$ a separable measure space. Assume that $\| u(Y_1, \ldots, Y_m) \|_{L_2} < \infty$. Then there exists a sequence of orthonormal $L_2(\cE,\mu)$ functions, $\{\phi_k\}_{k=1}^\infty$, and an array of real values $\{\lambda_{k_1 \ldots k_m}\}_{k_1,\ldots, k_m=1}^\infty$, such that 
\begin{align*}
    \varepsilon_{K;2} \;=\; \Big\| \msum_{k_1, \ldots, k_m=1}^K
    \lambda_{k_1 \ldots k_m}  \phi_{k_1}(Y_1) \ldots \phi_{k_m}(Y_m)
    -
    u(Y_1, \ldots, Y_m) \Big\|_{L_2}
    \;\xrightarrow{K \rightarrow \infty}\; 0 \;.
\end{align*}
\end{lemma}

\begin{proof} Separability of $(\cE,\mu)$ implies that $L_2(\cE,\mu)$ is a separable Hilbert space (see e.g.~Exercise 10(b), Chapter 1, \cite{stein2011functional}), which implies the existence of a countable orthonormal basis $\{\phi_k\}_{k=1}^\infty$. Consider the $m$-fold product space $(\cE^m, \mu^m)$. Then the collection of pointwise products $\{\phi_{k_1} \times \ldots \times \phi_{k_m}\}_{k_1,\ldots,k_m=1}^\infty$ forms an orthonormal basis in $L_2(\cE^m, \mu^m)$. The stated assumption implies $u \in L_2(\cE^m, \mu^m)$, which implies the desired result.
\end{proof}

\begin{remark*} We emphasise that the moment boundedness assumption is only required for every \emph{fixed} $n$, $m$ and $d$. Therefore, this does not contradict our overall analysis, which considers how the moments can be large relative to $n$ (through e.g.~dependence on $d$ and $m$).
\end{remark*}

In the next lemma, we give another choice of approximating functions provided that $u$ is well-approximated by its Taylor expansion in $L_\nu$. For each $k \in K$, let $M_k \in \N$ be the largest number such that 
\begin{align*}
    \msum_{M=1}^{M_k - 1} d^M 
    \;<\; 
    k 
    \;\leq\;  
    \msum_{M=1}^{M_k} d^M\;.
\end{align*}
For $y \in \R^d$, let $(y^{\otimes M})_t$ be the $t$-th coordinate of the tensor $y^{\otimes M}$ according to a fixed total order on $\N^m$, and define
\begin{align*}
    \tilde \phi_k( y ) \;\coloneqq\;  \big( y^{\otimes M_k} \big)_{k - \sum_{M=1}^{M_k-1} d^M}\;.
\end{align*}
For example, $\tilde \phi_1(y), \ldots, \tilde \phi_d(y)$ are the $d$ coordinates of $y$, $\tilde \phi_{d+1}(y), \ldots, \tilde \phi_{d+d^2}(y)$ are the $d^2$ coordinates of $y^{\otimes 2}$ and so on.

\begin{lemma} \label{lem:assumption:taylor} Fix $\nu \in (2,3]$. Assume that $u$ is infinitely differentiable and
\begin{align*}
    \big\| 
        u(Y_1, \ldots, Y_m) 
        -  
        \msum_{l=0}^L 
        \partial^l u(\bzero) (Y_1, \ldots, Y_m)^{\otimes l}
    \big\|_{L_\nu} 
    \;\xrightarrow{L \rightarrow \infty}\; 0\;.
\end{align*}
Then \Cref{assumption:L_nu} 
% \cref{assumption:L_nu} 
holds with $\{\tilde \phi_k\}_{k \in \N}$ defined above and some $\{ \tilde \lambda_{k_1 \ldots k_m}\}_{k_1, \ldots, k_m \in \N}$.
\end{lemma}

\begin{proof}[Proof of \cref{lem:assumption:taylor}] Each coordinate of $(Y_1, \ldots, Y_m)^{\otimes l}$ can be written as a product of $\tilde \phi_{k_1}(Y_1), \ldots, \tilde \phi_{k_m}(Y_m)$ for some $k_1, \ldots, k_m \leq \sum_{M=1}^l d^M$. By identifying $\tilde \lambda_{k_1 \ldots k_m}$'s as the corresponding coordinate in $\partial^l u(\bzero)$, we see that the Taylor approximation error above is exactly $\epsilon_{K;\nu}$, which converges to zero by assumption.
\end{proof}

In the case of the MMD statistic considered in \Cref{sec:MMD:imbalance}, for balanced sample sizes, \cite{huang2023high} verifies that \Cref{assumption:L_nu:MMD} holds with the Mercer decomposition of a kernel \cite{steinwart2012mercer} under mild moment conditions, and also adopts the Taylor expansion approach of \Cref{lem:L2:decomposition} to verify it directly for the Gaussian kernel and the linear kernel. Their proof directly extends to the case with imbalanced sample sizes.

\subsection{Higher-order fluctuations of subgraph counts} 
\label{appendix:subgraph:u} 
We apply our results to characterise the possible asymptotic distributions of subgraph counts. We shall see that variance domination recovers and extends the geometric conditions considered by Hladk{\`y} et al.~\cite{hladky2021limit} and Bhattacharya et al.~\cite{bhattacharya2023fluctuations} for obtaining different limits, and our results also apply to the edge-level fluctuations considered by Kaur and R{\"o}llin \cite{kaur2021higher}.
 
\vspace{.5em} 
 
\noindent
 \textbf{Graph model. } Given a symmetric, measurable function $w: [0,1]^2 \rightarrow [0,1]$, we generate a graph $G_n$ with vertex set $[n]$ by drawing independent (not necessarily identically distributed) random variables $ (U_i)_{i \leq n}$ and $(V_{ij})_{1 \leq i < j \leq n}$, and joining an edge $i \sim j$ according to
 \begin{align*}
    Y_{ij} \;\coloneqq\; \ind_{\{V_{ij} \leq w(U_i, U_j)\}}\;.   
 \end{align*}
When $U_i, V_{ij} \overset{\rm i.i.d.}{\sim} \textrm{Uniform}[0,1]$, this generates an exchangeable graph from $w$ \citep{diaconis2007graph}.

\vspace{.5em} 

\noindent
\textbf{Subgraph count statistics. } Denote $K(S)$ as the complete graph on the vertex set $S \subseteq [n]$, $K_n \coloneqq K(\{1,\ldots,n\})$ and $E(H')$ as the edge set of a graph $H'$. Given a non-empty simple graph $H$ with $m = m(n)$ vertices and $k = k(n)$ edges, we are interested in the number of subgraphs in $G_n$ that are isomorphic to $H$, described by
\begin{align*}
    \kappa(Y)
    \;\coloneqq\;
    \msum_{1 \leq i_1 < \ldots < i_m \leq n}
    \; 
    \msum_{H' \subseteq \cG_H(\{i_1, \ldots, i_m\})}
    \; 
    \mprod_{(i_s,i_t) \in E(H')} Y_{i_s i_t}\,,
\end{align*}
where $\cG_{H}(S)$ denotes the collection of all subgraphs of $K(S)$ that are isomorphic to $H$.

\vspace{.5em} 

Consider the conditionally centred indicator variable $\bar Y_{ij} \coloneqq Y_{ij} - w(U_i,U_j)$. To derive the asymptotic limit of $\kappa(Y)$, we first split $\kappa(Y)$ into a sum over vertices and a sum over edges:
\begin{align*}
    \kappa(Y)
    \;=&\;
    \msum_{1 \leq i_1 < \ldots < i_m \leq n}
    \; 
    \msum_{H' \subseteq \cG_H(\{i_1, \ldots, i_m\})}
    \; 
    \mprod_{(i_s,i_t) \in E(H')} w(U_{i_s}, U_{i_t})
    \\
    &\; 
    +
    \msum_{\substack{
        \{(i_1, j_1), \ldots, (i_k,j_k)\} \subseteq E(K_n) \\
        \text{is a set of distinct edges} }}
    \; 
    \delta_H\big( \{(i_s, j_s)\}_{s \in [k]} \big)
    \,
    \mprod_{s=1}^k \bar Y_{i_s j_s} 
    \;\eqqcolon\; 
    \kappa_1(U) + \kappa_2(\bar Y)\;,
\end{align*}
where $\delta_H(I)$ is the indicator function on whether the graph formed by the edge set $I$ is isomorphic to $H$. $\kappa_1(U)$ 
represents vertex-level fluctuations, and  $\kappa_2(\bar Y)$ represents edge-level fluctuations. We first study the two fluctuations separately, and show how they affect the distribution on $\kappa(Y)$ by variance domination.

\vspace{.5em} 

\noindent
\textbf{Vertex-level fluctuations. } Let $U \coloneqq (U_i)_{i \leq n}$ be i.i.d.~for notational simplicity. $\kappa_1(U)$ is now the degree-$m$ U-statistic considered in 
\Cref{prop:higher:u},
% \cref{prop:higher:u}, 
with the kernel
\begin{align*}
    u_1(x_1, \ldots, x_m) \;\coloneqq\;  
    \mbinom{n}{m}
    \;
    \msum_{H' \subseteq \cG_H(\{1, \ldots, m\})}
    \; 
    \mprod_{(i_s,i_t) \in E(H')} w(x_{i_s}, x_{i_t})\;.
\end{align*}
Applying \Cref{prop:higher:u}
% \cref{prop:higher:u} 
directly gives the following corollary, which approximates $\kappa_1(U)$ by its $M$-th Hoeffding decomposition. The notation $U^{(K)}_M(\Xi^{(K)})$, $\tilde \beta_{M,\nu}$, $\sigma^2_{m,n;M}$ and $\rho_{m,n;M}$ are defined in terms of the U-statistic $\kappa_1(U)$.

\begin{corollary}\label{cor:subgraph:vertex} Suppose 
    \Cref{assumption:L_nu}
    % \cref{assumption:L_nu} 
    holds for the kernel function $u_1$ with some $\nu \in (2,3]$. Then there is some absolute constant $C > 0$ s.t.~for every $n,m,d \in \N$ and $M \in [m]$, 
\begin{align*}
    \sup_{t \in \R} 
    \;
    \Big| \P\big( 
         \kappa_1(U) - \mean[ \kappa_1(U)]  \leq t
    \big) 
    - 
    \lim_{K \rightarrow \infty}
    \P\Big( & \,\mbinom{m}{M}   U_M^{(K)}(\Xi^{(K)})  \leq t\Big) \Big|
    \\
    &\;\leq\; 
    C m
    \Big(
    n^{-\frac{\nu-2}{2(\nu M + 1)}} 
    \tilde \beta_{M,\nu}^{\frac{\nu}{\nu M +1}}
    +
    \rho_{m,n;M}^{\frac{1}{2M+1}}
    \Big)
    \;.
\end{align*}
The variance ratio $\rho_{m,n;M} = \frac{\sum_{r \in [m] \setminus M}  \sigma_{m,n;r}^2 }{\sigma_{m,n;M}^2}$ can be computed in terms of 
\begin{align*}
    \sigma_{m,n;r}^2
    \;=&\; 
    \mbinom{m}{r}
    \mbinom{n}{m}
    \mbinom{n-r}{m-r}
    \,
    \Var\, \mean\Bigg[    
    \sum_{H' \subseteq \cG_H(\{1, \ldots,m\})}
    \; 
    \prod_{(i_s,i_t) \in E(H')} w(U_{i_s}, U_{i_t})
    \, \Bigg|\, 
    U_1, \ldots, U_r \Bigg]\;.
\end{align*}
\end{corollary}

\begin{remark} A common tool for proving limit theorems for subgraph counts \cite{kaur2021higher,bhattacharya2023fluctuations} is the orthogonal decomposition of a generalised U-statistic proposed by Janson and Nowicki \cite{janson1991asymptotic}. Here, this assumption takes the form of
\Cref{assumption:L_nu}: 
% \cref{assumption:L_nu}: 
the difference is that we focus only on vertices, and perform a decomposition in an $L_\nu$ space with $\nu > 2$ (see 
Remark 5).
% \cref{remark:decomposition}).

% The or\cref{assumption:L_nu} 

% In fact, the orthogonal decomposition proposed by \cite{janson1991asymptotic} is exactly (i) with $\nu=2$ with i.i.d.~$U_i$'s and i.i.d.~$V_{ij}$'s: This is analogous to how \cref{assumption:L_nu} for degree-two polynomials is motivated by the $L_2$ decomposition of a Hilbert-Schmidt operator used for studying a degenerate degree-two U-statistic. The structure of $u_*$ will allow the decompositions in (ii) to be interpreted as projections onto different subgraphs of $H$. Penultimate approximations can then be obtained via \cref{thm:main} and \cref{thm:VD} with finite-sample guarantees.

\end{remark}

% \vspace{.5em}

As before, the size of $\sigma_{m,n;M}^2$ -- the variance of the $M$-th order Hoeffding decomposition -- determines whether $\kappa_1(U)$ can be approximated by a degree-$M$ polynomial of Gaussians, This is known in the literature as the $n^{m-\frac{M}{2}}$-th order fluctuation \citep{kaur2021higher}, as evident from the $\binom{n}{m}\binom{n-M}{m-M} = O(n^{2m - M})$ factor in $\sigma^2_{m,n;M}$. A guiding example is when $H = K_2$, i.e.~$\kappa(Y)$ is the edge count: In this case, the variances of the Gaussian component and the quadratic-form-of-Gaussian component are respectively
\begin{align*}
    \sigma_{m,n;1}^2
    \;=&\; 
    n(n-1)^2
    \,
    \Var \, \mean[ w(U_1, U_2) | U_1 ]
    &\text{ and }&&
    \sigma_{m,n;2}^2
    \;=&\; 
    \mfrac{n(n-1)}{2}
    \Var [ w(U_1, U_2) ]\;,
\end{align*}
which are the variances of the $n^{3/2}$-th order fluctuation and the $n$-th order fluctuation. 

\vspace{.5em}

The next result gives conditions under which variance domination by the Gaussian component ($M=1$, or $n^{m-\frac{1}{2}}$-th order fluctuation) fails.

\vspace{.5em}

\begin{lemma} \label{lem:subgraph:vertex:conditions} Let $w$, $H$ and $m \geq 2$ be fixed. Then $\sigma^2_{m,n;r} = O(n^{2m - r})$ for all $r \in [m]$. Moreover, the following statements are equivalent:
    \begin{proplist}
        \item $\rho^2_{m,n;1} = \Omega(1)$;
        \item $\sigma^2_{m,n;1} = 0$;
        \item $\sum_{H' \subseteq \cG_H(\{1, \ldots, m\})}
        \; 
        \mean\Big[
        \prod_{(i_s,i_t) \in E(H')} w(U_{i_s}, U_{i_t})
        \, \Big|\, 
        U_1 \Big]$
         is constant almost surely;
         \item For almost every $x \in [0,1]$,
         \begin{align*}
             \mfrac{1}{m}
             \msum_{i=1}^m
             \; 
             \mean\Big[
             \mprod_{(i_s,i_t) \in E(H)} w(U_{i_s}, U_{i_t})
             \, \Big|\, 
             U_i = x\Big]
             \;=\; 
             \mean\Big[
             \mprod_{(i_s,i_t) \in E(H)} w(U_{i_s}, U_{i_t})
             \Big]\;.
         \end{align*}
    \end{proplist}
\end{lemma}

% \vspace{.5em}

\cref{lem:subgraph:vertex:conditions}(i) is the condition derived from variance domination, whereas \cref{lem:subgraph:vertex:conditions}(iv) is the definition of $H$-regularity of $w$, introduced by \cite{hladky2021limit} for a complete subgraph $H$ and extended to a general subgraph $H$ by \cite{bhattacharya2023fluctuations}. Indeed, Theorem 1.2 of \cite{hladky2021limit} and Theorem 2.9 of \cite{bhattacharya2023fluctuations} establish that $\kappa_1(U)$ exhibit fluctuations of an order higher than $n^{m-\frac{1}{2}}$ if and only if $H$-regularity holds -- a geometric condition that is recovered by variance domination. Meanwhile, variance domination provides more: We can now characterise all fluctuations in $\kappa_1(U)$ on the orders $n^{m- \frac{1}{2}}, n^{m- 1}, \ldots, n^{\frac{m}{2}}$, all with finite-sample bounds. \cref{cor:subgraph:vertex} may also be applied to a setup where $w$ and $H$ are allowed to vary in $n$.

% \vspace{.5em}

\begin{remark} In the case of $H = K_2$, \cref{lem:subgraph:vertex:conditions}(iii) says that $\mean[ w(U_1, U_2) \,|\, U_1 ]$ is constant almost surely. In other words, the edge count statistic $\kappa(Y)$ exhibits $n$-th order fluctuations if and only if the random graph is regular on average.
\end{remark}

% \vspace{.5em}

\noindent
\textbf{Edge-level fluctuations. } In practice, one may want to analyse edge properties of the random graph, which requires us to study $\kappa_2(\bar Y)$. Kaur and R\"ollin \cite{kaur2021higher} provide error bounds on the Gaussian approximation of $\kappa_2(\bar Y)$ by using the orthogonal decomposition of a generalised U-statistic of $U$ and $V$ proposed by \cite{janson1991asymptotic}. We consider a variant of their setup. For convenience, we re-index $(\bar Y_{ij})_{i,j \in E(K_n)}$ as $\bar Y_1, \ldots, \bar Y_{n_*}$, where $n_* = | E(K_n) | = \binom{n}{2}$; recall that $k$ is the number of edges of $H$. Conditioning on $U = (U_i)_{i \leq n}$, let $Z=(Z_i)_{i \leq n_*}$ be conditionally normal variables defined by 
\begin{align*}
    Z_i \,|\, U 
    \;\;\overset{i.i.d.}{\sim}\;\;
    \cN\big( 
        \mean[\bar Y_i \,|\, U] 
        \,,\,
        \Var[\bar Y_i \,|\, U] 
    \big) 
    \;.
\end{align*}
The penultimate approximation for $\kappa_2(\bar Y)$ is the incomplete U-statistic 
\begin{align*}
    \kappa_2(Z) 
    \;=\;
    \msum_{\substack{
        \{e_1, \ldots, e_k\} \subseteq E(K_n) \\
        \text{is a set of distinct edges} }}
    \; 
    \delta_H\big( \{e_s\}_{s \in [k]} \big)
    \,
    \mprod_{s=1}^k 
    Z_{e_s}\;.
\end{align*}

\vspace{.5em}

\begin{proposition} \label{prop:subgraph:u:two} If $\Var[\kappa_2(Z) \,|\, U] > 0$ almost surely, there exist some absolute constant $C > 0$ such that for every $n \geq 2$ and $m, k \in \N$, almost surely,
    \begin{align*}
    & \qquad \qquad \qquad \qquad \msup_{t \in \R} \Big| \P\big( \kappa_2(\bar Y)  \leq  t \,\big|\, U \big) 
        - 
        \P\big(  \kappa_2(Z)  \leq t \,\big|\, U \big)  \Big| 
        \;\leq\;
        \Delta_{\bar Y}(U)
        \;,
    \end{align*}
where $ \Delta_{\bar Y}(U) 
\coloneqq 
C 
    \, m \,
    n^{ - \frac{\nu-2}{\nu k + 1}}
    \,
    \rho_{\bar Y}(U)^{\frac{1}{\nu k +1}} $
and 
\begin{align*}
    \rho_{\bar Y}(U) 
    \coloneqq
    \mfrac
    {\mmax_{i_1, \ldots, i_k \in [n_*] \text{ distinct}} \, \mprod_{l=1}^{k}  \mean\big[ \big| \bar Y_{i_l} \big|^\nu \, \big|\, U  \big]  }
    {\mmin_{i_1, \ldots, i_{k'} \in [n_*] \text{ distinct}} \, \mprod_{l=1}^{k'}  \mean \big[ \big| \bar Y_{i_l} \big|^2  \, \big|\, U  \big]^{\nu/2} }\;.
\end{align*}
Moreover, writing $\Phi$ as the c.d.f.~of $\cN(0,1)$, we have that almost surely
\begin{align*}
    \sup_{t \in \R} \Big| \P\big(  \kappa_2(Y) - \mean [\kappa_2(Y) \,|\, U] \leq  t \,\big|\, U \big) 
    - 
    \Phi\big(  & \Var[\kappa_2(Z) \,|\, U]^{-1/2} \, t  \big)
     \Big|
     \\
    \;\leq&\;
    \Delta_{\bar Y}(U)
    + 
    \Big(\mfrac{4m-4}{3m} \, \big| \Kurt [\kappa_2(Z) \,|\, U ] \big| \Big)^{1/2}\;,
\end{align*}
where we defined
$\Kurt[ \kappa_2(Z) | U] \coloneqq \mean[ (\kappa_2(Z) - \mean [\kappa_2(Z)|U] )^4 | U ] \,/\, \Var[\kappa_2(Z) | U]^2 - 3$.
\end{proposition}

% \vspace{.5em}

\begin{remark} (i) Observe that $\kappa(Y) - \mean[\kappa(Y) |U] = \kappa_2(\bar Y)$ almost surely, since $\kappa_1(U)$ is fixed under the conditioning on $U$. (ii) $\kappa_2(\bar Y)$ is an incomplete U-statistic, where the asymmetry arises from $\delta_H$, and $(\bar Y_i)_{i \leq n_* }$ conditioning on $U$ are non-identically distributed. This is an example of how our main results can be applied to asymmetric estimators and non-identically distributed data. 
\end{remark}

% \vspace{.5em}

\begin{remark}\label{remark:subgraph:asymmetric:notations} The max-min ratio $\rho_{\bar Y}$ can be avoided by directly using the error bound from 
    \Cref{thm:main}:
    % \cref{thm:main}: 
    In exchange, the bound is given as a sum of non-identical moment terms and the convergence rate cannot be read off directly. This tradeoff is common for asymmetric estimators and non-identically distributed data, although simplifications may be possible for specific statistics.
\end{remark}

% \vspace{.5em}

\noindent
\textbf{Overall fluctuations. } We now consider the asymptotic distribution of $\kappa(Y) = \kappa_1(U) + \kappa_2(\bar Y)$. Specifically, we investigate whether the vertex-level fluctuations $\kappa_1(U)$ or the edge-level fluctuations $\kappa_2(\bar Y)$ dominates in variance in $\kappa(Y)$. The variances of the different Hoeffding's decompositions of $\kappa_1(U)$ have been provided in \cref{lem:subgraph:vertex:conditions}, whereas the next result provides the variance of $\kappa_2(\bar Y)$.

% \vspace{.5em}

\begin{lemma} \label{lem:subgraph:u:var:size} There are some absolute constants $c, C >0$ such that
\begin{align*}
    \Var[\kappa_2(\bar Y) ]
    \;\leq&\;
    C^k 
    \, | \cG_H([n])| \;
    \mean\Big[ \mmax_{i_1, \ldots, i_k \in [n_*] \text{ distinct}} \, \mprod_{l=1}^{k}  \Var[ \bar Y_{i_l} \, |\, U ]  \Big] \;, 
    \\
    \Var[\kappa_2(\bar Y) ]
    \;\geq&\;
    c^k 
    \, | \cG_H([n])| \;
    \mean \Big[ \mmin_{i_1, \ldots, i_k \in [n_*] \text{ distinct}} \, \mprod_{l=1}^{k}  \Var[ \bar Y_{i_l} \, |\, U ]    \Big]\;.
\end{align*}
\end{lemma}

\vspace{.5em}

Consider again the setup with $w$, $H$, $m$ and $k$ fixed. Let $\rm{Aut}(H)$ be the set of all automorphisms of $H$. Provided that the expectation terms in \cref{lem:subgraph:u:var:size} are $\Theta(1)$, we have
\begin{align*}
    \Var[\kappa_2(\bar Y)]
    \;=\; 
    \Theta( 
     |\cG_H([n])| 
    ) 
    \;=\; 
    \Theta\Big( 
    \mbinom{n}{m} \,\mfrac{m!}{|\rm{Aut}(H)|}
    \Big) 
    \;=\; 
    \Theta(n^m)
    \;.
\end{align*}
Meanwhile, when the dominating term of $\kappa_1(U)$ is its $M$-th Hoeffding's decomposition, $\sigma^2_{m,n;M}$ is assumed to be non-zero, and \cref{lem:subgraph:vertex:conditions} implies 
$$\sigma_{m,n;M}^2 \;=\; \Theta(n^{2m-M})\;. $$
As a result, $\kappa_1(U)$ always dominates $\kappa_2(\bar Y)$ in variance when $M < m$. This agrees with the observation in \cite{kaur2021higher}, who provide a detailed illustration for simple subgraphs. Consequently, analysing the distribution of $\kappa(Y)$ usually means that edge-level randomness is ignored except for the case with the highest level of degeneracy.

\vspace{.5em}

Some difficulties persist in the most general case: (i) Due to asymmetry and non-identically distributed data, the bounds can be loose or not immediately interpretable, as discussed in \cref{remark:subgraph:asymmetric:notations}.
Under specific random graph models, bounds can simplify considerably \citep{hladky2021limit}.
(ii) Even if a penultimate approximation is established, the limiting distribution for an asymmetric polynomial of Gaussians is highly dependent on the weights and may not be immediately obvious.
Different choices of such penultimate approximations can yield natural graph-based interpretations \citep{kaur2021higher,bhattacharya2023fluctuations}. In concurrent work, Chatterjee, Dan, and Bhattacharya \citep{chatterjee2024higher}
obtain limiting distributions for graphon models with higher-order degeneracies and for joint subgraph counts, and additional geometric insights.

\section{Proof of Theorem \ref{thm:main} and Corollary \ref{cor:non:multilinear}% Theorem \ref{thm:main}
} \label{sec:proof:main}

\subsection{Tensor notation and multilinear polynomial $\tilde q_m$ in \Cref{cor:non:multilinear}} \label{appendix:tensor:notation}

To specify polynomials of vectors, we use the tensor notation:
By a tensor $T$ of size ${d_1\times\ldots\times d_k}$, we mean an element
of ${\mathbb{R}^{d_1\times\ldots\times d_k}}$. 
The tensor product of $k$ vectors ${x_i\in\mathbb{R}^{d_i}}$ is the tensor $\otimes_ix_i$ with entries
\begin{equation*}
  \otimes_{i\leq k}\,x_i({j_1,\ldots,j_k})\;=\;x_1({j_1})\cdots x_k({j_k})\quad\text{ for }j_1\leq d_1,\ldots,j_k\leq d_k\;,
\end{equation*}
and we write ${x_1^{\otimes r}:=\otimes_{i\leq r}\,x_1}$ for short. The scalar product of two tensors of equal size is
\begin{equation*}
  \langle S,T\rangle\;:=\;\msum_{i_1\leq d_1,\ldots,i_k\leq d_k}S(i_1,\ldots,i_k)T(i_1,\ldots,i_k)
  \qquad\text{ for }S,T\in\mathbb{R}^{d_1\times\ldots\times d_k}\;,
\end{equation*}
or equivalently the Euclidean scalar product in ${\mathbb{R}^{d_1\times\ldots\times d_k}}$.
Any polynomial $p_m$ of degree ${\leq m}$ can be represented as
\begin{align*} 
    p_m(x_1, \ldots, x_n) 
    \;=\;
    T_0
    +
    \msum_{\substack{r_1 + \ldots + r_n \leq m }} \langle T_{r_1, \ldots, r_n}\,,\, x_1^{\otimes r_1} \otimes \ldots \otimes x_n^{\otimes r_n} \rangle\;,
\end{align*}
where $T_0$ is a scalar and each $T_{r_1, \ldots, r_n}$ is a tensor of size ${d^{r_1}\times\ldots\times d^{r_n}}$. 

\vspace{.5em}

To identify $\tilde q_m$, we observe that for the augmented random tensor $\bX_i$ in \Cref{cor:non:multilinear}, 
we can write $p_m$ as 
\begin{align*}
  p_m(X)
  \;=\;
  \mean[\,p_m(X)]
  \,+\,
  \langle T_m \,,\, (1, \bX_1)^\top \otimes \ldots \otimes (1, \bX_n)^\top 
  \rangle
\end{align*}
for some tensor $T_m$. $T_m$ is determined completely by $T_0$ and $(T_{r_1, \ldots, r_n})_{r_1, \ldots, r_n \leq m}$. Since $p_m$ has degree $m$, $T_m$ is such that for $m' > m$, all $m'$-fold products of $\bX_1, \ldots, \bX_n$ correspond to zero coefficients in $T_m$. Since $p_m(X)$ is linear in each $\bX_1$, we can directly identify 
\begin{align*}
    \tilde q_m(\bX) \;=\; p_m(X) \;=\; 
    \mean[\,p_m(X)]
    \,+\,
    \langle T_m \,,\, (1, \bX_1)^\top \otimes \ldots \otimes (1, \bX_n)^\top \rangle \;.
\end{align*}

\subsection{Proofs}

The proof idea is standard: We first compare the distributions on a class of smooth functions via Lindeberg's principle (e.g.~Chapter 9 of \cite{van2014probability}). The smooth function bound is then combined with the Carbery-Wright inequality \citep{carbery2001distributional}, an anti-concentration bound for a polynomial of Gaussians, to obtain the bound in the Kolmogorov metric. In this section, we also include several lemmas that will simplify subsequent proofs.

\vspace{.5em}

We use a smooth approximation of an indicator for bounding the probability of a given event; see a standard usage in e.g. the proof of Theorem 3.3 in \citep{chen2011normal}. 

\vspace{.5em}

\begin{lemma}[Lemma 34 of \cite{huang2023high}] \label{lem:smooth:approx:ind} Fix any $m \in \N \cup \{0\}$, $\tau \in \R$ and $\delta > 0$. Then there exists an $m$-times differentiable $\R \rightarrow \R$ function $h_{m;\tau;\delta}$ such that $h_{m;\tau+\delta;\delta}(x) \;\leq\; \ind_{\{x > \tau\}} \leq h_{m;\tau;\delta}(x)$. For $0 \leq r \leq m$, the $r$-th derivative $h^{(r)}_{m;\tau;\delta}$ is continuous and bounded above by $\delta^{-r}$. Moreover, for every $\epsilon \in [0,1]$, $h^{(m)}$ satisfies that
    \begin{align*}
        | h^{(m)}_{m;\tau;\delta}(x) - h^{(m)}_{m;\tau;\delta}(y) |\;\leq\; C_{m,\epsilon} \, \delta^{-(m+\epsilon)}  \, |x - y|^{\epsilon}\;,
    \end{align*}
    with respect to the constant $C_{m,\epsilon} = \binom{m}{\lfloor m/2 \rfloor} (m+1)^{m+\epsilon}$.
\end{lemma}

\vspace{.5em}

The next result allows us to control the difference in expectation of the above approximation functions evaluated at independent random quantities:

\vspace{.5em}

\begin{lemma} \label{lem:main:taylor} Fix $t \in \R$, $\delta > 0$ and define $h_{t;\delta} \equiv h_{2;t;\delta}$ as in \cref{lem:smooth:approx:ind}.  Let $\bV, \bZ, \bW$ be some random vectors in $\R^b$ and $Y$ be a random variable in $\R$, with dependence allowed. Then there exists an absolute constant $C$ such that for any $\nu \in (2,3]$ and $\delta > 0$, we have
    \begin{align*}
        | \mean[ h_{t;\delta}(\bW^\top \bV + Y) - h_{t;\delta}(\bW^\top \bZ + Y) ] |
        \;\leq\; 
        Q_1 + Q_2 + Q_3\;,
    \end{align*} 
where 
\begin{align*}
    Q_1 \;\coloneqq&\; 
    \big| \mean\big[ h'_{t;\delta}(Y) \,\bW^\top ( \bV - \bZ) \big] \big|\;, 
    \quad\quad
    Q_2 \;\coloneqq\; 
    \mfrac{1}{2} \,
    \big| \mean\big[ h''_{t;\delta}(Y) \, \big( 
    \bW^\top \, \big( \bV \bV^\top - 
    \bZ \bZ^\top \big)\, \bW \big) \big] \big|\;,
    \\
    Q_3 \;\coloneqq&\; 
    54 \delta^{-\nu}
    (
    \|\bW^\top \bV \|_{L_\nu}^\nu 
    +
    \|\bW^\top \bZ\|_{L_\nu}^\nu 
    ) \;.
\end{align*}
Assume additionally that $(\bV, \bZ)$ is independent of $(\bW,Y)$. If $\mean[\bV]=\mean[\bZ]$, then $Q_1=0$. If $\Var[\bV]=\Var[\bZ]$, then $Q_2=0$.
\end{lemma}

\vspace{.5em}

\begin{proof}[Proof of \cref{lem:main:taylor}] Since $h_{t;\delta}$ is twice continuously differentiable, by a second-order Taylor expansion with the integral remainder, we get that for any $w, r \in \R$, 
\begin{align*}
    h_{t;\delta}(u) 
    \;=&\;
    h_{t;\delta}(r) 
    + 
    h'_{t;\delta}(r) \, (u-r)
    +
    \mint_r^u h''_{t;\delta}(t) \, (u-t) \, dt
    \\
    \;=&\;
    h_{t;\delta}(r) 
    + 
    h'_{t;\delta}(r) \, (u-r)
    +
    \mint_0^1 h''_{t;\delta}\big( (u-r)\theta + r \big) \, (u-r)^2 (1-\theta) \, d \theta 
    \\
    \;=&\;
    h_{t;\delta}(r) 
    + 
    h'_{t;\delta}(r) \, (u-r)
    +
    \mean\Big[ h''_{t;\delta}\big( (u-r)\Theta + r \big) \, (u-r)^2 (1-\Theta) \Big]
    \;, \tagaligneq \label{eq:twice:taylor:h}
\end{align*}
where $\Theta \sim \textrm{Uniform}[0,1]$. By applying this expansion once with $u=\bW^\top \bV + Y, r=Y$ and once with $u=\bW^\top \bZ + Y, r=Y$, we get that 
\begin{align*}
    \big| \mean&\big[ h_{t;\delta}(\bW^\top \bV + Y) - h_{t;\delta}(\bW^\top \bZ + Y) \big] \big|
    \\
    \;=&\;
    \Big| \mean\Big[ 
    h'_{t;\delta}(Y) \, (\bW^\top \bV)
    - h'_{t;\delta}(Y) \, (\bW^\top \bZ)
    + 
    h''_{t;\delta}\big( \Theta \bW^\top \bV + Y \big) \, (\bW^\top\bV )^2 (1-\Theta)     
    \\
    &\;\;\;\;\;\;\;
    - 
    h''_{t;\delta}\big( \Theta \bW^\top \bZ + Y \big) \, (\bW^\top\bZ )^2 (1-\Theta)     
    \Big] \Big|\;.
\end{align*}
By adding and subtracting two second derivative terms evaluated at $t$, we can further obtain
\begin{align*}
     \big| \mean\big[ h_{t;\delta}&(\bW^\top \bV + Y) - h_{t;\delta}(\bW^\top \bZ + Y) \big] \big|
    \\
    \;=&\;
    \Big| \mean\big[ h'_{t;\delta}(Y) \,\bW^\top ( \bV - \bZ) \big]
    + \mean\big[ h''_{t;\delta}(Y) \,
    \big( (\bW^\top\bV )^2 - (\bW^\top\bZ )^2 \big) (1-\Theta)  \big]
    \\
    &\;\,
    +
    \mean\big[ \big( 
    h''_{t;\delta}\big( \Theta \bW^\top \bV + Y \big) 
    - 
    h''_{t;\delta}(Y) 
    \big) \,
    (\bW^\top\bV )^2
    (1-\Theta) \big]  
    \\
    &\;\,
    -
    \mean\big[ \big( 
    h''_{t;\delta}\big( \Theta \bW^\top \bZ + Y \big) 
    - 
    h''_{t;\delta}(Y )
    \big) \, 
    (\bW^\top\bZ)^2
    (1-\Theta)    
    \big] \Big|
    \\
    \;\eqqcolon&\;
    | Q_1 + Q_2 + Q_V - Q_Z |
    \;\leq\; |Q_1| + |Q_2| + |Q_V| + |Q_Z|
    \;. \tagaligneq \label{eqn:main:taylor:difference}
\end{align*} 
We handle the four terms individually. $Q_1$ is already in the desired form. $Q_2$ can be simplified by noting that $\Theta$ is independent of all other variables:
\begin{align*}
    Q_2
    \;=\;
    \mfrac{1}{2} \,
    \mean\big[ h''_{t;\delta}(Y) \, \big( 
    \bW^\top \, \big( \bV \bV^\top - 
    \bZ \bZ^\top \big)\, \bW \big) \big]
    \;.
\end{align*}
Now to handle $Q_V$ and $Q_Z$, we use a Jensen's inequality to move the absolute sign inside the expectation and note that $1-\Theta$ is bounded in norm by $1$:
\begin{align*}
    |Q_V|
    \;=&\;
    \big|
    \mean\big[  \big( 
    h''_{t;\delta}\big( \Theta \bW^\top \bV + Y \big) 
    - 
    h''_{t;\delta}(Y)
    \big) \,
    (\bW^\top\bV )^2
    (1-\Theta) \big]  
    \big|
    \\
    \;\leq&\;
    \mean\big[
    \big|  h''_{t;\delta}\big( \Theta \bW^\top \bV + Y \big) 
    - 
    h''_{t;\delta}(Y) \big| \times 
    \big(\bW^\top \bV\big)^2 \big]  \;.
\end{align*}
Recall the H\"older property of $h''_{t;\delta}$ from \cref{lem:smooth:approx:ind}: For any $\epsilon \in [0,1]$ and $x, y \in \R$, we have
\begin{align*}
    \big| h''_{t;\delta}(x) - h''_{t;\delta}(y) \big| 
    \;\leq&\; 
    54 \delta^{-2-\epsilon} \, |x-y|^{\epsilon} \,
    \;,
\end{align*}
Applying this to the bound above with $\epsilon=\nu-2 \in (0,1]$, $x=\Theta\bW^\top\bV+Y$ and $y=Y$, while noting that $|\Theta|$ is bounded above by $1$ almost surely, we get that
\begin{align*}
    |Q_V|
    \;\leq&\;
    54 \, \delta^{-\nu}
    \mean\big[
     |  \Theta \bW^\top \bV |^{\nu-2} \, (\bW^\top \bV)^2 
    \big] 
    \;\leq\;
    54 \, \delta^{-\nu} \,
    \big\|\bW^\top \bV \big\|_{L_\nu}^\nu 
    \;.
\end{align*}
To deal with $Q_Z$, the argument is the same except $\bV$ is replaced by $\bZ$:
\begin{align*}
    |Q_Z|
    \;\leq\;
    54 \, \delta^{-\nu} \,
    \big\|
    \bW^\top \bZ
    \big\|_{L_\nu}^\nu 
    \;.
\end{align*}
Applying the bounds to \eqref{eqn:main:taylor:difference} then gives the first desired bound:
\begin{align*}
    \big| \mean\big[ h_{t;\delta}(\bW^\top \bV + Y) &- h_{t;\delta}(\bW^\top \bZ + Y) \big] \big|
    \\
    \;\leq&\;
    |Q_1|
    +
    |Q_2|
    +
    54 \, \delta^{-\nu} 
    \big( 
    \big\|\bW^\top \bV \big\|_{L_\nu}^\nu 
    +
    \big\|\bW^\top \bZ \big\|_{L_\nu}^\nu 
    \big) \;.
\end{align*} 
In the case where $\mean[\bV]=\mean[\bZ]$ and $(\bV, \bZ)$ is independent of $(\bW,Y)$, we get that
\begin{align*}
    Q_1
    \;=&\;
    \mean\big[ h'_{t;\delta}(Y) \, \bW^\top (\bV - \bZ) \big] 
    \;=\;
    \mean\big[ h'_{t;\delta}(Y) \, \bW^\top \mean[\bV - \bZ] \big] 
    \;=\; 0\;. 
\end{align*}
Similarly in the case where $\Var[\bV]=\Var[\bZ]$ and $(\bV, \bZ)$ is independent of $(\bW,Y)$, we have
\begin{equation*}
    Q_2 
    \;=\;
    \mfrac{1}{2} \,
    \mean\big[ h''_{t;\delta}(Y) \, \big( 
    \bW^\top \, \mean[ \bV \bV^\top - 
    \bZ \bZ^\top] \, \bW \big) \big]
    \;=\; 0 \;.  \qedhere
\end{equation*}
\end{proof}

\vspace{.5em}

The next lemma is convenient for simplifying moment terms involving Gaussians:

\begin{lemma} \label{lem:gaussian:linear:moment} Consider a zero-mean $\R^b$-valued Gaussian vector $\eta$. Suppose $\Var[\eta]=\Var[V]$ for some $\R^b$ zero-mean random vector $V$ and let $W$ be a random vector in $\R^b$ independent of $\eta$ and $V$. Then for any real number $\nu \geq 2$, we have
    \begin{align*}
        \mean\big[ \big| W^\top \eta \big|^{\nu} \big]
        \;\leq\; 
        \mfrac{2^{\nu/2} \, \Gamma\big( \frac{\nu+1}{2}\big) }{\sqrt{\pi}} \, \mean\big[ |W^\top V|^{\nu} \big]\;,
    \end{align*}
    where $\Gamma$ represents the Gamma function. In the case $\nu=2$, we have $\Var[W^\top \eta] = \Var[W^\top V]$.
\end{lemma}

\vspace{.5em}

\begin{proof}[Proof of \cref{lem:gaussian:linear:moment}] Note that $W$ is independent of $\eta$ and $V$. Conditioning on $W$, $W^\top \eta$ is a zero-mean normal random variable with variance given by
    \begin{align*}
        W^\top \Var[\eta] W
        \;=&\; 
        W^\top \Var[V ] W
        \;=\;
        \mean\big[ W^\top V  V^\top W \big| W \big]
        \;=\;
        \mean\big[ (W^\top V)^2 \big| W \big] \;.
    \end{align*}
    Applying the formula of the $\nu$-th moment of a Gaussian random variable followed by a Jensen's inequality with $\nu \geq 2$, we get that
    \begin{align*}
        \mean\big[ \big| W^\top \eta \big|^{\nu} \big] 
        \;=\;
        \mean\big[ \mean\big[ \big| W^\top \eta \big|^{\nu} \big| W \big] \big]
        \;=&\;
        \mfrac{2^{\nu/2} \, \Gamma\big( \frac{\nu+1}{2}\big) }{\sqrt{\pi}} \mean\big[ \mean\big[ (W^\top V)^2 \big| W \big]^{\nu/2} \big]
        \\
        \;\leq&\;
        \mfrac{2^{\nu/2} \, \Gamma\big( \frac{\nu+1}{2}\big) }{\sqrt{\pi}} \mean\big[ |W^\top V|^{\nu} \big] \;.
    \end{align*}
    For $\nu=2$, the above becomes an equality and we get $\Var[W^\top \eta] = \Var[W^\top V]$.
\end{proof}
    
\vspace{.5em}

We also restate the Carbery-Wright inequality (Theorem 8 of \cite{carbery2001distributional}) with our notation:

\vspace{.5em}

\begin{lemma}[Carbery-Wright inequality] \label{lem:carbery:wright} Let $Q(\bx)$ be a polynomial of $\bx \in \R^d$ of degree at most $m$ taking values in $\R$, and $\eta$ be a standard Gaussian vector taking values in $\R^d$. Then there exists a constant $C$ independent of $Q$, $d$ or $m$ such that, for every $t \in \R$,
\begin{align*}
    \P \big( | Q(\eta) | \leq t \big) 
    \;\leq\; 
    C \, m\, t^{1/m} (\mean[|Q(\eta)|^2])^{-1/2m}
    \;.
\end{align*}
\end{lemma}

\vspace{.5em}

We are ready to present the proof for \Cref{thm:main}, which directly implies \Cref{cor:non:multilinear} by applying \Cref{thm:main} to $\tilde q_m$.

\begin{proof}[Proof of \Cref{thm:main} and \Cref{cor:non:multilinear}] We first note that since $q_m$ is affine in each argument, $\mean[q_m(Z)] = \mean[q_m(X)]$. Meanwhile, by applying \cref{lem:gaussian:linear:moment} repeatedly, $\Var[q_m(X)] =\Var[q_m(Z)]$. This proves the second part of \Cref{thm:main}.
% \cref{thm:main}.

\vspace{.5em}
    
To prove the distributional approximation, for simplicity, we denote 
\begin{align*}
    \bar q_m(x_1, \ldots, x_n) 
    \;\coloneqq\; 
    \sigma^{-1} \big( q_m(x_1, \ldots, x_n)  - \mean[ q_m(X)]  \big) 
    \;.
\end{align*}
We first approximate the probability terms. For $\tau \in \R$ and $\delta > 0$, let $h_{\tau; \delta} \equiv h_{2;\tau;\delta}$ be the twice continuously differentiable function defined in \cref{lem:smooth:approx:ind}, which satisfies that $ h_{\tau+\delta;\delta}(x) \leq \ind_{\{x > \tau\}}  \leq h_{\tau;\delta}(x)$ for any $\tau \in \R$. This allows us to bound 
\begin{align*}
    \P(\bar q_m(X) > t) - \P(\bar q_m(Z) > t - \delta) 
    \;=&\;
    \mean\big[ \ind_{\{  \tilde q_m(\bX) > t \}} - \ind_{\{  \tilde q_m(\Xi) > t - \delta\}} \big]
    \\
    \;\leq&\;
    \mean\big[ h_{t;\delta}\big(  \bar q_m(X)   \big) 
    - 
    h_{t;\delta}\big( \bar q_m(Z)  \big) \big] \;, 
\end{align*}
and similarly
\begin{align*}
    \P(\bar q_m(Z) > t+\delta) - \P(\bar q_m(X) > t ) 
    \;\leq&\;
    \mean\big[ h_{t+\delta;\delta}\big(  \bar q_m(Z)  \big) 
    - 
    h_{t+\delta;\delta}\big(  \bar q_m(X)  \big) \big] \;.
\end{align*}
By expressing
$
    \P( \bar q_m(Z)  > t-\delta )\;=\; \P( \bar q_m(Z)  > t ) + \P( t-\delta < \bar q_m(Z)  \leq t )
$
and performing a similar decomposition for $\P( \bar q_m(Z)  > t+\delta )$, we obtain that
\begin{align*}
    \big| \P( \bar q_m(X) \leq & t) - \P( \bar q_m(X) \leq t) \big|
    \;=\;
    | \P(\bar q_m(X) > t) - \P(\bar q_m(Z) > t)|
    \\
    \;\leq&\;
    \max\{ \P( t-\delta < \bar q_m(Z)  \leq t ) \,,\,  \P( t \leq \bar q_m(Z)  < t+\delta ) \} 
    +
    \max\{ E_t, E_{t+\delta} \}
    \\
    \;\leq&\;
    \P( t-\delta < \bar q_m(Z)  < t+\delta  )
    +
    E_{t}
    +
    E_{t+\delta}
    \;, \tagaligneq \label{eqn:replacement:intermediate}
\end{align*}
where we have defined $E_{\tau} \coloneqq | \mean[ h_{\tau;\delta}( \bar q_m(X) ) - h_{\tau;\delta}( \bar q_m(Z) ) ] |$ for $\tau \in \R$. The next step is to control quantities of the form $E_{\tau}$ by Lindeberg's principle. We recall $W_i = (X_1, \ldots, X_{i-1}, 0, Z_{i+1}, \ldots, Z_n) \in \R^{nd}$ and define $W_i^* = (X_1, \ldots, X_{i-1}, X_i, Z_{i+1}, \ldots, Z_n)$. Then by expanding the difference into a telescoping sum and applying a triangle inequality, 
\begin{align*}
    E_\tau \;=&\; | \mean[ h_{\tau;\delta}( \bar q_m(X) ) - h_{\tau;\delta}( \bar q_m(Z) ) ] |
    \\
    \;\leq&\; 
    \msum_{i=1}^n  | \mean[ h_{\tau;\delta}( \tilde q_m(W_i^*) ) - h_{\tau;\delta}( \tilde q_m(W_{i-1}^*) ) ] |
    \\
    \;\overset{(a)}{=}&\;
    \msum_{i=1}^n  | \mean[ h_{\tau;\delta}( \tilde q_m(W_i^*) ) - h_{\tau;\delta}( \tilde q_m(W_{i-1}^*) ) ] |
    \;\eqqcolon\; \msum_{i=1}^n E_{\tau;i} \;.
\end{align*}
We focus on bounding each $E_{\tau;i}$. Since $q_m$ is affine in the $i$-th argument, we get 
\begin{align*}
    \bar q_m( W_i^*) 
    \;=&\;
    \partial_i \bar q_m( W_i)^\top X_i + \bar q_m (W_i)
    &\text{ and }&& 
    \bar q_m (W_{i-1}^*) 
    \;=&\; 
    \partial_i \bar q_m( W_i)^\top Z_i + \bar q_m(W_i)\;.
\end{align*}
Note that $X_i$ and $Z_i$ are zero-mean with the same variance, and are independent of $W_i$. This allows us to apply \cref{lem:main:taylor}: For any $\nu \in (2,3]$ and $\delta > 0$, we have
\begin{align*}
    E_{\tau;i}
    \;=&\;
    \big| \mean\big[ h_{\tau;\delta}\big(
    \partial_i \bar q_m( W_i)^\top X_i + \bar q_m(W_i)
    \big) 
    - h_{\tau;\delta}
    \big(
    \partial_i \bar q_m( W_i)^\top Z_i + \bar q_m(W_i)
    \big)
    \big] \big| 
    \tagaligneq \label{eqn:lindeberg:intermediate}
    \\
    \;\leq&\;
    54 \, \delta^{-\nu} \,
    \big(
    \big\|
        \partial_i \bar q_m( W_i)^\top X_i
    \big\|_{L_\nu}^\nu
    +
    \big\|
    \partial_i \bar q_m( W_i)^\top Z_i
    \big\|_{L_\nu}^\nu 
    \big)
    \;.
\end{align*}
Since $Z_i$ is a Gaussian with the same mean and variance as $X_i$ and both are independent of $\partial_i f( W_i)$, by Lemma \ref{lem:gaussian:linear:moment} and noting that $\nu \leq 3$, there is an absolute constant $C' > 0$ such that
\begin{align*}
    \mean\big[
        |\partial_i \bar q_m( W_i)^\top Z_i|^\nu 
    \big]
    \;\leq&\; 
    C'
    \,
    \big\|
    \partial_i \bar q_m( W_i)^\top X_i
    \big\|_{L_\nu}^\nu 
     \;.
\end{align*}
By additionally noting that $M_{\nu;i} = \sigma \| \partial_i \bar q_m( W_i)^\top  X_i \|_{L_{\nu}}$ by definition, we get that 
\begin{align*}
    E_{\tau;i}
    \;\leq&\;
    54 (C'+1) \,  \delta^{-\nu} 
    \sigma^{-\nu} M_{\nu;i}^\nu
    \;. 
\end{align*}
Summing over $i=1,\ldots,n$ then gives
\begin{align*}
    E_\tau 
    &\;\leq\; \msum_{i=1}^n E_{\tau;i}r
    \;\leq\;  
    54 (C'+1) \, \mfrac{\sum_{i=1}^n M_{\nu;i}^\nu}{ \delta^\nu 
    \sigma^\nu  }
    \;. \tagaligneq \label{eq:main:smooth:intermediate}
\end{align*}
On the other hand, by Lemma \ref{lem:carbery:wright}, there exists an absolute constant $C^* > 0$ such that 
\begin{align*}
    \P( t-\delta < \bar q_m(Z)  < &t+\delta  )
    \;\leq\; 
    C^* m \, \delta^{1/m} (\mean[| \bar q_m(Z) - t|^2])^{-1/2m}
    \\
    \;\overset{(a)}{=}&\;
    C^* m \, \delta^{1/m} (\sigma^{-2} \Var[q_m(Z)] + t^2)^{-1/2m}
    \;\overset{(b)}{=}\;
    C^* m \, \delta^{1/m} (1 + t^2)^{-1/2m}
    \;. \tagaligneq \label{eq:main:proof:carbery:wright}
\end{align*}
In $(a)$ and $(b)$, we have used  that $\mean[\bar q_m(Z)]=0$ and $\Var[q_m(Z)] = \Var[q_m(X)] = \sigma^2$. Substituting \eqref{eq:main:smooth:intermediate} and \eqref{eq:main:proof:carbery:wright} into \eqref{eqn:replacement:intermediate}, we get that there exists some absolute constant $C > 0$ such that
\begin{align*}
    \big| \P( \bar q_m(X) \leq t) - \P( \bar q_m(Z) \leq t) \big|
    \;\leq\;
    \mfrac{C m}{2}
    \Big(
    \mfrac{\delta^{1/m}}{(1 + t^2)^{1/2m}}
    +
    \mfrac{\sum_{i=1}^n M_{\nu;i}^\nu}{ \delta^\nu 
    \sigma^\nu  }
    \Big)
    \;. 
\end{align*}
Finally by choosing $\delta = (1+t^2)^{\frac{1}{2 + 2\nu m}} \big( \sigma^{-\nu} \sum_{i=1}^n M_{\nu;i}^\nu \big)^{\frac{m}{\nu m +1}}$, we get the desired bound
\begin{equation*}
    \big| \P( \bar q_m(X) \leq t) - \P( \bar q_m(Z) \leq t) \big|
    \;\leq\;
    C m
    \,
    \Big( \mfrac{ \sum_{i=1}^n M_{\nu;i}^\nu }
    { (1+t^2)^{\nu/2} \sigma^\nu
    }
    \Big)^{\frac{1}{\nu m +1}}
    \;.   \qedhere
\end{equation*}
\end{proof}

\section{Proof of Theorem \ref{thm:VD}: Variance domination
% \ref{thm:VD}
} \label{sec:proof:VD}

Let $X'$ and $Y'$ be two real-valued random variables. In the proof of \Cref{thm:VD}, we will make use of the following bound for controlling the influence of $Y'$ on $X'+Y'$.

\begin{lemma} \label{lem:approx:XplusY:by:Y} For any $a, b \in \R$ and $\epsilon > 0$, we have
\begin{align*}
    \P( a \leq X'+Y' \leq b) 
    \;\leq\;
    \P( a - \epsilon \leq X' \leq b+\epsilon ) + \P(  |Y'| \geq \epsilon )\;,
    \\
    \P( a \leq X'+Y' \leq b) 
    \;\geq\;
    \P( a + \epsilon \leq X' \leq b-\epsilon ) - \P(  |Y'| \geq \epsilon )
    \;.
\end{align*}
\end{lemma}

\begin{proof}[Proof of \Cref{lem:approx:XplusY:by:Y}] By conditioning on the size of $Y$, we have 
\begin{align*}
    \P( a \leq X'+Y' \leq b) 
    \;=&\; 
    \P( a \leq X'+Y' \leq b \,,\, |Y'| \leq \epsilon ) + \P(  a \leq X' \leq b \,,\, |Y'| \geq \epsilon ) 
    \\
    \;\leq&\;
    \P( a - \epsilon \leq X' \leq b+\epsilon ) + \P(  |Y'| \geq \epsilon )\;,
\end{align*}
and by using the order of inclusion of events, we have the lower bound
\begin{equation*}
\begin{aligned}
    \P( a \leq X'+Y' \leq b) 
    \;\geq&\;
    \P( a + \epsilon \leq X' \leq b-\epsilon \,,\,|Y'| \leq \epsilon  )
    \\
    \;=&\;
    \P( a + \epsilon \leq X' \leq b-\epsilon ) - \P(  |Y'| \geq \epsilon )\;.
\end{aligned} 
\qedhere
\end{equation*}
\end{proof}

The proof of \Cref{thm:VD} is now a simple combination of \Cref{lem:approx:XplusY:by:Y} and \Cref{thm:main}.

\begin{proof}[Proof of \Cref{thm:VD}] Since the Kolmogorov distance is invariant under scaling and shifts, it suffices to prove the bound for $\bar q_m$ used in the proof of \Cref{thm:main} and 
\begin{align*}
    \bar f(x_1, \ldots, x_n) 
    \;\coloneqq\; 
    \sigma^{-1} \big( f(x_1, \ldots, x_n)  - \mean[ q_m(X)]  \big) 
    \;.
\end{align*}
By \cref{lem:approx:XplusY:by:Y}, for any $\epsilon > 0$ and $t \in \R$, we have
\begin{align*}
    \P\big(  \bar f(X) > t\big)
    \;\leq&\;
    \P\big( \bar q_m(X) > t - \epsilon \big) 
    +
    \P\big( | \bar f(X) - \bar q_m(X) | \geq \epsilon)
    \;,
    \\
    \P\big(  \bar f(X) > t\big)
    \;\geq&\;
    \P\big( \bar q_m(X) > t + \epsilon \big) 
    -
    \P\big( | \bar f(X) - \bar q_m(X) | \geq \epsilon)
    \;.
\end{align*}
By additionally expressing 
\begin{align*}
    \P\big(  \bar q_m(Z) > t\big) 
    \;=&\; 
    \P\big(  \bar q_m(Z) > t+\epsilon \big)
    +
    \P\big(  t \leq \bar q_m(Z) < t+\epsilon \big)
    \\
    \;=&\;
    \P\big(  \bar q_m(Z) > t-\epsilon \big)
    -
    \P\big(  t-\epsilon < \bar q_m(Z) \leq t \big)\;,
\end{align*}
we can bound 
\begin{align*}
    \big| \P\big(  \bar f(X) \leq t\big) - &\P\big( \bar q_m(Z) \leq t\big) \big|
    \;=\; 
    \big| \P\big(  \bar f(X) > t\big) - \P\big(  \bar q_m(Z) > t\big) \big|
    \\
    \;\leq&\;
    \P\big( | \bar f(X) - \bar q_m(X) | \geq \epsilon)
    +
    \P\big( |\bar q_m(Z) - t| < \epsilon \big) 
    \\
    &\;+
    \mmax_{s \,\in \{+1, -1\}}
    \big| 
    \P\big(  \bar q_m(X) > t + s\epsilon \big) 
    - 
    \P\big( \bar q_m(Z) > t + s\epsilon\big) 
    \big|
    \\
    \;\eqqcolon&\; T_1 + T_2 + T_3\;. 
\end{align*}
By Markov's inequality, we have  
\begin{align*}
    T_1
    \;\leq\;
    \mfrac{\| \bar f(X) - \bar q_m(X) \|_{L_2}^2}{ \epsilon^2}  
    \;=\;
    \mfrac{\| f(X) - q_m(X) \|_{L_2}^2}{\sigma^2 \epsilon^2}  
    \;.
\end{align*}
By the Carbery-Wright inequality from Lemma \ref{lem:carbery:wright} and noting that $\mean q_m(X) = 0$ and $\Var\,q_m(X) = \Var\,q_m(Z)$ by \Cref{thm:main}, 
% \cref{thm:main}, 
we have that for some absolute constant $C' > 0$,
\begin{align*}
    T_2
    \;\leq&\;
    C' m \, \epsilon^{1/m} (\mean[( \sigma^{-1} q_m(Z) - t)^2])^{-1/2m}
    \;\leq\;
    C' m \, \mfrac{\epsilon^{1/m}}{(1 + t^2 )^{1/2m}}
    \;.
\end{align*}
By \Cref{thm:main}, 
% \cref{thm:main}, 
there exists some absolute constant $C'' > 0$ such that
\begin{align*}
    T_3 
    \;\leq&\;
    C''
    m
    \,
    \max_{s \in \{+1, -1\}}
    \Big( \mfrac{ \sum_{i=1}^n M_{\nu;i}^\nu }
    { (1 + (t + s \epsilon)^2)^{\nu/2} \, \sigma^\nu
    }
    \Big)^{\frac{1}{\nu m +1}}
    \;. 
\end{align*}
Combining the bounds and choosing $\epsilon = \epsilon_t = (\frac{\| f(X) - q_m(X) \|_{L_2}}{\sigma})^{2m/(2m+1)} (1+ t^2)^{1/2(2m+1)}$, we obtain the non-uniform bound: There exists some absolute constant $C > 0$ such that
\begin{align*}
    \big| \P\big(  \bar f(X) &\leq t\big) - \P\big(  \bar q_m(Z) \leq t\big) \big|
    \\
    \;\leq&\;
    C m
    \Big( 
        \,
    \Big(\mfrac{\| f(X) - q_m(X) \|_{L_2}}{\sigma(1 + t^2)^{1/2}} \Big)^{\frac{2}{2m+1}}
    +
    \max_{s \in \{+1, -1\}}
    \Big( \mfrac{ \sum_{i=1}^n M_{\nu;i}^\nu }
    { (1 + (t + s\epsilon_t)^2)^{\nu/2} \, \sigma^\nu
    }
    \Big)^{\frac{1}{\nu m +1}}
    \Big)
    \;.
\end{align*}
Taking a supremum over $t$ gives the desired bound. 
\end{proof}

\section{Moment computation for U-statistics} \label{sec:auxiliary:u}

The proofs for many of our examples
% \cref{sec:simple:v} and \cref{sec:applications} 
rely extensively on moment formula for U-statistics, which are collected next. We follow the notation of \Cref{sec:degree:m:U:stats}: 
% \cref{sec:higher:u}: 
$Y=(Y_i)_{i \leq n}$ are i.i.d.~variables in a general measurable space $\cE$, and $U_j^{\rm H}(Y)$ and $u_j^{\rm H}(y_1, \ldots, y_j)$ are defined as in Hoeffding's decomposition (11) 
% \eqref{eq:hoeffding:decompose} 
of $u_m$ in (10). 
% \eqref{eq:defn:u}. 
The next result states that Hoeffding decompositions of different orders are orthogonal in $L_2$.

\vspace{.5em}

\begin{lemma}[Lemma 1.2.3 of \cite{denker1985asymptotic}] \label{lem:u:orthogonal} Fix some $j \neq j' \in [m]$. Suppose $\| U_j^{\rm H}(Y)\|_{L_2}$ and $\| U_{j'}^{\rm H}(Y)\|_{L_2}$ are bounded. Then $\mean[ U_j^{\rm H}(Y) U_{j'}^{\rm H}(Y) ] = 0$.
\end{lemma}

\vspace{.5em}

Denote $\sigma^2_r \coloneqq \Var\, \mean[ u(Y_1, \ldots, Y_m) \, |\, Y_1, \ldots, Y_r ]$, where $u$ is the kernel of $u_m$  in (10). 
% \eqref{eq:defn:u}. 
The next two results compute the variances associated with $u^{\rm H}_j$, $U^{\rm H}_j$ and $u_m$.

\vspace{.5em}

\begin{lemma}[Lemma 1.2.4 of \cite{denker1985asymptotic}] \label{lem:u:hoeffding:variance} For each $j \in [m]$, we have
    \begin{align*}
        \Var\big[ u_j^{\rm H}(Y_1, \ldots, Y_j) \big]
        \;=&\;
        \msum_{r=1}^{j} \, \mbinom{j}{j-r} \, (-1)^{j-r} \sigma_r^2 \;.
    \end{align*}
\end{lemma}

\vspace{.5em}

\begin{lemma}[Theorem 1.2.2 of \cite{denker1985asymptotic}] \label{lem:u:variance} We have that
    \begin{align*}
        \Var\big[ U_j^{\rm H}(Y) \big]
        \;=&\;
        \mbinom{n}{j}^{-1} \sigma_j^2
        &\text{ and }&&
        \Var[u_m(Y)] \;=&\; \mbinom{n}{m}^{-1} \msum_{r=1}^m \mbinom{m}{r} \mbinom{n-m}{m-r} \sigma^2_r\;.
    \end{align*}
\end{lemma}

\section{Proofs for Section \ref{sec:non:multilinear}: Universality for V-statistics} 
% \ref{sec:simple:v}}  
\label{appendix:proof:v:var:approx}

This section collects additional results on the V-statistic $v_m$ in \Cref{sec:non:multilinear}, which is used in the proof for \Cref{prop:simpler:V}. We also state and prove the universality result for the case where the data have i.i.d.~coordinates in \Cref{appendix:v:iid:coord}. The proofs for the results in this appendix are mathematically straightforward but tedious, as it involves expanding out a degree-$m$ V-statistic and comparing moment terms of different orders. Throughout, we express $\xi_i = (\xi_{i1}, \ldots, \xi_{im})$, where each $\xi_{ij}$ matches $\overline{ X_i^{\otimes j}} =  X_i^{\otimes j} - \mean[X_i^{\otimes j}]$ in mean and variance.

\subsection{Additional results} \label{sec:proof:v:additional}

We present several results on the universality approximation of $v_m(X)$ without the sub-Weibull condition in \Cref{assumption:sub:Weibull}. The results are given in terms of a rather cumbersome moment quantity, namely
\begin{align*}
    R_n 
    \;\coloneqq&\;
    \max_{1 \leq j \leq m-1}
    \,
    \max_{q_1 + \ldots + q_j = m\,,\, q_l \in \N  }
    \\
    &\;\; \qquad 
    \mfrac{ m^j }{n^m \Var[ v_m(X)]}
    \,\times \,
    \max \bigg\{
        \Var\bigg[ \bigg\langle S 
        \,, \, 
        \bigotimes_{l=1}^j 
        X_l^{\otimes q_l} 
        \bigg\rangle  \bigg]
        \,,\, 
        \Var\bigg[ \bigg\langle S 
        \,, \,  
        \bigotimes_{l=1}^j 
        Z_l^{\otimes q_l} \bigg\rangle \bigg]
    \bigg\} 
    \\
    &\;\; 
    +
    \\
    &\;\;
    \max_{1 \leq j \leq \lfloor \frac{m-1}{2} \rfloor}
    \,
    \max_{q_1 + \ldots + q_j = m\,,\, q_l \in \N  }
    \\
    &
    \qquad 
    \mfrac{n^{2j+1 - m}}{n^m \Var[v_m(X)]}
    \,\times \,
    \;\max\bigg\{
        \bigg( \mean\bigg[ \bigg\langle S 
        \,, \, 
        \bigotimes_{l=1}^j 
        X_l^{\otimes q_l} 
        \bigg\rangle  \bigg] \bigg)^2  
        \,,\,
        \bigg( \mean\bigg[ \bigg\langle S 
        \,, \, 
        \bigotimes_{l=1}^j 
        Z_l^{\otimes q_l} \bigg\rangle \bigg] \bigg)^2 
    \bigg\} 
    \;.
\end{align*}
Roughly speaking, it measures the error incurred by turning the universality approximation in \Cref{cor:non:multilinear} on the augmented vectors back into the universality approximation in \Cref{prop:simpler:V} on the original vectors. The proof of \Cref{prop:simpler:V} then reduces to using the uniform sub-Weibull condition to control $R_n$. 

\vspace{.5em}

We first state the main result, which controls various distributional approximation errors in terms of $R_n$. Here and below, we define 
\begin{align*}
    q^v_m(\bx_1, \ldots, \bx_n) \;\coloneqq\; 
    \tilde q_m(\bx_1, \ldots, \bx_n) - \mean[ v_m(X) ]
\end{align*}
where $\tilde q_m$ is the multilinear representation of $v_m$ in \Cref{cor:non:multilinear}.

\begin{proposition} \label{prop:universality:v:stat} Fix $\nu \in (2,3]$ and let $(X_i)_{i \leq n}$ be i.i.d.~zero-mean. 
    % If $v_m$ is $\delta$-regular with respect to $X$ for some $\delta \in [0,1)$, 
Under \Cref{assumption:V:stat}, there exist some absolute constants $C, C' > 0$ such that
\begin{align*}
    &\msup_{t \in \R} 
    \big|  \P\big( q^v_m(\Xi) + \mean[v_m(X)] \leq t\big) - \P\big( v_m(Z) \leq t  \big) \big|
    \,\;\leq\;\,
    C m \Big( \mfrac{R_n}{n} \Big)^{-\frac{1}{2m+1}}\;, 
    \\
    &\msup_{t \in \R} \big| \P( v_m(X) \leq t) - \P( v_m(Z) \leq t) \big|
    \,\;\leq\;\,
    C m \Big( \mfrac{R_n}{n} \Big)^{-\frac{1}{2m+1}}
    +
    C' m \Delta_v
    \;,
    \\
    &\msup_{t \in \R} \big| \P( v_m(X) - \mean[v_m(X)] \leq t) - \P( v_m(Z) - \mean[v_m(Z)] \leq t) \big|
    \\
    &\hspace{20em}\,\;\leq\;\,
    C m \Big( \mfrac{R_n}{n} \Big)^{-\frac{1}{2m+1}}
    +
    C' m \Delta_v
    \;,
\end{align*}
where the error term is defined as
\begin{align*}
    \Delta_v
    \;\coloneqq\;
    \Big( \mfrac{ \sum_{i=1}^n M_{\nu;i}^\nu }
    {  \sigma^\nu
    }
    \Big)^{\frac{1}{\nu m +1}}
    \;.
\end{align*}
\end{proposition}

\begin{remark*} (i) \cref{prop:universality:v:stat} is proved by an adaptation of variance domination (\Cref{thm:VD}) together with the variance ratio bounds in \cref{lem:v:var:approx}. (ii) While $\mean[v_m(X)]$ does not necessarily equal $\mean[v_m(Z)]$, the third bound of \cref{prop:universality:v:stat} implies that the difference is asymptotically negligible. (iii) An explicit bound on $\Delta_v$ is given in \cref{lem:universality:v:stat:aug}.
\end{remark*}

The proof of \cref{prop:universality:v:stat} is included in \cref{sec:proof:universality:v:stat}. The proof makes use of the next two lemmas, which are also useful for subsequent applications. The first lemma shows that $R_n$ can be used to control two variance ratios, which are similar to the moment ratio obtained from variance domination (\Cref{thm:VD}). Its proof is included in \Cref{sec:proof:v:var:approx}.

\begin{lemma}  \label{lem:v:var:approx} Suppose $(X_i)_{i \leq n}$ are i.i.d.~zero-mean. 
    % Assume that $v_m$ is $\delta$-regular with respect to $X$ for some $\delta \in [0,1)$. 
Assume a coupling between $\xi_i$ and $Z_i$ such that $\xi_{i1} = Z_i$ almost surely. Then under \Cref{assumption:V:stat}, there exists some absolute constants $C, C_* > 0$ such that
\begin{align*}
    \mfrac{ \| q^v_m(\Xi) + \mean[v_m(X)] - v_m(Z) \|_{L_2}^2 }{\Var[q^v_m(\Xi)]} \;\leq\; \mfrac{R_n C^m  }{ n}
\end{align*}
and $\big(1 - (C_*)^m R_n^{1/2} n^{-1/2}\big)^2 \leq \frac{\Var[v_m(Z)]}{\Var[v_m(X)]} \;\leq\; \big(1 + (C_*)^m R_n^{1/2} n^{-1/2} \big)^2$.
\end{lemma}

The next lemma applies \Cref{cor:non:multilinear} to show Gaussian universality with respect to the augmented variables $\bX$, and use the simple structure of $v_m$ to bound the moment ratio explicitly. The error bounds are given in terms of both $\tilde M_{\nu;i}$ defined in \Cref{thm:main} for $q^v_m$ and
\begin{align*}
    \alpha_\nu(S,k) 
    \;\coloneqq\; 
    \bigg\| 
    \msum_{\substack{p_1 + \ldots + p_k = m \\ 
    p_1, \ldots, p_k \geq 1
    }} 
    \Big\langle 
        S
        \,,\,
    \overline{X_1^{\otimes p_1}}
    \otimes \ldots  \otimes 
    \overline{X_k^{\otimes p_k}}
    \,
    \Big\rangle
    \bigg\|_{L_\nu} 
    \quad \text{ for } \nu \in [2,3] \text{ and } k \in [m]
    \;.
\end{align*}
We also provide upper and lower bounds on $\Var[v_m(X)]$ in terms of $\alpha_2(S,k)$. The proof is included in \Cref{sec:proof:universality:v:stat:aug}.

\begin{lemma} \label{lem:universality:v:stat:aug} Fix $\nu \in (2,3]$ and let $(X_i)_{i \leq n}$ be i.i.d.~zero-mean. Then under \Cref{assumption:V:stat}, there exists some absolute constant $C > 0$ such that, for every $n, m, d \in \N$, $t \in \R$ and $\sigma = \Var[v_m(X)] > 0$, 
\begin{align*}
    &\msup_{t \in \R} \big| \P( q^v_m(\bX) \leq t) - \P(  q^v_m(\Xi) \leq t) \big|
    \;\leq\;
    C m
    \Big( \mfrac{ \sum_{i=1}^n M_{\nu;i}^\nu }
    {  \sigma^\nu
    }
    \Big)^{\frac{1}{\nu m +1}}
    \;=\;
    C m
    \Delta_v
    \;.
\end{align*}
Meanwhile, if $n \geq 2m^2$, there exist absolute constants $C', C_1, C_2 > 0$ such that
\begin{align*}
    &\qquad \qquad\qquad\quad
    \Delta_v
    \;\leq\;
    C'
    \,
    n^{-\frac{\nu-2}{2\nu m + 2}}
    \,
    \beta_{m,\nu}^{\frac{\nu}{2 \nu m + 2}}\;,
    \qquad 
    \beta_{m,\nu} \;\coloneqq\; 
    \mfrac{
            \sum_{k=1}^m
            \binom{n}{k}
            (\alpha_\nu(S, k))^2
    }
    {
    \sum_{k=1}^m
    \binom{n}{k}
    (\alpha_2(S,k))^2
    }\;,
    \\
    &
    \text{ and } \quad
    \mfrac{(C_1)^m}{n^{2m}} 
    \msum_{k=1}^m \mbinom{n}{k} \, (\alpha_2(S,k))^2
    \;\leq\;
    \Var[v_m(X)]
    \;\leq\;
    \mfrac{(C_2)^m}{n^{2m}} 
    \msum_{k=1}^m \mbinom{n}{k} \, (\alpha_2(S,k))^2\;.
\end{align*}
\end{lemma}

\begin{remark*} The moment bound in \cref{lem:universality:v:stat:aug} involves a ratio of the $\nu$-th moment versus the second moment of the same quantity. This is a direct generalisation of the moment ratio from the classical Berry-Ess\'een bound for sample averages ($m=1$).
\end{remark*}

The proof for the above lemmas require two auxiliary results. The first one expresses a V-statistic as a sum of U-statistics:

\begin{lemma}[Theorem 1, p.183, \cite{lee1990u}] \label{lem:v:decompose} Given a function $u: (\R^d)^m \rightarrow \R$, we have
    \begin{align*}
        \mfrac{1}{n^m} \msum_{i_1, 
    \ldots, i_m \in [n]} u(y_{i_1}, \ldots, y_{i_m}) \;=&\; n^{-m} \msum_{j=1}^m \mbinom{n}{j} \tilde u_{j;m}(y_1, \ldots, y_n)\;,
    \end{align*}
    where $ \tilde u_{j;m}$ is a degree-$j$ U-statistic defined by 
    \begin{align*}
        \tilde u_{j;m}(y_1, \ldots, y_n) 
        \;\coloneqq&\; 
        \mfrac{1}{n(n-1) \ldots (n-j+1)} \msum_{i_1, \ldots, i_j \in [n] \text{ distinct}} 
        u_{j;m}(y_{i_1}, \ldots, y_{i_j}) 
        \\
        u_{j;m}(y_1, \ldots, y_j) 
        \;\coloneqq&\; 
        \msum_{\substack{(p_1, \ldots, p_m) \in [j]^m \\ \text{with exactly $j$ distinct elements} }} u(y_{p_1}, \ldots, y_{p_m})\;.
    \end{align*}
\end{lemma}

The next result controls the moment of a martingale, and is also used in \Cref{appendix:bootstrap,sec:proof:higher:u,appendix:MMD,appendix:high:d:average,sec:proof:delta,sec:proof:subgraph:u} to prove results for \Cref{sec:applications}, \Cref{appendix:delta:method} and \Cref{appendix:subgraph:u}. 
% \cref{sec:applications}. 
The original result is due to Burkholder \cite{burkholder1966martingale}. \cite{von1965inequalities} and \cite{dharmadhikari1968bounds} provides a further upper bound to Burkholder's result with an explicit constant $C_\nu$, and Lemma 26 of \cite{huang2023high} provides a further lower bound to Burkholder's result; these are summarised by the first statement of \cref{lem:martingale:bound}. We also include Burkholder's original upper bound, which is used when a finer control is required. The boundedness of $C'_\nu$ over a bounded interval is because \cite{burkholder1966martingale}'s constant arises from  Marcinkiewicz interpolation theorem and the Khintchine inequality.

\vspace{.5em}

\begin{lemma} \label{lem:martingale:bound} Fix $\nu > 1$. For a martingale difference sequence $Y_1,\ldots,Y_n$ taking values in $\R$,
    \begin{align*}
        c_{\nu} \, n^{\min\{0,\,\nu/2-1\}} \msum_{i=1}^n \mean[ |Y_i |^\nu ]
        \;\leq\;
        \mean \big[ \big| \msum_{i=1}^n Y_i \big|^\nu \big] 
        \;\leq\; 
        C_{\nu} \, n^{\max\{0,\,\nu/2-1\}} \msum_{i=1}^n \mean[ |Y_i |^\nu ]\;,
    \end{align*}
    for $C_{\nu} \coloneqq \max\big\{ 2,  (8(\nu-1)\max\{1,2^{\nu-3}\})^{\nu} \big\}$ and a constant $c_\nu > 0$ depending only on $\nu$. Moreover, there exists some constant $C'_\nu > 0$ depending only on $\nu$ such that 
    \begin{align*}
        \mean \big[ \big| \msum_{i=1}^n Y_i \big|^\nu \big] 
        \;\leq\; 
        C'_{\nu} \,\mean\Big[  \big| \msum_{i=1}^n Y_i^2 \big|^{\nu/2} \Big]  \;,
    \end{align*}
    where $\msup_{\nu \in I} C'_\nu$ is bounded whenever $I \subset (1,\infty)$ is a fixed bounded interval.
\end{lemma}

\subsection{Proof of \Cref{prop:simpler:V}} \label{appendix:proof:simpler:V}
Under the sub-Weibull condition of \Cref{assumption:sub:Weibull}, Theorem 1 of \cite{vladimirova2020sub} shows that there exists some constant $K_2$ that depends only on $K_1$ such that, for all $l \leq d$ and $q' \geq 1$,
\begin{align*}
    \mean \, | X_{1l} |^{q'}
    \;\leq\; 
    (K_2  (q')^\theta)^{q'}  \;.
    \tagaligneq \label{eq:sub:weibull:X1:k}
\end{align*}
Since $X_1, \ldots, X_n$ are i.i.d., we have
\begin{align*}
    &
    \;\mean \, \big|  \big\langle S 
    \,, \, 
    X_1^{\otimes q_1} \otimes \ldots \otimes  X_j^{\otimes q_j}   \big\rangle \big|^2
    \\
    &\;=\;
    \,
    \msum_{l_1, \ldots, l_k, l'_1, \ldots, l'_k \leq d}  
    \, 
    S_{l_1 \ldots l_k} \, S_{l'_1 \ldots l'_k}
    \mean 
    \Big[ 
    \mprod_{r=1}^{q_1} X_{1l_r} X_{1l'_r} 
    \Big]
    \ldots 
    \mean\Big[
    \mprod_{r=q_j-1}^{q_j} X_{jl_r} X_{jl'_r} 
    \Big] 
    \\
    &\;\leq\;
    \msum_{l_1, \ldots, l_k, l'_1, \ldots, l'_k \leq d}  
    \, 
    |S_{l_1 \ldots l_k}| \, |S_{l'_1 \ldots l'_k}|
    \;
    \mean 
    \Big[ 
    \mprod_{r=1}^{q_1} |X_{1l_r}| \, |X_{1l'_r} |
    \Big]
    \ldots 
    \mean 
    \Big[
    \mprod_{r=q_j-1}^{q_j} |X_{1l_r}| \,  |X_{1l'_r} |
    \Big]
    \\
    &\;\overset{(a)}{\leq}\;  
    \msum_{l_1, \ldots, l_k, l'_1, \ldots, l'_k \leq d}  
    \, 
    |S_{l_1 \ldots l_k}| \, |S_{l'_1 \ldots l'_k}|
    \;
    \Big( 
    \max_{l \leq d}
    \|X_{1l} \|_{L_{2m}}
    \Big)^{2 q_1 + \ldots + 2 q_j} 
    \\
    &\;=\;
    \| S \|_{l_1}^2 
    \,
    \Big( 
    \max_{l \leq d}
    \|X_{1l} \|_{L_{2m}}
    \Big)^{2m}
    \\
    &\;\leq\; 
    \| S \|_{l_1}^2  \, ( K_2  (2m)^\theta)^{2m}
    \;. \tagaligneq \label{eq:simper:V:crude:bound}
\end{align*}
In the last line, we have used the sub-Weibull condition. This allows us to control the $X_i$-dependent quantities in $R_n$:
% \eqref{eq:delta:reg:var} and
% \eqref{eq:delta:reg:mean}: 
% For any fixed $\delta \in (0,1)$, 
There exists some absolute constant $\tilde C > 0$ such that 
\begin{align*} 
    \max_{j \leq m-1} \,
    \max_{q_1 + \ldots + q_j = m\,,\, q_l \in \N  }
    \max \bigg\{ 
    \Var\bigg[ \bigg\langle S 
    \,, \, 
    \bigotimes_{l=1}^j 
    X_l^{\otimes q_l} 
    \bigg\rangle  \bigg]
    \,,\, 
    \bigg( \mean\bigg[ \bigg\langle S 
    \,, \,
    & 
    \bigotimes_{l=1}^j 
    X_l^{\otimes q_l} 
    \bigg\rangle  \bigg] \bigg)^2  
    \bigg\} 
    \\
    \;\leq&\; 
    \| S \|_{l_1}^2 \, ( K_2  (2m)^\theta)^{2m} 
    \;.
\end{align*} 
% for every $j \leq m-1$. In the last line, we have used $m \log m = o(\log n)$ and \Cref{assumption:V:stat}(iii) that $\Var[ v_m(X)] = \Omega(n^{-m})$.

\vspace{.5em}

To obtain an analogous statement for $Z_1, \ldots, Z_j$, it suffices to establish a bound analogous to \eqref{eq:sub:weibull:X1:k}. By \Cref{lem:bound:normal:moment}, there exists some absolute constant $\tilde C' > 0$ such that for all $l \leq d$ and $q' \geq 1$,
\begin{align*}
    \mean | Z_{1l} |^{q'} 
    \;\leq\;
    (\Var[X_{1l}])^{q'/2} \, C' \, (q')^{q' / 2}
    \;\leq&\;
    (K_2 2^\theta)^{q'}
    \,
    \tilde C' \, (q')^{q' / 2}
    \\
    \;\leq&\;
    \tilde C' 
    \,
    (K_2 2^\theta (q')^{1/2} )^{q'}
    \;.
\end{align*}
By the same argument as above with $K_2$ replaced by $(\tilde C')^{1/q'} K_2 2^\theta$ and $\theta$ replaced by $\frac{1}{2}$, 
\begin{align*} 
    \max_{j \leq m-1} \,
    \max_{q_1 + \ldots + q_j = m\,,\, q_l \in \N  }
    \max \bigg\{ 
    \Var\bigg[ \bigg\langle S 
    \,, \, 
    \bigotimes_{l=1}^j 
    Z_l^{\otimes q_l} 
    \bigg\rangle  \bigg]
    \,,\, 
    \bigg( \mean\bigg[ \bigg\langle S 
    \,, \,
    & 
    \bigotimes_{l=1}^j 
    Z_l^{\otimes q_l} 
    \bigg\rangle  \bigg] \bigg)^2  
    \bigg\} 
    \\
    \;\leq&\; 
     \| S \|_{l_1}^2 \, \tilde C' \, ( K_2 2^\theta  (2m)^\frac{1}{2} )^{2m}
     \;.
    % \\
    % \;\leq&\; 
    % \tilde C'' n^{m+\delta} m^{-j} \, \Var[v_m(X)]
\end{align*}
% for every $j \leq m-1$ and for some absolute constant $\tilde C'' > 0$. We have again used $m \log m = o(\log n)$ and $\Var[v_m(X)] = \Omega(n^{-m})$ above.

\vspace{.5em}

Combining the bounds and noting that $\theta$ is an absolute constant, we obtain that there exists some absolute constant $\tilde C'' > 0$ (that may depend on $\theta$) such that
\begin{align*}
    R_n 
    \;\leq\; 
    \tilde C''
    \mfrac{ 
         m^{m - 1 + \max\{2 \theta , 1\} m }  \, \| S \|_{l_1}^2
    }{
        n^m \Var[v_m(X)]
    }\;.
    \tagaligneq \label{eq:simpler:V:Rn:crude}
\end{align*}
% Under \Cref{assumption:V:stat}, $S$ is symmetric, which allows us to apply 
% and the bounds above verify \eqref{eq:delta:reg:var} and
% \eqref{eq:delta:reg:mean} for any fixed $\delta \in (0,1)$. This implies that $v_m$ is $\delta$-regular. Choose $\delta = 1 - \frac{2}{\nu}$. By 
By \Cref{prop:universality:v:stat}, we obtain that for some absolute constants $C_1, C'_1, C_2 > 0$,
\begin{align*}
    &\;
    \sup_{t \in \R} \big| \P( \sigma^{-1} v_m(X) \leq t) - \P( \sigma^{-1} v_m(Z) \leq t) \big|
    \\
    &\,\;\leq\;\,
    C_1 m n^{-\frac{1}{2m+1}} \bigg( 
        \mfrac{ 
            m^{m - 1 + \max\{2 \theta , 1\} m }  \,  \| S \|_{l_1}^2
        }{
            n^m \Var[v_m(X)]
        } 
    \bigg)^{\frac{1}{2m+1}}
    +
    C'_1 m  \Big( \mfrac{ \sum_{i=1}^n M_{\nu;i}^\nu }
    {  \sigma^\nu
    }
    \Big)^{\frac{1}{\nu m +1}}
    \\
    &\,\;\leq\;\, 
    C_2 m \Big( 
        m^{1+\theta} \,
        n^{ - \frac{1}{2 m + 1 } } \,
        \Big( 
        \mfrac{ 
            \| S \|_{l_1}^2
        }{
            n^m \Var[v_m(X)]
        } 
        \Big)^{\frac{1}{2m+1}}
    + 
    \Big( \mfrac{ \sum_{i=1}^n M_{\nu;i}^\nu }
    {  \sigma^\nu
    }
    \Big)^{\frac{1}{\nu m +1}} \Big)
    \;.
\end{align*}
Since $n \geq 2 m^2$, we may apply \Cref{lem:universality:v:stat:aug} to get that, for some absolute constants $C_3, C_4 > 0$, 
\begin{align*}
    \Big( \mfrac{ \sum_{i=1}^n M_{\nu;i}^\nu }
    {  \sigma^\nu
    }
    \Big)^{\frac{1}{\nu m +1}}
    \;\leq&\;
    C_3
    \,
    n^{-\frac{\nu-2}{2\nu m + 2}}
    \,
    \beta_{m,\nu}^{\frac{\nu}{2 \nu m + 2}}
    \;,
    \\
    \Var[v_m(X)]
    \;\geq&\;
    \mfrac{(C_4)^m}{n^{2m}} 
    \msum_{k=1}^m \mbinom{n}{k} \, (\alpha_2(S,k))^2
    \;,
\end{align*}
where we have followed the notation of \Cref{lem:universality:v:stat:aug} to write 
\begin{align*}
    \beta_{m,\nu} \;\coloneqq\; 
    \mfrac{
            \sum_{k=1}^m
            \binom{n}{k}
            (\alpha_\nu(S, k))^2
    }
    {
    \sum_{k=1}^m
    \binom{n}{k}
    (\alpha_2(S,k))^2
    }\;,
\end{align*}
and
\begin{align*}
    \alpha_{\nu'}(S,k) 
    \;\coloneqq&\; 
    \bigg\| 
    \msum_{\substack{p_1 + \ldots + p_k = m \\ 
    p_1, \ldots, p_k \geq 1
    }} 
    \Big\langle 
        S
        \,,\,
    \overline{X_1^{\otimes p_1}}
    \otimes \ldots  \otimes 
    \overline{X_k^{\otimes p_k}}
    \,
    \Big\rangle
    \bigg\|_{L_{\nu'}} 
    \quad \text{ for } \nu' \in [2,3] \text{ and } k \in [m]
    \;,
    \\
    \overline{X_i^{\otimes p_i}}
    \;\coloneqq&\;
    X_i^{\otimes p_i}
    -
    \mean \, X_i^{\otimes p_i}
    \;.
\end{align*}
This implies that, for some absolute constant $C_5 > 0$,
\begin{align*}
    &\;
    \sup_{t \in \R} \big| \P( \sigma^{-1} v_m(X) \leq t) - \P( \sigma^{-1} v_m(Z) \leq t) \big|
    \\
    &\,\;\leq\;\, 
    C_5 m \Big( 
        m^{1 + \theta} \,  n^{ - \frac{1}{2 m + 1 } } \,
        \Big( 
        \mfrac{ 
            \| S \|_{l_1}^2
        }{
            n^{-m}
            \sum_{k=1}^m \binom{n}{k} \, (\alpha_2(S,k))^2
        } 
        \Big)^{\frac{1}{2m+1}}
        +
        n^{-\frac{\nu-2}{2\nu m + 2}}
        \,
        \beta_{m,\nu}^{\frac{\nu}{2 \nu m + 2}}
        \Big)
    \;.
    \tagaligneq \label{eq:simple:v:intermediate}
\end{align*} 
We are left with simplifying $\alpha_{\nu'}(S,k)$ for $\nu' \in [2,3]$ via the sub-Weibull condition. To this end, note that by the triangle inequality and that $\nu' \leq 3 < 4$,
\begin{align*}
    \alpha_{\nu'}(S,k) 
    \;\leq&\;
    \msum_{\substack{p_1 + \ldots + p_k = m \\ 
    p_1, \ldots, p_k \geq 1
    }} 
    \bigg\| 
    \Big\langle 
        S
        \,,\,
    \overline{X_1^{\otimes p_1}}
    \otimes \ldots  \otimes 
    \overline{X_k^{\otimes p_k}}
    \,
    \Big\rangle
    \bigg\|_{L_{\nu'}} 
    \\
    \;\leq&\;
    m^m \, 
    \mmax_{\substack{p_1 + \ldots + p_k = m \\ 
    p_1, \ldots, p_k \geq 1
    }} 
    \,
    \underbrace{
    \bigg\| 
    \Big\langle 
        S
        \,,\,
    \overline{X_1^{\otimes p_1}}
    \otimes \ldots  \otimes 
    \overline{X_k^{\otimes p_k}}
    \,
    \Big\rangle
    \bigg\|_{L_4} }_{\eqqcolon (\star)} 
    \;.
\end{align*} 
First note that by \eqref{eq:sub:weibull:X1:k}, any entry of $\bar X_1^{\otimes p}$ satisfies that for all $l_1, \ldots, l_p \leq d$ and $q' \geq 1$,
\begin{align*}
    \mean \, \big| (\bar X_1^{\otimes p})_{l_1 \ldots l_p}  \big|^{q'} 
    \;=&\;
    \mean \big| X_{1l_1} \ldots X_{1 l_p} - \mean[X_{1l_1} \ldots X_{1 l_p}]  \big|^{q'}
    \\
    \;\leq&\;
    2^{q'} \, 
    \mean \big| X_{1l_1} \ldots X_{1 l_p} \big|^{q'}
    \\
    \;\leq&\;
    (2 K_2 (q')^\theta)^{q'}\;.
\end{align*}
Therefore as before, we can control $(\star)$ by explicitly expanding the inner product and using that $X_1, \ldots, X_k$ are i.i.d.:
\begin{align*}
    (\star)^4
    \;=&\;
    \mean 
    \bigg[ 
        \bigg|
        \Big\langle 
            S
            \,,\,
        \overline{X_1^{\otimes p_1}}
        \otimes \ldots  \otimes 
        \overline{X_k^{\otimes p_k}}
        \,
        \Big\rangle
        \bigg|^4
    \bigg]
    \;\leq\;
    \| S \|_{l_1}^4 (2 K_2 (4m)^\theta )^{4m}
    \;,
\end{align*}
which in turn yields that for all $\nu' \in [2,3]$, 
\begin{align*}
    \alpha_{\nu'}(S,k) 
    \;\leq\;
    m^m \| S \|_{l_1} (2 K_2 (4m)^\theta )^{m}
    \;.
    \tagaligneq \label{eq:simpler:V:crude:alpha}
\end{align*}
This implies that for any $\nu' \in [2,3]$,
\begin{align*}
    \msum_{k=1}^{m-1} \mbinom{n}{k} (\alpha_{\nu'}(S,k))^2
    \;\leq&\;
    m^{2m} 
    \, \| S \|_{l_1}^2
    \,
     (2 K_2 (4m)^\theta )^{2m} 
    \,
    \msum_{k=1}^{m-1} \mbinom{n}{k}
    \\
    \;\leq&\; 
    m^{2m} 
    \, \| S \|_{l_1}^2
    \,
     (2 K_2 (4m)^\theta )^{2m} 
    \,
    \msum_{k=1}^{m-1} \mfrac{e^k n^k}{k^k}
    \\
    \;\leq&\;
    K_3  \, 
    m^{m (3 + 2 \theta)} \, \| S \|_{l_1}^2 \, n^{m-1}
\end{align*}
for some absolute constant $K_3 > 0$. By further using the bound that $\frac{n^m}{m^m} \leq \binom{n}{m} \leq \frac{e^m n^m}{m^m}$, this implies 
\begin{align*}
    \beta_{m, \nu} 
    \;\leq&\; 
    \mfrac{
        \frac{e^m n^m}{m^m} 
        (\alpha_\nu(S,m))^2 
        +
        K_3  \, 
        m^{m (3 + 2 \theta)} \, \| S \|_{l_1}^2 \, n^{m-1}
    }
    {
        \max\Big\{ 
        \frac{n^m}{m^m}
        (\alpha_2(S,m))^2 
        -
        K_3  \, 
        m^{m (3 + 2 \theta)} \, \| S \|_{l_1}^2 \, n^{m-1}
        \,,\, 
        0
        \Big\} 
    }
    \\
    \;\leq&\;
    e^m 
    \, 
    \mfrac{
        (\alpha_\nu(S,m))^2 
        +
        K_3  \, 
        m^{m (4 + 2 \theta)} \, \| S \|_{l_1}^2 \, n^{-1}
    }
    {
        \max\Big\{ 
        (\alpha_2(S,m))^2 
        -
        K_3  \, 
        m^{m (4 + 2 \theta)} \, \| S \|_{l_1}^2 \, n^{-1}
        \,,\, 
        0
        \Big\} 
    }
    \\
    \;=&\;
    e^m 
    \, 
    \mfrac{
        \big\| \langle S, X_1 \otimes \ldots \otimes X_m \rangle \big\|_{L_\nu}^2 
        +
        K_3  \, 
        m^{m (4 + 2 \theta)} \, \| S \|_{l_1}^2 \, n^{-1}
    }
    {
        \max\Big\{ 
            \big\| \langle S, X_1 \otimes \ldots \otimes X_m \rangle \big\|_{L_2}^2 
        -
        K_3  \, 
        m^{m (4 + 2 \theta)} \, \| S \|_{l_1}^2 \, n^{-1}
        \,,\, 
        0
        \Big\} 
    }
\end{align*}
where we have used division by zero to denote a vacuous bound. Substituting the previous two bounds into \eqref{eq:simple:v:intermediate} gives that for some absolute constant $C_6 > 0$,
\begin{align*}
    &\;
    \sup_{t \in \R} \big| \P( \sigma^{-1} v_m(X) \leq t) - \P( \sigma^{-1} v_m(Z) \leq t) \big|
    \\
    &\,\;\leq\;\, 
    C_5 m \Big( 
        m^{1 + \theta} \,  n^{ - \frac{1}{2 m + 1 } } \,
        \Big( 
        \mfrac{ 
            \| S \|_{l_1}^2
        }{
            \max\big\{ 
            \frac{1}{m^m} \, (\alpha_2(S,k))^2
            -
            K_3  \, 
            m^{m (3 + 2 \theta)} \, \| S \|_{l_1}^2 \, n^{-1}
            \,,\,
            0 
            \big\}
        } 
        \Big)^{\frac{1}{2m+1}}
    \\
    &\hspace{5em} +
        n^{-\frac{\nu-2}{2\nu m + 2}}
        \,
        \beta_{m,\nu}^{\frac{\nu}{2 \nu m + 2}}
        \Big)
    \\
    &\,\;\leq\;\,
    C_6 m \bigg( 
        m^{2 + \theta} \,  n^{ - \frac{1}{2 m + 1 } } \,
        \Big( 
        \mfrac{ 
            \| S \|_{l_1}^2
        }{
            \max\big\{ 
                \big\| \langle S, X_1 \otimes \ldots \otimes X_m \rangle \big\|_{L_2}^2 
            -
            K_3  \, 
            m^{m (4 + 2 \theta)} \, \| S \|_{l_1}^2 \, n^{-1}
            \,,\,
            0 
            \big\}
        } 
        \Big)^{\frac{1}{2m+1}}
    \\
    &\hspace{4em} +
        n^{-\frac{\nu-2}{2\nu m + 2}}
        \,
        \bigg( 
            \mfrac{
                \big\| \langle S, X_1 \otimes \ldots \otimes X_m \rangle \big\|_{L_\nu}^2 
                +
                K_3  \, 
                m^{m (4 + 2 \theta)} \, \| S \|_{l_1}^2 \, n^{-1}
            }
            {
                \max\Big\{ 
                    \big\| \langle S, X_1 \otimes \ldots \otimes X_m \rangle \big\|_{L_2}^2 
                -
                K_3  \, 
                m^{m (4 + 2 \theta)} \, \| S \|_{l_1}^2 \, n^{-1}
                \,,\, 
                0
                \Big\} 
            }
        \bigg)^{\frac{\nu}{2 \nu m + 2}}
        \bigg)
    \;.
\end{align*}
Finally, we can use the bound $\alpha_\nu(S,m) = \| \langle S, X_1 \otimes \ldots \otimes X_m \rangle \|_{L_\nu} \leq m^m \| S \|_{l_1} (2K_2 (4m)^\theta)^m$ to control the above further by 
\begin{align*}
    &\;
    \sup_{t \in \R} \big| \P( \sigma^{-1} v_m(X) \leq t) - \P( \sigma^{-1} v_m(Z) \leq t) \big|
    \\
    &\,\;\leq\;\,
    C_7 m \bigg( 
        m^{2 + \theta} \,  n^{ - \frac{1}{2 m + 1 } } \,
        \Big( 
        \mfrac{ 
            \| S \|_{l_1}^2
        }{
            \max\big\{ 
                \big\| \langle S, X_1 \otimes \ldots \otimes X_m \rangle \big\|_{L_2}^2 
            -
            K_3  \, 
            m^{m (4 + 2 \theta)} \, \| S \|_{l_1}^2 \, n^{-1}
            \,,\,
            0 
            \big\}
        } 
        \Big)^{\frac{1}{2m+1}}
    \\
    &\hspace{4em} +
        n^{-\frac{\nu-2}{2\nu m + 2}}
        \,
        \bigg( 
            \mfrac{
                m^{m(4 + 2 \theta)} \| S \|_{l_1}^2 
                +
                K_3  \, 
                m^{m (4 + 2 \theta)} \, \| S \|_{l_1}^2 \, n^{-1}
            }
            {
                \max\Big\{ 
                    \big\| \langle S, X_1 \otimes \ldots \otimes X_m \rangle \big\|_{L_2}^2 
                -
                K_3  \, 
                m^{m (4 + 2 \theta)} \, \| S \|_{l_1}^2 \, n^{-1}
                \,,\, 
                0
                \Big\} 
            }
        \bigg)^{\frac{\nu}{2 \nu m + 2}}
        \bigg)
    \\
    &\,\;\leq\;\, 
    C m^{3 + \theta} n^{- \frac{\nu-2}{2\nu m + 2}}
    \,
    \bigg( 
            \mfrac{
                \| S \|_{l_1}^2 
                +
                \frac{K_3}{n} \,
                \| S \|_{l_1}^2 
            }
            {
                \max\Big\{ 
                    \big\| \langle S, X_1 \otimes \ldots \otimes X_m \rangle \big\|_{L_2}^2 
                -
                \frac{K_3}{n} \,
                m^{m (4 + 2 \theta)} \, \| S \|_{l_1}^2
                \,,\, 
                0
                \Big\} 
            }
        \bigg)^{\frac{\nu}{2 \nu m + 2}}
    \;.
\end{align*}
for some absolute constants $C_7, C \geq C_6 > 0$. Note that the bound above holds without the additional assumptions that $\| S \|_{l_1} = \Theta(1)$ and $m \log m = o(n)$. Under these two simplifying assumptions, we can rewrite the bounds above as 
\begin{align*}
    &\;
    \sup_{t \in \R} \big| \P( \sigma^{-1} v_m(X) \leq t) - \P( \sigma^{-1} v_m(Z) \leq t) \big|
    \\
    &\,\;\leq\;\,
    C m \bigg( 
        m^{2 + \theta} \,  n^{ - \frac{1}{2 m + 1 } } \,
        \Big( 
        \mfrac{ 
            \| S \|_{l_1}^2
        }{
                \| \langle S, X_1 \otimes \ldots \otimes X_m \rangle \|_{L_2}^2 
                + O(m^{m(4+2\theta)} n^{-1} )
        } 
        \Big)^{\frac{1}{2m+1}}
    \\
    &\hspace{4em} +
        n^{-\frac{\nu-2}{2\nu m + 2}}
        \,
        \bigg( 
            \mfrac{
                \big\| \langle S, X_1 \otimes \ldots \otimes X_m \rangle \big\|_{L_\nu}^2 
                + O(m^{m(4+2\theta)} n^{-1} )
            }
            {
                    \| \langle S, X_1 \otimes \ldots \otimes X_m \rangle \|_{L_2}^2 
                    + O(m^{m(4+2\theta)} n^{-1} )
            }
        \bigg)^{\frac{\nu}{2 \nu m + 2}}
        \bigg)
    \\
    & \,\;\leq\; \, 
    C m^{3 + \theta} n^{- \frac{\nu-2}{2\nu m + 2}}
    \,
    \bigg( 
            \mfrac{
                \| S \|_{l_1}^2 
                +
                O(n^{-1})
            }
            {
                \| \langle S, X_1 \otimes \ldots \otimes X_m \rangle \|_{L_2}^2 
                    + O(m^{m(4+2\theta)} n^{-1} )
            }
        \bigg)^{\frac{\nu}{2 \nu m + 2}}
    \;.
\end{align*}
\qed

\subsection{Proof of Lemma \ref{lem:v:var:approx}} \label{sec:proof:v:var:approx}
By \cref{lem:v:decompose}, we have
    \begin{align*}
        v_m(x_1, \ldots, x_n) 
        \;=&\; 
        n^{-m} \msum_{j=1}^m \mbinom{n}{j} \tilde u_{j;m}(x_1, \ldots, x_n)\;,
    \end{align*}
    where
    \begin{align*}
        \tilde u_{j;m}(x_1, \ldots, x_n) 
        \;\coloneqq&\; 
        \mfrac{1}{n(n-1) \ldots (n-j+1)} \msum_{i_1, \ldots, i_j \in [n] \text{ distinct}} 
         u_{j;m}(x_{i_1}, \ldots, x_{i_j}) 
        \\
         u_{j;m}(x_1, \ldots, x_j) 
        \;\coloneqq&\; 
        \msum_{\substack{(p_1, \ldots, p_m) \in [j]^m \\ \text{with exactly $j$ distinct elements} }} \langle S, \, x_{p_1} \otimes \ldots \otimes x_{p_m} \rangle
        \;.
    \end{align*}
    By the symmetry of $S$ in \Cref{assumption:V:stat}, we can express
    \begin{align*}
        &u_{j;m}(x_1, \ldots, x_j) 
        \;=\; 
        \msum_{q_1 + \ldots + q_j = m\,,\,q_l \in \N }
        \, \langle S, \, x_1^{\otimes q_1} \otimes \ldots \otimes x_j^{\otimes q_j} \rangle
        \\
        &\;=\;
        \sum_{\substack{q_1 + \ldots + q_j = m \\ q_l \in \N }}
        \, \langle S, \, \big( x_1^{\otimes q_1} - \mean[ X_1^{\otimes q_1} ]  + \mean[ X_1^{\otimes q_1} ]  \big) 
        \otimes \ldots \otimes 
        \big( x_j^{\otimes q_j} - \mean[ X_1^{\otimes q_j} ]  + \mean[ X_1^{\otimes q_j} ]  \big) \rangle\;.
    \end{align*}
    Analogously, we can express, for $y_i = (y_{i1}, \ldots, y_{im})$ with $y_{ij} \in \R^{d^j}$,
    \begin{align*}
        q^v_m(y_1, \ldots, y_n) + \mean[v_m(X)]
        \;=&\; 
        n^{-m} \msum_{j=1}^m \mbinom{n}{j} \tilde u^q_{j;m}(y_1, \ldots, y_n)
        \;,
    \end{align*}
    where
    \begin{align*}
        \tilde u^q_{j;m}(y_1, \ldots, y_n) 
        \;\coloneqq&\; 
        \mfrac{1}{n(n-1) \ldots (n-j+1)} \msum_{i_1, \ldots, i_j \in [n] \text{ distinct}} 
         u_{j;m}^q(y_{i_1}, \ldots, y_{i_j}) 
        \\
         u^q_{j;m}(y_1, \ldots, y_j) 
        \;\coloneqq&\; 
        \msum_{\substack{q_1 + \ldots + q_j = m \\ q_l \in \N }}
        \, \big\langle S, \, \big(y_{1 q_1} + \mean[ X_1^{\otimes q_1} ]  \big)  
        \otimes \ldots \otimes 
        \big( y_{j q_j} + \mean[ X_1^{\otimes q_j} ] \big)  \big\rangle\;.
    \end{align*}
    This allows us to express the difference of interest as a sum of centred U-statistics: 
    \begin{align*}
        \| q^v_m(\Xi) + \mean[v_m(X)] - v_m(Z) \|_{L_2}^2
        \;=&\; 
        \Big\| 
            n^{-m} \msum_{j=1}^m \mbinom{n}{j} 
            \tilde u^\Delta_{j;m}(\Xi)
        \Big\|_{L_2}^2
        \\
        \;\leq&\; 
        \Big( n^{-m} \msum_{j=1}^m \mbinom{n}{j} \,  \big\|  \tilde u^\Delta_{j;m}(\Xi) \big\|_{L_2} \Big)^2
        \\
        \;\leq&\; 
        m \,n^{-2m} \msum_{j=1}^m  \mbinom{n}{j}^2  \,
        \big\|  \tilde u^\Delta_{j;m}(\Xi) \big\|_{L_2}^2
        \;,
    \end{align*}
    where we have noted the coupling $\xi_{i1} = Z_i $ a.s.~and defined the degree-$j$ U-statistic 
    \begin{align*}
        \tilde u^\Delta_{j;m}(y_1, \ldots, y_n)
        \;\coloneqq&\; 
        \mfrac{1}{n(n-1) \ldots (n-j+1)} \msum_{i_1, \ldots, i_j \in [n] \text{ distinct}} 
        u_{j;m}^\Delta(y_{i_1}, \ldots, y_{i_j}) 
        \;,
        \\
        u_{j;m}^\Delta(y_1, \ldots, y_j)
        \;\coloneqq&\; 
        \msum_{\substack{q_1 + \ldots + q_j = m \\ q_l \in \N }}
        \Big( \big\langle S, \, \big(y_{1 q_1} + \mean[ X_1^{\otimes q_1} ]  \big)  
        \otimes \ldots \otimes 
        \big( y_{j q_j} + \mean[ X_1^{\otimes q_j} ] \big)  \big\rangle
        \\
        &\; 
        \qquad\qquad\qquad\quad
        -
        \big\langle S, \, y_{11}^{\otimes q_1} 
        \otimes \ldots \otimes 
        y_{j1}^{\otimes q_j}   \big\rangle 
        \Big)\;.
    \end{align*}
    The task now is to control $\| \tilde u^\Delta_{j;m}(\Xi) \|_{L_2}$. By the variance formula in \cref{lem:u:variance}, we have that 
    \begin{align*}
        \| \tilde u^\Delta_{j;m}&(\Xi) \|_{L_2}^2
        \;=\; 
        \big(\mean\big[ \tilde u^\Delta_{j;m}(\Xi) \big]\big)^2 
        + 
        \Var\big[ \tilde u^\Delta_{j;m}(\Xi) \big]
        \\ 
        &\;=\;
        \big(\mean\big[ u^\Delta_{j;m}(\Xi) \big]\big)^2 
        +
        \mbinom{n}{j}^{-1} \msum_{r=1}^j \mbinom{j}{r} \mbinom{n-j}{j-r} 
        \, \Var \, \mean\big[  u_{j;m}^\Delta(\xi_1, \ldots, \xi_j) \, \big| \, \xi_1, \ldots, \xi_r \big]
        \;.
    \end{align*}
    To compute the moments of $u^\Delta_{j;m}$, first note that we can express, for $0 \leq r \leq j$,
    \begin{align*}
        \mean\big[  u_{j;m}^\Delta(\xi_1, \ldots, \xi_j)  \, \big| &\, \xi_1, \ldots, \xi_r \big] \tagaligneq \label{eq:v:var:approx:u:moment}
        \\
        &\;=\; 
        \sum_{\substack{q_1 + \ldots + q_j = m \\ q_l \in \N  }}
        \bigg( \bigg\langle 
        S \,, \, 
        \bigotimes_{l=1}^r 
        \big(\xi_{l q_l} + \mean[ X_1^{\otimes q_l} ]  \big)  
        \otimes 
        \bigotimes_{l'=r+1}^j \mean[ X_1^{\otimes q_{l'}} ] 
        \bigg\rangle
        \\
        &\; 
        \qquad\qquad\qquad\qquad
        -
        \bigg\langle S 
        \,, \, 
        \bigotimes_{l=1}^r 
        \xi_{l1}^{\otimes q_l}  
        \otimes 
        \bigotimes_{l'=r+1}^j \mean \big[ \xi_{11}^{\otimes q_{l'}}  \big] 
        \bigg\rangle 
        \bigg)\;. 
    \end{align*} 
    Meanwhile, recall that $\xi_{i1} = Z_1$ almost surely,
    \begin{align*}
        \mean[ \xi_{i1} ] \;=\; \mean[X_1] \;=\; 0 \;,
    \end{align*}
    and, since $\mean[ X_i^{\otimes 2} ] = \mean[Z_i^{\otimes 2}] = \mean[\xi_{i1}^{\otimes 2}] $ and $\mean[\xi_{i2}]=0$, we have
    \begin{align*}
        \mean\big[ 
            \xi_{i2} + \mean\big[ X_1^{\otimes 2} \big]\big] -  \mean\Big[\big( \xi_{i1} + \mean[X_1]\big)^{\otimes 2}
        \Big] 
        \;=&\; 
        \mean\Big[ \mean\big[ X_1^{\otimes 2} \big]  - (Z_1 - \mean[Z_1])^{\otimes 2} -  \mean[X_1]^{\otimes 2} \Big]
        \\
        \;=&\;
        0
        \;.
    \end{align*}
    This implies that the summand in \eqref{eq:v:var:approx:u:moment} vanishes if
    \begin{proplist}
        \item $q_{l'} = 1$ for any $l=r+1,\ldots, j$, or
        \item  $q_l = 1$ for all $l=1,\ldots, r$ and $q_{l'} =2$ for all $l=r+1,\ldots, j$.
    \end{proplist}
    For \eqref{eq:v:var:approx:u:moment} to be non-zero, (i) and (ii) imply that $r + 2(j-r) < m$ by the pigeonhole principle. In other words, $r \geq 2j - m+1$, which allows us to rewrite
    \begin{align*}
        \Var\big[ \tilde u^\Delta_{j;m}(\Xi) \big]
        \;=&\;
        \mbinom{n}{j}^{-1} \msum_{r=2j-m+1}^j \mbinom{j}{r} \mbinom{n-j}{j-r} 
        \, \Var \, \mean\big[  u_{j;m}^\Delta(\xi_1, \ldots, \xi_j) \, \big| \, \xi_1, \ldots, \xi_r \big]
        \;,
        \\
        \mean\big[ u^\Delta_{j;m}(\Xi) \big]
        \;=&\;
        \mean\big[ u^\Delta_{j;m}(\Xi) \big] \;
        \ind_{\{ j \leq \lfloor \frac{m-1}{2} \rfloor \}}\;.
    \end{align*}
    Therefore by a Jensen's inequality combined with the above bound, we get that
    \begin{align*}
        \| q^v_m(\Xi) & +  \mean[v_m(X)] - v_m(Z) \|_{L_2}^2
        \\
        \;\leq&\;   
        \mfrac{m}{n^{2m}}  \, \msum_{j=1}^{\lfloor \frac{m-1}{2} \rfloor}  \mbinom{n}{j}^2  \,
        \big( \mean\big[ u^\Delta_{j;m}(\Xi) \big]\big)^2 
        \\
        &\; 
        + 
        \mfrac{m}{n^{2m}} \,  \msum_{j=1}^m \mbinom{n}{j}
        \msum_{r=2j-m+1}^j \mbinom{j}{r} \mbinom{n-j}{j-r} 
        \, \Var \, \mean\big[  u_{j;m}^\Delta(\xi_1, \ldots, \xi_j) \, \big| \, \xi_1, \ldots, \xi_r \big]
        \\
        \;\leq&\;   
        \mfrac{m}{n^{2m}} \, \msum_{j=1}^{\lfloor \frac{m-1}{2} \rfloor}  \, \mfrac{e^{2j} n^{2j}}{j^{2j}} \, \big( \mean\big[ u^\Delta_{j;m}(\Xi) \big]\big)^2 
        \\
        &\; 
        + 
        \mfrac{m}{n^{2m}} \,  \msum_{j=1}^{m-1} \mfrac{e^j n^j}{j^j}
        \msum_{r=2j-m+1}^j \mbinom{j}{r}  n^{m-1-j}
        \, \Var \big[  u_{j;m}^\Delta(\xi_1, \ldots, \xi_j) \big]
        \\
        \;\leq&\; 
        \mfrac{m}{n^{2m}}  \, \msum_{j=1}^{\lfloor \frac{m-1}{2} \rfloor}  \, \mfrac{e^{2j} n^{2j}}{j^{2j}} \, \big( \mean\big[ u^\Delta_{j;m}(\Xi) \big]\big)^2 
        \\
        &\; + 
        \mfrac{m}{n^{m+1}}\,  \msum_{j=1}^{m-1} \mfrac{(2e)^j}{j^j} \, \Var \big[  u_{j;m}^\Delta(\xi_1, \ldots, \xi_j) \big]
        \;. \tagaligneq \label{eq:v:var:approx:intermediate}
    \end{align*}
    To get a further control on the moment terms, note that 
    \begin{align*}
        \Var \big[  u_{j;m}^\Delta( \xi_1, \ldots, \xi_j) \big]
        \;=&\; 
        \Var \big[  u_{j;m}^q(\xi_1, \ldots, \xi_j) - u_{j;m}(\xi_{11}, \ldots, \xi_{j1})  \big]
        \\
        \;\leq&\; 
        2 \, \Var \,\big[  u_{j;m}^q(\xi_1, \ldots, \xi_j)  \big]
        +
        2 \,\Var \, \big[  u_{j;m}(\xi_{11}, \ldots, \xi_{j1}) \big]
        \;,
    \end{align*}
    and similarly 
    \begin{align*}
        \big( \mean\big[ u^\Delta_{j;m}(\Xi) \big]\big)^2 
        \;\leq\;
        2
        \big( \mean\big[ u^q_{j;m}(\Xi) \big]\big)^2 
        +
        2
        \big( \mean\big[ u_{j;m}(\Xi) \big]\big)^2 \;.
    \end{align*}
    Also recall that 
    \begin{align*}
        R_n 
        \;\coloneqq&\;
        \max_{j \leq m-1}
        \,
        \max_{q_1 + \ldots + q_j = m\,,\, q_l \in \N  }
        \\
        &\;\; \qquad 
        \mfrac{ m^j }{n^m \Var[ v_m(X)]}
        \,\times \,
        \max \bigg\{
            \Var\bigg[ \bigg\langle S 
            \,, \, 
            \bigotimes_{l=1}^j 
            X_l^{\otimes q_l} 
            \bigg\rangle  \bigg]
            \,,\, 
            \Var\bigg[ \bigg\langle S 
            \,, \,  
            \bigotimes_{l=1}^j 
            Z_l^{\otimes q_l} \bigg\rangle \bigg]
        \bigg\} 
        \\
        &\;\; 
        +
        \\
        &\;\;
        \max_{j \leq \lfloor \frac{m-1}{2} \rfloor}
        \,
        \max_{q_1 + \ldots + q_j = m\,,\, q_l \in \N  }
        \\
        &
        \qquad 
        \mfrac{n^{2j+1 - m}}{n^m \Var[v_m(X)]}
        \,\times \,
        \;\max\bigg\{
            \bigg( \mean\bigg[ \bigg\langle S 
            \,, \, 
            \bigotimes_{l=1}^j 
            X_l^{\otimes q_l} 
            \bigg\rangle  \bigg] \bigg)^2  
            \,,\,
            \bigg( \mean\bigg[ \bigg\langle S 
            \,, \, 
            \bigotimes_{l=1}^j 
            Z_l^{\otimes q_l} \bigg\rangle \bigg] \bigg)^2 
        \bigg\} 
        \;.
    \end{align*}
    Meanwhile, note that for any $j \in [m-1]$,
    \begin{align*}
        \Var \big[  u^q_{j;m}(\xi_{11}, \ldots, \xi_{j1})  \big]
        \;=&\;
        \Var\bigg[
        \msum_{\substack{q_1 + \ldots + q_j = m \\ q_l \in \N  }}
        \bigg\langle S 
        \,, \, 
        \bigotimes_{l=1}^j 
        \big( \xi_{1q_l} + \mean\big[ X_1^{\otimes q_l} \big] \big)
        \bigg\rangle 
        \bigg]
        \\
        \;\leq&\;
        \mbinom{m-1}{j-1}^2 \max_{q_1 + \ldots + q_j = m\,,\, q_l \in \N  }
        \Var\bigg[ \bigg\langle S 
        \,, \, 
        \bigotimes_{l=1}^j 
        \big( \xi_{lq_l} + \mean\big[ X_1^{\otimes q_l} \big] \big)
        \bigg\rangle  \bigg]
        \\
        \;\leq&\; 
        m^{2j} \max_{q_1 + \ldots + q_j = m\,,\, q_l \in \N  }
        \Var\bigg[ \bigg\langle S 
        \,, \, 
        \bigotimes_{l=1}^j 
        \big( \xi_{lq_l} + \mean\big[ X_1^{\otimes q_l} \big] \big)
        \bigg\rangle  \bigg]
        \\
        \;\overset{(a)}{=}&\;
        m^{2j} \max_{q_1 + \ldots + q_j = m\,,\, q_l \in \N  }
        \Var\bigg[ \bigg\langle S 
        \,, \, 
        \bigotimes_{l=1}^j 
         X_l^{\otimes q_l} 
        \bigg\rangle  \bigg]
        \\
        \;\overset{(b)}{\leq}&\; 
        m^j  n^m  \Var[v_m(X)]  R_n
        \;=\; 
        m^j  n^m   \Var[q^v_m(\Xi)] R_n \;,
    \end{align*}
    where we have used \cref{lem:gaussian:linear:moment} in $(a)$ and the definition of $R_n$ in $(b)$. Similarly 
    \begin{align*}
        \Var \big[  u_{j;m}(\xi_{11}, \ldots, \xi_{j1})  \big]
        \;\leq&\; 
        m^{2j} 
        \; 
        \max_{q_1 + \ldots + q_j = m\,,\, q_l \in \N  }
        \Var\bigg[ \bigg\langle S 
        \,, \, 
        \bigotimes_{l=1}^j 
        Z_l^{\otimes q_l} \bigg\rangle \bigg]
        \\
        \;\leq&\; 
        m^j \, n^m \, \Var[q^v_m(\Xi)] R_n \;,
    \end{align*}
    whereas
    \begin{align*}
        &\; 
        \max\big\{ 
        \big( \mean\big[ u^q_{j;m}(\Xi) \big] \big)^2
        \,,\, 
        \big( \mean\big[ u_{j;m}(\Xi) \big] \big)^2
        \big\}
        \\
        &\;\leq\; 
        m^{2j} 
        \, 
        \max_{ \substack{q_1 + \ldots + q_j = m \\ q_l \in \N } }
        \max\bigg\{
        \bigg( \mean\bigg[ \bigg\langle S 
        \,, \, 
        \bigotimes_{l=1}^j 
        \big( \xi_{lq_l} + \mean\big[ X_1^{\otimes q_l} \big] \big)
        \bigg\rangle  \bigg] \bigg)^2  \,,\,
        \bigg( \mean\bigg[ \bigg\langle S 
        \,, \, 
        \bigotimes_{l=1}^j 
        Z_l^{\otimes q_l} \bigg\rangle \bigg] \bigg)^2 
        \bigg\}
        \\
        &\;=\;
        m^{2j} 
        \, 
        \max_{q_1 + \ldots + q_j = m\,,\, q_l \in \N  }
        \max\bigg\{
        \bigg( \mean\bigg[ \bigg\langle S 
        \,, \, 
        \bigotimes_{l=1}^j 
         X_1^{\otimes q_l} 
        \bigg\rangle  \bigg] \bigg)^2  \,,\,
        \bigg( \mean\bigg[ \bigg\langle S 
        \,, \, 
        \bigotimes_{l=1}^j 
        Z_l^{\otimes q_l} \bigg\rangle \bigg] \bigg)^2 
        \bigg\}
        \\
        &\;\leq\; 
        m^{2j} \,  n^{2m-1-2j}   \, \Var[q^v_m(\Xi)] \, R_n \;.
    \end{align*}
    Substituting these bounds into \eqref{eq:v:var:approx:intermediate}, we get that 
    \begin{align*}
        \| q^v_m(\Xi) +  & \mean[v_m(X)] -  v_m(Z) \|_{L_2}^2
        \\
        \;\leq&\; 
        4  R_n \,  m \, n^{-2m}  \, \msum_{j=1}^{\lfloor \frac{m-1}{2} \rfloor}  \, \mfrac{e^{2j} n^{2j + (2m-1-2j)} m^{2j}}{j^{2j}} \, \Var[q^v_m(\Xi)] 
        \\
        &\; + 
        4 R_n \, m \, n^{-m-1}\,  \msum_{j=1}^{m-1} \mfrac{ (2e)^j m^j n^m}{j^j}  \, \Var[q^v_m(\Xi)] 
        \\
        \;\leq&\;
        4 R_n \, m \, n^{-1} \, \Var[q^v_m(\Xi)] 
        \Big( 
            \msum_{j=1}^{\lfloor \frac{m-1}{2} \rfloor}  \, \mfrac{e^{2j} m^{2j}}{j^{2j}} 
            + 
            \msum_{j=1}^{m-1} \mfrac{ (2e)^j m^j}{j^j}
        \Big)
        \\
        \;\leq&\; 
        4 R_n \, m \, n^{-1} \,  \Var[q^v_m(\Xi)] 
        \Big( 
            \msum_{j=0}^{m} e^{2j} \mbinom{m}{j}^2
            + 
            \msum_{j=0}^{m} (2e)^j \mbinom{m}{j}
        \Big)
        \\
        \;\leq&\;
        4 R_n \, m \, n^{-1} \,  \Var[q^v_m(\Xi)] 
        \Big( 
            \Big( \msum_{j=0}^{m} e^j \mbinom{m}{j} \Big)^2
            + 
            \msum_{j=0}^{m} (2e)^j \mbinom{m}{j}
        \Big)
        \\
        \;\leq&\;
        C^m \, R_n \, n^{-1} \, \Var[q^v_m(\Xi)]  
    \end{align*}
    for some absolute constant $C > 0$. This proves the first bound. To show the second bound, we use a triangle inequality to get that
    \begin{align*}
        \Var[&v_m(Z)]
        \;=\; 
        \| q^v_m(\Xi) +  \mean[v_m(X)] -  v_m(Z)  + \mean[v_m(Z)] - q^v_m(\Xi) - \mean[v_m(X)] \|_{L_2}^2
        \\
        \;\leq&\; 
        \big( 
            \| q^v_m(\Xi) +  \mean[v_m(X)] -  v_m(Z) \|_{L_2} 
            +
            \| q^v_m(\Xi) \|_{L_2} 
            +
            |  \mean[v_m(Z)] - \mean[v_m(X)] |
        \big)^2
        \\
        \;=&\; 
        \Big( 
            \| q^v_m(\Xi) +  \mean[v_m(X)] -  v_m(Z) \|_{L_2} 
            +
            \sqrt{\Var[q^v_m(\Xi)]}
        \\
        &\qquad\qquad
            + 
            \big|  \mean\big[v_m(Z) - q^v_m(\Xi) - \mean[v_m(X)]\big] \big|
        \Big)^2
        \\
        \;\leq&\; 
        \Big( 
            \, 2 \| q^v_m(\Xi) +  \mean[v_m(X)] -  v_m(Z) \|_{L_2} +
        \sqrt{\Var[q^v_m(\Xi)]}
        \; \Big)^2
        \\
        \;\leq&\; 
        \Var[q^v_m(\Xi)] \big( 1 + (C_*)^m  R_n^{1/2} n^{-1/2}  \big)^2
        \;=\; 
        \Var[v_m(X)] \big( 1 + (C_*)^m R_n^{1/2} n^{-1/2} \big)^2
        \;,
    \end{align*}
    where we have written $(C_*)^m = 2 C^{m/2}$. Similarly by a reverse triangle inequality,
    \begin{align*}
        \hspace{6em}
        \Var[v_m(Z)] \;\geq&\;  
        \big(  \sqrt{\Var[q^v_m(\Xi)]} 
        - 2 \big\| q^v_m(\Xi) +  \mean[v_m(X)] -  v_m(Z) \big\|_{L_2} 
        \big)^2
        \\
        \;\geq&\; 
        \Var[v_m(X)] \big( 1 - (C_*)^m R_n^{1/2} n^{-1/2} \big)^2\;.
        \hspace{8em} \qed 
    \end{align*}

\subsection{Proof of Lemma \ref{lem:universality:v:stat:aug}} 
\label{sec:proof:universality:v:stat:aug}
The first result holds directly by \Cref{thm:main}: 
% \cref{thm:main}: 
There exists an absolute constant $C > 0$ such that 
\begin{align*}
    \msup_{t \in \R} \big| \P(  q^v_m(\bX) \leq t) - \P( q^v_m(\Xi) \leq t) \big|
    \;\leq&\;
    C m\,
    \Big( \mfrac{ \sum_{i=1}^n M_{\nu;i}^\nu }
    {  \sigma^\nu
    }
    \Big)^{\frac{1}{\nu m +1}}
    \;. \tagaligneq \label{eq:universality:v:stat:first:bound}
\end{align*}
We now show that, if $n \geq 2m^2$, there exists some absolute constant $C' > 0$ such that 
\begin{align*}
    \Big( \mfrac{ \sum_{i=1}^n M_{\nu;i}^\nu }
    { \sigma^\nu
    }
    \Big)^{\frac{1}{\nu m +1}}
    \;\leq\;
    C' 
    \,
    n^{-\frac{\nu-2}{2\nu m + 2}}
    \,
    \Bigg(
    \mfrac{
            \sum_{k=1}^m
            \binom{n}{k}
            (\alpha_\nu(S, k))^2
    }
    {
    \sum_{k=1}^m
    \binom{n}{k}
    (\alpha_2(S,k))^2
    }
    \Bigg)^{\frac{\nu}{2 \nu m + 2}}\;.
    \tagaligneq \label{prop:universality:v:stat:second:bound}
\end{align*}
We proceed by induction. For $i \leq n$ and $\cI_0, \cI_1 \subseteq [n]$ with $|\cI_1| \leq m$ and $\cI_0 \cap \cI_1 = \emptyset$,  define 
    \begin{align*}
        T(i,\cI_0, &\cI_1)
        \;\coloneqq\;
        \\
        & \msum_{\substack{p_1 + \ldots + p_n = m \\ 0 \leq p_l \leq m \;\; \forall \, l \in [n]
        \\ p_{l'} \geq 1 \;\; \forall \,l' \in \cI_1 
        \\ p_{l'} = 0 \;\; \forall \,l' \in \cI_0 
        }} 
        \mbinom{m}{p_1 \, \ldots\, p_n}
        \Big\langle 
            S
            \,,\,
        \overline{X_1^{\otimes p_1}}
        \otimes \ldots  \otimes 
        \overline{X_{i}^{\otimes p_{i}}}  \otimes \xi_{(i+1)p_{i+1}} \otimes \ldots \otimes \xi_{n p_n} 
        \Big\rangle\;,
    \end{align*}
    where $\binom{m}{p_1 \, \ldots\, p_n} \coloneqq \frac{m!}{p_1! \ldots p_n!}$ is the multinomial coefficient. The quantities of interest become
    \begin{align*}
        \sigma \;=\; \sqrt{\Var[v_m(X)]} \;=\; 
        \mfrac{ \big\| T(n, \emptyset, \emptyset) \big\|_{L_2}}{n^m}\;, 
        \;\;\;
        M_{\nu;i} \;=\; 
        \big\| \partial_i q^v_m(\bW_i)^\top \bX_i \big\|_{L_\nu} 
        \;=\; 
        \mfrac{\big\|  T(i ,  \{i\}, \emptyset )  \big\|_{L_\nu}}{n^m}
        \;.
    \end{align*}
    Write $\alpha_\nu(k) = \alpha_\nu(S,k)$ for short. We first show by induction on $|\cI_1|$ that for $\nu \in [2,3]$ 
    \begin{align*}
        \big\|  T(i,  \cI_0, \cI_1) \big\|_{L_\nu}^2
        \;\leq&\; 
        (C_*)^m
        \msum_{k=0}^{m - |\cI_1|}
        \mfrac{(C'')^{k+1}}{k!}
        (n-|\cI_0|-|\cI_1|+1)^k
        ( \alpha_\nu(|\cI_1|+k))^2
        \;, 
        \tagaligneq \label{eq:simple:v:moment:induct:nu}
        \\
        \big\|  T(i, \cI_0, \cI_1) \big\|_{L_2}^2
        \;\geq&\; 
        \msum_{k=0}^{m - |\cI_1|}
        \Big( \mfrac{(c'')^{k+1}}{k!}
        ( \alpha_2(|\cI_1|+k))^2
        \tagaligneq \label{eq:simple:v:moment:induct:two}
        \\
        &\quad\;\; \,
        \times
        \max\big\{ (n-|\cI_0|-|\cI_1|+1)^k - k^2 (n-|\cI_0|-|\cI_1|+1)^{k-1} , 0\big\}  
        \Big) 
        \;.
    \end{align*}
    $C_* \geq 1$ is the supremum of the constant in the upper bound of \cref{lem:gaussian:linear:moment} over $\nu \in [2,3]$, $C''$ is the supremum of the constant in the first upper bound of \cref{lem:martingale:bound} over $\nu \in [2,3]$, and $c''$ is the constant in the lower bound of \cref{lem:martingale:bound} when $\nu=2$. 
    
    \vspace{.5em}
    
    For the base case, if $|\cI_1| = n - |\cI_0| $, i.e.~$\cI_0 = [n] \setminus \cI_1$, by using the independence of $X_1, \ldots, X_n$ and applying \cref{lem:gaussian:linear:moment} to replace each $\xi_{i}$ by $\big( \overline{X_i^{\otimes 1}} , \ldots ,   \overline{X_i^{\otimes m}}\big)$, we get that
    \begin{align*}
        \|  T(i, \cI_0, \cI_1) \|_{L_2}^2
        \;=&\; 
        \bigg\| 
        \msum_{\substack{ \sum_{l \in \cI_1} p_l = m \\ 
        p_l \geq 1 \;\;\forall\, l \in \cI_1
        }} 
        \,\mfrac{m!}{\prod_{l \in \cI_1} (p_l !)}\,
        \Big\langle 
            S
            \,,\,
            \bigotimes_{l \in \cI_1}
        \overline{X_l^{\otimes p_l}}
        \,
        \Big\rangle
        \bigg\|_{L_2}^2
        \\
        \;\geq&\;
         \big( \alpha_2(|\cI_1|)\big)^2
        \;.
    \end{align*}
    By applying \cref{lem:gaussian:linear:moment} again with $\nu \in [2,3]$, we get that 
    \begin{align*}
        \|  T(i, \cI_0, \cI_1) \|_{L_\nu}^2
        \leq
        (C_*)^{|\cI_1|}
        \bigg\| 
        \sum_{\substack{ \sum_{l \in \cI_1} p_l = m \\ 
        p_l \geq 1 \;\;\forall\, l \in \cI_1
        }} 
        \,\mfrac{m!}{\prod_{l \in \cI_1} (p_l !)}\,
        \Big\langle 
            S
            \,,\,
            \bigotimes_{l \in \cI_1}
        \overline{X_l^{\otimes p_l}}
        \,
        \Big\rangle
        \bigg\|_{L_\nu}^2 
        \leq
        (C_*)^m
        \big( \alpha_\nu(|\cI_1|)\big)^2.
    \end{align*}
    Meanwhile, if $|\cI_1| = m$, we automatically have $|\cI_1| = n - |\cI_0| $, so the above also holds.
    
    \vspace{.5em}

    For the inductive step, fix $\cI_1 \subseteq [n]$ with $|\cI_1| < \min\{n-|\cI_0|,m\}$. Suppose \eqref{eq:simple:v:moment:induct:nu} and \eqref{eq:simple:v:moment:induct:two} hold for all disjoint $\cI'_0, \cI'_1 \subseteq [n]$ with $|\cI'_1| > |\cI_1|$. For convenience, write $\bV^{(i)}_j \coloneqq \bX_j$ for $j \leq i$ and $\bV^{(i)}_j \coloneqq \xi_j$ for $j > i$. We also denote $V^{(i)}_{jp_j} = \overline{X_j^{\otimes p_j}}$ for $j \leq i$ and $V^{(i)}_{jp_j} = \xi_{j p_j}$ for $j > i$.

    \vspace{.5em} 
    
    Notice that by definition, $T(i,\cI_0,\cI_1)$ does not depend on $\bV^{(i)}_j$ for $j \in \cI_0$. We enumerate the elements of $[n] \setminus (\cI_0 \cup \cI_1)$ as $i_1 < \ldots < i_N$ for $N \coloneqq n-|\cI_0|-|\cI_1|$. Also denote $\bV^{(i)}_\cI \coloneqq \big\{ \bV^{(i)}_j \,\big|\, j \in \cI \big\} $. Since $\mean[ T(i,\cI_0,\cI_1)] = 0$, we can define a martingale difference sequence by
    \begin{align*}
         D_0(i, \cI_0, \cI_1)
        \;\;\coloneqq&\;\;
        \mean\Big[  T(i, \cI_0, \cI_1) \,\Big|\, \bV^{(i)}_{\cI_1} \Big]
        \; 
    \end{align*}
    and, for $1 \leq j \leq N$,
    \begin{align*}
         D_j(i, \cI_0, \cI_1)
        \;\coloneqq&\;
        \mean\Big[  T(i, \cI_0, \cI_1) \,\Big|\, \bV^{(i)}_{\cI_1} ,\bV^{(i)}_{i_1}, \ldots, \bV^{(i)}_{i_j}  \Big]
        \\
        &\quad 
        -
        \mean\Big[  T(i,  \cI_0, \cI_1) \,\Big|\, \bV^{(i)}_{\cI_1} , \bV^{(i)}_{i_1}, \ldots, \bV^{(i)}_{i_{j-1}} \Big].
    \end{align*}
    Then almost surely,
    \begin{align*}
        \bar T(i,\cI_0,\cI_1) \;=\; \msum_{j=0}^N \bar D_j(i,\cI_0,\cI_1)\;.
    \end{align*}
    By applying \cref{lem:martingale:bound} with $\nu \in [2,3]$ followed by a triangle inequality, we get that 
    \begin{align*}
        \big\|  T(i, \cI_0, \cI_1) \big\|_{L_\nu}^2
        \;\leq&\; 
        C''_* \,\Big\| \msum_{j=0}^N  D_j(i, \cI_0, \cI_1)^2 \Big\|_{L_{\nu/2}}
        \;\leq\; 
        C'' \, \msum_{j=0}^N \|  D_j(i, \cI_0, \cI_1) \|_{L_\nu}^2
        \;, 
        \tagaligneq \label{eq:simple:v:moment:induct:mds:upper}
        \\
        \big\|  T(i, \cI_0, \cI_1)  \big\|_{L_2}^2
        \;\geq&\; 
        c'' \, \msum_{j=0}^N \| D_j(i, \cI_0, \cI_1) \|_{L_2}^2
        \;.
        \tagaligneq \label{eq:simple:v:moment:induct:mds:lower}
    \end{align*}
    To control the martingale difference terms, recall that $\mean\big[V^{(i)}_{jp_j}\big]=0$ for all $i,j \in [n]$. Therefore
    \begin{align*}
        D_0(i, \cI_0, \cI_1) 
        \;=&\; 
        \msum_{\substack{p_1 + \ldots + p_n = m \\ 0 \leq p_l \leq m \;\; \forall \, l \in [n]
        \\ p_{l'} \geq 1 \;\; \forall \,l' \in \cI_1
        \\ p_{l'} = 0\;\; \forall \, l' \in \cI_0 }
         } 
        \mean\Big[
        \big\langle 
            S
            \,,\,
        V^{(i)}_{1 p_1}
        \otimes \ldots  \otimes 
        V^{(i)}_{n p_n}
        \big\rangle
        \,\Big|\,  \bV^{(i)}_{\cI_1} \Big] 
        \;=\; 
         T(i, \, [n] \setminus \cI_1, \, \cI_1)
        \;,
    \end{align*}
    i.e.~there is no dependence on $\bV^{(i)}_j$ for any $j \not\in \cI_1$. Similarly for $1 \leq j \leq N$,
    \begin{align*}
        D_j(i, \cI_0, \cI_1) 
        \;=&\; 
        \msum_{\substack{p_1 + \ldots + p_n = m \\ 0 \leq p_l \leq m \;\; \forall \, l \in [n]
        \\ p_{l'} \geq 1 \;\; \forall \,l' \in \cI_1
        \\ p_{l'} = 0 \;\; \forall \, l' \in \cI_0 }
         } 
        \Big( 
        \mean\Big[
        \big\langle 
            S
            \,,\,
        V^{(i)}_{1 p_1}
        \otimes \ldots  \otimes 
        V^{(i)}_{n p_n}
        \big\rangle
        \,\Big|\,   \bV^{(i)}_{\cI_1} , \bV^{(i)}_{i_1}, \ldots, \bV^{(i)}_{i_j} \Big] 
        \\
        &\qquad\qquad\qquad\quad\;
        -
        \mean\Big[
        \big\langle 
            S
            \,,\,
        V^{(i)}_{1 p_1}
        \otimes \ldots  \otimes 
        V^{(i)}_{n p_n}
        \big\rangle
        \,\Big|\,  \bV^{(i)}_{\cI_1} ,  \bV^{(i)}_{i_1}, \ldots, \bV^{(i)}_{i_{j-1}} \Big]
        \Big)
        \\[.5em]
        \;=&\; 
        T(i, \cI_0 \cup \{i_{j+1} , \ldots, i_N\}, \cI_1 \cup \{i_j\}  )
        \;.
    \end{align*}
    By  noting that $N= n-|\cI_0|-|\cI_1|$ and applying the inductive statement , we get that
    \begin{align*}
        &\; \msum_{j=0}^N \|  D_j(i,\cI_0, \cI_1) \|_{L_\nu}^2
        \\
        &\;=\; 
        \big\|  T(i, [n] \setminus \cI_1, \cI_1 ) \big\|_{L_\nu}^2
        +
        \msum_{j=1}^N \big\|  T(i, \cI_0 \cup \{i_{j+1} , \ldots, i_N\} , \cI_1 \cup \{i_j\} ) \big\|_{L_\nu}^2
        \\
        &\;\leq\;
        ( C_*)^m ( \alpha_\nu(|\cI_1|))^2
        +
        (C_*)^m 
        \msum_{j=1}^N 
        \msum_{k=0}^{m - (|\cI_1| +1) }
        \Big(\mfrac{(C'')^{k+1}}{k!} 
        \times 
        \\
        & \qquad \qquad  \qquad \qquad  \qquad \quad  (n-(|\cI_0|+N-j)-(|\cI_1|+1)+1)^k
        ( \alpha_\nu(|\cI_1|+1+k))^2 
        \Big)
        \\
        &\;=\; 
        (C_*)^m ( \alpha_\nu(|\cI_1|))^2
        +
        (C_*)^m 
        \msum_{j=1}^N 
        \msum_{k=1}^{m - |\cI_1|  }
        \mfrac{(C'')^k}{(k-1)!}
        j^{k-1}
        (\alpha_\nu(|\cI_1|+k))^2
        \\
        &\;=\;
        (C_*)^m ( \alpha_\nu(|\cI_1|))^2
        +
        (C_*)^m 
        \msum_{k=1}^{m - |\cI_1|}
        \mfrac{(C'')^k}{(k-1)!}
        \big(
            \msum_{j=1}^N j^{k-1}
        \big)
        (\alpha_\nu(|\cI_1|+k))^2 \;.
    \end{align*}
    Since $k-1 \geq 0$, we have
    \begin{align*}
        \msum_{j=1}^N j^{k-1}
        \;\leq&\;
        \msum_{j=1}^N \mint_{j}^{j+1} x^{k-1} dx
        \;=\; 
        \msum_{j=1}^N \mfrac{(j+1)^k - j^k }{k}
        \;<\; 
        \mfrac{(N+1)^k}{k}\;.
    \end{align*}
    Substituting these into \eqref{eq:simple:v:moment:induct:mds:upper}, we get that
    \begin{align*}
        \big\|  T(i, \cI_0, \cI_1) \big\|_{L_\nu}^2
        \;\leq&\; 
        C''  \msum_{j=0}^N \|  D_j(i, \cI_0, \cI_1) \|_{L_\nu}^2
        \\
        \;\leq&\; 
        C'' (C_*)^m
        \msum_{k=0}^{m - |\cI_1|}
        \mfrac{(C'')^k}{k!}
        (N+1)^k
        ( \alpha_\nu(|\cI_1|+k))^2
        \\
        \;=&\;
        (C_*)^m
        \msum_{k=0}^{m - |\cI_1|}
        \mfrac{(C'')^{k+1}}{k!}
        (n-|\cI_0|-|\cI_1|+1)^k
        ( \alpha_\nu(|\cI_1|+k))^2
        \;, 
    \end{align*}
    which finishes the induction for \eqref{eq:simple:v:moment:induct:nu}. Similarly for \eqref{eq:simple:v:moment:induct:two}, using the inductive statement gives
    \begin{align*}
        &\; \msum_{j=0}^N \|  D_j(i,\cI_0, \cI_1) \|_{L_2}^2 
        \\
        &\;=\; 
        \big\|  T(i, [n] \setminus \cI_1,\cI_1 ) \big\|_{L_2}^2
        +
        \msum_{j=1}^N \big\|  T(i, \cI_0 \cup \{i_{j+1} , \ldots, i_N\},\cI_1 \cup \{i_j\} ) \big\|_{L_2}^2
        \\
        &\;\geq\; 
        ( \alpha_2(|\cI_1|))^2
        +
        \sum_{j=1}^N 
        \sum_{k=0}^{m - (|\cI_1| +1) }
        \Big( 
        \mfrac{( c'' )^{k+1}}{k!}
        ( \alpha_2(|\cI_1|+1+k))^2 
        \max\{ j^k - k^2 j^{k-1} \,,\, 0 \} 
        \Big)
        \\
        &\;=\; 
        ( \alpha_2(|\cI_1|))^2
        +
        \sum_{k=1}^{m - |\cI_1|}
        \mfrac{(c'' )^k}{(k-1)!}
        \Bigg(
            \sum_{j=1}^N 
            \max\big\{ j^{k-1}
            - (k-1)^2 j^{k-2} \,,\, 0\big\}
        \Bigg)
        ( \alpha_2(|\cI_1|+k))^2\;.
    \end{align*}
    Since $k-1 \geq 0$, we have
    \begin{align*}
        & 
        \msum_{j=1}^N 
        \max\big\{ j^{k-1}
        - (k-1)^2 j^{k-2} \,,\, 0\big\} 
        \\
        &\;\geq\; 
        \max\big\{ 
        \msum_{j=1}^N 
        \big( j^{k-1}
        - (k-1)^2 j^{k-2} \big) \,,\, 0\big\} 
        \\ 
        &\;=\;
        \max\Big\{ 
            \msum_{j=1}^{N+1}  
            j^{k-1}
            -
            (N+1)^{k-1}
            - 
            \msum_{j=1}^N  
            (k-1)^2 j^{k-2} 
            \,,\,
            0
        \Big\} 
        \\
        &\;\geq\;  
        \max\Big\{ 
            \msum_{j=1}^{N+1}  
            \mint_{j-1}^{j} x^{k-1} dx
            -
            (N+1)^{k-1}
            - 
            \msum_{j=1}^N  
            (k-1)^2 
            \mint^{j+1}_j x^{k-2} dx
            \,,\,
            0
        \Big\} 
        \\
        &\;=\; 
        \max\bigg\{ 
            \sum_{j=1}^{N+1}  
            \mfrac{j^k - (j-1)^k}{k}
            -
            (N+1)^{k-1}
            - 
            \sum_{j=1}^N  
            (k-1)
            ((j+1)^{k-1} - j^{k-1})
            \,,\,
            0
        \bigg\} 
        \\
        &\;>\; 
        \max\Big\{ 
            \mfrac{(N+1)^k}{k}
            -
            (N+1)^{k-1}
            - 
            (k-1)
            (N+1)^{k-1}
            \,,\,
            0
        \Big\} 
        \\
        &
        \;=\; 
        \mfrac{1}{k} 
        \max\big\{ (N+1)^k - k^2 (N+1)^{k-1} \,,\, 0 \big\}
        \;,
    \end{align*}
    Substituting these into \eqref{eq:simple:v:moment:induct:mds:lower}, we get that
    \begin{align*}
        \big\|  T(i, \cI_0, \cI_1) & \big\|_{L_2}^2
        \;\geq\; 
        c'' \, \msum_{j=0}^N \|  D_j(i, \cI_0, \cI_1) \|_{L_2}^2
        \\
        \;\geq&\; 
        c'' \msum_{k=0}^{m - |\cI_1|}
        \mfrac{(c'')^k}{k!}
        \max\{ (N+1)^k - k^2 (N+1)^{k-1} , 0\} 
        ( \alpha_2(|\cI_1|+k))^2
        \\
        \;=&\;
        \msum_{k=0}^{m - |\cI_1|}
        \Big( \mfrac{(c'')^{k+1}}{k!}
        ( \alpha_2(|\cI_1|+k))^2
        \\
        &\qquad\qquad 
        \times
        \max\big\{ (n-|\cI_0|-|\cI_1|+1)^k - k^2 (n-|\cI_0|-|\cI_1|+1)^{k-1} , 0\big\}  
        \Big) 
        \;.
    \end{align*}
    This finishes the induction for \eqref{eq:simple:v:moment:induct:two}. In particular, we can now obtain 
    \begin{align*}
        \mfrac{\sum_{i=1}^n M_{\nu;i}^\nu}{\sigma^\nu} 
        \;=&\;
        \mfrac{\sum_{i=1}^n \| T(i, \{i\}, \emptyset)\|_{L_\nu}^\nu}
        { \| T(n,\emptyset,\emptyset) \|_{L_2}^\nu} 
        \\
        \;\leq&\;
        \mfrac{
            n 
            \Big( 
                (C_*)^m
                \sum_{k=0}^{m - 1}
                \frac{(C'')^{k+1} (\alpha_\nu( k+1))^2}{k!}
                n^k
            \Big)^{\frac{\nu}{2}}
        }
        {
        \Big(
        \sum_{k=0}^m
        \frac{(c'')^{k+1} (\alpha_2(k))^2}{k!}
        \max\big\{ (n+1)^k - k^2 (n+1)^{k-1} , 0\big\}  
        \Big)^{\frac{\nu}{2}}
        }\;.
    \end{align*} 
    Since $n \geq 2 m^2$ by assumption, we have
    \begin{align*}
        (n+1)^k - k^2 (n+1)^{k-1} 
        \;\geq&\; 
        (n+1)^{k-1} ( n + 1 - m^2) 
        \;\geq\;  
        \mfrac{1}{2} (n+1)^k \;. 
    \end{align*}
    By further noting that $\alpha_2(0) =0$, we can simplify the ratio as
\begin{align*}
    \mfrac{\sum_{i=1}^n M_{\nu;i}^\nu}{\sigma^\nu} 
    \;\leq&\; 
    \mfrac{
        2^\nu 
        (C_*)^{\frac{m\nu}{2}}
        n 
        \Big( 
            \msum_{k=0}^{m - 1}
            \mfrac{(C'')^{k+1}}{k!}
            n^k
            (\alpha_\nu( k+1))^2
        \Big)^{\nu/2}
    }
    {
    \Big(
    \msum_{k=0}^m
    \mfrac{(c'')^{k+1}}{k!}
    (n+1)^k
    (\alpha_2(k))^2
    \Big)^{\nu/2}
    }
    \\
    \;=&\;
    \mfrac{
        2^\nu 
        (C_*)^{\frac{m\nu}{2}}
        n 
        \Big( 
            \msum_{k=1}^m
            \mfrac{(C'')^{k}}{(k-1)!}
            n^{k-1}
            (\alpha_\nu( k))^2
        \Big)^{\nu/2}
    }
    {
    \Big(
    \msum_{k=1}^m
    \mfrac{(c'')^{k+1}}{k!}
    (n+1)^k
    (\alpha_2(k))^2
    \Big)^{\nu/2}
    }
    \\
    \;=&\;
    \mfrac{2^\nu 
    (C_*)^{\frac{m\nu}{2}}}
    {n^{\frac{\nu-2}{2}}}
    \,
    \Bigg(
    \mfrac{
            \sum_{k=1}^m
            \frac{(C'')^{k}}{(k-1)!}
            n^k
            (\alpha_\nu(k))^2
    }
    {
    \sum_{k=1}^m
    \frac{(c'')^{k+1}}{k!}
    (n+1)^k
    (\alpha_2(k))^2
    }
    \Bigg)^{\frac{\nu}{2}}\;. 
\end{align*}
Meanwhile by Stirling's approximation, $\frac{n^k}{(k-1)!} \leq 
A^m \binom{n}{k}$ and $\frac{(n+1)^k}{k!} \geq 
\frac{n^k}{k!} \geq B^m \binom{n}{k}$ for some absolute constants $A, B > 0$. Since $\nu < 3$, we get that for some absolute constant $C'>0$,
\begin{align*}
    \Big( \mfrac{ \sum_{i=1}^n M_{\nu;i}^\nu }
    { \sigma^\nu
    }
    \Big)^{\frac{1}{\nu m +1}}
    \;\leq\;
    C'
    \,
    n^{-\frac{\nu-2}{2\nu m + 2}}
    \,
    \Bigg(
    \mfrac{
            \sum_{k=1}^m
            \binom{n}{k}
            (\alpha_\nu(S, k))^2
    }
    {
    \sum_{k=1}^m
    \binom{n}{k}
    (\alpha_2(S,k))^2
    }
    \Bigg)^{\frac{\nu}{2 \nu m + 2}}
    \;=\;
    C'
     \,
    n^{-\frac{\nu-2}{2\nu m + 2}}
    \,
    \beta_{m,\nu}^{\frac{\nu}{2 \nu m + 2}}
    \;.
\end{align*}
By applying \eqref{eq:simple:v:moment:induct:nu} and \eqref{eq:simple:v:moment:induct:two} to $\sigma^2 = \Var[v_m(X)]$ and using the same argument as above,
we also get that if $n \geq 2 m^2$, there exist some absolute constants $C_1, C_2 > 0$ such that
\begin{align*}
    \hspace{1.5em}
    \mfrac{(C_1)^m}{n^{2m}} 
    \msum_{k=1}^m \mbinom{n}{k} (\alpha_2(S,k))^2
    \;\leq\;
    \Var[v_m(X)]
    \;\leq\;
    \mfrac{(C_2)^m}{n^{2m}}
    \msum_{k=1}^m \mbinom{n}{k} (\alpha_2(S,k))^2\;.
    \hspace{1em} \qed
\end{align*}

\vspace{.5em}

\subsection{Proof of Proposition \ref{prop:universality:v:stat}} 
\label{sec:proof:universality:v:stat}

Recall that 
\begin{align*}
    q^v_m(x_1, \ldots, x_n) \;\coloneqq\; 
    \tilde q_m(x_1, \ldots, x_n) - \mean[ v_m(X) ]
\end{align*}
where $\tilde q_m$ is the multilinear representation of $v_m$ in \Cref{cor:non:multilinear}, and WLOG assume the coupling between $\xi_i$ and $Z_i$ considered in \cref{lem:v:var:approx}. 
By the exact same argument as the proof of 
the lower bound result (\Cref{thm:lower:bound})
leading up to \eqref{eq:near:opt:delta:reg:argument} (\cref{lem:approx:XplusY:by:Y}, the Markov's inequality and choosing $\epsilon$ appropriately), we get that for some absolute constant $C_1 > 0$,
\begin{align*}
    \big|  \P\big(  q^v_m(\Xi) + \mean[v_m(X)] \leq t\big) - \P\big(v_m(Z) \leq t  \big) \big|
    \;\leq&\; 
    C_1 m \Big( \mfrac{R_n}{n} \Big)^{\frac{1}{2m+1}} \;.
    \tagaligneq \label{eq:universality:v:stat:third:bound}
\end{align*}
By applying \cref{lem:universality:v:stat:aug} 
to replace $q^v_m(\Xi)$ by $q^v_m(\bX)$, noting that $v_m(X) = q^v_m(\bX) + \mean[v_m(X)]$ and using a triangle inequality to combine the bounds, we get that for some absolute constant $C' > 0$,
\begin{align*}
    &\msup_{t \in \R} \big| \P( v_m(X) \leq t) - \P(  v_m(Z) \leq t) \big|
    \;\leq\;
    C_1 m \Big( \mfrac{R_n}{n} \Big)^{\frac{1}{2m+1}}
    +
    C' m \Delta_v
    \;. 
\end{align*}
To prove the final bound, we first note that by \cref{lem:universality:v:stat:aug}, \eqref{eq:universality:v:stat:third:bound} and the triangle inequality,
\begin{align*}
    &\msup_{t \in \R} \big| \P( v_m(X) - \mean[v_m(X)] \leq t) - \P( v_m(Z)-\mean[v_m(Z)] \leq t) \big|
    \\
    \;\leq&\;
    C' m \Delta_v
    +
    \msup_{t \in \R} \big| \P( q^v_m(\Xi) \leq t) - \P( v_m(Z)-\mean[v_m(Z)] \leq t) \big|
    \\
    \;\leq&\;
    C' m \Delta_v
    +
    C_1 m \Big( \mfrac{R_n}{n} \Big)^{\frac{1}{2m+1}}
    \\
    &\; \qquad
    +
    \msup_{t \in \R} \big| \P(  q^v_m(\Xi) \leq t) - \P( q^v_m(\Xi)+\mean[v_m(X)]-\mean[v_m(Z)] \leq t) \big|\;.
\end{align*}
By the same Carbery-Wright inequality argument used in the proof of \Cref{thm:main} in \eqref{eq:main:proof:carbery:wright}, the last term can be bounded by 
\begin{align*}
    &\;\msup_{t \in \R} \P\big( |\sigma^{-1} q^v_m(\Xi) - t| \leq  \sigma^{-1} | \mean[v_m(X)- v_m(Z)]| \big) 
    \\
    &\;\leq\;
    C_2 m \, \mfrac{| \sigma^{-1} \mean[v_m(X)] -  \sigma^{-1} \mean [v_m(Z)]|^{1/m}}{\mean[(\sigma^{-1} q^v_m(\Xi) - t )^2]^{1/2m}}
    \;=\; 
    C_2 m \sigma^{-1/m} \big| \mean[ v_m(X) ] - \mean[ v_m(Z)] \big|^{1/m}
    \\
    &\;=\;
    C_2 m\sigma^{-1/m} \big| \mean\big[ q^v_m(\Xi) + \mean[ v_m(X) ] - v_m(Z) \big] \big|^{1/m}
    \\
    &\;\leq\; 
    C_2 m \bigg( \mfrac{\big\| q^v_m(\Xi) + \mean[ v_m(X) ] - v_m(Z)  \big\|_{L_2}^2}{\Var[q^v_m(\Xi)]} \bigg)^{1/(2m)}
    \;\leq\; 
    C_3 m  \Big( \mfrac{R_n}{n} \Big)^{-\frac{1}{2m}}
    \;\leq\; 
    C_3 m \Big( \mfrac{R_n}{n} \Big)^{-\frac{1}{2m+1}}
\end{align*}
for some absolute constants $C_2, C_3 > 0$; in the last line, we noted that $\sigma^2 = \Var[v_m(X)] = \Var[q^v_m(\Xi)]$ and applied \cref{lem:v:var:approx} again. By taking $C = C_1 + C_3 > 0$, we get that
\begin{align*}
    \msup_{t \in \R} \big| \P( v_m(X) - \mean[v_m(X)] \leq t) - \P( v_m(Z)- &\mean[v_m(Z)] \leq t) \big|
    \\
    & \hspace{-2em}
    \;\leq\; 
    C m \Big( \mfrac{R_n}{n} \Big)^{-\frac{1}{2m+1}}
    +
    C' m \Delta_v
    \;. \qed
\end{align*}

\subsection{Proof of \Cref{lem:simple:V:independent:coords}} \label{appendix:proof:simple:V:iid}

Recall that the Kolmogorov distance is invariant under scaling. Under the setup of \Cref{example:simple:V} and the condition that $X_1$ consists of i.i.d.~coordinates, it suffices to consider
\begin{align*}
    \tilde v_d(X)
    \;\coloneqq\;
    \mfrac{1}{d} v_2(X)
    \;=\; 
    \mfrac{1}{n^2}  \msum_{i,j \leq n} \mfrac{X_i^\top X_j}{d} 
    \;=&\;
    \mfrac{1}{d n^2} \msum_{l \leq d}  \msum_{i,j \leq n} X_{il} X_{jl}
    \\
    \;=&\;
    \mfrac{1}{d n^2} \msum_{l \leq d} 
    \Big( 
        \msum_{i \neq j} X_{il} X_{jl}
        +
        \msum_{i \leq n} X_{il}^2
    \Big)
    \\
    \;=&\;
    \mfrac{1}{d n^2} \msum_{l \leq d} 
    \Big( 
        \msum_{i \neq j} \bX_{il 1} \bX_{jl1}
        +
        \msum_{i \leq n} \bX_{il 2}
    \Big)
    \\
    \;\eqqcolon&\;
    \tilde q_d(\bX)
    \;,
\end{align*}
where we have denoted $\bX_{il} \coloneqq (X_{il}, X_{il}^2 )$. $\tilde q_d(\bX)$ is a multilinear polynomial of $nd$ i.i.d.~random vectors $(\bX_{il})_{i \leq n, l \leq d}$, to which \Cref{cor:non:multilinear} is applicable. Write $\bW_{il}$ as the analogue of $\bW_i$ in \Cref{cor:non:multilinear}, with the $il$-th vector set to zero, all preceding vectors set to $\bX_{i'l'}$ and all subsequent vectors set to $\bZ_{i'l'}$. Then for some absolute constant $C' > 0$, we have
\begin{align*}
    \sup_{t \leq R} 
    \big| 
        \P(  v_2(X) \leq t ) 
        \,&-\,
        \P(  v_2(\Xi) \leq t) 
    \big| 
    \;\leq\; 
    C'  \, \Big( 
        \mfrac{
            \msum_{i \leq n, l \leq d} \| \partial_{il} \tilde q_d(\bW_{il}) \bX_{il} \|_{L_\nu}^\nu 
        }{ 
            \Var[ \tilde q_d(X)]^{\nu/2}  
        } 
    \Big)^{\frac{1}{2\nu + 1}}
    \\
    \;\overset{(a)}{=}&\;
    C' \, \bigg( 
        \mfrac{
            d^{1 - \nu}    
            \msum_{i \leq n}
            \| \partial_i \tilde q_1(\bX_{11}, \ldots, \bX_{(i-1)1}, \bzero, \xi_{(i+1)1}, \ldots, \xi_{n1}) \bX_{i1}
            \|_{L_\nu}^\nu
        }{ 
            d^{-\nu/2} \, \Var\big[ \tilde q_1(\bX_{11}, \ldots, \bX_{n1}) \big]^{\nu/2}  
        } 
    \bigg)^{\frac{1}{2 \nu + 1}}
    \\
    \;=&\;
    C' d^{-\frac{\nu-2}{4 \nu + 2}} \,  \bigg( 
        \mfrac{
            \msum_{i \leq n}
            \| \partial_i \tilde q_1(\bX_{11}, \ldots, \bX_{(i-1)1}, \bzero, \xi_{(i+1)1}, \ldots, \xi_{n1}) \bX_{i1}
            \|_{L_\nu}^\nu
        }{ 
            \Var\big[ \tilde q_1(\bX_{11}, \ldots, \bX_{n1}) \big]^{\nu/2}  
        } 
    \bigg)^{\frac{1}{2 \nu + 1}}
    \;.
\end{align*}
In $(a)$, we have used that $\bX_{il}$'s are i.i.d.~across $l \leq d$ and that both $\tilde v_d$ and $\tilde q_d$ are linear sums over $l \leq d$. Notice that, without the additional $d^{-\frac{\nu-2}{4 \nu + 2}} $ factor, this is exactly the same upper bound obtained from applying \Cref{cor:non:multilinear} to $\tilde q_1(\bX_{11}, \ldots, \bX_{n1}) = \tilde v_1(X_{11}, \ldots, X_{n1})$, which is a V-statistic of univariate random variables. Meanwhile, applying \Cref{thm:VD} to approximate $\tilde v_d(Z)$ by $\tilde q_d(\Xi)$ introduces an additional bound of the form 
\begin{align*}
    &\;
    C'' \Big( \mfrac{\| \tilde v_d(Z) - \tilde q_d(\Xi) \|_{L_2}^2}{\Var[\tilde q_d(\Xi)]} \Big)^{\frac{1}{2m+1}}
    \;=\;
    C'' \Big( \mfrac{ \mean[\tilde v_d(Z) - \tilde q_d(\Xi)]^2 +  \Var[ \tilde v_d(Z) - \tilde q_d(\Xi) ]}{ d^{-1} \, \Var[\tilde q_1(\xi_{11}, \ldots, \xi_{n1})]} \Big)^{\frac{1}{2m+1}}
    \\
    &\quad 
    \;=\;
    C'' \Big( \mfrac{ 0 +  d^{-1} \Var[ \tilde v_1(Z_{11}, \ldots, Z_{n1}) - \tilde q_1(\xi_{11}, \ldots, \xi_{n1} ) ]}{ d^{-1} \, \Var[\tilde q_1(\xi_{11}, \ldots, \xi_{n1})]} \Big)^{\frac{1}{2m+1}}
    \\
    &\quad 
    \;=\;
    C'' \Big(  \mfrac{\|  \tilde v_1(Z_{11}, \ldots, Z_{n1}) - \tilde q_1(\xi_{11}, \ldots, \xi_{n1} ) \|_{L_2} }{ \Var[\tilde q_1(\xi_{11}, \ldots, \xi_{n1})]} \Big)^{\frac{1}{2m+1}}
    \;,
\end{align*}
where $C'' > 0$ is an absolute constant. This is exactly the same term obtained by applying \Cref{thm:VD} to approximate $\tilde v_1(Z)$ by $\tilde q_1(\Xi)$. Since $\tilde v_1(X)$ is a V-statistic of univariate variables, the proof of \Cref{prop:simpler:V} (\cref{appendix:proof:simpler:V}) applies for controlling the above quantities. This allows us to recycle the bound obtained for $\tilde v_1(X)$ from \Cref{prop:simpler:V} and obtain, for some absolute constant $C > 0$,
\begin{align*}
    &\;
        \msup_{t \in \R} \big| \P( v(X) \leq t) - \P(  v(Z) \leq t) \big|
        \;=\;
        \msup_{t \in \R} \big| \P( \tilde v_d(X) \leq t) - \P( \tilde v_d(Z) \leq t) \big|
        \\
        &\,\;\leq\;\,
        C  (nd)^{-\frac{\nu-2}{4\nu+ 2}}
            \,
            \bigg( 
                \mfrac{
                    \| X_{11} \|_{L_\nu}^4 
                    + O(n^{-1} )
                }
                {
                        \| X_{11} \|_{L_2}^4 
                        + O(n^{-1} )
                }
            \bigg)^{\frac{\nu}{4 \nu + 2}}
        +
        C  n^{ - \frac{1}{5 } } \,
        \Big( 
        \mfrac{ 
            1
        }{
                \| X_{11} \|_{L_2}^2 
                + O(n^{-1} )
        } 
        \Big)^{\frac{1}{5}}
        \;.
\end{align*}

\section{Constructions for Theorems \ref{thm:lower:bound} and \ref{cor:matching:bound:V}}  \label{sec:construct:lower:bound}

The key ingredient of the proof of \Cref{thm:lower:bound} % \cref{thm:lower:bound} 
is a suitable sequence of heavy-tailed probability measures $\mu^{(n)}_{\nu}$ and statistic $p^*_m$. Before stating the exact constructions, we first motivate them by discussing our proof technique.

\vspace{.5em}

\noindent
\textbf{Adapting an asymmetry argument from Senatov \cite{senatov1998normal}. } The overall idea is to construct a mixture of an average of a heavy-tailed random variable, which is poorly approximated via Gaussian universality, and a degree-$m$ V-statistic of Gaussians, which has good approximation by the same V-statistic of an i.i.d.~copy of Gaussians at small $m$ but poor approximation at large $m$. The technical steps to obtain a tight lower bound, however, are non-trivial. Our strategy is inspired by Example 9.1.3 of Senatov \cite{senatov1998normal}: They construct a sequence of probability measures on $\R^d$ to demonstrate how a multivariate normal approximation bound on a chosen sequence of Euclidean balls may depend on eigenvalues of the covariance matrix. Two key mathematical ingredients of their work are
\begin{proplist}
    \item  a heavy-tailed univariate variable, whose i.i.d.~average is poorly approximated by a Gaussian, and 
    \item an asymmetry argument by considering Euclidean balls not centred at the origin while studying averages of mean-zero random vectors.
\end{proplist}
To adapt the construction in \cite{senatov1998normal} for our problem, we make one key observation: The probability of a mean-zero average lying in a radius-$r$ Euclidean ball centred at the origin is exactly the probability of a degree-2 V-statistic taking values in $[0,r^2)$. In other words, (i) gives the heavy-tailed variable we need, and (ii) is almost our V-statistic except that the Euclidean balls in Senatov \cite{senatov1998normal} are not centred at the origin. Unfortunately, while the asymmetry in (ii) is key to the proof by Senatov \cite{senatov1998normal}, it breaks the connection to V-statistics, and a naive application of Senatov's results gives very loose bounds. Instead, we create a different asymmetry by asking $p^*_m(X)$ to be a mixture of an odd-degree V-statistic and an even-degree V-statistic, which allows Senatov's argument to be adapted. An additional distinction of our proof from \cite{senatov1998normal} is that, to accommodate a growing degree $m$, we need to derive finer moment controls.

\vspace{1em}

We are ready to state the constructions. For $m=m(n)$ even, consider the polynomial 
\begin{align*}
    p^*_m(x_1,\ldots, x_n) 
    \;\coloneqq&\; 
    \mfrac{1}{\sqrt{n}} \msum_{i=1}^n x_{i1}  
    +  
    \Big(\mfrac{1}{\sqrt{n}}
    \msum_{i=1}^n x_{i2}\Big)^m  
    \;\; 
    \text{ for } x_i = (x_{i1}, x_{i2}) \in \R^2\;.
\end{align*}
The distribution $\mu^{(n)}_{\nu}$ is the law of a bivariate random vector with a heavy-tailed coordinate and a Gaussian coordinate. We first construct the heavy-tailed variable.

\vspace{.5em}

\noindent 

\textbf{A heavy-tailed univariate distribution $ \tilde \mu_{\nu, \sigma} $.} For $\nu > 2$ and $\sigma \in (0,1]$, we define
\begin{align*}
    p \;=&\; \sigma^{2\nu / (\nu-2)} \in (0,1]
    &\text{ and }&&
    x_0 \;=&\; \sigma / \sqrt{6 p} \;=\;x \;=\; \mfrac{1}{\sqrt{6}} \sigma^{-\frac{2}{\nu-2}}
     \;,
\end{align*}
and let $U_1$ be a discrete random variable taking values in $\{-x_0, 0, 2x_0\}$ with
\begin{align*}
    \P( U_1 = -x_0 ) \;=\; 2p \;,
    \qquad
    \P( U_1 = 0 ) \;=\; 1-3p
    \quad\text{ and }\quad
    \P( U_1 = 2x_0 ) \;=\; p\;.
\end{align*}
Let $Z_1 \sim \cN(0, \sigma^2)$ and define a smoothed version of $U$ as
\begin{align*}
    V_1 \;\coloneqq\; 2^{-1/2} U_1 + 2^{-1/2} Z_1\;. 
\end{align*}
Note that the only role $Z_1$ plays in the proof is for $V_1$ to have a continuous density function, which allows us to apply results relating distribution functions to characteristic functions.
Write $ \tilde \mu_{\nu, \sigma} $ as the law of $V_1$. When $\nu=3$, this is the heavy-tailed distribution constructed in Example 9.1.3 of \cite{senatov1998normal}. Roughly speaking, the construction is such that if we set $\sigma=\sigma_n \rightarrow 0$ as $n \rightarrow \infty$, then $\Var[V_1]= \sigma_n^2 \rightarrow 0$, but the $\nu$-th central moment of $V_1$ remains $\Theta(1)$.  

\vspace{.5em}

Let $V_2, \ldots, V_n \overset{i.i.d.}{\sim} \tilde \mu_{\nu, \sigma}$ and draw an independent $Z'_1 \sim \cN(0,\sigma^2)$. We now list three results on $\tilde \mu_{\nu, \sigma}$, which extend the results by \cite{senatov1998normal} to a general $\nu$ and admit similar proofs. Since the original proofs are highly condensed, we include proofs for all lemmas with more detailed steps and intuitions in \cref{appendix:univariate:heavy}. \cref{lem:heavy:uni:moments} below controls moments of $V_1$. In particular, $\mean |V_1|^\nu = O(1)$ and the upper bound on $\mean |V_1|^\omega$ diverges for $\omega > \nu$. This upper bound will allow us to handle moments of polynomials of $V_i$'s of growing degrees.

\vspace{.5em}

\begin{lemma} \label{lem:heavy:uni:moments} $\mean \, V_1 = 0$ and $\Var \, V_1 = \sigma^2$. Moreover, there exist absolute constants $c_1, c_2 > 0$ such that for all $\omega \geq 1$, $\mean |V_1|^\omega \leq c_1^\omega \sigma^{- \frac{2(\omega-\nu)}{\nu -2} } + c_2^\omega \sigma^\omega \omega^{\omega / 2}$.
\end{lemma}

\vspace{.5em}

\cref{lem:heavy:uni:normal:approx:with:remainder} below provides a finer control on the normal approximation error of an empirical average, by performing a higher-order Taylor expansion in the space of characteristic functions. This higher-order Taylor term is then inverted back to the space of distribution functions and captured by $F_q$ below.

\begin{lemma} \label{lem:heavy:uni:normal:approx:with:remainder} There exists some absolute constant $A > 0$ such that, defining the function   $F_q(x) \coloneqq \frac{A}{n^{1/2} \sigma^{\nu/(\nu-2)}} \big(1 - \frac{x^2}{\sigma^2} \big) e^{-x^2/(2\sigma^2)}$, we have
    \begin{align*}
        \msup_{x \in \R}
        \Big| 
        \P\Big( \mfrac{1}{\sqrt{n}} \msum_{i=1}^n V_i < x \Big) 
        -
        \P( Z'_1 < x ) - F_q(x)
        \Big|
        \;\leq\; \mfrac{2}{n \sigma^{2\nu/(\nu-2)}}\;.
    \end{align*}
\end{lemma}

\cref{lem:heavy:uni:normal:approx:non:uniform} below provides a finer upper bound on the normal approximation error of an average by exploiting the distribution of $V_i$'s we constructed.

\vspace{.5em}

\begin{lemma} \label{lem:heavy:uni:normal:approx:non:uniform} Suppose there exists some constant $M \geq 10$ such that $ \sigma^{\nu/(\nu-2)} \geq M n^{-1/2}$ and $n \geq 6M^2$. Then there is some constant $C_M > 0$ that depends only on $M$ such that, for all $x \in \R$, we have
    \begin{align*}
        \Big| 
        \P\Big( \mfrac{1}{\sqrt{n}} \msum_{i=1}^n V_i < x \Big) 
        -
        \P( Z'_1 < x ) 
        \Big|
        \;\leq\; 
        C_M \Big(
        \mfrac{ \max\{1, \sigma^3\} }{ n^{1/2} \sigma^{\nu/(\nu-2)}} e^{-\frac{x^2}{ 16 \sigma^2}}
        +
        \mfrac{1 }{n^{3/2} x^4 \sigma^{(8-\nu)/(\nu-2)} }
        \Big)\;.
    \end{align*}
\end{lemma}

\vspace{.5em}

\noindent
\textbf{The bivariate distribution $\mu^{(n)}_\nu \equiv \mu^{(n)}_{\nu, \sigma_0}$.} Now for a fixed $\nu \in (2,3]$ and some absolute constant $\sigma_0 > 0$, set a standard deviation parameter
\begin{align*}
    \sigma_n 
    \;\coloneqq&\; 
    \min\big\{ \sigma_0 \, n^{-\frac{\nu-2}{2\nu}} \,,\, 1 \big\}
    \;\in\; (0,1]
    \;.
\end{align*}
With a slight abuse of notation, we consider 
\begin{align*}
    V_1, \ldots, V_n \;\overset{\rm i.i.d.}{\sim}&\; \tilde \mu_{\nu, \sigma_n}
    &\;\text{ and }\;&&
    Y_1, \ldots, Y_n \;\overset{\rm i.i.d.}{\sim}&\; \cN(0,1)\;,
\end{align*}
and define $\mu^{(n)}_\nu \equiv \mu^{(n)}_{\nu, \sigma_0}$ in \Cref{thm:lower:bound} as the law of $(V_1, Y_1)$.

\vspace{1em}

In summary, the statistic we consider is  a mixture of a heavy-tailed average and an even-degree V-statistic of Gaussian random variables: 
\begin{align*}
    p^*_m(X) 
    \;\coloneqq&\; 
    \mfrac{1}{\sqrt{n}} \msum_{i=1}^n V_i
    +  
    \Big(\mfrac{1}{\sqrt{n}}
    \msum_{i=1}^n Y_i \Big)^m 
    \;=\;
    v^*_1(V)
    +
    v^*_m(Y)
    \;,
\end{align*}
where we have denoted $V \coloneqq (V_i)_{i \leq n}$, $Y \coloneqq (Y_i)_{i \leq n}$ and the rescaled degree-$m'$ V-statistic as
\begin{align*}
    v^*_{m'}(x'_1, \ldots, x'_n) 
    \;\coloneqq\; 
    \Big( 
    \mfrac{1}{\sqrt{n}}
    \msum_{i=1}^n x'_i \Big)^{m'} 
    \quad \text{ for } x'_1, \ldots, x'_n \in \R \;. 
\end{align*}

\vspace{.5em}

\noindent 
\textbf{Rewrite $p^*_m(X)$ and $q^*_m(\bX)$ as the V-statistic and the U-statistic in \Cref{cor:matching:bound:V}.} \Cref{cor:matching:bound:V} follows directly from \Cref{thm:lower:bound}, by observing that $p^*_m(X)$ and $q^*_m(\bX)$ in \Cref{thm:lower:bound} can be directly re-expressed as a V-statistic and a U-statistic. To do this, we rewrite the polynomial $p^*_m$ above as 
\begin{align*}
    p^*_m(x_1,\ldots, x_n) 
    \;=&\; 
    \mfrac{1}{\sqrt{n}} \msum_{i=1}^n x_{i1}  
    +  
    \Big(\mfrac{1}{\sqrt{n}}
    \msum_{i=1}^n x_{i2}\Big)^m  
    \\
    \;=&\;
    \mfrac{1}{n^{m/2}} \msum_{i_1, \ldots, i_m \leq n} 
    \underbrace{
    \Big( 
        \mfrac{1}{m n^{(m-1)/2}} \msum_{l=1}^m x_{i_l 1}    
        +
        x_{i_1 2} \ldots x_{i_m 2}
    \Big)
    }_{\eqqcolon \tilde  k^*_v(x_{i_1}, \ldots, x_{i_m})}
    \;.
\end{align*}
This is exactly a rescaled V-statistic, to which the bounds of \Cref{thm:lower:bound} applies to give a distributional approximation. Since $d > m \geq 2$ in \Cref{cor:matching:bound:V}, the V-statistic $\tilde v^*_m$ in \Cref{cor:matching:bound:V} can be obtained by augmenting the bivariate random vectors $(V_i, Y_i)$ to $\R^d$-valued vectors (e.g.~by padding with zeros), and by extending $\tilde k^*_v: (\R^2)^m \rightarrow \R$ to form the function $k^*_v: (\R^d)^m \rightarrow \R$. 

\vspace{.5em}

Meanwhile, observe that the second bound of \Cref{thm:lower:bound} also applies to the centred multilinear polynomial 
\begin{align*}
    q^{**}_m(\bx_1, \ldots, \bx_n)
    \;\coloneqq\;
    q^*_m(\bx_1, \ldots, \bx_n) 
    -
    \mean[q^*_m(\bX)]
    \;,
\end{align*}
where $q^*_m$ is the multilinear representation of $p^*_m$ defined in \Cref{thm:lower:bound}. Writing $\bx_i = (v_1, y_1, y_1^2, \ldots, y_1^m) \in \R^{m+1}$, $q^{**}$ can be expressed explicitly as 
\begin{align*}
    q^{**}_m( \bx_1, \ldots, \bx_n ) 
    \;=&\; 
    \mfrac{1}{\sqrt{n}} \msum_{i=1}^n v_i 
    +
    \mfrac{1}{n^{m/2}} 
    \msum_{i_1, \ldots, i_m =1}^n y_{i_1} \ldots y_{i_m}
    \\
    \;=&\;
    \mfrac{1}{\sqrt{n}} \msum_{i=1}^n v_i 
    +
    \mfrac{1}{n^{m/2}} 
    \msum_{\substack{i_1, \ldots, i_m \\ \text{ distinct}}} 
    \langle S_n, \bx_{i_1} \otimes \ldots \otimes \bx_{i_m} \rangle 
    \\
    \;=&\;
    \mfrac{1}{n (n-1) \ldots (n-m+1)} \msum_{\substack{i_1, \ldots, i_m \\ \text{ distinct}}}  k^*_u(\bx_{i_1}, \ldots, \bx_{i_m})
    \;.
\end{align*} 
for a suitable symmetric tensor $S_n \in \R^{(m+1)^m}$ and a suitable symmetric function $k^*_u: (\R^{m+1})^m \rightarrow \R$. The distributional approximation of \Cref{thm:lower:bound} applies to $q^{**}_m$ and therefore automatically to the U-statistic above. Since $m < d$ in \Cref{cor:matching:bound:V}, this again gives the required U-statistic $\tilde u^*_m$ by augmenting $\tilde k^*_u: (\R^{m+1})^m \rightarrow \R$ to form the function $k^*_u: (\R^d)^m \rightarrow \R$. 

\section{Proofs concerning Theorem \ref{thm:lower:bound}: Nearly matching bounds for universality
% \ref{thm:lower:bound}
} \label{sec:proof:lower:bound} 
The proof consists of three main steps:

\begin{proplist}
    \item By carefully exploiting the heavy tail of $V_i$'s and the asymmetry of $v^*_{m'}$, we can obtain the first lower bound for the approximation of $p^*_m(X)$ by $p^*_m(Z)$ (\cref{lem:lower:V});
    \item By applying \Cref{cor:non:multilinear} 
    % \cref{thm:main} 
    to approximate $p^*_m(X) = q^{**}_m(\bX) + \mean[p^*_m(X)]$ by $q^{**}_m(\Xi) + \mean[p^*_m(X)]$, the upper bound involving $q^{**}_m$ in \Cref{thm:lower:bound} 
    % \cref{thm:lower:bound} 
    reduces to a moment control (\cref{lem:near:opt:partial:upper:bound});
    \item By directly controlling $R_n$ in \Cref{sec:proof:v:additional} and modifying the argument of \cref{prop:universality:v:stat}, 
    we can approximate $q^{**}_m(\Xi) + \mean[p^*_m(X)]$ further by $p^*_m(Z)$ (\cref{lem:near:opt:delta:reg}).
\end{proplist}

We introduce some more notation. Note that we can express $X_i = (V_i, Y_i)$ and $\bX_i = (V_i, \bY_i)$, 
where $\bY_i = (Y_i, \ldots, Y_i^m - \mean[Y_i^m])$. Correspondingly, we express 
\begin{align*}
    Z_i \;=&\; (Z_{V;i}, Z_{Y;i})
    &\text{ where }&&
    Z_{V;i} \;\sim\; \cN(0, \Var\,V_i)
    \;\text{ and }\;&
    Z_{Y;i} \sim \cN(0, \Var\,Y_i)
    \;,
    \\
    \xi_i \;=&\; (\xi_{V;i}, \xi_{\bY;i})
    &\text{ where }&&
    \xi_{V;i} \;\sim\; \cN(0, \Var\,V_i)
    \;\text{ and }\;&
    \xi_{\bY;i} \sim \cN(\bzero, \Var\,\bY_i)
    \;.
\end{align*}
Denote the collections $Z_V \coloneqq (Z_{V;i})_{i \leq n}$, $Z_Y \coloneqq (Z_{Y;i})_{i \leq n}$ and $\Xi_\bY \coloneqq (\xi_{\bY;i})_{i \leq n}$. It is also convenient to denote $q^{**}_{v;m}$ as the centred multilinear polynomial 
\begin{align*}
    q^{**}_{v;m}(\by_1, \ldots, \by_n)
    \;\coloneqq\;
    \tilde q(\by_1, \ldots, \by_n) 
    -
    \mean[v^*_m(Y)]
    \;,
\end{align*}
where $\tilde q$ is the multilinear representation of $v^*_m$ defined in \Cref{cor:non:multilinear} with respect to $Y$, since we can express
\begin{align*}
    q^{**}_m(\bX) \;=&\; v^*_1(V) + q^{**}_{v;m}(\bY)
    &\text{ and }&&
    q^{**}_m(\Xi) \;\overset{d}{=}&\; v^*_1(Z_V) + q^{**}_{v;m}(\Xi_{\bY})\;.
\end{align*}

\begin{lemma} \label{lem:lower:V} Fix $\nu > 2$. Assume that $m$ is even with $m=o(\log n)$. Then there exist some absolute constants $c, \sigma_0 > 0$ and $N \in \N$ such that, for any $n \geq N$,
    \begin{align*}
        \P\big(  p^*_m(Z) < - 2 \sigma_n  & \big) - \P\big( p^*_m(X) < - 2\sigma_n \big) 
        \;\geq\;
        c n^{-\frac{\nu-2}{2\nu m}}
        \;. 
    \end{align*}  
\end{lemma}

\begin{proof}[Proof of \cref{lem:lower:V}] Since the second coordinate of $X_i$ is already Gaussian, the main hurdle is to approximate the heavy-tailed average by a Gaussian. Let $F_V$ be the c.d.f.~of $\frac{1}{\sqrt{n}} \sum_{i=1}^n V_i$ -- an empirical average of the heavy-tailed coordinates -- and $F_Z$ be the c.d.f.~of $\frac{1}{\sqrt{n}} \sum_{i=1}^n Z_i$, where $Z_i$'s are i.i.d.~zero-mean random variables with the same variance as $V_1$. Write $\varphi$ be the p.d.f.~of $\frac{1}{\sqrt{n}} \sum_{i=1}^n Y_i \sim \cN(0, 1)$. Recalling that $m$ is even and denoting  
\begin{align*}
    I(x) \;\coloneqq\; F_Z( - 2 \sigma_n - x^m ) -  F_V( - 2 \sigma_n - x^m) \;,
\end{align*}
we can bound the quantity of interest as 
\begin{align*} 
    \P\big( p^*_m(Z)  < -2 \sigma_n \big) - \P\big( p^*_m(X) &< - 2 \sigma_n \big) 
    \;=\;  
    \mint_{0 \leq x^m < \infty}
    I(x) \varphi(x) dx
    \\
    \;\geq&\; 
    \mint_{0 \leq x^m \leq \kappa^m \sigma_n}
    I(x) \varphi(x) dx
    -
    \Big| \mint_{  |x| > \kappa \sigma_n^{1/m}  }  I(x) \varphi(x) dx \Big| 
    \\
    \;\eqqcolon&\; J_1 - J_2\;,
\end{align*}
for some $\kappa \geq 1$ to be chosen later.

\vspace{.5em}

\textbf{Bounding $J_1$. } Recall that for $y \in \R$, \cref{lem:heavy:uni:normal:approx:with:remainder} provides an error bound for approximating the difference $F_Z(y) - F_n(y)$ by
\begin{align*}
    - F_q(y) 
    \;=\; 
    \mfrac{A}{n^{1/2} \sigma_n^{\nu/(\nu-2)}} 
    \,\Big(\mfrac{y^2}{\sigma_n^2} -1 \Big) \, 
    e^{-y^2/(2\sigma_n^2)} 
\end{align*}
for some absolute constant $A > 0$. Therefore
\begin{align*}
    J_1 
    \;\geq&\; 
    \mint_{x^m \leq \kappa^m \sigma_n }  - F_q(-2\sigma_n-x^m) \varphi(x) dx 
    -
    \mfrac{2}{n \sigma_n^{2\nu/(\nu-2)}}
    \mint_{x^m \leq \kappa^m \sigma_n }  \varphi(x) dx 
    \\
    \;\overset{(a)}{\geq}&\; 
    \mfrac{A}{n^{1/2} \sigma_n^{\nu/(\nu-2)}}
    \mint_{x^m \leq  \sigma_n }  
    \Big(\mfrac{(2\sigma_n+x^m)^2}{\sigma_n^2} -1 \Big) \, 
    e^{-(2\sigma_n+x^m)^2/(2\sigma_n^2)} 
    \varphi(x) dx 
    \\
    &\;
    -
    \mfrac{2}{n \sigma_n^{2\nu/(\nu-2)}}
    \mint_{x^m \leq \kappa^m \sigma_n }  \varphi(x) dx 
    \\
    \;\overset{(b)}{\geq}&\;  
    \mfrac{3 e^{-9/2} A}{n^{1/2} \sigma_n^{\nu/(\nu-2)}}
    \mint_{x^m \leq \sigma_n } \varphi(x) dx
    -
    \mfrac{2}{n \sigma_n^{2\nu/(\nu-2)}}
    \mint_{x^m \leq \kappa^m \sigma_n} \varphi(x) dx \;. 
\end{align*}
In $(a)$, we have restricted the first integral to a smaller ball by noting that $\kappa \geq 1$; in $(b)$, we have noted in the first integral that $4 \sigma_n^2 \leq (2\sigma_n + x^m)^2 \leq 9 \sigma_n^2$ in the domain of integration. Since both integrals involve c.d.f.~of a standard normal, we obtain
\begin{align*}
    J_1 
    \;\geq&\; 
    \mfrac{3 e^{-9/2} A}{n^{1/2} \sigma_n^{\nu/(\nu-2)}}
    \, 
    \mfrac{2 \sigma_n^{1/m}}{\sqrt{2\pi}} \exp\Big( - \mfrac{\sigma_n^{2/m}}{2} \Big)
    -
    \mfrac{4 \kappa \sigma_n^{1/m} }{n \sigma_n^{2\nu/(\nu-2)}}
    \\
    \;\geq&\; 
    \mfrac{\sigma_n^{1/m} }{ n^{1/2} \sigma_n^{\nu/(\nu-2)}}
    \Big(A_1 e^{- \frac{1}{2} \sigma_n^{2/m}} - \mfrac{A_2 \kappa }{n^{1/2} \sigma_n^{\nu/(\nu-2)}} \Big)
\end{align*}
for some absolute constants $A_1, A_2 > 0$.

\vspace{.5em}
    
\textbf{Bounding $J_2$.} Recall that $\sigma_n =\min \big\{  \sigma_0 n^{-\frac{\nu-2}{2\nu}} \,,\,1 \big\}$, which implies 
\begin{align*}
    \sigma_n^{\frac{\nu}{\nu-2}}
    \;\geq\; 
    \min\big\{ \sigma_0^{\frac{\nu}{\nu-2}} n^{-1/2} \,,\, 1  \big\}
    \;\geq\;
    \sigma_0^{\frac{\nu}{\nu-2}} n^{-1/2}\;.
\end{align*}
Suppose we choose $\sigma_0$ such that $\sigma_0^{\frac{\nu}{\nu-2}}  > 6^{-1/2}$. Since $6(6^{-1/2})^2 =1 \leq n$, by applying \cref{lem:heavy:uni:normal:approx:non:uniform} with $M=6^{-1/2}$, there exists some absolute constant $C>0$ such that
\begin{align*}
    J_2 \;=&\; 
    \Big| \mint_{ x^m > \kappa^m \sigma_n }  \big(  F_Z(-2\sigma_n-x^m) -  F_n(-2\sigma_n-x^m) \big)   \varphi(x) dx \Big| 
    \\ 
    \;\leq&\;
    \mfrac{C}{ n^{1/2} \sigma_n^{\nu/(\nu-2)}}
        \mint_{ x^m > \kappa^m \sigma_n } 
        e^{-\frac{(2\sigma_n+x^m)^2}{ 16 \sigma_n^2}} 
        \varphi(x) dx 
    \\
    &\qquad\; +
    \mfrac{C}{n^{3/2} \sigma_n^{(8-\nu)/(\nu-2)} }
        \mint_{  x^m > \kappa^m \sigma_n } 
        \mfrac{1}{(2\sigma_n+x^m)^4}  \varphi(x) dx
    \\
    \;\eqqcolon&\; J_{21} + J_{22}\;.
\end{align*}
By applying a change-of-variable, we get that 
\begin{align*}
    J_{21} 
    \;\leq&\; 
    \mfrac{C}{ n^{1/2} \sigma_n^{\nu/(\nu-2)}}
        \mint_{ x^m > \kappa^m \sigma_n } 
        e^{-\frac{x^{2m}}{ 16 \sigma_n^2}} 
        \varphi(x) dx 
    \\
    \;=&\; 
    \mfrac{C\sigma_n^{1/m}}{ n^{1/2} \sigma_n^{\nu/(\nu-2)}}
        \mint_{ x^m > \kappa^m } 
        e^{-\frac{x^{2m}}{16}} 
        \varphi \big( \sigma_n^{1/m} x \big) dx 
    \\ 
    \;\overset{(a)}{\leq}&\; 
    \mfrac{2 C \sigma_n^{1/m}}{ n^{1/2} \sigma_n^{\nu/(\nu-2)}}
        \mint_{ x > \kappa } 
        e^{-\frac{x^{2m}}{16}} 
        dx
    \\ 
    \;\leq&\; 
    \mfrac{2 C \sigma_n^{1/m}}{ n^{1/2} \sigma_n^{\nu/(\nu-2)}}
        \mint_{ x > \kappa } 
        (2m) \Big(\mfrac{x}{\kappa}\Big)^{2m - 1}
        e^{-\frac{x^{2m}}{16}} 
        dx
    \\
    \;\leq&\; 
    \mfrac{A_3 \, \sigma_n^{1/m}}{ n^{1/2} \sigma_n^{\nu/(\nu-2)} \kappa^{ 2 m - 1}} 
\end{align*}
for some absolute constant $A_3 > 0$. In $(a)$, we have noted that $\varphi(x) \leq 1$ and the integrand is symmetric in $x$. Similarly, we can bound
\begin{align*}
    J_{22} 
    \;\leq&\; 
    \mfrac{C}{n^{3/2} \sigma_n^{(8-\nu)/(\nu-2)} }
        \mint_{  x^m > \kappa^m \sigma_n } 
        \mfrac{1}{x^{4m}}  \varphi(x) dx
    \\
    \;\leq&\; 
    \mfrac{C}{n^{3/2} \sigma_n^{(8-\nu)/(\nu-2)} }
        \mint_{  x^m > \kappa^m \sigma_n } 
        \mfrac{1}{x^{4m}}  dx 
    \\
    \;=&\; 
    \mfrac{(4m-1) C }{n^{3/2} \sigma_n^{(8-\nu)/(\nu-2)} \,  \sigma_n^{4 - 1 / m} \, \kappa^{4m - 1} } 
    \\
    \;=&\; 
    \mfrac{A_4 \, \sigma_n^{1/m} (4m-1) }{n^{3/2} \sigma_n^{3\nu/(\nu-2)} \kappa^{4m - 1} } 
\end{align*}
for some absolute constant $A_4 > 0$.
    
\vspace{.5em}

\textbf{Combining the bounds.} Combining the bounds, we get that if $\sigma_0^{\frac{\nu}{\nu-2}}  > 6^{-1/2}$, then
    \begin{align*}
        \P\big(  &p^*_m(Z) < - 2 \sigma_n  \big) - \P\big( p^*_m(X) < - 2 \sigma_n \big) 
        \;\geq\; J_1 - J_2
        \\
        \;\geq&\;
        \mfrac{\sigma_n^{1/m}}{ n^{1/2} \sigma_n^{\nu/(\nu-2)}}
        \Bigg(A_1  e^{- \frac{1}{2} \sigma_n^{2/m}} - \mfrac{A_2 \kappa}{n^{1/2} \sigma_n^{\nu/(\nu-2)}} - \mfrac{A_3}{\kappa^{2m - 1} } -  \mfrac{A_4 (4m-1) }{n  \,\sigma_n^{2\nu/(\nu-2)}  \, \kappa^{4m - 1}  } \Bigg)
        \;. 
    \end{align*}
 Recall that $\sigma_n =\min \big\{  \sigma_0 n^{-\frac{\nu-2}{2\nu}} \,,\,1 \big\}$ and $m=o(\log n)$. Take the absolute constant $N_2 \in \N$ sufficiently large such that 
$ \sigma_{N_2} = \sigma_0 N_2^{-\frac{\nu-2}{2\nu}} < 1$ and therefore 
\begin{align*}
    N_2^{1/2} \sigma^{\nu/(\nu-2)}_{N_2} 
    \;=\;
    \sigma_0^{\nu/(\nu-2)} \;.
\end{align*}
Also take the absolute constants
$\kappa > 2$ sufficiently large such that $\frac{A_3}{\kappa^{2m-1}} < \frac{1}{9} A_1$, $\sigma_0 > 0$ sufficiently large such that 
$\frac{A_2 \kappa}{\sigma_0^{\nu/(\nu-2)}} 
<
\frac{1}{9} A_1$ 
and 
$
\frac{A_4  }{\sigma_0^{2\nu/(\nu-2)} } \frac{4m-1}{ \kappa^{4m - 1}  }    
<
\frac{1}{9} A_1
$, and finally $N_1 \geq N_2$ sufficiently large such that $ A_1 e^{-\frac{1}{2} \sigma_{N_1}^{2/m}} =  A_1 e^{- \frac{1}{2} \sigma_0^{2/m} N_1^{-\frac{\nu-2}{m \nu}}} \geq \frac{1}{2} A_1$. 
Then we get the desired bound that, for some absolute constant $c > 0$ and all $n \geq N \coloneqq N_1$,
\begin{equation*}
    \P\big(  p^*_m(Z) < - 2 \sigma_n  \big) - \P\big( p^*_m(X) < - 2 \sigma_n \big) 
    \;\geq\;
    \mfrac{A_1}{6}
    \mfrac{ \sigma_n^{1/m}}{ n^{1/2} \sigma_n^{\nu/(\nu-2)}}
    \;\geq\; 
    c n^{-\frac{\nu-2}{2\nu m}}\;.  \qedhere
\end{equation*}
\end{proof}

To obtain the desired moment controls, we will need to use tight bounds on the moments of a univariate Gaussian. The proof of the following result is included in \cref{appendix:proof:bound:normal:moment}.

\vspace{.5em}

\begin{lemma} \label{lem:bound:normal:moment} Let $Z \sim \cN(0,1)$. Then there exist some absolute constants $C, c>0$ such that 
\begin{proplist}
    \item for any $\nu \geq 1$, $\mean |Z|^\nu \leq C \nu^{\nu/2}$;
    \item for any $m_1, m_2 \in \N$ with different parities, $ \Cov[ Z^{m_1}, Z^{m_2} ]  =0$;
    \item for $m_1, m_2 \in \N$ with the same parity, $\Cov[ Z^{m_1}, Z^{m_2} ] > c^{m_1+m_2} (m_1+m_2)^{(m_1+m_2)/2}$.
\end{proplist}
\end{lemma}

\vspace{.5em}

The next result controls the moment ratio arising from \Cref{thm:main}.
% \cref{thm:main}.

% \vspace{.5em}

\begin{lemma} \label{lem:near:opt:partial:upper:bound} Fix $\nu \in (2,3]$ and let $n \geq 2m^2$. Then for some absolute constant $C > 0$,
\begin{align*}
    \mfrac{ \sum_{i=1}^n \big\| \partial_i \,q_m^*(\bW_i)^\top \bX_i \big\|_{L_\nu}^\nu }
    { \big(\Var[p_m^*(X)]\big)^{\nu/2} 
    }
    \leq  C^m  n^{-\frac{\nu-2}{2}}, 
    \quad
    \text{ where }
    \bW_i \coloneqq \big(\bX_1, \ldots, \bX_{i-1}, 0, \xi_{i+1}, \ldots, \xi_n\big)\;.
\end{align*}
\end{lemma}

% \vspace{.5em}

\begin{proof}[Proof of \cref{lem:near:opt:partial:upper:bound}]
     First note that by independence,
\begin{align*}
    \sigma
    \;=\; 
    \sqrt{\Var[q_m^*(\bX)] }
    \;=\; 
    \sqrt{\Var[v^*_1(V)] + \Var[q^{**}_{v;m}(\Xi_\bY)] }
    \;\geq\;  
    \sqrt{\Var[q^{**}_{v;m}(\Xi_\bY)]}\;.
\end{align*}
Next to bound the $\nu$-th moment, denote $W_{V;i} \coloneqq (V_1, \ldots, V_{i-1}, 0, Z_{V;i+1}, \ldots, Z_{V;n} )$ and  $\bW_{\bY;i} \coloneqq (\bY_1, \ldots, \bY_{i-1}, 0, \xi_{\bY;i+1}, \ldots, \xi_{\bY;n} )$. By a triangle inequality, we get that
\begin{align*}
    \big\| \partial_i \,q_m^*(\bW_i)^\top \bX_i \big\|_{L_\nu}
    \;=&\; 
    \big\| \partial_i \,v^*_1(W_{V;i}) V_i + \partial_i \,q^{**}_{v;m}(\bW_{Y;i})^\top \bY_i  \big\|_{L_\nu}
    \\
    \;\leq&\; 
    n^{-1/2} \|  V_i \|_{L_\nu} +  \big\| \partial_i \,q^{**}_{v;m}(\bW_{Y;i})^\top \bY_i  \big\|_{L_\nu}\;.
\end{align*}
Since $\nu \in (2,3]$, the quantity to control can be bounded for some absolute constant $C_1>0$ as 
\begin{align*}
    \mfrac{ \sum_{i=1}^n \big\| \partial_i \,q_m^*(\bW_i)^\top \bX_i \big\|_{L_\nu}^\nu }
    { \big(\Var[p_m^*(X)]\big)^{\nu/2} 
    }
    \;\overset{(a)}{\leq}&\;  
    \mfrac{  4 n^{-\frac{\nu-2}{2}} \| V_1 \|_{L_\nu} + 4 \msum_{i=1}^n \big\| \partial_i \,q^{**}_{v;m}(\bW_{Y;i})^\top \bY_i  \big\|_{L_\nu}^\nu }
    {  (\Var[q^{**}_{v;m}(\Xi_\bY)])^{\nu/2} }
    \\
    \;\overset{(b)}{\leq}&\; 
    \mfrac{ C_1 n^{-\frac{\nu-2}{2}}}{  (\Var[v^*_m(Y)])^{\nu/2} }
    +
    \mfrac{ 4 \msum_{i=1}^n \big\| \partial_i \,q^{**}_{v;m}(\bW_{Y;i})^\top \bY_i  \big\|_{L_\nu}^\nu }
    {  (\Var[v^*_m(Y)])^{\nu/2} } 
    \;. \tagaligneq \label{eq:near:opt:upper:bound:initial}
\end{align*}
In $(a)$, we have used a Jensen's inequality to note that $(a+b)^\nu \leq 2^{\nu-1} (a^\nu + b^\nu) \leq 4  (a^\nu + b^\nu)$ and noted that $V_i$'s are i.i.d.; in $(b)$, we have used the moment bound on $V_1$ from \cref{lem:heavy:uni:moments} and noted that $\Var[q^{**}_{v;m}(\Xi_\bY)] = \Var[q^{**}_{v;m}(\bY)] = \Var[v^*_m(Y)]$ by a repeated application of \cref{lem:gaussian:linear:moment} and property of the multilinear representation. The first term can be controlled by using the lower bound from \cref{lem:bound:normal:moment}:
\begin{align*}
    \Var[v^*_m(Y)] 
    \;=\; 
    \Var\Big[ \Big( \mfrac{1}{\sqrt{n}} \msum_{i=1}^n Y_i \Big)^m \Big]
    \;=\; 
    \Var[Y_1^m]
    \;\geq\; 
    (c_2')^m (2m)^m 
    \;\geq\; 
    c_2  \tagaligneq \label{eq:near:opt:upper:bound:first:term}
\end{align*}
for some absolute constants $c'_2, c_2 > 0$. The second term can be controlled by applying \cref{lem:universality:v:stat:aug}: 
Since $n \geq 2m^2$, there exists some absolute constant $C_3 > 0$ such that
\begin{align*}
    \mfrac{ 
        \msum_{i=1}^n \big\| \partial_i \,q^{**}_{v;m}(\bW_{Y;i})^\top \bY_i  \big\|_{L_\nu}^\nu }
    {  (\Var[q^{**}_{v;m}(\Xi_\bY)])^{\nu/2} }
    \;\leq\;
    (C_3)^m
    \,
    n^{-\frac{\nu-2}{2}}
    \,
    \Bigg(
    \mfrac{
            \sum_{k=1}^m
            \binom{n}{k}
            (\alpha_\nu^*(k))^2
    }
    {
    \sum_{k=1}^m
    \binom{n}{k}
    (\alpha_2^*(k))^2
    }
    \Bigg)^{\frac{\nu}{2}}\;, 
    \tagaligneq \label{eq:near:opt:upper:bound:second:term}
\end{align*}
where 
\begin{align*}
    \alpha_\nu^*(k) 
    \;\coloneqq\;
    \bigg\| 
    \msum_{\substack{p_1 + \ldots + p_k = m \\ 
    p_1, \ldots, p_k \geq 1
    }} 
    \,
    \mprod_{l=1}^k 
    \big( Y_l^{p_l} - \mean Y_l^{p_l} \big)
    \bigg\|_{L_\nu} \;.
\end{align*}
We are left with bounding $\alpha_2^*(k)$ and $\alpha_\nu^*(k)$. Since $Y_i$'s are i.i.d.~standard normals, the proof is down to using the moment bounds for standard normals from \cref{lem:bound:normal:moment}. For convenience, we denote $\overline{Y_l^{p_l}} \coloneqq  Y_l^{p_l} - \mean Y_l^{p_l}$ from now on. We can then express 
\begin{align*}
    (\alpha_2^*(k))^2 \;=\; 
    \msum_{\substack{p_1 + \ldots + p_k = m \\ 
    p_1, \ldots, p_k \geq 1
    }} 
    \msum_{\substack{q_1 + \ldots + q_k = m \\ 
    q_1, \ldots, q_k \geq 1
    }}
    \,
    \Cov\bigg[ 
        \mprod_{l=1}^k 
        \overline{Y_l^{p_l} }
        \,,\,       
        \mprod_{l=1}^k 
        \overline{Y_l^{p_l} }
    \bigg]\,.
\end{align*}
We now show by induction that for some absolute constant $c > 0$,
\begin{align*}
    \Cov\bigg[ 
        \mprod_{l=1}^k 
        \overline{Y_l^{p_l} }
        \,,\,       
        \mprod_{l=1}^k 
        \overline{Y_l^{p_l} }
    \bigg] 
    \;\geq\; 
    \mprod_{l=1}^k  \ind_{\{p_l \equiv q_l \, (\textrm{mod } 2)\}} \, c^{p_l+q_l}\, (p_l+q_l)^{(p_l+q_l)/2}
    \;.
\end{align*}
For $k=1$, this directly holds for all $p_l, q_l \in [m]$ by the lower bound of \cref{lem:bound:normal:moment}. Suppose this holds for $k-1$. By the total law of covariance and noting that all variables are centred and independent,
\begin{align*}
    \Cov\bigg[ 
        \mprod_{l=1}^k 
        \overline{Y_l^{p_l} }
        \,,\,       
        \mprod_{l=1}^k 
        \overline{Y_l^{p_l} }
    \bigg] 
    \;=&\; 
    \mean 
    \,  \Cov\bigg[ 
        \mprod_{l=1}^k 
        \overline{Y_l^{p_l} }
        \,,\,       
        \mprod_{l=1}^k 
        \overline{Y_l^{p_l} }
        \, \bigg| \, Y_1, \ldots, Y_{k-1}
    \bigg] 
    \\
    \;=&\;
    \mean\Big[ \overline{Y_k^{p_k} } \times \overline{Y_k^{q_k} } \Big]
    \,
    \Cov\bigg[ 
        \mprod_{l=1}^{k-1} 
        \overline{Y_l^{p_l} }
        \,,\,       
        \mprod_{l=1}^{k-1} 
        \overline{Y_l^{p_l} }
    \bigg] 
    \\
    \;\geq&\;
    \Cov[ Y_k^{p_k} \,,\, Y_k^{q_k} ] 
    \mprod_{l=1}^{k-1}  \ind_{\{p_l \equiv q_l \, (\textrm{mod } 2)\}} \, c^{p_l+q_l}\, (p_l+q_l)^{(p_l+q_l)/2}
    \\
    \;\geq&\; 
    \mprod_{l=1}^k  \ind_{\{p_l \equiv q_l \, (\textrm{mod } 2)\}} \, c^{p_l+q_l}\, (p_l+q_l)^{(p_l+q_l)/2}\;.
\end{align*}
This finishes the induction, and in particular implies 
\begin{align*}
    (\alpha_2^*(k))^2 
    \;\geq&\; 
    c^{2m}
    \msum_{\substack{p_1 + \ldots + p_k = m \\ 
    p_1, \ldots, p_k \geq 1
    }} 
    \msum_{\substack{q_1 + \ldots + q_k = m \\ 
    q_1, \ldots, q_k \geq 1
    }}
    \mprod_{l=1}^k  \ind_{\{p_l \equiv q_l \, (\textrm{mod } 2)\}} \, (p_l+q_l)^{(p_l+q_l)/2}
    \\
    \;\overset{(a)}{\geq}&\; 
    c^{2m}
    \msum_{\substack{p_1 + \ldots + p_k = m \\ 
    p_1, \ldots, p_k \geq 1
    }} 
    \msum_{\substack{q_1 + \ldots + q_k = m \\ 
    q_1, \ldots, q_k \geq 1
    }}
    \\
    &\qquad \times 
    \mprod_{l=1}^k  \ind_{\{p_l \equiv q_l \, (\textrm{mod } 2)\}} \, 
    \mfrac{
        (p_l+q_l)^{(p_l+q_l)/2}
        +
        (c')^{p_l+q_l} (p_l+q_l+1)^{(p_l+q_l+1)/2}
    }{2}
    \\
    \;\geq&\;
    (c'')^{2m}
    \msum_{\substack{p_1 + \ldots + p_k = m \\ 
    p_1, \ldots, p_k \geq 1
    }} 
    \msum_{\substack{q_1 + \ldots + q_k = m \\ 
    q_1, \ldots, q_k \geq 1
    }}
    (p_l+q_l)^{(p_l+q_l)/2}
\end{align*}
for some absolute constant $c'' > 0$. In $(a)$, we have noted that 
\begin{align*}
    (p_l+q_l)^{(p_l+q_l)/2}
    \;=&\; 
    (p_l+q_l)^{-1/2}
    (p_l+q_l)^{(p_l+q_l+1)/2}
    \\
    \;\geq&\; 
    \big( (p_l+q_l)^{-1/2} 2^{-(p_l+q_l+1)/2}\big) (p_l+q_l+1)^{(p_l+q_l+1)/2}
    \\
    \;\geq&\; 
    (c')^{p_l+q_l} (p_l+q_l+1)^{(p_l+q_l+1)/2}
\end{align*}
for some absolute constant $c' > 0$. On the other hand, by a triangle inequality and noting that $Y_i$'s are i.i.d., we have
\begin{align*}
    \big(\alpha_\nu^*(k) \big)^2
    \;=&\;
    \bigg\| 
    \msum_{\substack{p_1 + \ldots + p_k = m \\ 
    p_1, \ldots, p_k \geq 1
    }} 
    \,
    \mprod_{l=1}^k 
    \overline{ Y_l^{p_l} }
    \bigg\|_{L_\nu}^2
    \\
    \;\leq&\; 
    \msum_{\substack{p_1 + \ldots + p_k = m \\ 
    p_1, \ldots, p_k \geq 1
    }} 
    \msum_{\substack{q_1 + \ldots + q_k = m \\ 
    q_1, \ldots, q_k \geq 1
    }}
    \Big\| \mprod_{l=1}^k \overline{ Y_l^{p_l} }
    \Big\|_{L_\nu} 
    \Big\| \mprod_{l=1}^k \overline{ Y_l^{q_l} }
    \Big\|_{L_\nu}
    \\
    \;=&\; 
    \msum_{\substack{p_1 + \ldots + p_k = m \\ 
    p_1, \ldots, p_k \geq 1
    }} 
    \msum_{\substack{q_1 + \ldots + q_k = m \\ 
    q_1, \ldots, q_k \geq 1
    }}
    \mprod_{l=1}^k \Big\|\overline{ Y_1^{p_l} } \Big\|_{L_\nu} \Big\|\overline{ Y_1^{q_l} } \Big\|_{L_\nu}\;.
\end{align*}
By noting that $\nu \leq 3$ and using the moment bound in \cref{lem:bound:normal:moment}, we get that for some absolute constants $C', C''> 0$,
\begin{align*}
    \Big\|\overline{ Y_1^p } \Big\|_{L_\nu}^\nu  \Big\|\overline{ Y_1^q } \Big\|_{L_\nu}^\nu 
    \;=&\; 
    \mean\big[ \big| Y_1^p - \mean[Y_1^p]\big|^\nu \big]
    \mean\big[ \big| Y_1^q - \mean[Y_1^q]\big|^\nu \big]
    \\
    \;\leq&\; 
    C' \mean[ |Y_1|^{p \nu}]  \mean[ |Y_1|^{q \nu}] 
    \;\leq\; 
    C'' p^{p\nu/2} q^{q \nu /2}
    \;\leq\; 
    C'' (p+q)^{(p+q) \nu /2 }\;.
\end{align*}
This implies that 
\begin{align*}
    \big(\alpha_\nu^*(k) \big)^2
    \;\leq&\; 
    \sum_{\substack{p_1 + \ldots + p_k = m \\ 
    p_1, \ldots, p_k \geq 1
    }} 
    \sum_{\substack{q_1 + \ldots + q_k = m \\ 
    q_1, \ldots, q_k \geq 1
    }}
    \mprod_{l=1}^k 
    (C'')^{1/\nu} (p_l+q_l)^{(p_l+q_l) /2 }
    \;\leq\; 
    (\tilde c)^{2m} \big(\alpha_2^*(k) \big)^2
\end{align*}
for some absolute constant $\tilde c > 0$. \eqref{eq:near:opt:upper:bound:second:term} then becomes
\begin{align*}
    \mfrac{ 
        \msum_{i=1}^n \big\| \partial_i \,q^{**}_{v;m}(\bW_{Y;i})^\top \bY_i  \big\|_{L_\nu}^\nu }
    {  (\Var[q^{**}_{v;m}(\Xi_\bY)])^{\nu/2} }
    \;\leq&\;
    (C_3)^m
    \,
    n^{-\frac{\nu-2}{2}}
    \,
    \Bigg(
    \mfrac{
            \sum_{k=1}^m
            \binom{n}{k}
            (\alpha_\nu^*(k))^2
    }
    {
    \sum_{k=1}^m
    \binom{n}{k}
    (\alpha_2^*(k))^2
    }
    \Bigg)^{\frac{\nu}{2}}
    \;\leq\;
    ( \tilde c^\nu C_3 )^{\nu m}
    \,
    n^{-\frac{\nu-2}{2}}\;.
\end{align*}
Plugging in this bound and \eqref{eq:near:opt:upper:bound:first:term} into \eqref{eq:near:opt:upper:bound:initial}, we get that 
\begin{align*}
    \mfrac{ \sum_{i=1}^n \big\| \partial_i \,q_m^*(\bW_i)^\top \bX_i \big\|_{L_\nu}^\nu }
    { \big(\Var[p_m^*(X)]\big)^{\nu/2} 
    }
    \;\leq&\; 
    \mfrac{ C_1 n^{-\frac{\nu-2}{2}}}{  c_2^{\nu/2} }
    +
    4  ( \tilde c^\nu C_3 )^{\nu m}
    \,
    n^{-\frac{\nu-2}{2}}
    \;\leq\;  C^m  n^{-\frac{\nu-2}{2}}
\end{align*}
for some absolute constant $C > 0$. This finishes the proof.
\end{proof}

\vspace{.5em}

The next result allows $q^*_m(\Xi) = q^{**}_m(\Xi) + \mean[p^*_m(X)]$ to be approximated by $p^*_m(Z)$.

% \vspace{.5em}

\begin{lemma} \label{lem:near:opt:delta:reg}
    % Let $C$ and $c$ be the absolute constants in the upper and lower bounds in \cref{lem:bound:normal:moment}
    % , and assume that $m \leq  \delta (\log n) (2 + \max\{\log c,0\} + \max\{\log C, 0\})^{-1} $ for some $\delta \in [0, 1)$. 
Assume that $Z_Y$ and $\Xi_\bY$ are coupled such that $Z_{Y;i} = \xi_{Y;i1}$ almost surely. Then there exists some absolute constant $C' > 0$ such that 
\begin{align*}
    \mfrac{ \| q^{**}_m(\Xi) + \mean[p^*_m(X)] - p^*_m(Z) \|_{L_2}^2 }{\Var[q^{**}_m(\Xi)]}  
    \;\leq\; 
    \mfrac{ (C')^m }{ n^{1-\delta}}\;.
\end{align*} 
\end{lemma}

% \vspace{.5em}

\begin{proof}[Proof of \cref{lem:near:opt:delta:reg}] 
We first compute $R_n$ in \Cref{sec:proof:v:additional}
for the V-statistic $v^*_m$. It suffices to control 
\begin{align*}
    \max_{q_1 + \ldots + q_j = m\,,\, q_l \in \N  }
   & \max\bigg\{
        \Big(\mean\Big[ n^{m/2} \mprod_{l=1}^j Y_l^{q_l} \Big]\Big)^2  \,,\, 
        \Var\Big[ n^{m/2} \mprod_{l=1}^j Y_l^{q_l} \Big]
    \bigg\}
    \\
    & \hspace{-4em} \;\leq\;
    n^m
    \max_{q_1 + \ldots + q_j = m\,,\, q_l \in \N  }
        \Big\| \mprod_{l=1}^j Y_l^{q_l} \Big\|_{L_2}^2
    \;\;=\;\; 
    n^m
    \max_{q_1 + \ldots + q_j = m\,,\, q_l \in \N  } \mprod_{l=1}^j  \| Y_1^{q_l} \|_{L_2}^2
    \\
    & \hspace{-4em} \;\leq\; 
    n^m \, \max_{q_1 + \ldots + q_j = m\,,\, q_l \in \N  }  \, \mprod_{l=1}^j \big( C (2q_l)^{q_l} \big)
    \;\;\leq\;\;
    n^m \, (2C)^m \max_{q_1 + \ldots + q_j = m\,,\, q_l \in \N  }  \, \mprod_{l=1}^j q_l^{q_l} 
\end{align*}
for some absolute constant $C > 0$. Note that we have used independence followed by \cref{lem:bound:normal:moment} again, and that we only need to bound the above for $j \in [m-1]$. We shall now perform an induction to show that 
\begin{align*}
    Q(j,m) \;\coloneqq\; \max_{q_1 + \ldots + q_j = m\,,\, q_l \in \N  }  \, \mprod_{l=1}^j q_l^{q_l}  \;=\; (m-j+1)^{m-j+1}\;.
\end{align*}
This holds trivially for $j=1 < m$. To prove this for $j \geq 2$ and $m > j$, suppose the statement holds for all $j' < j$ and all $m' > j'$. By applying the pigeonhole principle, $q_1$ can only take values in $[m-j+1]$, and
\begin{align*}
    Q(j,m) 
    \;=&\; 
    \max_{q_1 \in [m-j+1]} q_1^{q_1} Q(j-1, m-q_1)
    \\
    \;=&\; 
    \max_{q_1 \in [m-j+1]} q_1^{q_1} (m-q_1-j+2)^{m-q_1-j+2}
    \;=\; 
    \max_{q_1 \in [m-j+1]} \psi_{m-j+2}(q_1)\;,
\end{align*}
where we have denoted $\psi_{m'}(x) = x^x (m'-x)^{m'-x}$, defined for $x \in [0,m']$. Since $\psi_{m'}(x)$ is strictly convex with a minimum at $x=m'/2$, its maximum is obtained at the boundary, so 
\begin{align*}
    Q(j,m) \;=\; \psi_{m-j+2}(1) \;=\; (m-1-j+2)^{m-1-j+2} \;=\; (m-j+1)^{m-j+1}\;.
\end{align*}
This finishes the induction. By additionally recalling that, by using the lower bound from \cref{lem:bound:normal:moment} (as proved in \cref{eq:near:opt:upper:bound:first:term}), we have
\begin{align*}
    \Var[v^*_m(Y)] 
    \;=\; 
    \Var\Big[ \Big( \mfrac{1}{\sqrt{n}} \msum_{i=1}^n Y_i \Big)^m \Big]
    \;=\; 
    \Var[Y_1^m]
    \;\geq\; 
    (1/c)^m m^m 
\end{align*}
for some absolute constant $c> 0$. This implies that 
\begin{align*}
   & \max_{q_1 + \ldots + q_j = m\,,\, q_l \in \N  }
    \max\bigg\{
        \Big(\mean\Big[ n^{m/2} \mprod_{l=1}^j Y_l^{q_l} \Big]\Big)^2  \,,\, 
        \Var\Big[ n^{m/2} \mprod_{l=1}^j Y_l^{q_l} \Big]
    \bigg\}
    \\
    &\;\leq\;
    n^m \, (2C)^m \, m^{m-j+1}
    \\
    &\;\leq\; 
    n^m \, (2 \max\{C, 1\})^m \, m^{m-j+1} c^m m^{-m}  \Var[v^*_m(Y)]
    \\
    &\;\leq\; 
    \tilde C^m n^m  m^{-j}  \Var[v^*_m(Y)]
    % \\
    % &\;=\; 
    % e^{m (\log 4 + \max\{\log c, 0\} + \max\{\log C, 0\})}  n^m m^{-j}  \Var[v^*_m(Y)]
    % \\
    % &\;<\; 
    % e^{m (2 + \max\{\log c, 0\} + \max\{\log C, 0\})}  n^m m^{-j}  \Var[v^*_m(Y)]
    % \\
    % &\;\leq\; 
    % n^{m + \delta }  m^{-j}  \Var[v^*_m(Y)]
\end{align*}
for some absolute constant $\tilde C > 0$. This implies that for the V-statistic $v^*_m(Y)$, 
\begin{align*}
    R_n \;\leq\;  \tilde C^m\;.
\end{align*}
% where we have used the assumption on $m$ in the last line. This implies that $v^*_m$ in (6) of \Cref{sec:non:multilinear} 
% \eqref{eq:simple:v:defn}
% is $\delta$-regular with respect to $Y$. 
By \cref{lem:v:var:approx}, 
there is some absolute constant $C' > 0$ such that
\begin{align*}
    \mfrac{ \| q^{**}_{v;m}(\Xi_\bY) + \mean[v^*_m(Y)] - v^*_m(Z_Y) \|_{L_2}^2 }
    {\Var[q^{**}_{v;m}(\Xi_\bY)]} 
    \;\leq\; 
    \mfrac{ (C')^m }{ n}\;.
\end{align*}
Recall that $v^*_m(Z_Y) - \mean[v^*_m(Y)] = p^*_m(Z) - \mean[p^*_m(X)] - v^*_1(Z_V)$ and, by the coupling assumption,  $q^{**}_{v;m}(\Xi_\bY) = q^{**}_m(\Xi) - v^*_1(Z_V)$ . Also, since $v^*_m(Y) = \big(n^{-1/2} \sum_{i=1}^n Y_i \big)^m \overset{d}{=} Y_1^m$, by \cref{lem:bound:normal:moment}, $\Var[q^{**}_{v;m}(\Xi_\bY) ] = \Var[v^*_m(Y)]  > c^m \, (2m)^m $ for some absolute constant $c > 0$. These imply the desired bound that
\begin{equation*}
    \mfrac{ \| q^{**}_{v;m}(\Xi) - (p^*_m(Z) - \mean[p^*_m(X)]) \|_{L_2}^2 }{\Var[f^*(\Xi)]}  
    \;=\;
    \mfrac{ \| q^{**}_{v;m}(\Xi_\bY) + \mean[v^*_m(Y)] - v^*_m(Z_Y) \|_{L_2}^2 }{\Var[v^*_1(Z_V)] + \Var[q^{**}_{v;m}(\Xi_\bY)]}  
    \;\leq\; 
    \mfrac{ (C')^m }{ n}\;.
    \qedhere
\end{equation*}
\end{proof}

\vspace{.5em}

Combining the results in this section allows us to prove \Cref{thm:lower:bound}.
% \cref{thm:lower:bound}.

\begin{proof}[Proof of \Cref{thm:lower:bound}
    % \cref{thm:lower:bound}
] By a direct application of \Cref{cor:non:multilinear}
% \cref{thm:main} 
to $q^{**}_m(\bX) = p^*_m(X) - \mean[p^*_m(X)]$ followed by the moment bound computed in \cref{lem:near:opt:partial:upper:bound}, we get that if $n \geq 2m^2$, there is an absolute constant $C' > 0$ such that
\begin{align*}
    \msup_{t \in \R} \big| \P( q^{**}_m(\bX) \leq t) - \P( q^{**}_m(\Xi) \leq t) \big|
    \;\leq&\; 
    C' m n^{-\frac{\nu-2}{2\nu m +2 }}\;.  \tagaligneq \label{eq:near:opt:q:upper}
\end{align*}
Since $m=o(\log n)$, there is some absolute constant $N' \in \N$ such that the above holds for all $n \geq N'$. Meanwhile since $m$ is even, by \cref{lem:lower:V}, there are absolute constants $c', \sigma_0 > 0$ and $N'' \in \N$ such that, for any $n \geq N''$,
\begin{align*}
    \msup_{t \in \R} \big| \P( p^*_m(X) \leq t) - \P( p^*_m(Z) \leq t) \big|
    \;\geq\; 
    c' n^{-\frac{\nu-2}{2\nu m}}
    \;.  \tagaligneq \label{eq:near:opt:p:lower}
\end{align*}  
We are left with approximating $p^*_m$ by $q^{**}_m$. WLOG assume that $Z_{Y;i} = \xi_{Y;i1}$ almost surely. By \cref{lem:near:opt:delta:reg}, there is an absolute constant $C_* > 0$ such that
\begin{align*}
    \mfrac{ \| q^{**}_m(\Xi) + \mean[p^*_m(X)] - p^*_m(Z)  \|_{L_2}^2 }{\Var[q^{**}_m(\Xi)]}  
    \;\leq\; 
    \mfrac{ (C_*)^m }{ n}\;. \tagaligneq \label{eq:near:opt:delta:reg}
\end{align*} 
Since $m= \log(n)$, there is some absolute constant $N''' \in \N$ such that the above holds for all $n \geq N'''$ and any fixed $\delta \in [0,1)$ to be specified. Now by \cref{lem:approx:XplusY:by:Y}, for any $\epsilon > 0$ and $t \in \R$, 
\begin{align*}
    \P\big(  p^*_m(Z) \leq t\big)
    \;\leq&\;\,
    \P\big(  q^{**}_m(\Xi) + \mean[p^*_m(X)]  \leq t + \epsilon \big) 
    +
    \P\big( | q^{**}_m(\Xi) + \mean[p^*_m(X)] - p^*_m(Z) | \geq \epsilon)
    \;,
    \\
    \P\big(  p^*_m(Z) \leq t\big)
    \;\geq&\;\,
    \P\big( q^{**}_m(\Xi) + \mean[p^*_m(X)]  \leq t - \epsilon \big) 
    - 
    \P\big( | q^{**}_m(\Xi) + \mean[p^*_m(X)] - p^*_m(Z) | \geq \epsilon)\;. 
\end{align*}
This implies that
\begin{align*}
    &\;\big|  \P\big( q^{**}_m(\Xi) + \mean[p^*_m(X)] \leq t\big) - \P\big( p^*_m(Z) \leq t  \big) \big|
    \\
    &\;\;\leq\; 
    \P\big(  t-\epsilon \leq q^{**}_m(\Xi) + \mean[p^*_m(X)] \leq t+ \epsilon \big)
    +
    \P\big( | q^{**}_m(\Xi) + \mean[p^*_m(X)] - p^*_m(Z) | \geq \epsilon)
    \\
    &\;\;\overset{(a)}{\leq}\; 
    C_1 m \Big( \mfrac{\epsilon}{\sqrt{\Var[q^{**}_m(\Xi)]}} \Big)^{1/m}
    +
    \mfrac{\| q^{**}_m(\Xi) + \mean[p^*_m(X)] - p^*_m(Z) \|_{L_2}^2}{\epsilon^2}
    \\ 
    &\;\;\overset{(b)}{\leq}\; 
    C_1 m \Big( \mfrac{\epsilon}{\sqrt{\Var[q^{**}_m(\Xi)]}} \Big)^{1/m}
    +
    \mfrac{(C_{*})^m}{n} 
    \Big( \mfrac{\sqrt{\Var[q^{**}_m(\Xi)]}}{\epsilon} \Big)^2
\end{align*}
for some absolute constant $C_1 > 0$. In $(a)$, we used the Carbery-Wright argument in \eqref{eq:main:proof:carbery:wright} in the proof of \Cref{thm:main}
% \cref{thm:main} 
for the first term, and Markov's inequality for the second term. In $(b)$, we have applied \eqref{eq:near:opt:delta:reg}. Choosing $\epsilon = (C_*)^{m^2/(2m+1)} n^{-m / (2m+1)} \sqrt{\Var[q^{**}_m(\Xi)]}$, we get that for some absolute constant $C_{**} > 0$,
\begin{align*}
    \big|  \P\big(  q^{**}_m(\Xi) + \mean[p^*_m(X)] \leq t\big) - \P\big( p^*_m(Z) \leq t  \big) \big|
    \;\leq&\; (C_1 m +1) (C_{**})^{\frac{m}{2m+1}} n^{-\frac{1-\delta}{2m+1}}
    \\
    \;\leq&\; 
    C_{**} m n^{-\frac{1}{2m+1}}\;. \tagaligneq \label{eq:near:opt:delta:reg:argument}
\end{align*}
Combining this with \eqref{eq:near:opt:q:upper} via a triangle inequality, we get that for all $n \geq \max\{ N', N''' \}$,
\begin{align*}
    \msup_{t \in \R} \big| \P( p^*_m(X) \leq t) - \P( p^*_m(Z) \leq t) \big|
    \;\leq&\;
    C''' m n^{-\frac{\nu-2}{2\nu m +2 }} + C_{**} m n^{-\frac{1}{2m+1}}  
    \;,
\end{align*}
whereas combining the bound with \eqref{eq:near:opt:p:lower} via a reverse triangle inequality, we get that  for all $n \geq \max\{ N', N'' \}$,
\begin{align*}
    \msup_{t \in \R} \big| \P\big(  q^{**}_m(\bX) \leq t  \big) - \P\big( q^{**}_m(\Xi) \leq t \big) \big|
    \;\geq&\;
    c'' n^{-\frac{\nu-2}{2\nu m}} -  C_{**} m n^{-\frac{1}{2m+1}} \;.
\end{align*}
Since $\nu \in (2,3]$ and $m$ is even,
\begin{align*}
    \mfrac{1}{2m+1} \;>\; \mfrac{1}{6 m} \;\geq\; \mfrac{\nu-2}{2\nu m} \;\geq\; \mfrac{\nu-2}{2 \nu m +2 }\;,
\end{align*}
and that the $m n^{-\frac{1}{2m+1}} = o\big( (\log n) n^{-\frac{1}{2m+1}} \big)$ term can be ignored. This finishes the proof by noting that 
\begin{align*}
    q^{**}_m(\bx_1, \ldots, \bx_n)
    \;\coloneqq\;
    q^*_m(\bx_1, \ldots, \bx_n) 
    -
    \mean[q^*_m(\bX)]
    \;.
\end{align*}
\end{proof}

\section{Properties of univariate distributions in Theorem \ref{thm:lower:bound}} 
% \ref{thm:lower:bound}}  
\label{appendix:senatov}

\subsection{Proof of Gaussian moment bound in Lemma \ref{lem:bound:normal:moment}} \label{appendix:proof:bound:normal:moment} Since $ \mean |Z|^\nu = \pi^{-1/2} 2^{\nu/2} \Gamma(\frac{\nu+1}{2} )$, the proof boils down to approximating the Gamma function: \cite{alzer2003ramanujan} proves that for all $x \geq 0$,
    \begin{align*}
        \sqrt{\pi} \Big( \mfrac{x}{e} \Big)^x \Big( 8 x^3 + 4 x^2 + x + \mfrac{1}{100} \Big)^{1/6}
        <
        \Gamma(x+1)
        <
        \sqrt{\pi} \Big( \mfrac{x}{e} \Big)^x \Big( 8 x^3 + 4 x^2 + x + \mfrac{1}{30} \Big)^{1/6}
        \;.
    \end{align*}
    Since $\nu \geq 1$, $\frac{\nu+1}{2} \geq 1$. As $ 8 x^3 + 4 x^2 + x + \frac{1}{30} \leq 14(1+x)^3$ for $x \geq 0$, we have
    \begin{align*}
        \Gamma\Big(\mfrac{\nu+1}{2}\Big) 
        \;\leq&\; 
        14^{1/6} \sqrt{\pi} 
        \Big( \mfrac{\nu-1}{2e} \Big)^{(\nu-1)/2} \Big( \mfrac{\nu+1}{2}\Big)^{1/2}
        \\ 
        \;\leq&\; 
        14^{1/6} \sqrt{\pi} 
        \Big( \mfrac{\nu-1}{2} \Big)^{\frac{\nu}{2} - 1 } \Big( \mfrac{\nu^2 - 1}{4}\Big)^{1/2}
        \;\leq\; 
        14^{1/6} \sqrt{\pi} 
        \Big( \mfrac{\nu}{2} \Big)^{\frac{\nu}{2}}
        \;,
    \end{align*}
    which implies the desired bound that, for some absolute constant $C > 0$, 
    \begin{align*}
        \mean |Z|^\nu 
        \;\leq\; 
        C
        \nu^{\nu/2}
        \;.
    \end{align*}
    For the second bound, we note that if $m_1$ and $m_2$ have different parities,
    \begin{align*}
        \Cov[ Z^{m_1}, Z^{m_2}] 
        \;=\; \mean[Z^{m_1 + m_2} ] - \mean[Z^{m_1}] \mean[Z^{m_2}] \;=\; 0
    \end{align*}
    since odd moments of $Z$ vanish. Now focus on the case when $m_1$ and $m_2$ have the same parity, and recall that $m_1 \geq m_2$ by assumption. For $x \geq 0$ and $\alpha \in [0,1]$, we define the function 
    $
        R(x;\alpha) \coloneqq 2^{x + \frac{1}{2} }   \Big( \mfrac{x}{e} \Big)^x ( 8 x^3 + 4 x^2 + x + \alpha )^{1/6}
    $, which implies that 
    \begin{align*}
        R\Big( \mfrac{\nu-1}{2}; \mfrac{1}{100} \Big) 
        \;<\; 
        \mean |Z|^\nu 
        \;<\; 
        R\Big( \mfrac{\nu-1}{2}; \mfrac{1}{30} \Big) \;.
    \end{align*}
    Then we can bound 
    \begin{align*}
        \Cov[ &Z^{m_1}, Z^{m_2} ] 
        \;\geq\; 
        \mean[ Z^{m_1+m_2}] 
        -
        \mean[ |Z|^{m_1} ] \mean[|Z|^{m_2}]
        \\
        \;\geq&\;
        R\Big( \mfrac{m_1+m_2-1}{2}; \mfrac{1}{100} \Big)  -
        R\Big( \mfrac{m_1-1}{2}; \mfrac{1}{30} \Big) R\Big( \mfrac{m_2-1}{2}; \mfrac{1}{30} \Big) 
        \;.
    \end{align*}
    First suppose $m_2 \geq 2$, and denote $\alpha=1/100$ and $\beta=1/30$.  Note that for $x \geq y \geq 3/4$,
    \begin{align*}
        &\mfrac{R(x+y; \alpha)}{R(x - 1/4;\beta)R(y - 1/4;\beta)}
        \\
        &\;=\; 
        e^{-1/2}  
        \mfrac{ (x+y)^{x+y} } 
        {
            (x-1/4)^{x-1/4}
            (y-1/4)^{y-1/4}
        }
        \\
        &\quad 
        \times
        \mfrac{
        \big(
            8(x+y)^3 
            +
            4(x+y)^2
            +
            (x+y)
            +
            \alpha    
        \big)^{1/6}
        }
        {
            \big(
            8(x-1/4)^3 
            +
            4(x-1/4)^2
            +
            (x-1/4)
            +
            \beta
            \big)^{1/6}
            \big(
            8(y-1/4)^3 
            +
            4(y-1/4)^2
            +
            (y-1/4)
            +
            \beta
            \big)^{\frac{1}{6}}
        }
        \\
        &\;\geq\; 
        e^{-1/2}  
        \Big( \mfrac{x+y}{y-1/4} \Big)^{y-1/4}
        \\
        &\quad 
        \times
        \bigg(
        \mfrac{
            8(x+y)^6
            +
            4(x+y)^5
            +
            (x+y)^4
            +
            \alpha (x+y)^3    
        }
        {
            \big(
            8(x-1/4)^3 
            +
            4(x-1/4)^2
            +
            (x-1/4)
            +
            \beta
            \big)
            \big(
            8(y-1/4)^3 
            +
            4(y-1/4)^2
            +
            (y-1/4)
            +
            \beta
            \big)
        }
        \bigg)^{\frac{1}{6}}
        \\
        &\;\eqqcolon\; 
        a(x,y) \Big( \mfrac{b(x,y)}{c(x,y)} \Big)^{1/6}
        \;.
    \end{align*}
    Note that since $x \geq y \geq 3/4$, we have
    \begin{align*}
        a(x,y) 
        \;=\; 
        e^{-1/2}  
        \Big( \mfrac{x+y}{y-1/4} \Big)^{y-1/4}
        \;\geq\; 
        e^{-1/2}  
        2^{y-1/4} 
        \;\geq\; 
        e^{-1/2} 2^2  \;>\; 2\;.
    \end{align*} 
    On the other hand, a lengthy computation gives
    \begin{align*}
        c(x,y) 
        \;<&\; 
        64x^3y^3
        +
        32 x^3 y^2
        +
        8 x^3 y 
        +
        8 \beta x^3
        +
        32x^2 y^3 
        +
        16x^2 y^2 
        +
        4 x^2 y 
        +
        4\beta x^2
        +
        8 x y^3
        \\
        &\; 
        +4 x y^2 
        +
        xy 
        +
        \beta x
        +
        8\beta y^3 
        +
        4\beta y^2 
        +
        \beta y 
        +
        \beta^2
        \\
        \;\leq&\;
        b(x,y) 
        - \big( 96 x^3 y^3 - 4x^3 y - (8\beta - \alpha) x^3 - 10x^2y^2 - x^2y - 4\beta x^2 - 4xy^3 - 4xy^2 
        \\
        &\qquad\qquad\quad - xy -\beta x -8\beta y^3 - 4\beta y^2 - \beta y - \beta^2 \big)
        \\ 
        \;\leq&\; 
        b(x,y)  - 46 x^3 y^3 \;<\; b(x,y) \;,
    \end{align*}
    where we have used that $x, y \geq 3/4$, $\alpha=1/100$ and $\beta=1/30$ in the last line. Therefore $b(x,y) / c(x,y) > 1$ and, for $x \geq y \geq 3/4$,
    \begin{align*}
        \mfrac{R(x+y; \alpha)}{R(x - 1/4;\beta)R(y - 1/4;\beta)} \;>\; 2\;.
    \end{align*}
    By identifying $x= \frac{m_1}{2} - \frac{1}{4}$ and $y= \frac{m_2}{2} - \frac{1}{4}$, we get that for $m_1 \geq m_2 \geq 2$ such that $m_1$ and $m_2$ have the same parity,
    \begin{align*}
        \Cov[ Z^{m_1}, Z^{m_2} ] 
        \;\geq&\;
        R\Big( \mfrac{m_1+m_2-1}{2}; \mfrac{1}{100} \Big)  -
        R\Big( \mfrac{m_1-1}{2}; \mfrac{1}{30} \Big) R\Big( \mfrac{m_2-1}{2}; \mfrac{1}{30} \Big) 
        \\
        \;>&\; 
        \mfrac{1}{2} R\Big( \mfrac{m_1+m_2-1}{2}; \mfrac{1}{100} \Big)\;. 
    \end{align*}
    Meanwhile for $m_2 = 1$ and $m_1$ odd, since $\mean[Z]=0$, we obtain directly that 
    \begin{align*}
        \Cov[ Z^{m_1}, Z^{m_2} ] 
        \;=\; 
        \mean[ Z^{m_1+m_2-1}] 
        \;\geq\; 
        R\Big( \mfrac{m_1+m_2-1}{2}; \mfrac{1}{100} \Big) 
        \;>\; 
        \mfrac{1}{2} R\Big( \mfrac{m_1+m_2-1}{2}; \mfrac{1}{100} \Big) \;.
    \end{align*}
    Therefore for any $m_1$ and $m_2$ with the same parity, we get the desired bound that 
    \begin{align*}
        \Cov[ Z^{m_1},& Z^{m_2} ] \;>\; 2^{\frac{m_1+m_2}{2}} \Big( \mfrac{m_1+m_2-1}{2e} \Big)^{\frac{m_1+m_2-1}{2}}
        \;=\; 
        2^{\frac{1}{2}}
        \Big( \mfrac{m_1+m_2-1}{e} \Big)^{\frac{m_1+m_2-1}{2}}
        \\
        \;\geq&\; 
        c_1^{m_1+m_2} (m_1 + m_2 -1)^{\frac{m_1+m_2}{2}}
        \;\geq\; 
    c_1^{m_1+m_2} \Big( \mfrac{m_1 + m_2 -1}{m_1+m_2} \Big)^{\frac{m_1+m_2}{2}}    (m_1 + m_2)^{\frac{m_1+m_2}{2}}
        \\
        \;\geq&\;
        c^{m_1+m_2}  (m_1 + m_2)^{\frac{m_1+m_2}{2}}  
    \end{align*}
    for some sufficiently small absolute constants $c_1, c > 0$. \qed

\subsection{Proofs for properties of the heavy-tailed distribution in Section \ref{sec:construct:lower:bound}} \label{appendix:univariate:heavy} 

\begin{proof}[Proof of \cref{lem:heavy:uni:moments}] The first two moments of $V_1$ can be obtained directly from construction. To control the $\omega$-th absolute moment for $\omega \geq 1$, first note that 
        \begin{align*}
            \mean | U_1 |^\omega 
            \;=&\;  
            2p |x_0|^\omega + p |2x_0|^\omega 
            \;=\;
            (2 + 2^\omega) p x_0^\omega 
            \;=\; 
            \mfrac{(2 + 2^\omega) \sigma^\omega}{6^{\omega/2} p^{\omega/2-1}}
            \;=\; 
            \mfrac{(2 + 2^\omega)}{6^{\omega/2} \sigma^{2(\omega-\nu)/(\nu-2)}} \;. \tagaligneq \label{eq:U:moment}
        \end{align*}
    Moreover, by Jensen's inequality, we have $|a+b|^\omega \leq 2^{\omega-1}(|a|^\omega + |b|^\omega)$ for $a, b \in \R$. Combining this with the upper bound on $\mean |\sigma^{-1} Z_1|^\omega$ from \cref{lem:bound:normal:moment}, we get that
    \begin{align*}
            \mean |V_1|^\omega
            \;\leq&\; 
            \mfrac{2^{-\omega/2} \, 2^{\omega-1}(2+2^\omega)}{6^{\omega/2} \sigma^{2(\omega-\nu)/(\nu-2)}}
            +
            2^{-\omega/2} 2^{\omega-1} \sigma^\omega \mean |\sigma^{-1} Z|^\omega
            \;\leq\;
            c_1^\omega \sigma^{- \frac{2(\omega-\nu)}{\nu -2} }
            +
            c_2^\omega \sigma^\omega \omega^{\omega / 2}
    \end{align*} 
    for some absolute constants $c_1, c_2 > 0$ as desired.    
\end{proof}

\cref{lem:heavy:uni:normal:approx:with:remainder} approximates an empirical average of $V_i$'s by a Gaussian $Z'_1$, and gives a finer control by considering an additional remainder term. The key idea is to perform a fourth-order Taylor expansion in the characteristic functions of both $n^{-1/2} \sum_{i=1}^n V_i$ and $Z'_1$, before turning back to the distribution functions. Note that the smoothing by $Z_1$ in the construction of $V_1$ makes the distribution of $V_1$ continuous, which enables this approach. 

\begin{proof}[Proof of \cref{lem:heavy:uni:normal:approx:with:remainder}] Write $S_n = n^{-1/2} \sum_{i=1}^n V_i$, $F_n(x) = \P(S_n < x)$ and $F_Z(x) = \P( Z'_1 < x )$. By the inversion formula for continuous random variables (Theorem 2.3.1., \cite{cuppens1975decomposition}), 
\begin{align*}
    F_n(x) \;=&\; \mfrac{1}{2\pi} \mint^\infty_{-\infty} \mfrac{ - e^{-i tx}}{it} \chi_n(t) \, dt
    &\text{ and }&&
    F_Z(x) \;=&\; \mfrac{1}{2\pi} \mint^\infty_{-\infty} \mfrac{ - e^{-i tx}}{it} \chi_Z(t) \, dt\; \tagaligneq \label{eqn:lower:bound:inversion:formula} 
\end{align*}  
where $\chi_n$ and $\chi_Z$ are the characteristic functions of $S_n$ and $Z$. We first compare the characteristic functions. Denoting $\theta \coloneqq 2^{-1/2}$ for simplicity, the characteristic function of $V$ satisfies 
    \begin{align*}
        &\chi_V(t) 
        \;=\; 
        \mean[e^{it ( \theta U + \theta \sigma Z)}] 
        \;=\; 
        e^{ - \theta^2 \sigma^2 t^2 / 2}\big( 2 p e^{-i \theta  x_0 t} + (1-3p) + p e^{2 i  \theta  x_0 t} \big) 
        \\
        &=\;
        e^{ - \theta^2 \sigma^2 t^2 / 2} \big(1-3p+2 p \cos( \theta  x_0 t) + p \cos(2  \theta  x_0 t)
        + i (- 2p \sin( \theta x_0 t) + p \sin(2  \theta x_0 t))
        \big) 
        \\
        &\overset{(a)}{=}\;
        e^{ - \frac{\theta^2 \sigma^2 t^2}{2}} 
        \Big(1-3p +2 p \big(1 
        - 
        \mfrac{\theta^2 x_0^2 t^2 }{2} 
        + 
        \mfrac{ \theta^4 x_0^4 t^4 \cos(t \theta x'_1)}{24}
        \big)
        \\
        &\hspace{6em} 
        + p \big( 1 - 2 t^2 \theta^2 x_0^2 
        + 
        \mfrac{2  \theta^4 x_0^4 t^4 \cos(2 t \theta x'_1) }{3} 
        \big)
        \\
        &\hspace{6em}
        - i (2p) 
        \big( 
           \theta x_0 t
           - 
           \mfrac{ \theta^3 x_0^3 t^3}{6} 
           + 
           \mfrac{ \theta^4 x_0^4 t^4\sin( t \theta x'_2 ) }{24}  
        \big)
        \\
        &\hspace{6em}
        + i p 
        \big( 
           2 \theta x_0 t
           - 
           \mfrac{4 \theta^3 x_0^3 t^3  }{3} 
           + 
           \mfrac{2 \theta^4 x_0^4 t^4 \sin( 2 t \theta x'_2 ) }{3} 
        \big)
        \Big) 
        \\
        &=\;
        e^{ - \theta^2 \sigma^2 t^2 / 2} 
        \Big(1 
        - 
        3p \theta^2 x_0^2 t^2 
        -
        ip \theta^3 x_0^3 t^3 
        +
        \mfrac{p  \theta^4 x_0^4 t^4}{12} \cos(t \theta x'_1)
        + 
        \mfrac{2p  \theta^4 x_0^4 t^4}{3} \cos(2 t \theta x'_1) 
        \\
        &\qquad\qquad\;\;
           - 
           \mfrac{ip \theta^4 x_0^4 t^4 }{12} \sin( t \theta x'_2 ) 
           + 
           \mfrac{2  i p \theta^4 x_0^4  t^4 }{3} \sin( 2 t \theta x'_2 ) 
        \Big) 
        \;.
    \end{align*} 
    In $(a)$, we have performed Taylor expansions on the real and imaginary parts with some $x'_1, x'_2 \in [0, x_0]$. Since $x_0 = \sigma / \sqrt{6p}$ and $p=\sigma^{2\nu/(\nu-2)}$, we have that 
    \begin{align*}
        p x_0^2 \;=\; \mfrac{\sigma^2}{6}\;, 
        \qquad 
        p x_0^3 \;=\; \mfrac{1}{6^{3/2}\sigma^{(6 - 2 \nu)/(\nu-2)}}\;,
        \qquad
        p x_0^4 \;=\; \mfrac{1}{6^2 \sigma^{(8- 2\nu )/(\nu-2)} }\;.
    \end{align*} 
    Then the characteristic function of $S_n$ can be expressed as   
    \begin{align*}
        \chi_n&(t) \;=\; (\chi_V( n^{-1/2}\,t))^n 
        \\
        \;=&\;
        e^{ - \theta^2 \sigma^2 t^2 / 2} 
        \Big(1 
        - 
        \mfrac{3p \theta^2 x_0^2  t^2}{n}
        -
        \mfrac{ip\theta^3 x_0^3 t^3 }{n^{3/2}}
        +
        \mfrac{p  \theta^4 x_0^4 t^4}{12 n^2} \cos(t \theta x'_1)
        + 
        \mfrac{2p \theta^4 x_0^4 t^4}{3 n^2 } \cos(2 t \theta x'_1) 
        \\
        &\qquad\qquad\;\;
           - 
           \mfrac{ip \theta^4 x_0^4 t^4 }{12 n^2 } \sin( t \theta x'_2 ) 
           + 
           \mfrac{2  i p \theta^4 x_0^4 t^4}{3 n^2} \sin( 2 t \theta x'_2 ) 
        \Big)^n
        \\
        \;=&\;
        e^{ - \theta^2 \sigma^2 t^2 / 2} 
        \Big(1 
        - 
        \mfrac{ \theta^2 \sigma^2 t^2 }{2 n}
        -
        \mfrac{i \theta^3 t^3 }{6^{3/2} n^{3/2} \sigma^{(6 - 2 \nu)/(\nu-2)}}
        +
        \mfrac{\theta^4 t^4  \cos(t \theta x'_1) }{432 n^2 \sigma^{(8-2\nu)/(\nu-2)}} 
        \\
        &\qquad\qquad\;\;
            + 
            \mfrac{\theta^4 t^4 \cos(2 t \theta x'_1) }{54 n^2 \sigma^{(8-2\nu)/(\nu-2)}}
           - 
           \mfrac{i \theta^4 t^4 \sin( t \theta x'_2 )  }{432 n^2 \sigma^{(8-2\nu)/(\nu-2)}} 
           + 
           \mfrac{ i \theta^4 t^4 \sin( 2 t \theta x'_2 ) }{54 n^2 \sigma^{(8-2\nu)/(\nu-2)}} 
        \Big)^n
        \\
        \;=&\;
        \exp\Big( 
            - \mfrac{ \theta^2 \sigma^2 t^2}{2}
            + 
            n \log 
            \Big( 
                1 - \mfrac{ \theta^2 \sigma^2 t^2 }{2n}  - i \mfrac{Q_n(t)}{n} + R_1(t) + i R_2(t) 
            \Big) 
            \,
            \Big)
        \;,
    \end{align*} 
    where we have defined 
    \begin{align*}
        Q_n(t)
        \;\coloneqq&\; 
        \mfrac{\theta^3 t^3}{6^{3/2} n^{1/2} \sigma^{(6 - 2 \nu)/(\nu-2)}}\;,
        \\
        R_1(t)
        \;\coloneqq&\;
        \mfrac{ \theta^4 t^4\cos(t \theta x'_1) }{432 n^2 \sigma^{(8-2\nu)/(\nu-2)}} 
        + 
        \mfrac{ \theta^4 t^4 \cos(2 t \theta x'_1) }{54 n^2 \sigma^{(8-2\nu)/(\nu-2)}}
        \;,
        \\
        R_2(t)
        \;\coloneqq&\;
         - 
        \mfrac{\theta^4 t^4 \sin( t \theta x'_2 )  }{432 n^2 \sigma^{(8-2\nu)/(\nu-2)}} 
        + 
        \mfrac{ \theta^4 t^4 \sin( 2 t \theta x'_2 ) }{54 n^2 \sigma^{(8-2\nu)/(\nu-2)}} 
        \;.
    \end{align*}
    Now define 
    \begin{align*}
        R_n(t)
        \;\coloneqq\; 
        n \log\Big( 1 - \mfrac{ \theta^2 \sigma^2 t^2 }{2n} - i \mfrac{Q_n(t)}{n} + R_1(t) + i R_2(t) \Big) + \mfrac{\theta^2 \sigma^2t^2}{2} + i Q_n(t)
        \;.
    \end{align*}
    By multiplying and dividing $e^{-\sigma^2 t^2/4}$ and recalling that $\theta=2^{-1/2}$, we get that 
    \begin{align*}
        \chi_n(t)
        \;=&\; 
        e^{-\frac{\sigma^2t^2}{2} - i Q_n(t)} e^{R_n(t)}
        \;=\;
        \chi_Z(t) \, e^{- i Q_n(t)} e^{ R_n(t)}  \;.
    \end{align*}
    Now define $q(t) \coloneqq - i Q_n(t) \chi_Z(t)$.  Then 
    \begin{align*}
        | \chi_n(t) - \chi_Z(t)  - q(t)| 
        \;=&\; 
        \Big| \chi_Z(t) \Big(  e^{- i Q_n(t)} e^{ R_n(t)}  - 1 + iQ_n(t) \Big) \Big| 
        \\
        \;\leq&\; 
        |\chi_Z(t) e^{-iQ_n(t)} ( e^{R_n(t)} - 1 )|                                       
        +
        |\chi_Z(t) ( e^{-iQ_n(t)} - 1 + i Q_n(t) ) |
        \\    
        \;\leq&\;
        e^{-\sigma^2t^2/2} \big( 
        | e^{R_n(t)} - 1|
        +
        | e^{-iQ_n(t)} - 1 + i Q_n(t) | 
        \big)
        \\
        \;\leq&\;
        e^{-\sigma^2t^2/2} \big( 
        | e^{R_n(t)} - 1|
        +
        | Q_n(t) |^2
        \big)
        \;. \tagaligneq \label{eq:lower:bound:univariate:char:diff:intermediate}
    \end{align*} 
    We seek to control $R_n(t)$ by a Taylor expansion of the complex logarithm around $1$, which is only permitted outside the branch cut $\R^- \cup \{0\}$. Set $T_n \coloneqq n^{1/2} \sigma^{(4-\nu)/(2\nu-4)} $. For $|t| \leq T_n$, 
    \begin{align*}
        \Big| - \mfrac{ \theta^2 \sigma^2 t^2 }{2n}  - \mfrac{i Q_n(t)}{n} + & R_1(t) + i R_2(t) \Big| 
        \\
        \;\leq&\;
        \mfrac{\theta^2 \sigma^2 t^2 }{2n} 
        +
        \mfrac{\theta^3 t^3 }{6^{3/2} n^{3/2} \sigma^{(6 - 2 \nu)/(\nu-2)}} 
        +
        \mfrac{\theta^4 t^4 }{24 n^2 \sigma^{(8-2\nu)/(\nu-2)}} 
        \\ 
        \;\leq&\; 
        \mfrac{1}{2}
        +
        \mfrac{1 }{6^{3/2}} 
        + 
        \mfrac{1 }{24} 
        \;\leq\;
        1
        \;.
    \end{align*} 
    In this case, the quantity in the complex logarithm in $R_n(t)$ is outside the branch cut, so by a Taylor expansion,
    \begin{align*}
        R_n(t)
        \;=&\;
        n R_1(t) + i nR_2(t) + n R_3(t)
    \end{align*} 
    where $R_3(t)$ is a remainder term that satisfies, for $|t| \leq T_n$,  
    \begin{align*}
        | R_3(t)| 
        \;\leq&\;\;
        \Big| - \mfrac{ \theta^2 \sigma^2 t^2 }{2n} - i \mfrac{Q_n(t)}{n} + R_1(t) + i R_2(t)  \Big|^2
        \\
        \;\overset{(a)}{\leq}&\;\;
        3 \Big(  
        \mfrac{t^4 \theta^4 \sigma^4}{4n^2} 
        +
        \mfrac{t^6 \theta^6 }{6^3 n^3 \sigma^{(12 - 4 \nu)/(\nu-2)}} 
        +
        \mfrac{t^8 \theta^8 }{24^2 n^4 \sigma^{(16-4\nu)/(\nu-2)}} 
        \Big)
        \\
        \;\overset{(b)}{\leq}&\;\;
        3 \theta^4 t^4 \Big(  \mfrac{ \sigma^4  }{4n^2} 
        +
         \mfrac{1 }{6^3 n^2 \sigma^{(8-3\nu)/(\nu-2)}}
        +  
        \mfrac{1 }{24^2 n^2 \sigma^{(8-2\nu)/(\nu-2)}}
        \Big)
        \\
         \;\overset{(c)}{\leq}&\;\; 
        \mfrac{t^4}{5 n^2 \sigma^{(8-2\nu)/(\nu-2)}}
        \;,
    \end{align*} 
    In $(a)$, we have noted that $(A+B+C)^2 \leq 3(A^2+B^2+C^2)$; in $(b)$, we have used $|t| \leq T_n$ and $\theta \leq 1$; in $(c)$, we have compared the powers of $\sigma \in (0,1]$ by noting that $\frac{2\nu - 8}{\nu-2} \leq \frac{3\nu-8}{\nu-2}$ and $\frac{2\nu - 8}{\nu-2} \leq 4$, and combined the constants while noting that $\theta^4 = 1/4$. By a further Taylor expansion of the complex exponential, we obtain that for $|t| \leq T_n$, 
    \begin{align*}
        \big| e^{R_n(t)} - 1 \big| \;\leq\; | R_n(t) |
        \;\leq&\; 
        n |R_1(t)| + n | R_2(t) | + n |R_3(t)|
        \\
        \;\leq&\; 
        \mfrac{\theta^4 t^4}{24 n \sigma^{(8-2\nu)/(\nu-2)}} 
         + 
         \mfrac{ t^4}{5 n \sigma^{(8-2\nu)/(\nu-2)}}
         \;\leq\;
         \mfrac{0.22 t^4}{n \sigma^{(8-2\nu)/(\nu-2)}}
        \;.
    \end{align*} 
    By plugging this and $Q_n(t) = \frac{ \theta^3 t^3 }{6^{3/2} n^{1/2} \sigma^{(6 - 2 \nu)/(\nu-2)}} $ into \eqref{eq:lower:bound:univariate:char:diff:intermediate}, we get that for $|t| \leq T_n$,
    \begin{align*}
        | \chi_n(t) - \chi_Z(t)  - q(t)| 
        \;\leq&\;
         \mfrac{1}{n} e^{-\sigma^2t^2/2} \Big( 
            \mfrac{0.22 t^4}{ \sigma^{(8-2\nu)/(\nu-2)}} 
            +
            \mfrac{0.03 t^6 }{\sigma^{(12 - 4 \nu)/(\nu-2)}} 
         \Big) 
         \;. \tagaligneq \label{eq:lower:bound:univariate:intermediate}
    \end{align*}  
    Now note that $F_q$ is a function analogous to \eqref{eqn:lower:bound:inversion:formula} for $q$:  
    \begin{align*}
        \mfrac{1}{2\pi} \mint^\infty_{-\infty} \mfrac{ - e^{-i tx}}{it } q(t) \, dt
         \;=&\; 
         \mfrac{1}{2\pi} \mint^\infty_{-\infty} \mfrac{ e^{-i tx}}{t} Q_n(t) \chi_Z(t) \, dt
         \\
         \;=&\;
         \mfrac{1}{2\pi} \mint^\infty_{-\infty} \mfrac{ e^{-i tx}}{t} 
         \mfrac{\theta^3 t^3}{6^{3/2} n^{1/2} \sigma^{(6 - 2 \nu)/(\nu-2)}}
         e^{-\sigma^2 t^2/2} \, dt
         \\
         \;=&\;
         \mfrac{\theta^3}{(2^{5/2}3^{3/2}\pi) \, n^{1/2} \sigma^{(6-2\nu)/(\nu-2)}} \mint^{\infty}_{-\infty} t^2 e^{-\sigma^2 t^2/2 - itx} dt
         \\
         \;=&\;
         \mfrac{\theta^3}{(2^2 3^{3/2}\pi^{1/2}) \, n^{1/2} \sigma^{\nu/(\nu-2)}} \Big(1 - \mfrac{x^2}{\sigma^2} \Big) e^{-x^2/(2\sigma^2)} \;=\; F_q(x)\;,
    \end{align*} 
    where we have set the constant $A$ in the definition of $F_q$ as
    $
        A  \coloneqq \frac{\theta^3}{2^2 3^{3/2} \pi^{1/2}}
    $ where $\theta=2^{-1/2}$.  Therefore by \eqref{eqn:lower:bound:inversion:formula}, a bound on the distribution functions can be given as 
    \begin{align*}
        | F_n(x&) - F_Z(x) - F_q(x) | \;\leq\;  \mfrac{1}{2\pi} \mint^{\infty}_{-\infty} \mfrac{| \chi_n(t) - \chi_Z(t) - q(t) |}{|t|} \, dt
        \\
        \;=&\;
        \mfrac{1}{2\pi} \mint_{|t| \leq T_n} \mfrac{| \chi_n(t) - \chi_Z(t) - q(t) |}{|t|} \, dt
        +
        \mfrac{1}{2\pi} \mint_{|t| > T_n} \mfrac{| \chi_n(t) - \chi_Z(t) - q(t) |}{|t|} \, dt
        \;\eqqcolon\; I_1 + I_2\;.
    \end{align*} 
    \eqref{eq:lower:bound:univariate:intermediate} allows the first integral to be controlled as
    \begin{align*}
        I_1 \;\leq&\; 
        \mfrac{1}{2 \pi n} 
        \mint_{|t| \leq T_n}
        e^{-\sigma^2t^2/2} 
        \Big( 
            \mfrac{0.22 |t|^3}{ \sigma^{(8-2\nu)/(\nu-2)}} 
            +
            \mfrac{0.03 |t|^5 }{\sigma^{(12 - 4 \nu)/(\nu-2)}} 
         \Big) 
        \, dt
        \\
        \;=&\; 
        \mfrac{1}{2\pi n} 
        \Big(
        \mfrac{0.44\big( 2 - (\sigma^2 T_n^2 + 2) e^{-\sigma^2 T_n^2/2}\big)}{\sigma^{2\nu/(\nu-2)}}
        +
        \mfrac{0.06\big(8 - (\sigma^4 T_n^4 + 4 \sigma^2 T_n^2 + 8) e^{-\sigma^2 T_n^2/2}  \big)}{\sigma^{2\nu/(\nu-2)}}
        \Big)
        \\
        \;\leq&\; \mfrac{0.22 }{n \sigma^{2\nu/(\nu-2)}}
        \;.
    \end{align*} 
    To deal with the case $|t| > T_n$, we first let $\{(U_i, V_i)\}_{i=1}^n$ be i.i.d.~copies of $(U,V)$ while applying independence and Jensen's inequality to obtain 
    \begin{align*}
        |\chi_n(t)| \;=\; \big| \mean\big[ e^{i \theta^2 t n^{-1/2} \sum_{i=1}^n U_i } \big] \mean\big[ e^{i \theta^2 t n^{-1/2} \sum_{i=1}^n V_i } \big] \big|
        \;\leq\;
        \big| \mean\big[ e^{i \theta^2 t n^{-1/2} \sum_{i=1}^n V_i } \big] \big| 
        \;=\;
        e^{- \frac{\theta^2 \sigma^2 t^2}{2}}\;.
    \end{align*} 
    By noting that $\chi_Z(t) = e^{-\sigma^2t^2/2}$ and $|q(t)| = |Q_n(t)| e^{-\sigma^2t^2/2} \leq \frac{ \theta^3 t^3 }{6^{3/2} n^{1/2} \sigma^{(6 - 2 \nu)/(\nu-2)}} e^{-\sigma^2t^2/2}$, we can bound $I_2$ via a triangle inequality: 
    \begin{align*}
        I_2 \;\leq&\; 
        \mfrac{1}{2\pi} \mint_{|t| > T_n} \mfrac{|\chi_n(t)|}{|t|} + \mfrac{|\chi_Z(t)|}{|t|} + \mfrac{|q(t)|}{|t|} dt 
        \\
        \;\leq&\;
        \mfrac{1}{2\pi} \mint_{|t| > T_n} 
        \mfrac{e^{- \theta^2 \sigma^ 2 t^2/2}}{|t|}
        +    
        \mfrac{e^{- \sigma^ 2 t^2/2}}{|t|}
        +
        \mfrac{ \theta^3 t^2 e^{- \sigma^ 2 t^2/2}}{6^{3/2} n^{1/2} \sigma^{(6-2\nu)/(\nu-2)}}
        dt
        \\
        \;\leq&\;
        \mfrac{1}{\pi \, T_n^2} \Big( \mint_{t > T_n} 
        t e^{- \theta^2 \sigma^ 2 t^2/2} dt\Big)
        +
        \mfrac{1}{\pi \, T_n^2} \Big( \mint_{t > T_n} 
        t e^{-  \sigma^ 2 t^2/2} dt\Big)
        \\
        &\;
        +
        \mfrac{2 }{ \pi (12)^{3/2} n^{1/2} \sigma^{(6-2\nu)/(\nu-2)} T_n}
        \Big(\mint_0^\infty t^3 e^{-\sigma^2t^2/2} dt\Big)
        \\
        \;=&\;
        \mfrac{2 e^{- \sigma^2 T_n^2/4}}{\pi \,T_n^2 \sigma^2}
        +
        \mfrac{e^{- \sigma^2 T_n^2/2}}{\pi \,T_n^2 \sigma^2}
        +
        \mfrac{4 }{\pi (12)^{3/2} n^{1/2} \sigma^{(6-2\nu)/(\nu-2) + 4} \, T_n}
        \\
        \;\leq&\; 
        \Big( \mfrac{3 }{\pi } + \mfrac{4}{\pi \, (12)^{3/2}}\Big) 
        \max\Big\{
        \mfrac{1 }{n \sigma^{\nu/(\nu-2)}} 
        \,,\,
        \mfrac{1}{n \sigma^{ 3\nu/(2\nu-4) }}
        \Big\} 
        \,\;\leq\,\; \mfrac{1}{n \sigma^{2\nu/(\nu-2)}} \;.
    \end{align*} 
    In the last line, we have recalled that $T_n = n^{1/2} \sigma^{(4-\nu)/(2\nu-4)}$ and noted that $\sigma \in (0,1]$. Combining the bounds for $I_1$ and $I_2$, we get that
    \begin{align*}
        \Big| \P\Big( \mfrac{1}{\sqrt{n}} \msum_{i=1}^n V_i < x\Big) - \P( Z < x) - F_q(x) \Big|
        \;=\; | F_n(x) - F_Z(x) - F_q(x) | 
        \;\leq\; 
        \mfrac{2}{n \sigma^{2\nu/(\nu-2)}}
    \end{align*} 
     as desired.
\end{proof}

\cref{lem:heavy:uni:normal:approx:non:uniform} provides a normal approximation error with tighter $x$-dependence than a typical non-uniform Berry-Ess\'een bound. The proof is tedious, but the key ideas are the following: 
\begin{proplist} 
    \item In the interval $|x| \leq \frac{2\sigma}{\sqrt{3} \, M}$, i.e.~$x$ is within some multiples of the standard deviation of the Gaussian from the mean, Berry-Ess\'een bound is sufficiently tight so we can apply it directly. The interesting case is when $x > \frac{2\sigma}{\sqrt{3} \, M}$ ($x < \frac{2\sigma}{\sqrt{3} \, M}$ can be handled by symmetry).
    \item To give a tighter control in the region $x > \frac{2\sigma}{\sqrt{3} \, M}$, the first part of the proof expands a telescoping sum of $n^{-1/2} \sum_{i=1}^n V_i$ and successively replacing the non-Gaussian part of each data point, i.e.~$U_i$ in $V_i=2^{-1/2} U_i + 2^{-1/2}Z_i$, by a Gaussian. 
    \item Since the support of $U_i$ has size $\Theta(\sigma^{-\frac{2}{\nu-2}})$, which is made to grow slower than $n^{1/2}$ by assumption of the lemma, the support of $n^{-1/2} U_i$ shrinks. Therefore on most part of the real line, the Gaussian approximation error is determined solely by the mass of the Gaussian counterpart of $n^{-1/2} U_i$. By carefully choosing a region moderately far from the origin, we can make use of the exponential tail of $n^{-1/2} U_i$. This is the argument leading up to \eqref{eq:heavy:uni:non:uniform:first}, and gives rise to the first exponential term in \cref{lem:heavy:uni:normal:approx:non:uniform}.
    \item At places where we cannot exploit the tail of $n^{-1/2} U_i$, we make use of the approximate Gaussianity of the remaining sum $n^{-1/2} \sum_{i'=1}^n V_{i'} - n^{-1/2} U_i$, which is the proof starting from \eqref{eq:heavy:uni:non:uniform:second}. The idea is that when $n^{-1/2} U_i$ is close to the origin, the remaining sum $n^{-1/2} \sum_{i'=1}^n V_{i'} - n^{-1/2} U_i$ is allowed to be far from the origin, and we can obtain a sharper tail. This is achieved by symmetrising the sum $n^{-1/2} \sum_{i'=1}^n V_{i'} - n^{-1/2} U_i$ in \eqref{eq:heavy:uni:non:uniform:symmetrization} and then exploiting the improved $n^{-3/2}$ rate of normal approximation of a symmetric sum via Lindeberg's technique in \eqref{eq:heavy:uni:non:uniform:lindeberg}. This yields the second $n^{-3/2} x^{-4}$ term in \cref{lem:heavy:uni:normal:approx:non:uniform}; as suggested by the $x^{-4}$ term, this error is only well-controlled at regions far from the origin. 
\end{proplist}
The proof below is an elaborate version of the arguments in Example 9.1.3 of Senatov \cite{senatov1998normal} and with their final step replaced by Lindeberg's technique.

\begin{proof}[Proof of \cref{lem:heavy:uni:normal:approx:non:uniform}] Recall that $V \coloneqq 2^{-1/2} U + 2^{-1/2} Z$. Let $\{U_i, Z_i\}_{i=1}^n$ be i.i.d.~copies of $(U, Z)$ and let $\{Z_i'\}_{i=1}^n$ be an independent copy of $\{Z_i\}_{i=1}^n$. Write 
    \begin{align*}
        W_i \;\coloneqq\; \mfrac{1}{\sqrt{2 n}} \big( \msum_{j=1}^n Z_j + \msum_{j=1}^{i-1}  U_i  + \msum_{j=i+1}^{n} Z'_i  \big)
    \end{align*}
    and denote its probability measure by $\mu_{W_i}$. By expressing the quantity to be bounded as a telescoping sum and noting the independence of $(W_i, U_i, V_i)$ across different $i$'s, we have
     \begin{align*}
        |F_n&(x) - F_Z(x)| \;=\;
        \big| \P\big( W_n + (2n)^{-1/2} U_n < x\big) - \P\big( W_1 + (2n)^{-1/2}  Z_1 < x \big)  \big|
        \\
        \;\leq&\; 
        \msum_{i=1}^n \big| \P\big( W_i + (2n)^{-1/2} U_i < x\big) - \P\big( W_i +  (2n)^{-1/2} Z_i < x \big)  \big|
        \\
        \;\leq&\;
        \msum_{i=1}^n
        \mint_{-\infty}^\infty
        \, \big| \P\big( (2n)^{-1/2} U_i < x - w\big) - \P\big( (2n)^{-1/2} Z_i < x - w\big) \big| d\mu_{W_i}(w)
        \\
        \;=&\;
        \msum_{i=1}^n
        \mint_{-\infty}^\infty
        \, \big| \P\big( (2n)^{-1/2} U_i < x - w \big) - \P\big(  (2n)^{-1/2} Z_i < x - w \big) \big| \, d\mu_{W_i}(w)
        \\
        \;\coloneqq&\; 
        \msum_{i=1}^n \mint_{-\infty}^\infty J_i(w) \, d\mu_{W_i}(w)
        \;=\; 
        \msum_{i=1}^n
        \Big( \mint_{-\infty}^{x/2} J_i(w) \, d\mu_{W_i}(w)
        +
        \mint_{x/2}^\infty J_i(w) \,  d\mu_{W_i}(w) \Big)
        \;. 
     \end{align*}
     
     \vspace{.5em}

    We first focus on the case $x > \frac{2 \sigma}{\sqrt{3} M }$, where $M$ is the constant in the assumption $\sigma^{\nu/(\nu-2)} \geq M n^{-1/2}$. Since $(2n)^{-1/2} U_i$ is bounded in norm by $(2n)^{-1/2} (2x_0) = \frac{\sigma}{\sqrt{3 n} \, \sigma^{\nu/(\nu-2)}} \leq \frac{\sigma}{\sqrt{3} M} < \frac{x}{2}$ almost surely, the probability $\P\big(n^{-1/2} U_i \geq x - w \big)$ for $w < x/2$ is zero. Therefore 
    \begin{align*}
        \mint_{-\infty}^{x/2} J_i(w) &d\mu_{W_i}(w)
        \\
        \;=&\;
        \mint_{-\infty}^{x/2}
        \big| \P\big( (2n)^{-1/2} U_i \geq x - w\big) - \P\big((2n)^{-1/2} Z_i \geq x - w\big) \big| d\mu_{W_i}(w)
        \\
        \;=&\;
        \mint_{-\infty}^{x/2}
        \P\big( (2n)^{-1/2} Z_i \geq x - w \big) d\mu_{W_i}(w)
        \;\;\;\leq\;\;\;
        \P\Big( (2n)^{-1/2} Z_i \geq \mfrac{x}{2} \Big) 
        \\
        \;\overset{(a)}{\leq}&\;
        \mfrac{ \sigma}{\sqrt{\pi n} \, x} e^{- \frac{n x^2}{4 \sigma^2 }} 
        \;\overset{(b)}{\leq}\;
        \mfrac{ \sigma}{\sqrt{\pi n} \, x}
        \mfrac{8 \sigma^2}{nx^2}
        e^{- \frac{n x^2}{8 \sigma^2}} 
        \;\overset{(c)}{\leq}\;
        \mfrac{3^{3/2} M^3 }{ \sqrt{\pi} \, n^{3/2}}
        e^{- \frac{n x^2}{8 \sigma^2}} 
        \;\leq\; 
        \mfrac{3M^3 }{n^{3/2}} e^{- \frac{n x^2}{ 8 \sigma^2}} 
        \;. \tagaligneq \label{eq:heavy:uni:non:uniform:first}
    \end{align*}
    In $(a)$, we used a standard bound for the complementary error function. In $(b)$, we have noted that $y \coloneqq \frac{n x^2}{8 \sigma^2} > \frac{ n}{6 M^2} \geq 1$ since $n \geq 6M^2$ and applied the bound  $e^{-y} \leq \frac{1}{y} $ for $y \geq 1 $. In $(c)$, we have used $x > \frac{2 \sigma}{\sqrt{3} M}$ again. Now for the other integral, by the standard Berry-Ess\'een bound,
    \begin{align*}
        \mint_{x/2}^\infty J_i(w) d\mu_{W_i}(w)
        \;\leq&\;
        \mfrac{\mean |U_i|^3}{n^{3/2} \sigma^3} \mint_{x/2}^\infty d\mu_{W_i}(w)
        \;=\; 
        \mfrac{10}{6^{3/2} n^{3/2} \sigma^{\nu/(\nu-2)}}
        \P( W_i \geq x/2 )\;, \tagaligneq \label{eq:heavy:uni:non:uniform:second}
    \end{align*}
    where we have used the moment formula from \eqref{eq:U:moment} to get that 
    \begin{align*}
        \mean |U_i |^3 \;=\; \frac{10}{6^{3/2} \sigma^{(6-2\nu)/(\nu-2)}}\;.
    \end{align*}

    \vspace{.5em}

    To handle the probability term, we first consider rewriting the sum $ W_i = \xi_i + S_i$ with a Gaussian component $\xi_i$ and an independent non-Gaussian component $S_i$ given by
    \begin{align*}
        \xi_i \;\coloneqq&\; \mfrac{1}{\sqrt{2n}}
        \big( \msum_{j=1}^n Z_i +
         \msum_{j=i+1}^n  Z'_i \big) 
        \sim
         \cN\Big(0,  \mfrac{(2n-i)}{2n} \sigma^2\Big)
        &\text{  and }&&
        S_i \;\coloneqq&\; \mfrac{1}{\sqrt{2n}} \msum_{j=1}^{i-1} U_i\;.
    \end{align*}
    We also write $W_i' = \xi_i'+S_i'$, where $\xi'_i$ and $S'_i$ are i.i.d.~copies of $\xi_i$ and $S_i$, and denote the symmetrisation $\bar W_i = W_i - W'_i$. Let $\mu_{W_i}$ be the measure associated with $W_i$. Then
    \begin{align*}
        \P( \bar W_i \geq &\,x/2) 
        \;=\; 
        \mint_{-\infty}^\infty \P( - W'_i \geq x/2 - t) \,d\mu_{W_i}(t)
        \;\geq\; 
        \mint_{x/2}^\infty \P( - W'_i \geq x/2 - t) \,d\mu_{W_i}(t)
        \\
        \;\geq&\; 
        \big( \minf_{t \in [\frac{x}{2}, \infty)} \, \P( - W'_i \geq x/2 - t)\big)
        \P( W_i \geq x/2 )
        \;=\; 
        \P( - W'_i \geq 0) \,
        \P( W_i \geq x/2 )
        \;.
    \end{align*}
    This rearranges to give 
    \begin{align*}
        \P( W_i \geq x/2 ) \;\leq\;  \P( - W'_i \geq 0 )^{-1} \, \P( \bar W_i \geq x/2) \;, \tagaligneq \label{eq:heavy:uni:non:uniform:symmetrization}
    \end{align*}
    The first probability can be lower bounded by
    \begin{align*}
       \P( - W'_i \geq 0 ) 
       \;=&\;
       \P( \xi_i + S_i \leq 0)
       \;\geq\; 
       \P( \xi_i \leq 0) \, \P( S_i \leq 0)\;.
    \end{align*}
    Note that $\P( \xi_i \leq 0) = 1/2$ since $\xi_i$ is symmetric and $\P( S_1 \leq 0) = 1$. For $i > 1$, by a Berry-Ess\'een bound and the assumption that $\sigma^{\nu/(\nu-2)} \geq M n^{-1/2}$ and $M \geq 10$,
    \begin{align*}
        \P( S_i \leq 0) 
        \;\geq\; 
        \mfrac{1}{2} - \mfrac{2^{3/2} \mean |U_i|^3}{n^{1/2} \sigma^3}
        \;=\; 
        \mfrac{1}{2} - \mfrac{10}{3^{3/2} n^{1/2} \sigma^{\nu/(\nu-2)}}
        \;\geq\;
        \mfrac{1}{2} - \mfrac{10}{3^{3/2} M} \;\geq\; \mfrac{1}{4}\;.
    \end{align*}
    This implies that $\P( - W'_i \geq 0 ) \geq \frac{1}{8}$ and therefore 
    \begin{align*}
        \P(W_i \geq x/2) \;\leq\; 8 \, \P( \bar W_i \geq x/2) \;.
    \end{align*}
    In other words, we have shown that $\P( W_i \geq x/2)$ in the bound \eqref{eq:heavy:uni:non:uniform:second} can now be controlled in terms of the symmetric variable $\bar W_i$.

    \vspace{.5em}

    The final step is to compare $\bar W_i$ to $\bar Z \sim \cN(0, \frac{\theta^2 (2 n-1)}{n} \sigma^2)$. We apply the standard Lindeberg's argument with a smooth test function $h$. By \cref{lem:smooth:approx:ind} with $\delta=x/4$ and $m=3$, there is a three-times differentiable function $h$ such that for some absolute constant $C'' >0$,
    \begin{align*}
        \ind_{\{t \geq x/2\}} \;&\leq\; h(t) \;\leq\; \ind_{\{t \geq x/4 \}}
        \;&\text{ and }&&\;
        |h'''(t) - h'''(s)| \;&\leq\; C'' x^{-4} |t-s|\;.
    \end{align*}
    Consider the symmetric variable $\bar U_i = \frac{ \theta}{\sqrt{n}}( U_i - U'_i )$, where $\{U'_i\}_{i=1}^n$ is an independent copy of $\{U_i\}_{i=1}^n$. Write $Y_k = \sum_{j=1}^{k-1}  \bar U_i + \bar \xi_i$, where $\bar \xi_i \sim \cN\big(0, \frac{\theta^2 (4n-2k)}{n} \sigma^2 \big)$ is independent of $\bar U_i$. Also denote $\zeta_i \sim \cN\big(0, \frac{2 \theta^2 }{n} \sigma^2 \big)$. Then
    \begin{align*}
        &\P\big( \bar W_i \geq x/2\big) 
        \;\leq\;
         \P(\bar Z \geq x/4) + | \mean[ h(\bar W_i) - h(\bar Z) ]|  \tagaligneq \label{eq:heavy:uni:non:uniform:lindeberg}
        \\
        &\leq
        \P(\bar Z \geq x/4) + \msum_{k=1}^i | \mean[ h(Y_k +\bar U_k) - h(Y_k + \zeta_k) ] |
        \eqqcolon
        \P(\bar Z \geq x/4) + \msum_{k=1}^i A_i.
    \end{align*}
Since $\bar Z \sim \cN(0, \frac{ \theta^2 (2n-1)}{n} \sigma^2)$, by noting $ \frac{\theta^2 (2n-1)}{n} \sigma^2 \leq \sigma^2 $ and  $x > \frac{2 \sigma}{\sqrt{3}\,M}$, we have
    \begin{align*}
        \P( \bar Z \geq x/4 ) 
        \;\leq\; 
        \mfrac{\sigma}{\sqrt{2\pi} \, (x/4) } 
        e^{-\frac{x^2}{2 \sigma^2} } 
        \;\leq\;
        \mfrac{\sqrt{6} \, M}{\sqrt{\pi} } 
        e^{-\frac{x^2}{2 \sigma^2} }  
         \;\leq\; 
        2 M
        e^{-\frac{x^2}{ 2\sigma^2}} 
        \;.
    \end{align*}
    By performing two third-order Taylor expansions around $Y_k$, each $A_i$ satisfies
    \begin{align*}
        A_i \;=\;
         \big| \mean\big[ h'(Y_k) (\bar U_k - \zeta_k) + \mfrac{1}{2} h''(Y_k) (\bar U_k^2 - \zeta_k^2) 
         + \mfrac{1}{6} h'''(Y_k + \tilde U_k) \bar U_k^3 - \mfrac{1}{6} h'''(Y_k + \tilde \zeta_k)  \zeta_k^3   
         \big] \big|
    \end{align*}
    for some $\tilde U_k \in [0, U_k]$ and $\tilde \zeta_k \in [0, \zeta_k]$ that exist almost surely. The first two terms vanish by independence of $Y_k$ and $(\bar U_k, \zeta_k)$ as well as the fact that $\bar U_k$ and $\zeta_k$ match in mean and variance. Since $\bar U_k$ and $\zeta_k$ are both symmetric, they both have zero third moments. By adding and subtracting a third moment term and applying the Lipschitz property of $h'''$, we obtain
    \begin{align*}
        A_i \;=&\; 
        \mfrac{1}{6}
        \big| 
        \mean\big[ \big(h'''(Y_k + \tilde U_k)- h'''(Y_k) \big) \bar U_k^3 
        - 
        \big( h'''(Y_k + \tilde \zeta_k) -  h'''(\zeta_k)  \big)  \zeta_k^3   \big]
        \\
        \;\leq&\;
        \mfrac{C''}{6 x^4} \big( \mean \, \bar U_k^4 +  \mean \, \zeta_k^4  \big)
        \;\overset{(a)}{\leq}\;
        \mfrac{C'}{2 n^2 x^4} \Big( 
            \mfrac{1}{ \sigma^{ (8-2\nu)/\nu -2} }
            + 
            \sigma^4 
        \Big)
        \;\leq\; 
        \mfrac{C'}{n^2 x^4 \sigma^{(8-2\nu)/(\nu-2)}}
        \;,
    \end{align*}
    where $C'$ is an absolute constant; in $(a)$, we have noted that
    $
        | U_i - U'_i|^4 
        \leq 
        ( | U_i|  +  |U'_i|)^4 
        \leq 
        8 (|U_i|^4 + |U'_i|^4)
    $
    and applied the moment bound from \eqref{eq:U:moment} in the proof of \cref{lem:heavy:uni:moments}. 
    Combining the bounds gives 
    \begin{align*}
        \P( W_i \geq x/2) 
        \;\leq\; 
        8 \P\big( \bar W_i \geq x/2\big) 
        \;\leq\; 
        2 M e^{-\frac{x^2}{ 2\sigma^2}} 
        +
        \mfrac{C' }{n x^4 \sigma^{(8-2\nu)/(\nu-2)}}\;,
    \end{align*}
    and therefore 
    \begin{align*}
        \mint_{x/2}^\infty J_i(w) d\mu_{W_i}(w)
        \;\leq&\;
        \mfrac{10 \, \P( W_i \geq x/2 )}{6^{3/2} n^{3/2} \sigma^{\nu/(\nu-2)}}
        \;\leq\; 
        \mfrac{20 M e^{-\frac{x^2}{ 2\sigma^2}}  }{6^{3/2} n^{3/2} \sigma^{\nu/(\nu-2)}}
        +
        \mfrac{10 C'  }{6^{3/2} n^{5/2} x^4 \sigma^{(8-\nu)/(\nu-2)} }
        \;.
    \end{align*}
    Combining this with the bound for $\int_{-\infty}^{x/2} J_i(w) d\mu_{W_i}(w)$ in \eqref{eq:heavy:uni:non:uniform:first}, we get that for $x > \frac{2 \sigma}{\sqrt{3} M}$,
    \begin{align*}
        |F_n(x) - F_Z(x)| \;\leq&\; \msum_{i=1}^n \Big( \mint_{-\infty}^{x/2} J_i(w) d\mu_{W_i}(w) + \mint_{x/2}^\infty J_i(w) d\mu_{W_i}(w) \Big)
        \\
        \;\leq&\; 
        \mfrac{3M^3 e^{- \frac{n x^2}{8\sigma^2}} }{n^{1/2}} 
        +
        \mfrac{20 M e^{-\frac{x^2}{ 2\sigma^2}}  }{6^{3/2} n^{1/2} \sigma^{\nu/(\nu-2)}}
        +
        \mfrac{10 C'   }{6^{3/2} n^{3/2} x^4 \sigma^{(8-\nu)/(\nu-2)} }
        \\
        \;\leq&\;
        C'_M \Big(\mfrac{1}{ n^{1/2} \sigma^{\nu/(\nu-2)}} e^{-\frac{x^2}{ 8 \sigma^2}}
        +
        \mfrac{1}{n^{3/2} x^4 \sigma^{(8-\nu)/(\nu-2)} }\Big)
        \;,
    \end{align*}
    where $C'_M > 0$ is a constant that depends only on $M$. Now for $|x| \leq \frac{2 \sigma}{\sqrt{3}\, M}$, we can use the standard Berry-Ess\'een bound directly and the moment bound from \cref{lem:heavy:uni:moments} to get that
    \begin{align*}
        |F_n(x) - F_Z(x)| \;\leq\;
         \mfrac{\mean|X|^3 }{n^{1/2} \sigma^3}
        \;\leq\; 
        \mfrac{a}{n^{1/2} \sigma^3} 
        \big( \sigma^{- \frac{6-2\nu}{\nu -2} } + 1\big)
        \;\leq\; 
        \mfrac{a \max\{1, \sigma^3\} }{n^{1/2}} \sigma^{-\frac{\nu}{\nu-2}}
    \end{align*}
    for some absolute constant $a > 0$. Since in this case, $e^{-x^2/(8 \sigma^2)} \geq e^{-1/6 M}$, we again obtain the desired bound that for some constant $C''_M > 0$ depending only on $M$,
    \begin{align*}
        |F_n(x) - F_Z(x)| 
        \;\leq&\; 
        \mfrac{C'_{M} \max\{1, \sigma^3\}  e^{-\frac{x^2}{ 8 \sigma^2}} }{n^{1/2} \sigma^{\nu/(\nu-2)}} 
        \;\leq\; 
        C''_M \Big( \mfrac{\max\{1, \sigma^3\} e^{-\frac{x^2}{ 8 \sigma^2}} }{ n^{1/2} \sigma^{\nu/(\nu-2)}} 
        +
        \mfrac{1}{n^{3/2} x^4 \sigma^{(8-\nu)/(\nu-2)} }  \Big)\;.
    \end{align*}
    The proof for $x < - \frac{2\sigma}{\sqrt{3}M}$ is exactly analogous to the proof for $x > \frac{2\sigma}{\sqrt{3} M}$.
\end{proof}

\section{Proof of Proposition \ref{prop:bootstrap:inconsistency}: Bootstrap inconsistency} \label{appendix:bootstrap}

The proof is exactly the same whether we proceed with $\nu=3$ or a general $\nu \in (2,3]$, except that in the latter case, the notation is much more complicated. For clarity of presentation, we proceed with $\nu=3$ and only comment on bounds where care is taken for a general $\nu$.

\vspace{.5em}

We first establish a generic Gaussian universality result for $u_2(Y) =  \frac{1}{n(n-1)} \sum_{i \neq j} Y_i^\top Y_j$, where $Y_1, \ldots, Y_n$ are i.i.d.~random vectors with mean $\mu_Y$ and variance $\Sigma_Y$. Since $u_2(Y)$ is a multilinear polynomial, by \Cref{thm:main}, there exists some absolute constants $C', C'' > 0$ such that 
\begin{align*}
    &\;\sup_{t \in \R} \big| \P( u_2(Y) \leq t ) - \P( u_2(Z) \leq t ) \big|
    \;\leq\;
    C' \Bigg( 
        \mfrac{ 
            \msum_{i \leq n} \big\| \frac{1}{n^2} \sum_{j \neq i} Y_j^\top (Y_i - \mu)  \big\|_{L_3}^3 
        }{
            \Var[u_2(Y)]^{3/2}
        } 
    \Bigg)^{\frac{1}{7}}
    \\
    &\;\leq\;
    C' \Bigg( 
        \mfrac{ 
            \frac{1}{n^2} \, \mean \big[ \big|  \mu^\top (Y_1 - \mu)  \big|^3 \big]
            \,+\,
            \frac{1}{n^5} \, \mean \big[ \mean \big[\big|  \sum_{j > 1} (Y_j - \mu)^\top (Y_1 - \mu)  \big|^3 \,\big|\, Y_1 \big] \big]
        }{
            \Var[u_2(Y)]^{3/2}
        } 
    \Bigg)^{\frac{1}{7}}
    \\
    &\;\leq\;
    C'' \Bigg( 
        \mfrac{ 
            \frac{1}{n^2}\, \mean \big[ \big|  \mu^\top (Y_1 - \mu)  \big|^3 \big]
            \,+\,
            \frac{1}{n^{7/2}}\,  \mean \big[\big|  (Y_2 - \mu)^\top (Y_1 - \mu)  \big|^3  \big] 
        }{
            \Var[u_2(Y)]^{3/2}
        } 
    \Bigg)^{\frac{1}{7}}
    \;.
    \tagaligneq \label{eq:bootstrap:prelim}
\end{align*}
In the last line, we have noted that $Y_j$'s are i.i.d.~and used \Cref{lem:martingale:bound}. The variance term can be computed by \Cref{lem:u:variance} as
\begin{align*}
     \Var[u_2(Y)]
     \;=&\;
     \mbinom{n}{2}^{-1} 
     \Big(  
        2 (n-2) 
        \, \Var \, \mean[ Y_1^\top Y_2 | Y_1 ]
        +
        \, \Var[ Y_1^\top Y_2 ]
    \Big)
    \\
    \;=&\;
    \mfrac{4(n-2)}{n(n-1)} 
    \, \mu_Y^\top \Sigma_Y \mu_Y
    \,+\,
    \mfrac{2}{n(n-1)}
    ( 
        \Tr( (\Sigma_Y + \mu \mu^\top)  (\Sigma_Y + \mu_Y \mu_Y^\top) )
        -
        (\mu_Y^\top \mu_Y)^2
    )
    \\
    \;=&\;
    \mfrac{4 \, \| \Sigma_Y^{1/2} \mu_Y \|_{l_2}^2}{n} 
    \,+\,
    \mfrac{2 \, \Tr(\Sigma_Y^2) }{n(n-1)} 
    \;.
    \tagaligneq \label{eq:bootstrap:var:prelim} 
\end{align*}

\vspace{.5em}

Denote $X^{(b)} \coloneqq (X_1^{(b)}, \ldots, X_n^{(b)} )$ and $\tilde X^{(b)} \coloneqq (X^{(b)}_1 - \bar X, \ldots, X^{(b)}_n - \bar X)$. The rest of the proof consists of three steps: 
\begin{proplist}
    \item The first is to apply \eqref{eq:bootstrap:prelim} and \eqref{eq:bootstrap:var:prelim} with $Y_i =X_i$ to provide a Gaussian polynomial approximation of $u_2(X)$ and the desired formula for $\Var[u_2(X)]$. We then apply \Cref{thm:main} again to provide a normal approximation of the Gaussian polynomial, which proves the asymptotic normality statement. This is presented in \Cref{appendix:bootstrap:OG};
    \item The second step is to apply \eqref{eq:bootstrap:prelim} with $Y_i =X^{(b)}_i$ to obtain a conditional Gaussian polynomial approximation of $U^{(b)}_1=u_2(X^{(b)})$ given $X$. We then use a combination of the Markov inequality and central limit theorem to study the size of $X$-dependent quantities, before again establishing the conditional asymptotic normality of $u_2(X^{(b)})$ given $X$. The consistency condition is then established by comparing the limiting variance of  $u_2(X^{(b)})$ given $X$ against the limiting variance of $u_2(X)$.  This is presented in \Cref{appendix:bootstrap:Uone};
    \item The third step is analogous to the second step except that we study instead $U^{(b)}_2=u_2(\tilde X^{(b)})$. Many of the calculations in the second step can be recycled, although the obtained consistency condition is different. This is presented in \Cref{appendix:bootstrap:Utwo}.
\end{proplist}

\subsection{Take $Y_i = X_i$ to analyse $u_2(X)$} \label{appendix:bootstrap:OG}

We first apply \eqref{eq:bootstrap:prelim} and \eqref{eq:bootstrap:var:prelim} with $Y_i = X_i = \mu + \tau U_i$. Plugging in $\mu_Y = \mu$ and $\Sigma_Y = \tau^2 I_d$, we get the desired variance formula 
\begin{align*}
    \Var[u_2(X)] \;=\; \mfrac{4 \tau^2 \| \mu \|_{l_2}^2}{n} + \mfrac{2 \tau^4 d}{n(n-1)}\;.
\end{align*}
We now seek to control the third moments. Using $(\argdot)_l$ to denote the $l$-th coordinate of a vector, we have that for some absolute constants $\tilde C_1, \tilde C_2 > 0$,
\begin{align*}
     \mean [ | \mu^\top (X_1 - \mu) |^3 ]
     \;=&\;
     |\tau|^3 \, \mean [ | \mu^\top U_i |^3 ]
     \;=\;
     |\tau|^3 \,  \mean \Big[ \Big| \msum_{l \leq d} \mu_l (U_i)_l \Big|^3 \Big]
     \\
     \;\overset{(a)}{\leq}&\;
     \tilde C_1 
    |\tau|^3   d^{1/2} \, \msum_{l \leq d} \mean[ |  \mu_l (U_i)_l  |^3  ]
    \\
    \;\overset{(b)}{=}&\;
    \tilde C_1  |\tau|^3  d^{1/2} \,  \msum_{l \leq d} |\mu_l |^3 \, \times \, 
    \mean [| (U_1)_1 |^3]
    \\
    \;\overset{(c)}{\leq}&\;
    \tilde C_2  |\tau|^3  \| \mu \|_{l_2}^3
    \;.
    \tagaligneq \label{eq:bootstrap:third:moment:non:degen}
\end{align*}
In $(a)$ above, we have used the martingale difference moment bound of \Cref{lem:martingale:bound} on the sum of independent zero-mean variables; in $(b)$, we have used that $(U_i)_l$'s are i.i.d.~across $i \leq n$ and $l \leq d$; in $(c)$, we have used \Cref{assumption:bootstrap:model}(ii) to note that $\mean [| (U_1)_1 |^3]$ is bounded and that $d^{1/2} \| \mu \|_{l_3} = \Theta( \| \mu \|_{l_2})$. Note that all these arguments work with a general $\nu \in (2,3]$ instead of just $\nu=3$.  By a similar argument,  for some absolute constant $\tilde C_3 > 0$,
\begin{align*}
     \mean [ | (X_1 - \mu)^\top (X_2 - \mu) |^3 ]
    \;=&\;
    \tau^6 \, \mean[  |U_1^\top U_2|^3  ]
    \\
    \;=&\; 
    \tau^6 \, \mean\Big[ \,\Big| \msum_{l \leq d} (U_1)_l (U_2)_l \Big|^3 \Big]
    \\
    \;\leq&\;
    \tilde C_1 \, \tau^6 \, d^{3/2} \mean[ | (U_1)_1 (U_2)_1  |^3 ]
    \\
    \;\leq&\;
    \tilde C_3 \, \tau^6 \, d^{3/2} \;.
    \tagaligneq \label{eq:bootstrap:third:moment:degen}
\end{align*}
In the last two lines, we have again used the martingale difference moment bound of \Cref{lem:martingale:bound} and \Cref{assumption:bootstrap:model}(ii) that each coordinate of $U_1$ has a bounded 6th moment. Plugging these computations into \eqref{eq:bootstrap:prelim} gives that for some absolute constant $C''' > 0$,
\begin{align*}
    \msup_{t \in \R} \big| \P( u_2(X) \leq t ) - \P( u_2(Z) \leq t ) \big|
    \;\leq&\;
    C''' \Bigg( 
        \mfrac{ 
            \frac{1}{n^2}\, |\tau|^3 \| \mu \|_{l_2}^3
            \,+\,
            \frac{1}{n^{7/2}}\, \tau^6 d^{3/2}
        }{
             \frac{1}{n^{3/2}} \, |\tau|^3 \| \mu \|_{l_2}^3 
             \,+\,
             \frac{1}{n^3}\, \tau^6 d^{3/2}
        } 
    \Bigg)^{1/7}
    \\
    \;=&\;
    C''' \, n^{-1/14} 
    \;\rightarrow\; 
    0
    \;.
\end{align*}
 Note that this also holds beyond $\nu=3$ for a general $\nu \in(2,3]$, in which case the rate of convergence is $d^{-\frac{\nu-2}{4 \nu + 2}}$. By \Cref{thm:main}, we also have 
\begin{align*}
    \mean[u_2(Z)] = \mean[u_2(X)] = \mu^\top \mu
    \;\;\; \text{ and } \;\;\;
    \Var[u_2(Z)] = \Var[u_2(X)] = \mfrac{4 \tau^2 \| \mu \|_{l_2}^2}{n} + \mfrac{2 \tau^4 d}{n(n-1)}
    \;.
\end{align*}
To show Gaussianity, rather than computing the excess kurtosis in \Cref{prop:gaussian}, it is easier in this case to study the limit of $u_2(Z)$ (up to an appropriate rescaling) directly. Write $Z_i = \mu + \tau \eta_i$ where $\eta_i$'s are i.i.d.~standard Gaussians in $\R^d$. Notice that 
\begin{align*}
    u_2(Z) - \mean[u_2(X)]
    \;=&\;
    u_2(Z) - \mean[u_2(Z)]
    \\
    \;=&\; 
    \underbrace{\mfrac{2 \tau }{n} \msum_{i \leq n} \eta_i^\top \mu}_{\eqqcolon \; p_1(Z)}
    +
    \underbrace{\mfrac{\tau^2}{n(n-1)} \msum_{i \neq j} \eta_i^\top \, \eta_j}_{\eqqcolon \; p_2(Z)}
    \;.
\end{align*}
$p_1(Z)$ is a centred normal variable, whereas the appropriately rescaled version of $p_2(Z)$ can be expressed as
\begin{align*}
    \tilde p_2(Z) \;\coloneqq&\; \mfrac{n}{\tau^2 d^{1/2}}  \, \times \, p_2(Z)
    \\
    \;=&\;
    \mfrac{n}{n-1}
    \,
    \mfrac{1}{d^{1/2}}
    \,
    \Big( \mfrac{1}{\sqrt{n}} \msum_{i \leq n} \eta_i \Big)^\top 
    \Big( \mfrac{1}{\sqrt{n}} \msum_{j \leq n} \eta_j \Big)
    \,-\, 
    \mfrac{1}{(n-1)} 
    \,
    \msum_{i \leq n}  
    \mfrac{\eta_i^\top \eta_i}{d^{1/2}}
    \;.
\end{align*}
Since $\Var[ \eta_1^\top \eta_1] = 2d$, the second sum above is an average of i.i.d.~zero-mean quantities with bounded variance, and we can apply the strong law of large numbers to obtain
\begin{align*}
    \bigg| 
        \mfrac{1}{(n-1)} \msum_{i \leq n}  
        \mfrac{\eta_i^\top \eta_i}{d^{1/2}}
        - 
        \mfrac{n d^{1/2}}{n-1} 
    \bigg| 
    \;\xrightarrow{d}\; 0 
    \qquad 
    \text{ almost surely\,. }
\end{align*}
Also write $\bar \eta = \frac{1}{\sqrt{n}} \sum_{i \leq n} \eta_i$, which is a standard Gaussian vector in $\R^d$. We can now approximate $\tilde p_2(Z)$ by the centred variable
\begin{align*}
    \tilde p'_2(Z) 
    \;\coloneqq&\;
    \mfrac{n}{(n-1) d^{1/2}}
    \, 
    \big(
    \bar \eta^\top \bar \eta
    \,-\, 
    d
    \big)
    \;.
\end{align*} 
Therefore to show asymptotic normality of $u_2(Z) - \mean[u_2(X)]$, it suffices to prove the asymptotic normality of  
\begin{align*}
    \Var[u_2(X)]^{-1/2} \, u^*_2(Z) 
    \;\coloneqq&\; 
    \Var[u_2(X)]^{-1/2} \,
    \Big(
    \mfrac{2 \tau}{\sqrt{n}} \, \bar \eta^\top  \mu
    \,+\, 
    \mfrac{\tau^2}{n-1} \, 
    \big( \bar \eta^\top \bar \eta  - d \big)
    \Big) 
    \\
    \;=&\;
    \Var[u_2(X)]^{-1/2} \, 
    \msum_{l=1}^d 
    \Big( 
    \underbrace{
        \mfrac{2 \tau}{\sqrt{n}} \, \bar \eta_l \mu_l 
        \,+\, 
        \mfrac{\tau^2}{n-1} \, 
        \big( \bar \eta_l^2  - 1 \big)    
    }_{\eqqcolon W_l}
    \Big)
    \;.
\end{align*}
This is a sum of independent zero-mean random variables $(W_l)_{l \leq d}$. Moreover, since $\bar \eta_l \sim \cN(0,1)$, we can compute
\begin{align*}
    \Var[ u^*_2(Z)]
    \;=&\;
    \msum_{l=1}^d \mean[ W_l^2]
    \;=\;
    \msum_{l=1}^d 
    \Big( 
        \mfrac{4 \tau^2}{n} \mu_l^2
        \,+\,
        \mfrac{2 \tau^4}{(n-1)^2} 
    \Big)
    \;=\;
    \mfrac{4 \tau^2 \| \mu \|_{l_2}^2}{n}
    +
    \mfrac{2 \tau^4 d }{(n-1)^2}
    \;,
\end{align*}
and, for some absolute constants $C^*_1, C^*_2 > 0$,
\begin{align*}
    \mean\big[ W_l^4 \big]
    \;\leq&\;
    C^*_1
    \Big(
        \,
        \mfrac{\tau^4}{n^2} \, \mu_l^4  
        \,+\,
        \mfrac{\tau^6}{n^3} \, \mu_l^2
        \,+\,
        \mfrac{\tau^8}{n^4} 
    \Big)
    \;\leq\; 
    C^*_2 \Big(
        \,
        \mfrac{\tau^4}{n^2} \, \mu_l^4  
        \,+\,
        \mfrac{\tau^8}{n^4} 
    \Big)
    \;.
\end{align*}
Write $\Phi(\argdot)$ as the c.d.f.~of a standard normal variable. By \Cref{thm:main}, there are absolute constants $C^*_3, C^*_4, C^*_5 > 0$ such that
\begin{align*}
    &\;
    \sup_{t \in \R} 
    \, 
    \big|   
        \P\big( \Var[ u^*_2(Z)]^{-1/2}  u^*_2(Z) \leq t \big) 
        -
        \Phi(t)
    \big|
    \;\leq\;
    C^*_3  \,
    \Big( 
        \mfrac{\sum_{l \leq d} \mean | W_l |^3 }{\Var[ u^*_2(Z)]^{3/2}} 
    \Big)^{\frac{1}{7}}
    \\
    &\;\leq\;
    C^*_3 \,
    \Big( 
        \mfrac{\sum_{l \leq d} \, (\mean [W_l^4] )^{3/4} }{\Var[ u_2^*(Z)]^{3/2}} 
    \Big)^{\frac{1}{7}}
    \\
    &\;\leq\;
    C^*_4 \,
    \bigg( 
        \,
        \mfrac{  
            n^{-3/2} \, |\tau|^3 \, \sum_{l \leq d} \, |\mu_l|^3
            \,+\,
            n^{-3} \, \tau^6 d
        }{
            n^{-3/2} |\tau|^3  \, \| \mu \|_{l_2}^3
            +
            n^{-3}  \tau^6 d^{3/2}
        } 
        \,
    \bigg)^{\frac{1}{7}}
    \\
    &\;\overset{(a)}{\leq}\;
    C^*_5
    \bigg( 
        \,
        \mfrac{  
            n^{-3/2} \, |\tau|^3 \, d^{-1/2} \,  \| \mu \|_{l_2}^3
            \,+\,
            n^{-3} \, \tau^6 d
        }{
            n^{-3/2} |\tau|^3  \, \| \mu \|_{l_2}^3
            +
            n^{-3}  \tau^6 d^{3/2}
        } 
        \,
    \bigg)^{\frac{1}{7}}
    \;=\; C^*_5 d^{-1/14} \;\rightarrow\; 0\;.
     \tagaligneq \label{eq:bootstrap:prove:gaussianity}
\end{align*}
In $(a)$, we have again used \Cref{assumption:bootstrap:model}(iii) to note that $d^{1/6}  \| \mu \|_{l_3} = \Theta(\| \mu \|_{l_2})$. Note that this also holds beyond $\nu=3$ for a general $\nu \in(2,3]$, in which case the rate of convergence is $d^{-\frac{\nu-2}{4 \nu + 2}}$. Meanwhile, 
\begin{align*}
    \mfrac{\Var[ u^*_2(Z)]}{\Var[u_2(Z)]} 
    \;=\; 
    \mfrac{
        4 \tau^2 \| \mu \|^2_{l_2} / n
        \,+\,
        2 \tau^4 d / ((n-1)^2)
    }{
        4 \tau^2 \| \mu \|^2_{l_2} / n
        \,+\,
        2 \tau^4 d / (n(n-1))
    }
    \;\rightarrow\; 
    1 \;.
\end{align*}
By Slutsky's lemma, we obtain that $\Var[u_2(Z)]^{-1/2} \,  u^*_2(Z) \xrightarrow{d} \cN(0,1)$, and therefore we can conclude that 
\begin{align*}
    \Var[u_2(X)]^{-1/2}  \, \big( u_2(X) - \mean[u_2(X)] \big) 
    \;\xrightarrow{d}\; 
    \cN(0,1)\;.
\end{align*}

\subsection{Take $Y_i = X^{(b)}_i$ to analyse $U^{(b)}_1=u_2(X^{(b)})$ under $d=o(n^{2/\nu})$} \label{appendix:bootstrap:Uone}

We shall apply \eqref{eq:bootstrap:prelim} and \eqref{eq:bootstrap:var:prelim} conditionally on $X$, since $X^{(b)} = (X_1^{(b)}, \ldots, X_n^{(b)} )$ is a conditionally i.i.d.~set of vectors given $X$. In this case, $\mu_Y = \bar X$ and 
\begin{align*}
    \Sigma_Y 
    \;=\; 
    \Var[ X_1^{(b)} \,|\, X]
    \;=&\;
    \mean[ X_1^{(b)} (X_1^{(b)})^\top \,|\, X ] - \bar X \bar X^\top
    \\
    \;=&\;
    \mfrac{1}{n} \msum_{i \leq n} (X_i - \bar X) (X_i - \bar X)^\top 
    \;\eqqcolon\; \hat \Sigma \;.
\end{align*}
Write $Z^{(b)} \coloneqq (Z^{(b)}_1, \ldots, Z^{(b)}_n)$, where 
\begin{align*}
    Z^{(b)}_i \;\coloneqq\; \bar X + (\hat \Sigma)^{1/2} \, \eta_i
\end{align*}
and $\eta_1, \ldots, \eta_n$ are i.i.d.~standard normal vectors that are independent of all other variables. Then by \eqref{eq:bootstrap:prelim}, there exists some absolute constant $\tilde C_1 > 0$ such that almost surely
\begin{align*}
    &\;\sup_{t \in \R} \big| \P( u_2(X^{(b)}) \leq t \,|\, X ) - \P( u_2(Z^{(b)}) \leq t \,|\, X ) \big|
    \\
    &\;\leq\;
    \tilde C_1 \Bigg( 
        \mfrac{ 
            \frac{1}{n^2} 
            \,
            \mean\big[ \big| \bar X^\top ( X_1^{(b)} - \bar X) \big|^3 \,\big|\, X  \big]
            \,+\,
            \frac{1}{n^{7/2}}
            \mean\big[ \big| ( X_1^{(b)} - \bar X)^\top ( X_2^{(b)} - \bar X) \big|^3 \,\big|\, X  \big]
        }{
            \Var[ u_2(X^{(b)}) \,|\, X ]^{3/2}
        } 
    \Bigg)^{\frac{1}{7}}
    \;.
    \tagaligneq \label{eq:bootstrap:X:boot:intermediate}
\end{align*}
As before, we seek to compute the ratio of $L_3$-norms to $L_2$-norms in the bound above. Recall that $\bar X =\frac{1}{n} \sum_{i \leq n} X_i$, and write $\bar U = \frac{1}{\sqrt{n}} \sum_{i \leq n} U_i$. By \eqref{eq:bootstrap:var:prelim}, we obtain that almost surely
\begin{align*}
    &\;
    \Var[ u_2(X^{(b)}) \,|\, X ]
    \;=\;
    \mfrac{4 \, \| \hat \Sigma^{1/2} \bar X \|_{l_2}^2}{n} 
    \,+\,
    \mfrac{2 \, \Tr( \hat \Sigma^2) }{n(n-1)} 
    \\
    &\;=\;
    \mfrac{4}{n^2} \msum_{i \leq n} (\bar X^\top (X_i - \bar X ))^2 
    +
    \mfrac{2}{n^3(n-1)} \msum_{i,j \leq n}  ( (X_i - \bar X)^\top (X_j - \bar X) )^2
    \\
    &\;=\;
    \mfrac{4 \tau^2 }{n^2} \msum_{i \leq n} \Big(  \Big( \mu + \mfrac{\tau}{\sqrt{n}} \,  \bar U \Big)^\top  \Big( U_i - \mfrac{1}{\sqrt{n}} \, \bar U \Big) \Big)^2 
    \\
    &\qquad 
    +
    \mfrac{2 \tau^4}{n^3(n-1)} \msum_{i,j \leq n}  \Big(  \, \Big(U_i -  \mfrac{1}{\sqrt{n}} \bar U \Big)^\top \Big(U_j -  \mfrac{1}{\sqrt{n}} \bar U \Big)  \,  \Big)^2
    \;.
\end{align*}
Similarly, the third moments can be computed as
\begin{align*}
    \mean\big[ \big| \bar X^\top ( X_1^{(b)} - \bar X) \big|^3 \,\big|\, X  \big]
    \;=&\;
    \mfrac{1}{n} \msum_{i \leq n} \big| \bar X^\top (X_i - \bar X) \big|^3
    \\
    \;=&\;
    \mfrac{|\tau|^3}{n} \msum_{i \leq n} \Big| \Big( \mu + \mfrac{\tau}{\sqrt{n}} \, \bar U \Big)^\top  \Big( U_i - \mfrac{1}{\sqrt{n}} \, \bar U \Big) \Big|^3
    \;,
    \\
    \mean\big[ \big| ( X_1^{(b)} - \bar X)^\top ( X_2^{(b)} - \bar X) \big|^3 \,\big|\, X  \big]
    \;=&\;
    \mfrac{1}{n^2} \msum_{i, j \leq n} 
    \, \big| (X_i - \bar X)^\top (X_j - \bar X) \big|^3
    \\
    \;=&\;
    \mfrac{\tau^6}{n^2} \msum_{i, j \leq n} 
    \, \Big| \Big(U_i -  \mfrac{1}{\sqrt{n}} \bar U \Big)^\top   \Big(U_j -  \mfrac{1}{\sqrt{n}} \bar U \Big) \Big|^3
    \;.
\end{align*}
When $d$ is fixed, one might apply the strong law of large numbers to $\bar X$ and $\hat \Sigma$ and use the continuous mapping theorem to show that the above quantities can be approximated by their population versions. In the $n, d \rightarrow \infty$ regime, however, the argument is more subtle. We start by observing that in both the third moment and the second moment terms, the quantity $\frac{1}{\sqrt{n}} \bar U$ --- which we would expect to converge to a deterministic quantity had $d$ been fixed --- appears through the quantities
\begin{align*}
    \bar V 
    \coloneqq
    \mfrac{\tau^2}{n} \, \bar U^\top \bar U 
    =
    \mfrac{\tau^2}{n^2} \msum_{i, j \leq n} U_i^\top  U_j 
    \;,
    \qquad 
    V^\mu 
    \coloneqq
    \mfrac{\tau}{\sqrt{n}} \, \mu^\top  \bar U\;,
    \qquad 
    V_i 
    \coloneqq \mfrac{\tau^2}{\sqrt{n}} U_i^\top \bar U \;.
\end{align*}
Note that $\mean[ \bar V] = \frac{\tau^2 d}{n}$, $\mean[V^\mu] = 0$ and $\mean[V_i ] = \frac{\tau^2 d}{n}$, whereas by using the martingale moment bound (\Cref{lem:martingale:bound}) and $d = o(n)$,
\begin{align*}
    \Var[ \bar V ] 
    \;=&\; 
    \Theta\Big( 
        \tau^4 
        \,\Big(
        \mfrac{\mean[U_1^\top U_2 U_1^\top U_3 ]}{n} 
        +
        \mfrac{\mean[ (U_1^\top U_2)^2 ]}{n^2} 
        +
        \mfrac{\mean[(U_1^\top U_1 - \mean[U_1^\top U_1])^2]}{n^3}
        \Big)
    \Big)
    \\
    \;=&\;
    \Theta\Big( \tau^4\Big( \mfrac{d^{1/2}}{n} + \mfrac{d}{n^2} + \mfrac{d}{n^3} \Big)\Big)
    \;=\;
    \Theta\Big( \tau^4 \mfrac{d^{1/2}}{n} \Big) 
    \;,
    \\
    \Var[V^\mu]
    \;=&\;
    \mfrac{\tau^2}{n^2}  \msum_{i \leq n} \Var[ \mu^\top U_i ] 
    \;=\;
    \mfrac{\tau^2}{n} \|  \mu \|_{l_2}^2\;,
    \\
    \Var[ V_i ]
    \;=&\;
    \tau^4 \mean \, \Var[ n^{-1/2} U_i^\top \bar U \,|\, U_i  ]
    \,+\,
    \tau^4 \Var \, \mean[  n^{-1/2} U_i^\top \bar U \,|\, U_i  ]
    \\
    \;=&\;
    \mfrac{\tau^4}{n} \mean [U_1^\top U_1]
    \,+\,
    \mfrac{\tau^4}{n^2} \Var[U_1^\top U_1 ]
    \;=\;
    \Theta\Big( \tau^4\Big(  \mfrac{d}{n}  + \mfrac{d}{n^2} \Big) \Big) 
    \;=\;
    \Theta\Big( \tau^4  \mfrac{d}{n}  \Big)
    \;.
\end{align*}
Therefore, choosing the normalising factor 
$$
    a_n \;\coloneqq\; 
    \Big( 
        |\tau| \, \| \mu \|_{l_2} 
        + 
        \tau^2 \mfrac{d^{1/2}}{n^{1/2}}
    \Big)^{-1}
\;,
$$ 
by the Markov inequality,  we obtain 
\begin{align*}
    a_n \,  V^\mu \;\xrightarrow{\P}\; 0 
    \;,
    \qquad 
    a_n\,   \bar V  \;\xrightarrow{\P}\; 0 
    \;,
\end{align*}
whereas by the central limit theorem, 
\begin{align*}
    a_n   V_i 
    -
    \mfrac{a_n \tau^2 d}{n}
    \;=&\; 
    \mfrac{a_n \tau^2}{\sqrt{n}}
    \msum_{l=1}^d  \big( (U_i)_l (\bar U)_l  - n^{-1/2} \big)
    \\
    \;=&\;
    O\Big(d^{-1/2} \msum_{l=1}^d  \big( (U_i)_l (\bar U)_l  - n^{-1/2} \big) \Big)
    \;=\;
    O_\P(1)\;,
    \tagaligneq \label{eq:bootstrap:replace:empirical:mean}
\end{align*}
where we adopt $o_\P, O_\P$ as the usual notation for stochastic domination. Although there are many variables $(a_n  V_i)_{i \leq n}$, they appear in the bound in \eqref{eq:bootstrap:X:boot:intermediate} only through averages, and we can use Borel-Cantelli lemma and the Markov inequality to replace all of them by $a_n \tau^2 d/n + O_\P(1)$ simultaneously in the bound of \eqref{eq:bootstrap:X:boot:intermediate}. Plugging the above computations into \eqref{eq:bootstrap:X:boot:intermediate}, while scaling the numerator and the denominator by $a_n^3$ simultaneously, we get that for some absolute constant  $\tilde C_2 > 0$, almost surely
\begin{align*}
    &\;\sup_{t \in \R} \big| \P( u_2(X^{(b)}) \leq t \,|\, X ) - \P( u_2(Z^{(b)}) \leq t \,|\, X ) \big|
    \\
    &\;\leq\;
    \tilde C_2 \Bigg( 
        \mfrac{ 
            a_n^3 |\tau|^3 n^{-3} \sum_{i \leq n} \big| \mu^\top U_i \big|^3 
            +
            a_n^3 \tau^6 
            n^{-11/2} \sum_{i, j \leq n} 
            \, \big| U_i^\top   \, U_j + \Theta( \frac{d}{n} ) \big|^3 
            +
            o_\P(n^{-2} )
        }{
            \Big( a_n^2 \tau^2 n^{-2} \sum_{i \leq n} (\mu^\top U_i)^2 \Big)^{3/2}
            \,+\,
            \Big( a_n^2 \tau^4 n^{-4} \sum_{i,j \leq n} ( U_i^\top U_j + \Theta( \frac{d}{n} ))^2 \Big)^{3/2}
            +
            o_\P( n^{-3/2} )
        } 
    \Bigg)^{\frac{1}{7}}
    \tagaligneq \label{eq:bootstrap:conditional:gaussian:bound}
    \;.
\end{align*}
To show that the random variables in the bound converge to their population analogues, we first note that upon a rescaling, their expectations can be controlled as
\begin{align*}
    \mean\Big[  n^{-1} \msum_{i \leq n} \big| \mu^\top U_i \big|^3\Big]
    \;=&\;
    \mean\big[   | \mu^\top  U_1 |^3 \big]
    \;\overset{\eqref{eq:bootstrap:third:moment:non:degen}}{\leq}\; 
    \tilde C_2 \,\| \mu \|_{l_2}^3 
    \;,
    \\
    \mean\Big[
        n^{-2}
        \msum_{i, j \leq n} 
        \, \big| U_i^\top  U_j \big|^3
    \Big]
    \;=&\;
    \mfrac{n-1}{n} \,\mean\big[ \big| U_1^\top  U_2 \big|^3 \big]
    +
    \mfrac{1}{n} \, \mean\big[ \big| U_1^\top  U_1 \big|^3 \big]
    \\
    \;\overset{\eqref{eq:bootstrap:third:moment:degen}}{\leq}&\;
    \tilde C_3 \, \mfrac{n-1}{n} \,  d^{3/2}
    +
    \mfrac{1}{n} \, \mean\big[ \big| U_1^\top  U_1 \big|^3 \big]
    \\
    \;\overset{(a)}{\leq}&\;
    \tilde C_3 \, \mfrac{n-1}{n} \,  d^{3/2}
    +
    C_* \,\mfrac{d^3}{n} 
    \;,
    \\
    \mean\Big[ 
        \mfrac{1}{n} \, \msum_{i \leq n} (\mu^\top U_i)^2     
    \Big]
    \;=&\; 
    \| \mu \|_{l_2}^2
    \;,
    \\
    \mean\Big[ 
        \mfrac{1}{n^2} \, \msum_{i,j \leq n} (U_i^\top  U_j)^2     
    \Big]
    \;=&\; 
    \mfrac{n-1}{n} \, \mean\big[ (U_1^\top  U_2)^2 \big] 
    \,+\, 
    \mfrac{1}{n} \, \mean\big[ (U_1^\top  U_1)^2 \big]
    \\
    \;\overset{(b)}{=}&\; 
    \mfrac{n-1}{n} \, d  
    \,+\,
    \Theta\Big( \mfrac{d^2}{n}\Big)
    \;.
    \tagaligneq \label{eq:bootstrap:barU:squared:mean}
\end{align*}
In $(a)$, we have used Jensen's inequality and that $U_1$ has i.i.d.~coordinates with unit variance and a bounded 6th moment to obtain
\begin{align*}
    \mean\big[ \big| U_1^\top  U_1 \big|^3 \big] 
    \;=&\;
    d^3 \, \mean\Big[ \Big| \mfrac{1}{d} \msum_{l \leq d} U_{1l}^2 \Big|^3 \Big] 
    \;\leq\;
    C_* d^3
\end{align*}
for some absolute constant $C_* > 0$. In $(b)$ we have used the same assumption to obtain
\begin{align*}
    \mean\big[ \big( U_1^\top  U_1 )^2 \big] 
    \;=\;
    \msum_{l_1, l_2 \leq d} \mean[ U_{1l_1}^2 U_{1l_2}^2 ] 
    \;=\; \Theta( d^2)\;.
\end{align*}
Meanwhile, the corresponding variances of these quantities can be controlled by the martingale moment bound (\Cref{lem:martingale:bound}) and by noting that $d^{1/4} \| \mu \|_{l_4} = \Theta( \| \mu \|_{l_2})\,,\, d^{1/\nu} \| \mu \|_{l_{2\nu}} = \Theta( \| \mu \|_{l_2})$ and the coordinates of $U_1$ are i.i.d.~mean-zero with bounded sixth moments (\Cref{assumption:bootstrap:model}): 
\begin{align*}
   &\; 
   \Var\Big[ n^{-1} \msum_{i \leq n} \big| \mu^\top U_i \big|^3\Big]
    \;=\;
    \mfrac{1}{n} \Var\big[ \big| \mu^\top U_1 \big|^3 \big]
    \\
    &\qquad\quad \;\leq\;
    \mfrac{1}{n}   
        \,
        \mean\Big[ \Big| \msum_{l \leq d} \mu_l \, (U_1)_l \Big|^6 \Big]
    \;=\; 
    O\Big( \mfrac{d^2}{n} \, \msum_{l \leq d} \mean[ | \mu_l |^6 | (U_1)_1 |^6 ]  \Big)
    \\
    &\qquad\quad 
    \;=\;
    O\Big( \mfrac{d^2}{n} \, \| \mu \|_{l_6}^6 \Big)
    \;=\;
    O\Big( \mfrac{1}{n} \, \| \mu \|_{l_2}^6  \Big)
    \;,
    \\
    &\; 
   \Var\Big[ n^{-2} \msum_{i,j \leq n} \big| U_i^\top U_j \big|^3\Big]
   \\
    &\qquad\quad 
    \;=\; 
    \Theta\Big( 
        \mfrac{
            \mean\big[  ( |U_1^\top U_2 |^3 - \mean[|U_1^\top U_2 |^3]) \, 
            ( |U_1^\top U_3 |^3 - \mean[|U_1^\top U_3 |^3]) \big]}{n} 
    \\
    &\qquad\qquad\qquad 
        +
        \mfrac{
            \mean\big[ ( |U_1^\top U_2 |^3 - \mean[ |U_1^\top U_2|^3])^2 \big]
        }{n^2} 
    \\
    &\qquad\qquad\qquad 
        +
        \mfrac{\mean\big[ ( |U_1^\top U_1 |^3 - \mean[ |U_1^\top U_1 |^3] )^2 \big]
        }{n^3} 
    \Big)
    \\
    &\qquad\quad 
    \;=\;
    O\Big( \mfrac{d^3}{n} + \mfrac{d^3}{n^2} + \mfrac{d^6}{n^3}  \Big)
    \;\overset{(a)}{=}\;
    O\Big( \mfrac{d^3}{n} \Big)
    \;,
    \\
    &\;
    \Var\Big[ 
        \mfrac{1}{n} \, \msum_{i \leq n} (\mu^\top U_i)^2     
    \Big]
    \;=\;
    \mfrac{1}{n} \, \Var[  (\mu^\top U_1)^2  ]
    \\
    &\qquad\qquad \qquad\qquad\quad
    \;=\;
    \mfrac{1}{n} \, 
    \mean\Big[ \big( \msum_{l \leq d} \mu_l (U_1)_l \big)^4 \Big]
    -
    \mfrac{1}{n} \, \| \mu \|_{l_2}^4
    \\
     &\qquad\qquad \qquad\qquad\quad\;=\;
    \mfrac{1}{n} \, \msum_{l \leq d} \mu_l^4  \, \mean[(U_1)_1^4]
    +
    \mfrac{3}{n} \,
    \msum_{l, l' \leq d} \mu_l^2 \mu_{l'}^2 
    -
    \mfrac{1}{n} \, \| \mu \|_{l_2}^4
    \\
    &\qquad\qquad \qquad\qquad\quad
    \;=\; 
    \Theta\big( \mfrac{1}{n} \| \mu \|_{l_2}^4 \big)
    \;,
    \\
    &\;
    \Var\Big[ 
        \mfrac{1}{n^2} \, \msum_{i,j \leq n} (U_i^\top U_j)^2     
    \Big]
    \\
    &\qquad\quad 
    \;=\; 
    \Theta\Big( 
        \mfrac{
            \mean\big[  ( (U_1^\top U_2 )^2 - \mean[(U_1^\top U_2 )^2]) \, 
            ( (U_1^\top U_3 )^2 - \mean[(U_1^\top U_3 )^2]) \big]}{n} 
    \\
    &\qquad\qquad\qquad 
        +
        \mfrac{
            \mean\big[ ( (U_1^\top U_2 )^2 - \mean[(U_1^\top U_2)^2])^2 \big]
        }{n^2} 
    \\
    &\qquad\qquad\qquad 
        +
        \mfrac{\mean\big[ ( (U_1^\top U_1 )^2 - \mean[ (U_1^\top U_1 )^2] )^2 \big]
        }{n^3} 
    \Big)
    \\
    &\qquad\quad 
    \;=\;
    O\Big( \mfrac{d^2}{n} + \mfrac{d^2}{n^2} + \mfrac{d^4}{n^3}  \Big)
    \;\overset{(b)}{=}\;
    O\Big( \mfrac{d^2}{n} \Big)
    \;.
    \tagaligneq \label{eq:bootstrap:barU:squared:var}
\end{align*}
In $(a)$ and $(b)$, we have used that $d=o(n^{2/3})$. Note that  step $(a)$ also works if $\nu=3$ is replaced by a general $\nu \in (2,3]$. As such, we can choose the normalising constant 
\begin{align*}
    b_n \;\coloneqq\; 
    \big( 
        |\tau|\, \| \mu \|_{l_2}
        + 
        \tau^2  d^{1/2}
    \big)^{-1}
\end{align*}
and apply the the Markov inequality again to obtain that 
\begin{align*}
    &b_n^3  |\tau|^3  n^{-1} \msum_{i \leq n} \big| \mu^\top U_i \big|^3\,,
    &\,&&
    &b_n^3 \tau^6 \, n^{-2} \msum_{i, j \leq n}  \big| U_i^\top  U_j \big|^3
    \,,\,
    \\
    &b_n^2 \tau^2  n^{-1}  \msum_{i \leq n} (\mu^\top U_i)^2\,,
    &\,&&
    &b_n^2 \tau^4  \, n^{-2} \msum_{i,j \leq n} (U_i^\top  U_j)^2\,,
\end{align*}
all converge to their limiting expectations in probability. Plugging in the values of their expectations,  we get that for some  absolute constant  $\tilde C_3 > 0$, almost surely
\begin{align*}
    &\;\sup_{t \in \R} \big| \P( u_2(X^{(b)}) \leq t \,|\, X ) - \P( u_2(Z^{(b)}) \leq t \,|\, X ) \big|
    \\
    &\;\leq
    \tilde C_3 \Bigg( 
        \mfrac{ 
            \frac{a_n^3}{n^2}
            \big( 
            b_n^3 |\tau|^3 \| \mu \|_{l_2}^3 
            + 
            o_\P(1)
            \big) 
            +
            \frac{a_n^3}{n^{7/2}}
            \big( 
            b_n^3 \tau^6 d^{3/2} 
            + 
            b_n^3 \tau^6 \frac{d^3}{n} 
            + 
            o_\P(1) 
            \big)
            +
            o_\P( b_n^3 n^{-2} )
        }{
            \frac{a_n^3 }{n^{3/2}}
            \big( 
                b_n^3 |\tau|^3 \| \mu \|_{l_2}^3 
                +
                o_\P(1)
            \big) 
            +
            \frac{a_n^3}{n^3}
            \big( 
                b_n^3 \tau^6 d^{3/2}
                +
                b_n^3  \tau^6 \frac{d^3}{n^{3/2}}
                + 
                o_\P( 1) 
            \big)
            +
            o_\P( b_n^3 n^{-3/2} )
        } 
    \Bigg)^{\frac{1}{7}}
    \;.
    \tagaligneq
    \label{eq:bootstrap:conditional:gaussian:bound:two} 
\end{align*}
The numerator and the denominator of the above fraction are almost identical, except for a difference in scaling of the order $n^{-1/2}$ as well as the term $\frac{d^3}{n}$ in the numerator and the term $\frac{d^3}{n^{3/2}}$ in the denominator. Under the stated condition  $d = o(n^{2/3})$, these terms are negligible, since 
\begin{align*}
     \mfrac{d^3}{n^{3/2}} \;\leq\; \mfrac{d^3}{n} \;=\; o(d^{3/2}) \;.
\end{align*}
Note that this also true in the case with a general $\nu \in (2,3]$, in which case the given condition is $d = o(n^{2/\nu})$ and the comparison is 
\begin{align*}
     \mfrac{d^\nu}{n^{\nu/2}} \;\leq\; \mfrac{d^\nu}{n} \;=\; o(d^{\nu/2}) \;.
\end{align*} 
Also note that the $o_\P(\argdot)$ terms are negligible by the choices of $a_n$ and $b_n$. This implies that 
\begin{align*}
    \msup_{t \in \R} \big| \P( u_2(X^{(b)}) \leq t \,|\, X ) - \P( u_2(Z^{(b)}) \leq t \,|\, X ) \big|
    \;\xrightarrow{\P}\; 0
    \;.
\end{align*}

\vspace{1em}

We now study the limit of 
\begin{align*}
    u_2(Z^{(b)})
    -
    \mean[ u_2(Z^{(b)})  \,|\, X]
    \;=&\;
    \mfrac{2}{n} \msum_{i \leq n} 
    \eta_i^\top (\hat \Sigma)^{1/2} \bar X
    +
    \mfrac{1}{n(n-1)} \msum_{i \neq j} \eta_i^\top \hat \Sigma \eta_j
\end{align*}
where $\eta_i$'s are i.i.d.~standard Gaussians. We first seek to replace $\hat \Sigma$ by $\tau^2 I_d$ and approximate the above by 
\begin{align*}
    \tilde u_2(Z^{(b)})
    \;\coloneqq\; 
    \mfrac{2 \tau}{n}
    \msum_{i \leq n} 
    \eta_i^\top \bar X
    +
    \mfrac{\tau^2}{n(n-1)} \msum_{i \neq j} \eta_i^\top \eta_j
    \;.
\end{align*}
For $t \in \R$, notice that (see e.g.~Proposition 2 of \cite{huang2024slow}) almost surely
\begin{align*}
    \Big| 
        \P\big(  u_2(Z^{(b)}) &\, - \mean[ u_2(Z^{(b)}) \,|\, X ] \leq t \,\big|\, X \big)
        -
        \P\big( \tilde u_2(Z^{(b)}) \leq t \,\big|\, X \big)
    \Big|
    \\
    \;\leq&\;
    \minf_{\epsilon > 0}
    \Big\{ 
        \;
        \P\big( \Var[\tilde u_2(Z^{(b)})\,|\, X ]^{-1/2} \tilde u_2(Z^{(b)}) \in (t-\epsilon, t+\epsilon] \,\big|\, X  \big)  
    \\
    &\hspace{4.5em}
        +
        \mfrac{\mean\big[ \big(  u_2(Z^{(b)}) - \mean[ u_2(Z^{(b)}) \,|\, X ] - \tilde u_2(Z^{(b)}) \big)^2 \,\big|\, X \big]}{\epsilon^2 \, \Var[\tilde u_2(Z^{(b)})\,|\, X ]}
    \Big\}
    \;,
    \tagaligneq \label{eq:bootstrap:approx:boot:u:two}
\end{align*}
and the conditional moments above can be computed as 
\begin{align*}
    &\; \mean\big[ \big(  u_2(Z^{(b)}) - \mean[ u_2(Z^{(b)}) \,|\, X ] - \tilde u_2(Z^{(b)}) \big)^2 \,\big|\, X \big]
    \\
    \;=&\;
    \mean\Big[ \, 
        \Big( 
            \mfrac{2}{n} \msum_{i \leq n} \eta_i^\top (\hat \Sigma^{1/2} - \tau  I_d) \bar X 
            +
            \mfrac{1}{n(n-1)} \msum_{i \neq j} \eta_i^\top (\hat \Sigma - \tau^2 I_d) \eta_j
        \Big)^2  
    \,\Big|\, X \Big]
    \\
    \;=&\;\,
    \mfrac{4}{n} \bar X^\top  (\hat \Sigma^{1/2} - \tau  I_d)^2 \bar X 
    +
    \mfrac{2}{n(n-1)} \Tr\Big(   (\hat \Sigma - \tau^2 I_d)^2 \Big)
    \\
    \;\leq&\;\,
    \mfrac{4 d}{n}
    \| \hat \Sigma^{1/2} - \tau I_d \|_{op}^2 \, \bar X^\top \bar X 
    \,+\, 
    \mfrac{2 d }{n(n-1)} \| \hat \Sigma - \tau^2 I_d \|_{op}^2 
    \;,
\end{align*}
and
\begin{align*}
     \Var[\tilde u_2(Z^{(b)})\,|\, X ]
     \;=&\;
     \mfrac{4 \tau^2 d}{n} \bar X^\top \bar X 
     +
     \mfrac{2 \tau^4 d}{n(n-1)} \;.
\end{align*}
To show that $\eqref{eq:bootstrap:approx:boot:u:two} = o_\P(1)$, it suffices to prove that 
\begin{align*}
    \tau^{-1} \| \hat \Sigma^{1/2} - \tau I_d \|_{op}
    \;=&\; o_\P(1)
    &\text{ and }&&
    \tau^{-2} \| \hat \Sigma - \tau^2 I_d \|_{op}
    \;=&\;
    o_\P(1)\;.
    \tagaligneq \label{eq:bootstrap:matrix:control}
\end{align*}
Let $s_{\rm min}(\argdot)$ and $s_{\rm max}(\argdot)$ denote the largest and smallest singular values of a matrix respectively. It suffices to show that 
\begin{align*}
    \big| \tau^{-2} s_{\rm min}(\hat \Sigma) - 1 \big| 
    \;=&\;
    o_\P(1)
    &\text{ and }&&
    \big| \tau^{-2} s_{\rm max}(\hat \Sigma) - 1 \big| 
    \;=&\;
    o_\P(1)
    \;.
\end{align*}
Denote the $\R^{d \times n}$ random matrix $U\coloneqq (U_1, \ldots, U_n)$. First note that by the triangle inequality,
\begin{align*}
    s_{\rm max}(\hat \Sigma)
    \;\leq&\; 
    \Big\| 
        \mfrac{1}{n} \msum_{i \leq n} (X_i - \mu) (X_i - \mu)^\top 
    \Big\|_{op}
    \\
    &\quad 
    \;
    +
    \Big\| 
        \mfrac{1}{n} \msum_{i \leq n} (X_i - \bar X) (X_i - \bar X)^\top 
        -
        \mfrac{1}{n} \msum_{i \leq n} (X_i - \mu) (X_i - \mu)^\top 
    \Big\|_{op}
    \\
    \;=&\;
    \mfrac{\tau^2}{n} s_{\rm max}(U)^2
    +
    \big\| 
        \bar X \bar X^\top - \bar X \mu^\top - \mu \bar X^\top + \mu \mu^\top
    \big\|_{op}
    \\
    \;=&\;
    \mfrac{\tau^2}{n} s_{\rm max}(U)^2
    +
    \mfrac{\tau^2}{n} \| \bar U \bar U \|_{op}
    \;=\;
    \mfrac{\tau^2}{n} s_{\rm max}(U)^2
    +
    o_\P(\tau^2)
    \;.
\end{align*}
In the last line, we have recalled $\bar U = n^{-1/2} \sum_{i \leq n} U_i$ and noted that $n^{-1} \mean \| \bar U \bar U \|_{op}^2 = \mean ( \frac{1}{n^2} \sum_{i, j \leq n} U_i^\top U_j ) = O(d/n) = o(1)$ and applied the Markov inequality. Meanwhile by the triangle inequality again, 
\begin{align*}
    s_{\rm min}(\hat \Sigma)
    \;=&\; 
    \minf_{v \in \R^d, \| v \|_{l_2}=1} \, v^\top \hat \Sigma v 
    \\
    \;=&\;
    \minf_{v \in \R^d, \| v \|_{l_2}=1}  
    \big\{ 
        v^\top \big( \mfrac{1}{n} \msum_{i \leq n} \tau^2 U_i U_i^\top \big) v  
        +
        v^\top \Big( \hat \Sigma - \mfrac{1}{n} \msum_{i \leq n} \tau^2 U_i U_i^\top \Big) v
    \big\}
    \\
    \;\geq&\;
    \mfrac{\tau^2}{n} \, s_{\rm min}(U)^2 
    - 
    \mfrac{\tau^2}{n} \, \| \bar U \bar U^\top \|_{op}
    \;=\;
     \mfrac{\tau^2}{n} \, s_{\rm min}(U)^2 + o_\P(\tau^2)
    \;.
\end{align*}
It now suffices to show that 
\begin{align*}
    \Big| \mfrac{s_{\rm min}(U)}{\sqrt{n}} - 1 \Big| 
    \;=&\;
    o_\P(1)
    &\text{ and }&&
    \Big| \mfrac{s_{\rm max}(U)}{\sqrt{n}} - 1 \Big| 
    \;=&\;
    o_\P(1)\;.
\end{align*}
This holds since $U$ is a $\R^{d \times n}$ random matrix with i.i.d.~mean-zero 1-sub-Gaussian entries, to which we can apply a standard matrix singular value control  under the condition $d=o(n^{\frac{2}{\nu}}) = o(n)$ (see Theorem 6.5~and (6.20) of \cite{wainwright2019high}). In summary this shows that $\eqref{eq:bootstrap:approx:boot:u:two} = o_\P(1)$ and therefore
\begin{align*}  
    \msup_{t \in \R} 
    \big| \P( u_2(X^{(b)}) - \mean[u_2(X^{(b)}) \,|\, X] \leq t \,|\, X) 
      - \P\big( \tilde u_2(Z^{(b)}) \leq t \,\big|\, X \big)
    \big|
    \;\xrightarrow{\P}\; 0\;.
\end{align*}
Moreover, since 
\begin{align*}
    \tilde u_2(Z^{(b)}) 
    \;=\; 
     \mfrac{2 \tau}{n}
    \msum_{i \leq n} 
    \eta_i^\top \bar X
    +
    \mfrac{\tau^2}{n(n-1)} \msum_{i \neq j} \eta_i^\top \eta_j
    \;,
\end{align*}
which is almost the same as $u_2(Z) - \mean[u_2(X)]$ except that $\mu$ is replaced by $\bar X$, we can apply the same argument leading up to \eqref{eq:bootstrap:prove:gaussianity} to obtain conditional Gaussianity of $\tilde u_2(Z^{(b)})$ by re-expressing the above as a sum of $d$ independent quantities. This implies that for some absolute constants $C^*_4, C^*_5 > 0$, almost surely
\begin{align*}
    &\;
    \sup_{t \in \R} 
    \, 
    \big|   
        \P\big( \Var[u_2(X^{(b)}) \,|\, X ]^{-1/2} \tilde u_2(Z^{(b)}) \leq t \,\big|\, X \big) 
        -
        \Phi(t)
    \big|
    \\
    &\;\leq\;
    C^*_4 \,
    \bigg( 
        \,
        \mfrac{  
            n^{-3/2} \, |\tau|^3 \, \sum_{l \leq d} \, |(\bar X)_l|^3
            \,+\,
            n^{-3} \, \tau^6 d
        }{
            n^{-3/2} |\tau|^3  \, \| \bar X \|_{l_2}^3
            +
            n^{-3}  \tau^6 d^{3/2}
        } 
        \,
    \bigg)^{\frac{1}{7}}
    \\
    &\;\leq\;
    C^*_5 \,
    \bigg( 
        \,
        \mfrac{  
            n^{-3/2} \, \sum_{l \leq d} |(n^{-1/2} \bar U)_l|^3 
            +
             n^{-3/2} \tau^{-3} \, d^{-1/2} \|\mu_l \|_{l_2}^3
            \,+\,
            n^{-3} \, d
        }{
            n^{-3/2} \, \big(n^{-1} \bar U^\top \bar U \big)^{3/2}
            +
            n^{-3/2}  \tau^{-3} \, \| \mu \|_{l_2}^3
            +
            n^{-3}  d^{3/2}
        } 
        \,
    \bigg)^{\frac{1}{7}}
    \;,
\end{align*}
where we have recalled again that $\bar U = n^{-1/2} \sum_{i \leq n} U_i$. Meanwhile by the Markov inequality argument again and noting that $d = o(n)$, we can show that 
\begin{align*}
    \msum_{l \leq d} |(n^{-1/2} \bar U)_l|^3 
    \;=&\;
    O_\P( d n^{-3/2}) 
    &\text{ and }&&
     \big(n^{-1} \bar U^\top \bar U \big)^{3/2}
     \;=&\;
     O_\P(d^{3/2} n^{-3/2})
     \;. 
\end{align*} 
This implies
\begin{align*}
    &\;
    \sup_{t \in \R} 
    \, 
    \big|   
        \P\big( \Var[u_2(X^{(b)}) \,|\, X ]^{-1/2} (u_2(X^{(b)}) - \mean[u_2(X^{(b)}) \,|\, X] ) \leq t \,\big|\, X \big) 
        -
        \Phi(t)
    \big|
    \\
    &\;\leq\;
    C^*_5 \,
    \bigg( 
        \,
        \mfrac{  
            n^{-3/2} \, O_\P(d n^{-3/2})
            +
             n^{-3/2} \tau^{-3} \, d^{-1/2} \|\mu \|_{l_2}^3
            \,+\,
            n^{-3} \, d
        }{
            n^{-3/2} \, O_\P(d^{3/2} n^{-3/2})
            +
            n^{-3/2}  \tau^{-3} \, \| \mu \|_{l_2}^3
            +
            n^{-3}  d^{3/2}
        } 
        \,
    \bigg)^{\frac{1}{7}}
    \;=\;
    o_\P(1)
    \;.
\end{align*}
Note that the same argument also works for a general $\nu \in (2,3]$ beyond just $\nu=3$. This implies that with high probability
\begin{align*}
    \Var[u_2(X^{(b)}) \,|\, X ]^{-1/2} (u_2(X^{(b)}) - \mean[u_2(X^{(b)}) \,|\, X] ) 
    \;\xrightarrow{d}\; 
    \cN(0,1)
    \;.
\end{align*}

In the definition of consistency, one compares the distribution of $U_1^{(b)} - \mean[ U_1^{(b)} | X ] = u_2(X^{(b)}) - \mean[u_2(X^{(b)})|X]$ against the distribution of $u_2(X^) - \mean[u_2(X)]$. In light of the asymptotic normality results above, it therefore suffices to compare the conditional variances. Recall the formulas of $\Var[ u_2(X)]$ and $\Var[u_2(X^{(b)}) | X]$ proved before, and that 
\begin{align*}
    a_n 
    \;=\; 
     \Big( 
        |\tau|  \| \mu \|_{l_2} 
        + 
        \tau^2 \mfrac{d^{1/2}}{n^{1/2}}
    \Big)^{-1}
    \;=\;
    \Theta( ( n \, \Var[u_2(X)]  )^{-1/2} )
    \;.
\end{align*}
The argument leading up to \eqref{eq:bootstrap:replace:empirical:mean} proves that 
\begin{align*}
    n a_n^2  \Var[ u_2(X^{(b)}) \,|\, X ]
    \;=&\;
    \mfrac{4 a_n^2 \tau^2 }{n} \msum_{i \leq n} ( \mu^\top U_i )^2
    +
    \mfrac{2  a_n^2 \tau^4}{n^2(n-1)} 
    \msum_{i,j \leq n} \big(U_i^\top U_j + \Theta\big(\mfrac{d}{n} \big)\big)^2
    \\
    &\;
    +
    o_\P(1)
    \;.
\end{align*}
Therefore the difference in variance can be computed as 
\begin{align*}
     &\;
     n a_n^2 \big( \Var[ u_2(X^{(b)}) \,|\, X ] - \Var[u_2(X)]  \big)
     \\
     &\;=
      \mfrac{4 a_n^2 \tau^2 }{n} \msum_{i \leq n} 
      \big( ( \mu^\top U_i )^2 - \| \mu \|_{l_2}^2 \big)
      +
      \mfrac{2  a_n^2 \tau^4}{n^2(n-1)} 
      \msum_{i,j \leq n} 
      \Big( 
        \Big(U_i^\top U_j + \Theta\Big(\mfrac{d}{n} \Big)\Big)^2
        -
        d
    \Big)
    \\
    &\quad 
        \;+
      o_\P(1)
      \,.
\end{align*}
Recycling the variance computations in \eqref{eq:bootstrap:barU:squared:var} and applying the central limit theorem again,
\begin{align*}
    \mfrac{1}{\sqrt{n}}
    \msum_{i \leq n} 
    \mfrac{ ( \mu^\top U_i )^2 - \| \mu \|_{l_2}^2 }{\| \mu \|_{l_2}^2}
    \;=\;
    \Theta_\P(1)\;.
\end{align*}
Meanwhile,  by noting that the coordinates of $U_i$ and $U_j$ are i.i.d.~zero-mean and $1$-sub-Gaussian and using the standard convergence result of a degenerate V-statistic \citep{serfling1980approximation}, for $i \neq j$,
\begin{align*}
    &\;
    \mfrac{1}{d}
    \Big(
    \Big(U_i^\top U_j + \Theta\Big(\mfrac{d}{n} \Big)\Big)^2
        -
        d
    \Big)
    \\
    &\;=\;
    \mfrac{1}{d}
    \msum_{l_1,l_2 \leq d} 
    (U_i)_{l_1} (U_j)_{l_1}
    (U_i)_{l_2} (U_j)_{l_2}
    -
    1
    +  
    \mfrac{2}{d} 
    \msum_{l \leq d} 
    (U_i)_{l} (U_j)_{l} \, \Theta\Big(\mfrac{d}{n} \Big)
    +
     \Theta\Big(\mfrac{d}{n^2} \Big)
     \\
     &\;=\;
     \Theta_\P\Big(1 + \mfrac{d^{1/2}}{n} + \mfrac{d}{n^2}  \Big)
     \;.
\end{align*}
 Meanwhile for $i=j$, by the central limit theorem,
\begin{align*}
    \mfrac{1}{d}
    \Big( \Big(U_i^\top U_i + \Theta\Big(\mfrac{d}{n} \Big)\Big)^2 - d \Big)
    \;=\;
    \Big( \mfrac{1}{\sqrt{d}} \msum_{l \leq d} (U_i)_l^2  + \Theta\Big(\mfrac{d^{1/2}}{n} \Big)\Big)^2
    - 1
    \;=\; \Theta_\P( d )\;.
\end{align*}
This implies that 
\begin{align*}
   \mfrac{1}{d} \, \mfrac{1}{n^2} 
      \msum_{i,j \leq n} 
      \Big( 
        \Big(U_i^\top U_j + \Theta\Big(\mfrac{d}{n} \Big)\Big)^2
        -
        d
    \Big)
    \;=&\;
    \Theta_\P
    \Big( 1 + \mfrac{d}{n} \Big) \;.
\end{align*}
Combining the calculations, we get that the variance difference, upon an appropriate normalisation, is 
\begin{align*}
    \Delta_1
    \;\coloneqq\;
     n a_n^2 \big( \Var[ u_2(X^{(b)}) \,|\, X ] - \Var[u_2(X)]  \big)
     \;=&\;
     \Theta_\P\Big( 
      \mfrac{a_n^2 \tau^2 \| \mu \|_{l_2}^2}{n^{1/2}}
      +
      \mfrac{a_n^2 \tau^4 d}{n}
      \Big( 1 + \mfrac{d}{n} \Big)
    \Big)
    \\
    \;=&\;
    \Theta_\P\bigg( 
      \mfrac{ 
        \tau^2 \| \mu \|_{l_2}^2 n^{-1/2}
        +
        \tau^4 d n^{-1} (1+d n^{-1})
      }{
        (
        \tau  \| \mu \|_{l_2} 
        + 
        \tau^2 d^{1/2}n^{-1/2}
        )^2
      }
    \bigg)
    \\
    \;=&\;
    \Theta_\P 
    \bigg( 
      \mfrac{ 
        \| \mu \|_{l_2}^2 n^{-1/2}
        +
        \tau^2 d n^{-1}  (1+d n^{-1})
      }{
        \| \mu \|_{l_2}^2
        + 
        \tau^2 d n^{-1}
      }
    \bigg)
    \;.
    \tagaligneq \label{eq:bootstrap:var:diff}
\end{align*}
Recall that $d=o(n)$. There are three cases: If $\tau d^{1/2} = o\big( n^{1/4} \| \mu \|_{l_2} \big)$, the above evaluates to 
\begin{align*}
    \Delta_1\;=\;
    \Theta_\P\Big( 
        \mfrac{\| \mu \|_{l_2}^2 n^{-1/2}}
        {\| \mu \|_{l_2}^2} 
    \Big) \;=\; o_\P(1)
    \;.
\end{align*}
If $\tau d^{1/2} = \Omega\big( n^{1/4} \| \mu \|_{l_2} \big)$ and  $\tau d^{1/2} = o\big( n^{1/2} \| \mu \|_{l_2} \big)$, the above bound evaluates to 
\begin{align*}
    \Delta_1\;=\;
     \Theta_\P\Big( 
        \mfrac{\tau^2 d n^{-1}}
        {\| \mu \|_{l_2}^2} 
    \Big) \;=\; o_\P(1)
    \;.
\end{align*}
If  $\tau d^{1/2} = \Omega\big( n^{1/2} \| \mu \|_{l_2} \big)$, the bound evaluates to 
\begin{align*}
    \Delta_1\;=\;
     \Theta_\P\Big( 
        \mfrac{\tau^2 d n^{-1}}
        {\tau^2 d n^{-1}} 
    \Big) \;=\; \Theta_\P(1)\;.
\end{align*}
As such, the variance agrees if and only if 
\begin{align*}
    \tau d^{1/2} \;=\; o\big( n^{1/2} \| \mu \|_{l_2} \big)
    \;.
\end{align*}
This proves the necessary and sufficient condition for $U^{(b)}_1 = u_2(X^{(b)})$ to be consistent. 

\subsection{Take $Y_i = X^{(b)}_i - \bar X$ to analyse $U^{(b)}_2=u_2(\tilde X^{(b)})$ under $d=o(n^{2/\nu})$} \label{appendix:bootstrap:Utwo}

Recall that $\tilde X^{(b)} = (X^{(b)}_1 - \bar X, \ldots, X^{(b)}_n - \bar X)$. Denote the corresponding conditional Gaussian surrogates as 
\begin{align*}
    \tilde Z^{(b)}
    \;\coloneqq\; 
    \big( \tilde Z^{(b)}_1, \ldots, \tilde Z^{(b)}_n \big)
    \quad \text{ where }
    \tilde Z^{(b)}_i \;\coloneqq\; (\hat \Sigma)^{1/2} \eta_i \;,
\end{align*}
and $\eta_i$'s are i.i.d.~standard normal vectors that are independent of all other variables. 

\vspace{.5em}

Most of the calculations from \Cref{appendix:bootstrap:Uone} can be recycled, except that in this case we have $\mu_Y = 0$ and $\Sigma_Y = \hat \Sigma$, and several computations can be omitted. We again deal with the case $\nu=3$, as the argument for the general $\nu \in (2,3]$ is identical and we have highlighted places in \Cref{appendix:bootstrap:Uone} where one should pay attention to the value of $\tau$.
The conditional Gaussian approximation bound analogous to \eqref{eq:bootstrap:X:boot:intermediate} gives that for some absolute constant $\tilde C_1' > 0$,
\begin{align*}
    &\;\msup_{t \in \R} \big| \P( u_2(\tilde X^{(b)}) \leq t \,|\, X ) - \P( u_2(\tilde Z^{(b)}) \leq t \,|\, X ) \big|
    \\
    &\hspace{8em}\;\leq\;
    \tilde C_1' \Bigg( 
        \mfrac{ 
            n^{-7/2}
            \mean\big[ \big| ( X_1^{(b)} - \bar X)^\top ( X_2^{(b)} - \bar X) \big|^3 \,\big|\, X  \big]
        }{
            n^{-3}
            \Tr(\hat \Sigma^2)^{3/2} 
        } 
    \Bigg)^{\frac{1}{7}}
    \;.
\end{align*}
By using a modified normalising constant 
\begin{align*}
    \tilde a_n \;\coloneqq&\;
    \Big( 
        \tau^2 \mfrac{d^{1/2}}{n^{1/2}}
    \Big)^{-1}
    \;,
\end{align*}
where the $\| \mu \|_{l_2}$ term is omitted from $a_n$, we obtain a bound analogus to \eqref{eq:bootstrap:conditional:gaussian:bound} that 
for some absolute constant  $\tilde C_2' > 0$, almost surely
\begin{align*}
    &\;\sup_{t \in \R} \big| \P( u_2(\tilde X^{(b)}) \leq t \,|\, X ) - \P( u_2(\tilde Z^{(b)}) \leq t \,|\, X ) \big|
    \\
    &\;\leq\;
    \tilde C_2' \Bigg( 
        \mfrac{ 
            \tilde a_n^3 \tau^6 
            n^{-7/2} \, n^{-2} \sum_{i, j \leq n} 
            \, \big| U_i^\top   \, U_j + \Theta( \frac{d}{n} ) \big|^3 
            +
            O_\P(n^{-7/2} )
        }{
            \tilde a_n^3 \tau^6 n^{-3} 
            \Big( n^{-2}  \sum_{i,j \leq n} ( U_i^\top U_j + \Theta( \frac{d}{n} ))^2 \Big)^{3/2}
            +
            O_\P( n^{-3} )
        } 
    \Bigg)^{\frac{1}{7}}
    \;.
\end{align*}
Similarly, by introducing a modified normalising constant
\begin{align*}
    \tilde b_n \;\coloneqq&\; (\tau^2 d^{1/2})^{-1}\;,
\end{align*}
we obtain a bound analogous to \eqref{eq:bootstrap:conditional:gaussian:bound:two} that for some  absolute constant  $\tilde C_3' > 0$, almost surely
\begin{align*}
    &\;\sup_{t \in \R} \big| \P( u_2(\tilde X^{(b)}) \leq t \,|\, X ) - \P( u_2(\tilde Z^{(b)}) \leq t \,|\, X ) \big|
    \\
    &\;\leq
    \tilde C_3 \Bigg( 
        \mfrac{ 
            \frac{\tilde a_n^3}{n^{7/2}}
            \big( 
            \tilde b_n^3 \tau^6 d^{3/2} 
            + 
            \tilde b_n^3 \tau^6 \frac{d^3}{n} 
            + 
            o_\P(1) 
            \big)
            +
            O_\P( \tilde b_n^3 n^{-7/2} )
        }{
            \frac{\tilde a_n^3}{n^3}
            \big( 
                \tilde b_n^3 \tau^6 d^{3/2}
                +
                \tilde b_n^3  \tau^6 \frac{d^3}{n^{3/2}}
                + 
                o_\P( 1) 
            \big)
            +
            O_\P( b_n^3 n^{-3} )
        } 
    \Bigg)^{\frac{1}{7}}
    \;.
\end{align*}
As with \Cref{appendix:bootstrap:Uone}, by the choice of $\tilde a_n$ and $\tilde b_n$, the $o_\P$ terms are negligible in the above bound, and under the condition $d=o(n^{2/3})$, the second term in the numerator and the second term in the denominator are both negligible. This shows that the above bound is on the order $O_\P(n^{-1/14}) = o_\P(1)$. The same argument also holds for a general $\nu \in (2,3]$ under the condition $d=o(n^{2/3})$, as discussed in \Cref{appendix:bootstrap:Uone}. This establishes conditional Gaussian universality.

\vspace{.5em}

As with \Cref{appendix:bootstrap:Uone}, the next step is to study the limit of 
\begin{align*}
    u_2(\tilde Z^{(b)})
    -
    \mean[ u_2(\tilde Z^{(b)})  \,|\, X]
    \;=&\;
    \mfrac{1}{n(n-1)} \msum_{i \neq j} \eta_i^\top \hat \Sigma \eta_j
    \;.
\end{align*}
We consider the approximation of the above by
\begin{align*}
    \tilde u_2(\tilde Z^{(b)})
    \;=&\;
    \mfrac{\tau^2}{n(n-1)} \msum_{i \neq j} \eta_i^\top \eta_j\;.
\end{align*}
By the same argument as \Cref{appendix:bootstrap:Uone}, it suffices to show that 
\begin{align*}
    \mfrac{
        \mean\big[ \big(  u_2(\tilde Z^{(b)}) - \mean[ u_2(\tilde Z^{(b)}) \,|\, X ] - \tilde u_2(\tilde Z^{(b)}) \big)^2 \,\big|\, X \big]
    }{
        \Var[\tilde u_2(\tilde Z^{(b)})\,|\, X ]
    }
    \;=\; 
    \mfrac{
        \frac{2d}{n(n-1)} \| \hat \Sigma - \tau^2 I_d\|_{op}^2
    }
    {
        \frac{2 \tau^4 d}{n(n-1)}
    }
\end{align*}
is $o_\P(1)$. This holds directly from the statement \eqref{eq:bootstrap:matrix:control} proved in \Cref{appendix:bootstrap:Uone}, and therefore 
\begin{align*}
    \msup_{t \in \R} 
    \big| \P( u_2(\tilde X^{(b)}) - \mean[u_2(\tilde X^{(b)}) \,|\, X] \leq t \,|\, X) 
      - \P\big( \tilde u_2(\tilde Z^{(b)}) \leq t \,\big|\, X \big)
    \big|
    \;\xrightarrow{\P}\; 0\;.
\end{align*} 

\vspace{.5em}

The next step is to obtain asymptotic normality. Notice that for each $l \leq d$,
\begin{align*}
     \Var\Big[ 
        \mfrac{1}{n(n-1)} \msum_{i \neq j} (\eta_i)_l (\eta_j)_l
    \Big] 
    \;=&\;
    \mfrac{2\Var[(\eta_1)_1 (\eta_2)_1]}{n(n-1)}
    \;=\;
    \mfrac{2}{n(n-1)}
    \;,
\end{align*}
since the coordinates of $\eta_i$'s are i.i.d.~zero-mean Gaussians. Asymptotic normality then holds by directly applying the central limit theorem:
\begin{align*}
    \mfrac{
        \tilde u_2(\tilde Z^{(b)})
    }{
       2^{1/2} \tau^2 d^{1/2} n^{-1/2} (n-1)^{-1/2}
    }
    \;=&\;
    \mfrac{1}{\sqrt{d}}
    \msum_{l=1}^d 
    \Big( \mfrac{1}{2^{1/2}  n^{1/2}(n-1)^{1/2}} \msum_{i \neq j} (\eta_i)_l (\eta_j)_l
    \Big)
    \;\xrightarrow{d}\; 
    \cN(0,1)\;.
\end{align*}
To show consistency of $U^{(b)}_2 = u_2(\tilde X^{(b)})$, we are again left with checking the limiting variances. 
% \vspace{.5em}
% \emph{Mean agreement.} We have 
% \begin{align*}
%     \mean[ \tilde u_2(\tilde Z^{(b)}) \,|\, X ]
%     - 
%     \mean[u_2(X)]
%     \;=&\; 
%     0
%     -
%     \| \mu \|_{l_2}^2
%     \;.
% \end{align*}
% Therefore the means agree if and only if $\| \mu \|_{l_2} = o(1)$. 
% \vspace{.5em}
% \emph{Variance agreement.} 
Recall that 
\begin{align*}
    a_n 
    \;=\; 
     \Big( 
        |\tau| \,  \| \mu \|_{l_2} 
        + 
        \tau^2 \mfrac{d^{1/2}}{n^{1/2}}
    \Big)^{-1}
    \;=\;
    \Theta( ( n \, \Var[u_2(X)]  )^{-1/2} )
    \;.
\end{align*}
The difference in variance, appropriately normalised, is therefore
\begin{align*}
    n a_n^2 
    \big( 
    \Var[ \tilde u_2(\tilde Z^{(b)}) \,|\, X ]
    \,-\,
    \Var[u_2(X)] 
    \big)
    \;=&\;
    n a_n^2 \Big( \mfrac{2 \tau^4 d}{n(n-1)}
    - 
    \mfrac{4 \tau^2 \| \mu \|_{l_2}^2}{n}
    -
    \mfrac{2 \tau^4 d}{n(n-1)}
    \Big)
    \\
    \;=&\;
    \mfrac{
        4 \tau^2 \| \mu \|_{l_2}^2
    }{
        \Big(
            \tau  \| \mu \|_{l_2} 
            + 
            \tau^2 \frac{d^{1/2}}{n^{1/2}}
        \Big)^2
    }
    \;=\;
    \mfrac{
        4  \| \mu \|_{l_2}^2
    }{
        \Big(
             \| \mu \|_{l_2} 
            + 
            \tau \frac{d^{1/2}}{n^{1/2}}
        \Big)^2
    }
    \;.
\end{align*}
Therefore, the limiting variances agree if and only if $\tau d^{1/2} = \omega( n^{1/2} \| \mu \|_{l_2} )$. This gives the necessary and sufficient condition for consistency of $U_2^{(b)} = u_2(\tilde X^{(b)})$.

\section{Proof of Proposition \ref{prop:higher:u}: Distributional approximations for degree-$m$ U-statistics} 
% \ref{prop:higher:u}} 
\label{sec:proof:higher:u}

Fix $K \in \N$. Denote 
\begin{align*}
    &\,
    \mu_k \;\coloneqq\; \mean[\phi^{(K)}_k(Y_1)] \in \R\;,
    \quad 
    \bar u \;\coloneqq\; \mean[u(Y_1, \ldots, Y_m)]
    \quad\text{ and }\quad
    V^K \;\coloneqq\; \{V^K_1, \ldots, V^K_n\}\;,
    \\
    &
    \hspace{12em}
    \text{ where }\quad
    V^K_i \;\coloneqq\; (\phi^{(K)}_1(Y_i) - \mu_1, \ldots, \phi^{(K)}_K(Y_i) - \mu_K)
    \;
    .
\end{align*}
By Hoeffding's decomposition (12) of \Cref{sec:degree:m:U:stats} 
% \eqref{eq:hoeffding:decompose} 
and writing 
$j \neq M$ in short for $j \in [m] \setminus \{M\}$, we have
    \begin{align*}
        u_m(Y) - \bar u
        \;=\; 
        \mbinom{m}{M} \,  U^{(K)}_M(V^K)
        +
        \msum_{j \neq M} \mbinom{m}{j} \, U_j^{\rm H}(Y)
        +
        \mbinom{m}{M} \,
        \Big( U_M^{\rm H}(Y) -   U^{(K)}_M(V^K) \Big)
        \;.
    \end{align*}
    By \Cref{thm:VD}, 
    % \cref{thm:VD}, 
    we get that for some absolute constant $C' > 0$ and for every $t \in \R$,  
    \begin{align*}
        \Big| \P\big( & u_m(Y) - \bar u \leq t\big) 
        - 
        \P\Big( \mbinom{m}{M} \,  U_M^{(K)}(\Xi^{(K)})  \leq t\Big) \Big|
        \\
        \;\leq&\;
        C' m
        \Bigg( 
            \mfrac{ 
                \sum_{i=1}^n 
                \big\| 
                \partial_i  U_M^{(K)}
                \Big(
                    V^K_1, \ldots, V^K_{i-1}, \bzero, \xi^{(K)}_{i+1}, \ldots, \xi^{(K)}_n
                \Big) 
                V^K_i
                \big\|_{L_\nu}^\nu 
                }
            {
                \Big\|  U_M^{(K)}(V^K) \Big\|_{L_2}^\nu
            }
        \Bigg)^{\frac{1}{\nu M +1}}
        \\
        & \hspace{-.5em}
        +
        C' m
        \Bigg(
            \mfrac{
                \Big\|  
                    \sum_{j \neq M} \mbinom{m}{j} \, U_j^{\rm H}(Y)
                    +
                    \mbinom{m}{M} \,
                    \Big( U_M^{\rm H}(Y) -   U^{(K)}_M(V^K) \Big)
                 \Big\|_{L_2}^2  
            }
            {
                \Big\| \mbinom{m}{M} \,  U_M^{(K)}(V^K) \Big\|_{L_2}^2
            }
        \Bigg)^{\frac{1}{2M+1}} \hspace{-.5em}
        \eqqcolon C' m (R_1 + R_2 )\;.
    \end{align*}
    Before proceeding, first note that $ U_M^{(K)}(V^K)$ is a U-statistic with the kernel
    \begin{align*}
         u^{(K)}_M (v_1, \ldots ,v_j) 
        \;=&\; 
        \msum_{k_1, \ldots, k_m=1}^K \lambda^{(K)}_{k_1 \ldots k_m} v_{1k_1} \ldots v_{Mk_M}
        \;=\; \langle T_M, v_1 \otimes \ldots \otimes v_M \rangle
    \end{align*}
    for some deterministic tensor $T_M \in \R^{K^M}$, and where $V^K_i$'s are zero-mean. By \cref{lem:u:variance},
    \begin{align*}
        \|  U_M^{(K)}(V^K) \|_{L_2}^2 \;\geq&\; \mfrac{1}{n^M}\big\|  u_m^{(K)}(V_1^K, \ldots, V_M^K) \big\|_{L_2}^2
        \;.
    \end{align*}
    We defer to \cref{lem:u:main:Lnu} to show that for some absolute constant $C_* > 0$,
    \begin{align*}
        \Big\| 
            \partial_i  U_M^{(K)}
            \Big(
                V^K_1, \ldots, V^K_{i-1}, \bzero, \xi^{(K)}_{i+1}, \ldots, \xi^{(K)}_n
            \Big) 
            V^K_i
        \Big\|_{L_\nu}^\nu 
        \leq
        \mfrac{ C_*^m  \big\|  u_m^{(K)}(V_1^K, \ldots, V_M^K)
        \big\|_{L_\nu}^\nu}{n^{(M+1)\nu/2}}\;.
        \tagaligneq \label{eq:u:main:Lnu}
    \end{align*}
    Let $Y'_1, \ldots, Y'_m$ be i.i.d.~copies of $Y_1$. Now by the definition of $V^K_i$, the binomial theorem, the triangle inequality and Jensen's inequality, the truncation error satisfies
    \begin{align*}
        \Big\| 
        &U_M^{\rm H}(Y)  
        - 
         U_M^{(K)}(V^K) \Big\|_{L_\nu}
        \;\leq\;
        \Big\| 
        u_M^{\rm H}(Y_1, \ldots, Y_M)
        - 
         u_M^{(K)}(V^K_1, \ldots, V^K_M)
        \Big\|_{L_\nu}
        \\
        =&
        \Big\| 
        \msum_{r=0}^M (-1)^{M-r} \msum_{1 \leq l_1 < \ldots < l_r \leq M }  
        \mean\big[ u(Y_{l_1}, \ldots, Y_{l_r}, Y'_1, \ldots, Y'_{m-r}) \big| Y_1, \ldots, Y_M \big]
        \\
        &\; 
        - 
        \msum_{k_1, \ldots, k_m=1}^K 
        \, \lambda^{(K)}_{k_1 \ldots k_m} 
        (\phi^{(K)}_{k_1} (Y_i) - \mu_{k_1} )
        \ldots 
        (\phi^{(K)}_{k_M} (Y_M) - \mu_{k_M} )
        \mu_{k_{M+1}}
        \ldots 
        \mu_{k_m}
        \Big\|_{L_\nu}
        \\
        =&
        \Big\| 
        \msum_{r=0}^M (-1)^{M-r} \msum_{1 \leq l_1 < \ldots < l_r \leq M }  
        \mean\big[ u(Y_{l_1}, \ldots, Y_{l_r}, Y'_1, \ldots, Y'_{m-r}) \big| Y_1, \ldots, Y_M \big]
        \\
        &\; 
        - 
        \msum_{r=0}^M (-1)^{M-r} 
        \msum_{1 \leq l_1 < \ldots < l_r \leq M }  
        \\
        &\hspace{8em}
        \Big( \msum_{k_1, \ldots, k_m=1}^K 
        \, \lambda^{(K)}_{k_1 \ldots k_m} 
        \phi_{k_{l_1}}(Y_{l_1}) \ldots \phi_{k_{l_r}}(Y_{l_r})
        \, 
        \mprod_{\substack{ k' \in [M] \\ k' \neq l_s \, \forall s}}
        \mu_{k'}
        \Big)
        \Big\|_{L_\nu}
        \\
        \leq&
        \msum_{r=0}^M \mbinom{M}{r}
        \Big\| 
        u(Y_{l_1}, \ldots, Y_{l_r}, Y'_1, \ldots, Y'_{m-r}) 
        \\
        &\hspace{6em}
        -
        \msum_{k_1, \ldots, k_m=1}^K 
        \Big(
        \, \lambda^{(K)}_{k_1 \ldots k_m} 
        \phi_{k_{l_1}}(Y_{l_1}) \ldots \phi_{k_{l_r}}(Y_{l_r})
        \mprod_{\substack{ k' \in [M] \\ k' \neq l_s \, \forall s}}
        \phi_{k'}(Y'_{k'}) 
        \Big)
        \Big\|_{L_\nu}
        \\
        \;\overset{(a)}{=}&\;  2^M  \epsilon_{K;\nu}\;.
    \end{align*}
    In $(a)$ above, we have noted that for any permutation $\pi$ on $\{1,\ldots, m\}$, the $L_\nu$ approximation error satisfies
    \begin{align*}
        \epsilon_{K;\nu}
        \;=&\;
         \big\| \msum_{k_1, \ldots, k_m=1}^K \lambda^{(K)}_{k_1 \ldots k_m} \phi^{(K)}_{k_1}(Y_1) \times \ldots \times \phi^{(K)}_{k_m}(Y_m)
        -
        u(Y_1, \ldots, Y_m) \big\|_{L_\nu}
        \\
        \;\overset{(b)}{=}&\;
        \big\| \msum_{k_1, \ldots, k_m=1}^K \lambda^{(K)}_{k_1 \ldots k_m} \phi^{(K)}_{k_1}(Y_{\pi(1)}) \times \ldots \times \phi^{(K)}_{k_m}(Y_{\pi(m)})
        -
        u(Y_{\pi(1)}, \ldots, Y_{\pi(m)}) \big\|_{L_\nu}
        \\
        \;\overset{(c)}{=}&\;
        \big\| \msum_{k_1, \ldots, k_m=1}^K \lambda^{(K)}_{k_1 \ldots k_m} \phi^{(K)}_{k_{\pi^{-1}(1)}}(Y_1) \times \ldots \times \phi^{(K)}_{k_{\pi^{-1}(m)}}(Y_m)
        -
        u(Y_{\pi(1)}, \ldots, Y_{\pi(m)}) \big\|_{L_\nu}
        \\
        \;\overset{(d)}{=}&\;
        \big\| \msum_{k_1, \ldots, k_m=1}^K \lambda^{(K)}_{k_1 \ldots k_m} \phi^{(K)}_{k_{\pi^{-1}(1)}}(Y_1) \times \ldots \times \phi^{(K)}_{k_{\pi^{-1}(m)}}(Y_m)
        -
        u(Y_1, \ldots, Y_m) \big\|_{L_\nu}\;,
    \end{align*}
    where we have used that $Y_1, \ldots, Y_m$ are i.i.d.~in $(b)$, the commutativity of scalar product in $(c)$ and that $u$ is symmetric in $(d)$. Combining the bounds above, we get that there is some absolute constant $C'' > 0$ such that 
    \begin{align*}
        R_1 
        &\;\leq\;
        C''
        n^{-\frac{\nu-2}{2(\nu M + 1)}} 
        \Bigg( 
          \mfrac{ 
            \Big\| 
                 u_M^{(K)} \big( V^K_1,\ldots, V^K_M\big)
            \Big\|_{L_\nu} }
          {
            \Big\| 
                 u_M^{(K)} \big( V^K_1,\ldots, V^K_M \big) 
            \Big\|_{L_2}
          }
        \Bigg)^{\frac{\nu}{\nu M +1}} 
        \\
        &\;\leq\; 
        C''
        n^{-\frac{\nu-2}{2(\nu M + 1)}} 
        \Bigg( 
          \mfrac{ 
            \Big\| 
                u_M^{\rm H} \big( Y_1,\ldots, Y_M \big)
            \Big\|_{L_\nu} + (2^M )  \, \epsilon_{K;\nu}
          }
          {
            \Big\| 
                u_M^{\rm H} \big( Y_1,\ldots, Y_M \big) 
            \Big\|_{L_2}
            - (2^M ) \, \epsilon_{K;\nu}
          }
        \Bigg)^{\frac{\nu}{\nu M +1}} \;.
    \end{align*}
    Denote $\sigma_j^2 \coloneqq \Var\, \mean[ u(Y_1, \ldots, Y_m) \, |\, Y_1, \ldots, Y_j ]$. By the truncation error bound again, we have
    \begin{align*}
        R_2 \;\leq&\; 
        \Bigg(
            \mfrac{ 
                2 
                \Big\|  
                    \sum_{j \neq M} \mbinom{m}{j} \, U_j^{\rm H}(Y)
                \Big\|_{L_2}^2 
                +
                \mbinom{m}{M}^2 
                2^{M+1} \epsilon_{K;\nu}^2
            }
            {
                \Big\| \mbinom{m}{M} \, U_M^{\rm H}(Y) \Big\|_{L_2}^2
                -
                \mbinom{m}{M}^2 
                2^{M+1} \epsilon_{K;\nu}^2
            }
        \Bigg)^{\frac{1}{2M+1}}
        \\
        \;\overset{(a)}{=}&\; 
        \Bigg(
            \mfrac{ 
                2 
                \sum_{j \neq M} \mbinom{m}{j}^2 
                \big\|  
                    \, U_j^{\rm H}(Y)
                \big\|_{L_2}^2 
                +
                \mbinom{m}{M}^2 
                2^{M+1} \epsilon_{K;\nu}^2
            }
            {
                \mbinom{m}{M}^2 \big\| \, U^{\rm H}_M(Y) \big\|_{L_2}^2
                -
                \mbinom{m}{M}^2 
                2^{M+1} \epsilon_{K;\nu}^2
            }
        \Bigg)^{\frac{1}{2M+1}}
        \\
        \;\overset{(b)}{=}&\; 
        \Bigg(
            \mfrac{ 
                2 
                \sum_{j \neq M} \mbinom{m}{j}^2 
                \mbinom{n}{j}^{-1} \sigma^2_j
                +
                \mbinom{m}{M}^2 
                2^{M+1} \epsilon_{K;\nu}^2
            }
            {
                \mbinom{m}{M}^2 \mbinom{n}{M}^{-1} \sigma^2_M
                -
                \mbinom{m}{M}^2 
                2^{M+1} \epsilon_{K;\nu}^2
            }
        \Bigg)^{\frac{1}{2M+1}}
        \\
        \;\overset{(c)}{=}&\; 
        \Bigg(
            \mfrac{ 
                2 
                \sum_{j \neq M} \sigma^2_{m,n;j}
                +
                \mbinom{m}{M}^2 
                2^{M+1} \epsilon_{K;\nu}^2
            }
            {
                \sigma^2_{m,n;M}
                -
                \mbinom{m}{M}^2 
                2^{M+1} \epsilon_{K;\nu}^2
            }
        \Bigg)^{\frac{1}{2M+1}}
        \;.
    \end{align*}
    In $(a)$, we have used the orthogonality of the degenerate U-statistics of different degrees by \cref{lem:u:orthogonal}; in $(b)$, we have used the variance formula of $U_j^{\rm H}(Y)$ in \cref{lem:u:variance}; in $(c)$, we have plugged in the definition of $\sigma^2_{m,n;j}$. Combining the bounds and taking $K \rightarrow \infty$, we get 
    \begin{align*}
        \Big| \P\big( u_m(Y) - & \bar u \leq t\big) 
         - 
        \lim_{K \rightarrow \infty}
        \P\Big( \mbinom{m}{M} \, \tilde U_M^{(K)}(\Xi^{(K)})  \leq t\Big) \Big|
        \\
        \;\leq&\;
        C m
        n^{-\frac{\nu-2}{2(\nu M + 1)}} 
        \Bigg( 
          \mfrac{ 
            \Big\| 
                u_M^{\rm H} \big( Y_1,\ldots, Y_M \big)
            \Big\|_{L_\nu} 
          }
          {
            \Big\| 
                u_M^{\rm H} \big( Y_1,\ldots, Y_M \big) 
            \Big\|_{L_2}
          }
        \Bigg)^{\frac{\nu}{\nu M +1}}
        +
        C m
        \Bigg(
            \mfrac{
                \sum_{j \neq M} \sigma^2_{m,n;j}
            }
            {
                \sigma^2_{m,n;M}
            }
        \Bigg)^{\frac{1}{2M+1}}
        \\
        \;=&\;
        C 
        \Big(
        m
        n^{-\frac{\nu-2}{2(\nu M + 1)}} 
        \tilde \beta_{M,\nu}^{\frac{\nu}{\nu M +1}}
        +
        m
        \rho_{m,n;M}^{\frac{1}{2M+1}}
        \Big)
    \end{align*}
    for some absolute constant $C > 0$.  This gives the finite-sample bound for the first asymptotic statement of \Cref{prop:higher:u}. To obtain the Gaussian approximation bound, first write $\Phi$ as the c.d.f.~of $\cN(0,1)$. We can apply \Cref{prop:gaussian} to get that for every $t \in \R$,
    \begin{align*}
        \Big| \P\big( & u_m(Y) - \mean[ u(Y_1, \ldots, Y_m) ] \leq t\big) 
        - 
        \lim_{K \rightarrow \infty} 
        \Phi\Big( \mbinom{m}{M}^{-1} \Var\big[   U_M^{(K)}(\Xi^{(K)}) \big]^{-1/2} \, t \Big)\Big|
        \\
        \;\leq&\; 
        C m
        \Big(
        n^{-\frac{\nu-2}{2(\nu M + 1)}} 
        \beta_{M,\nu}^{\frac{\nu}{\nu M +1}}
        +
        \rho_{m,n;M}^{\frac{1}{2M+1}}
        \Big)
        + 
        \lim_{K \rightarrow \infty}
        \Big(\mfrac{4m-4}{3m} \, \big| \Kurt\big[ U_M^{(K)}(\Xi^{(K)})\big] \big| \Big)^{1/2}\;.
    \end{align*}
    By continuity, we can move $\lim_{K \rightarrow \infty}$ outside the absolute value, and by rescaling $t$ appropriately, we get that the same upper bound holds for 
    \begin{align*}
        &\;
        \lim_{K \rightarrow \infty} 
        \Big| 
            \P\Big( 
                \mfrac{ u_m(Y) - \mean[ u(Y_1, \ldots, Y_m) ] }{\binom{m}{M}  \Var\big[   U_M^{(K)}(\Xi^{(K)}) \big]^{1/2}}
            \;\leq\; t
            \Big) 
            - 
            \Phi(t) 
        \Big|
        \\
        &\;=\;
        \Big| 
            \lim_{K \rightarrow \infty} 
            \P\Big( 
                \mfrac{ u_m(Y) - \mean[ u(Y_1, \ldots, Y_m) ] }{\binom{m}{M}  \Var\big[   U_M^{(K)}(\Xi^{(K)}) \big]^{1/2}}
            \;\leq\; t
            \Big) 
            - 
            \Phi(t) 
        \Big|
        \;.
    \end{align*}
    Therefore under the additional condition $\lim_{K \rightarrow \infty}  \Kurt\big[ U_M^{(K)}(\Xi^{(K)})\big] \xrightarrow{n \rightarrow \infty} 0$, we obtain the desired asymptotic statement.

    \qed

    \vspace{.5em}

\begin{lemma} \label{lem:u:main:Lnu} \eqref{eq:u:main:Lnu} holds.  
\end{lemma}

\begin{proof} Write $ W_j = V^K_j $ for $j < i$, $ W_i = 0$ and $W_j = \xi^{(K)}_j$ for $j > i$. Then we can express
    \begin{align*}
        \big\| \partial_i \,  U_M^{(K)}(  &V^K_1, \ldots,  V^K_{i-1}, \bzero, \xi^{(K)}_{i+1}, \ldots, \xi^{(K)}_n) X_i \big\|_{L_\nu}^\nu
        \;=\;
        \big\| \partial_i \,  U^{(K)}_M(  W_1, \ldots,  W_n) 
        V^K_i \big\|_{L_\nu}^\nu
        \\
        =&\;
        \Big( \mfrac{(n-m)!}{n!} \Big)^{\nu}
        m \,
        \bigg\|
            \msum_{\substack{j_1, \ldots, j_{m-1} \\ \text{\rm distinct and in } [n] \setminus \{i\} }} 
        \, \langle T_m,  V^K_i \otimes W_{j_1} \otimes \ldots \otimes \ldots\otimes  W_{j_{m-1}} \rangle
        \bigg\|_{L_\nu}^\nu\;,
    \end{align*}
    where we used the symmetry of $T_m$ in the last equality. Now by the independence of $W_j$'s and $V^K_i$ and noting that $\mean\, W_i = 0$ and $\nu \in (2,3]$, we can use \cref{lem:gaussian:linear:moment} repeatedly to get
    \begin{align*}
        \big\| \partial_i \,  U_M^{(K)}(  &V^K_1, \ldots,  V^K_{i-1}, \bzero, \xi^{(K)}_{i+1}, \ldots, \xi^{(K)}_n) X_i \big\|_{L_\nu}^\nu
        \\
        \leq&\;
        (C')^m
        \Big( \mfrac{(n-m)!}{n!} \Big)^{\nu}
        m \,
        \bigg\|
            \msum_{\substack{j_1, \ldots, j_{m-1} \\ \text{\rm distinct and in } [n] \setminus \{i\} }} 
        \, \langle T_m,  V^K_i \otimes  V^K_{j_1} \otimes \ldots \otimes  V^K_{j_{m-1}} \rangle
        \bigg\|_{L_\nu}^\nu
    \end{align*}
    for some absolute constant $C' > 0$.
    Denote $\| \argdot \|_{L_\nu | i} \coloneqq \mean[ \argdot | V^K_i ]$,  $V^{(-i)}_j \coloneqq V^K_j$ for $j < i$, $V^{(-i)}_j \coloneqq V^K_{j+1}$ for $j \geq i$ and  $V^{(-i)}=(V^{(-i)}_j)_{j \leq n-1}$. We can further express the above as
    \begin{align*}
        &\; 
        (C')^m
        \mfrac{m}{n^\nu}
        \mean \bigg[ 
        \Big\| 
        \mfrac{(n-m)!}{(n-1)!}
            \msum_{\substack{j_1, \ldots, j_{m-1} \\ \text{\rm distinct and in } [n] \setminus \{i\} }} 
        \, \langle T_m,  V^K_i \otimes  V^K_{j_1} \otimes \ldots \otimes  V^K_{j_{m-1}} \rangle
        \Big\|_{L_\nu | i}^\nu 
        \bigg]
        \\
        &\quad
        \;\eqqcolon\;
        (C')^m
        \mfrac{m}{n^\nu}
        \mean \bigg[ 
        \Big\| 
        u^{(i)}_{m-1}\big( V^{(-i)} \big)
        \Big\|_{L_\nu | i}^\nu 
        \bigg]\;.
    \end{align*}
    Conditioning on $V^K_i$, $u^{(i)}_{m-1}(V^{(-i)})$ is now a degree $m-1$ degenerate U-statistic of $n-1$ i.i.d.~zero-mean random vectors, $V^{(-i)}_1, \ldots, V^{(-i)}_{n-1}$. By a standard moment bound on degenerate U-statistics (Theorem 1.1 of \cite{ferger1996moment}), there is some absolute constant $C'' > 0$ such that 
    \begin{align*}
        &\Big\| 
        u^{(i)}_{m-1}\big( V^{(-i)} \big)
        \Big\|_{L_\nu | i}
        \\
        &\leq
        (C'')^{m-1}  \,
        \mbinom{n}{m-1}^{-1} n^{\frac{m-1}{2}}
        \Big( \mprod_{l=1}^{m-1} \mfrac{2}{(l-1) \nu + 2}  \Big)
        \big\|
        \big\langle T_m, V^K_i \otimes V^{(-i)}_1 \otimes \ldots \otimes V^{(-i)}_{m-1} \big\rangle \big\|_{L_\nu | i}
        \\
        &\leq
        (C'')^{m-1} \,
        n^{\frac{m-1}{2}}
        \mfrac{(m-1)^{m-1}}{n^{m-1}} \Big( \mprod_{l=1}^{m-1} \mfrac{\nu}{(l-1) \nu + \nu}  \Big)
        \big\| \big\langle T_m, V^K_i \otimes V^{(-i)}_1 \otimes \ldots \otimes V^{(-i)}_{m-1} \big\rangle  \big\|_{L_\nu | i}
        \\
        &= 
        (C'')^{m-1} \,
        n^{-\frac{m-1}{2}}
        \mfrac{(m-1)^{m-1}}{(m-1)!} 
        \big\| \big\langle T_m, V^K_i \otimes V^{(-i)}_1 \otimes \ldots \otimes V^{(-i)}_{m-1} \big\rangle  \big\|_{L_\nu | i}
        \\
        &\leq
        (C'')^{m-1} \,
        n^{-\frac{m-1}{2}}
        e^{m-1}
        \big\| \big\langle T_m, V^K_i \otimes V^{(-i)}_1 \otimes \ldots \otimes V^{(-i)}_{m-1} \big\rangle \big\|_{L_\nu | i}
        \qquad \qquad \text{ almost surely }\,,
    \end{align*}
    where we have used Stirling's approximation of the factorial in the last line. Combining the two bounds finishes the proof.  
\end{proof}

\section{Phase transition and mixture limit of MMD under imbalanced and high-dimensional data} \label{appendix:MMD}

\Cref{appendix:MMD:proof} presents the proof of \Cref{prop:MMD:imbalanced}, which provides a distributional approximation analogous to that of \Cref{prop:higher:u} and demonstrates the phase transition of MMD in high dimensions and under imbalanced sample sizes. \Cref{appendix:MMD:mean:shift} considers a high-dimensional distribution test, in which a MMD with the linear kernel can have a mixture limit specific to the data imbalance setup.

\subsection{Proof of \Cref{prop:MMD:imbalanced}: Phase transition of MMD} \label{appendix:MMD:proof}

The proof is exactly analogous to that of \Cref{sec:proof:higher:u}, except that we only need to deal with two polynomial components. We focus on highlighting the differences. We seek to study 
\begin{align*}
    U_{\rm MMD} 
    \;\coloneqq&\; 
    \mfrac{1}{n_1(n_1-1)} \msum_{i \neq j}^{n_1} \kappa(X_i, X_j)
    -
    \mfrac{2}{n_1 n_2} \msum_{i \leq n_1, j \leq n_2} \kappa(X_i, Y_j)
    \\
    &\;
    +
    \mfrac{1}{n_2 (n_2-1)} \msum_{i \neq j}^{n_2} \kappa(Y_i, Y_j)
    \;.
\end{align*}
\Cref{assumption:L_nu:MMD} says that there exist some choices of $(\phi^P_k, \lambda^P_{k,k'})$, $(\phi^{PQ}_k, \lambda^{PQ}_{k,k'})$ and  $(\phi^{Q}_k, \lambda^{Q}_{k,k'})$ such that
\begin{align*}
    &\;
    \Big\|
    \msum_{k, k' \leq K}
    \lambda^P_{k,k'} 
    \phi^{P}_k(X_1) 
    \phi^{P}_{k'}(X_2)
    -
    \kappa(X_1, X_2)
    \Big\|_{L_\nu}
    \;\xrightarrow{K \rightarrow \infty}\; 
    0\;,
    \\
    &\;
    \Big\|
    \msum_{k, k' \leq K}
    \lambda^{PQ}_{k,k'} 
    \phi^{PQ}_k(X_1) 
    \phi^{PQ}_{k'}(X_2)
    -
    \kappa(X_1, Y_1)
    \Big\|_{L_\nu}
    \;\xrightarrow{K \rightarrow \infty}\; 
    0\;,
    \\
    &\;
    \Big\|
    \msum_{k, k' \leq K}
    \lambda^Q_{k,k'} 
    \phi^{Q}_k(X_1) 
    \phi^{Q}_{k'}(X_2)
    -
    \kappa(Y_1, Y_2)
    \Big\|_{L_\nu}
    \;\xrightarrow{K \rightarrow \infty}\; 
    0\;.
\end{align*}
Note that in the full statement of \Cref{assumption:L_nu:MMD}, the $\lambda$ and $\phi$'s above also depend on $K$, but we omit the dependence for simplicity of presentation. To study $U_{\rm MMD}$, by the same truncation argument as in \Cref{sec:proof:higher:u}, it suffices to study the distributional approximation of
\begin{align*}
    U_{\rm MMD}^{(K)}
    \;\coloneqq&\;
    \msum_{k, k' \leq K} 
    U_{k,k'}^{(K)}
    \;,
    \\
    &\quad \text{ where }
    \quad
    U_{k,k'}^{(K)}
    \;\coloneqq\;
    \mfrac{\lambda^P_{k,k'}}{n_1(n_1-1)} \msum_{i \neq j}^{n_1} 
    \phi^P_k(X_i) \phi^P_k(X_j)
    \\
    &\hspace{9em}
    -
    \mfrac{2 \lambda^{PQ}_{k,k'}}{n_1 n_2} \msum_{i \leq n_1, j \leq n_2} 
     \phi^{PQ}_k(X_i) \phi^{PQ}_{k'}(Y_j)
    \\
    &\hspace{9em}
    +
    \mfrac{\lambda^Q_{k,k'} }{n_2 (n_2-1)} \msum_{i \neq j}^{n_2} 
    \phi^Q_k(Y_i) \phi^Q_{k'}(Y_j)
    \;.
\end{align*}
This is a multilinear polynomial in the $n_1+n_2$ independent $\R^{2K}$ random vectors,
\begin{align*}
    \big( 
        \phi^P_1(X_i)
        \,,\, 
        \ldots
        \,,\, 
        \phi^P_K(X_i) 
        \,,\,
        \phi^{PQ}_1(X_i)
        \,,\,
        \ldots 
        \,,\,
        \phi^{PQ}_K(X_i)
    \big) 
    \qquad 
    &&\text{ for } i \leq n_1 
    \;,
    \\
    \big( 
        \phi^Q_1(Y_i)
        \,,\, 
        \ldots
        \,,\, 
        \phi^Q_K(Y_i) 
        \,,\,
        \phi^{PQ}_1(Y_i)
        \,,\,
        \ldots 
        \,,\,
        \phi^{PQ}_K(Y_i)
    \big) 
    \qquad 
    &&\text{ for } i \leq n_2 
    \;.
\end{align*}
Since both $n_1, n_2 \rightarrow \infty$, this allows us to apply our universality result (\Cref{thm:main}) and variance domination (\Cref{thm:VD})  in an analogous way as we did for degree-$m$ U-statistics in \Cref{sec:proof:higher:u}, before using the truncation argument to rewrite the moment terms in terms of the original kernel $\kappa$. It remains to decompose $U^{(K)}_{\rm MMD}$ into the two components that respectively give rise to the variances $\sigma_1^2$ and $\sigma_2^2$. To this end, we denote
\begin{align*}
    &\;
    \mu_k^P
    \coloneqq
    \mean[\phi^P_k(X_1)]
    ,
    \quad
    \mu_k^{PQX}
    \coloneqq
    \mean[\phi^{PQ}_k(X_1)]
    ,
    \quad 
    \mu_k^{PQY}
    \coloneqq
    \mean[\phi^{PQ}_k(Y_1)]
    ,
    \quad 
    \mu_k^{Q}
    \coloneqq
    \mean[\phi^{Q}_k(Y_1)]
    ,
\end{align*}
as well as the centred random variables 
\begin{align*}
    \Phi^{PX}_{ki} 
    \;\coloneqq&\; 
    \phi^P_k(X_i) - \mu^P_k\;,
    &\,&& 
    \Phi^{PQX}_{ki} 
    \;\coloneqq&\; 
    \phi^{PQ}_k(X_i) - \mu^{PQX}_k
    \;,
    \\
    \Phi^{PQY}_{ki} 
    \;\coloneqq&\; 
    \phi^{PQ}_k(Y_i) - \mu^{PQY}_k\;,
    &\,&& 
    \Phi^{QY}_{ki} 
    \;\coloneqq&\; 
    \phi^{Q}_k(Y_i) - \mu^{Q}_k
    \;.
\end{align*}
This allows us to write 
\begin{align*}
    U_{k,k'}^{(K)}
    \;=&\;
    U_{k,k'}^{(K;1)}
    +
    U_{k,k'}^{(K;2)}
    \;,
\end{align*}
where we have denoted the first-order component as
\begin{align*}
    U_{k,k'}^{(K;1)}
    \;\coloneqq&\;
    \mfrac{2 \lambda^P_{k,k'}}{n_1} \msum_{i \leq n_1}
    \mu^P_k \, \Phi^{PX}_{ki}
    -
    \mfrac{2 \lambda^{PQ}_{k,k'}}{n_1} \msum_{i \leq n_1} 
     \mu^{PQY}_k \, \Phi^{PQX}_{ki}
    \\
    &\;
    +
    \mfrac{2 \lambda^Q_{k,k'}}{n_2} \msum_{i \leq n_2}
    \mu^Q_k \, \Phi^{QY}_{ki}
    -
    \mfrac{2 \lambda^{PQ}_{k,k'}}{n_2} \msum_{i \leq n_2} 
     \mu^{PQX}_k \, \Phi^{PQY}_{ki}
\end{align*}
and the second-order component as 
\begin{align*}
    U_{k,k'}^{(K;2)}
    \;\coloneqq&\;
    \mfrac{\lambda^P_{k,k'}}{n_1(n_1-1)} \msum_{i \neq j}^{n_1} 
    \Phi^{PX}_{ki} \Phi^{PX}_{k'j}
    -
    \mfrac{2 \lambda^{PQ}_{k,k'}}{n_1 n_2} \msum_{i \leq n_1, j \leq n_2} 
     \Phi^{PQX}_{ki} \Phi^{PQY}_{k'j}
    \\
    &\;
    +
    \mfrac{\lambda^Q_{k,k'} }{n_2 (n_2-1)} \msum_{i \neq j}^{n_2} 
    \Phi^{QY}_{ki} \Phi^{QY}_{k'j}
    \;.
\end{align*}
The first-order component is a sum of two independent empirical averages. The corresponding variance and $L_\nu$ stability terms, which appear in the universality upper bound of \Cref{thm:main}, can be directly computed, and are equivalent to $\sigma_1^2$ and $M_{\nu;1}^\nu$ under $K \rightarrow \infty$ by the truncation argument of \Cref{sec:proof:higher:u}. Meanwhile, the second-order component involve three dependent sums, but the variances do simplify by noting that 
\begin{align*}
    \Cov\big[ 
        \Phi^{PX}_{k1} \Phi^{PX}_{k2} \,,\, 
        \Phi^{PQX}_{k1} \Phi^{PQY}_{k2}
     \big]
    \;=\;
    \mean\big[ \Phi^{PX}_{k2} \Phi^{PX}_{k1} \Phi^{PQX}_{k1} \Phi^{PQY}_{k2}  \big]
    \;=\; 
    0
\end{align*}
since the zero-mean variables $\Phi^{PQY}_{k2}$ and $\Phi^{PX}_{k2}$ are independent of the other variables, and similarly 
\begin{align*}
    \Cov\big[ 
        \Phi^{QY}_{k1} \Phi^{QY}_{k2} \,,\, 
        \Phi^{PQX}_{k1} \Phi^{PQY}_{k2}
     \big]
     \;=\;
     0\;.
\end{align*}
Therefore the variance of $\sum_{k, k' \leq K} U^{(K;2)}_{k,k'}$ as $K \rightarrow \infty$ gives $\sigma_2^2$. Meanwhile, the $L_\nu$-th moment term of the second component corresponds to $M_{\nu;2}$ as $K \rightarrow \infty$ by noting that only $L_\nu$ terms only appear in the upper bound and by applying a triangle inequality to handle the three sums in $U^{(K;2)}_{k,k'}$ separately. In summary, under the moment ratio conditions that $\frac{M_{\nu;1}}{\sigma_1}$ and $\frac{M_{\nu;2}}{\sigma_2}$ are $O(1)$, we can apply the universality bound and variance domination in the same way as \Cref{sec:proof:higher:u} to show that if $\sigma_1 / \sigma_2 = \omega(1)$, the first-order component dominates, and if $\sigma_1 / \sigma_2 = o(1)$, the second-order component dominates. 

In the case $\sigma_1 / \sigma_2 = \omega(1)$, the approximating first-order Gaussian polynomial is automatically Gaussian, and therefore the asymptotic normality statement holds under centering and rescaling by $\sigma_1^{-1}$. This proves (i).

In the case $\sigma_1 / \sigma_2 = o(1)$, the approximating second-order Gaussian polynomial is a Gaussian quadratic form given by 
\begin{align*}
    \tilde U_{K, n_1, n_2}
    \;\coloneqq\;
    \msum_{k, k' \leq K} 
    &\;
    \Big(
    \mfrac{\lambda^P_{k,k'}}{n_1(n_1-1)} \msum_{i \neq j}^{n_1} 
    Z^{PX}_{ki} Z^{PX}_{k'j}
    -
    \mfrac{2 \lambda^{PQ}_{k,k'}}{n_1 n_2} \msum_{i \leq n_1, j \leq n_2} 
     Z^{PQX}_{ki} Z^{PQY}_{k'j}
    \\
    &\quad
    +
    \mfrac{\lambda^Q_{k,k'} }{n_2 (n_2-1)} \msum_{i \neq j}^{n_2} 
    Z^{QY}_{ki} Z^{QY}_{k'j}
    \Big) 
    \;,
\end{align*}
where $Z^{PX}_{ki}$, $Z^{PQX}_{ki}$, $Z^{PQY}_{ki}$ and $Z^{QY}_{ki}$ are all centred univariate Gaussian variables, such that the covariance matrix of the $\R^{2K}$ random vector  
\begin{align*}
   ( Z^{PX}_{1i}, \ldots,  Z^{PX}_{Ki}, Z^{PQX}_{1i}, \ldots,  Z^{PQX}_{Ki})
\end{align*}
matches that of
\begin{align*}
    ( \phi^P_1(X_1), \ldots, \phi^P_K(X_1), \phi^{PQ}_1(X_1), \ldots, \phi^{PQ}_K(X_1)  )\;,
\end{align*}
and the covariance matrix of the $\R^{2K}$ random vector  
\begin{align*}
   ( Z^{QY}_{1i}, \ldots,  Z^{QY}_{Ki}, Z^{PQY}_{1i}, \ldots,  Z^{PQY}_{Ki})
\end{align*}
matches that of
\begin{align*}
    ( \phi^Q_1(Y_1), \ldots, \phi^Q_K(Y_1), \phi^{PQ}_1(Y_1), \ldots, \phi^{PQ}_K(Y_1)  )
    \;.
\end{align*}
This finishes the proof of (ii) by identifying 
\begin{align*}
    U_{K, n_1, n_2} 
    \;\coloneqq\; 
    \mfrac{\tilde U_{K, n_1, n_2} - q_{\rm MMD}}{\sigma_2}
    \;.
\end{align*}

\subsection{Mixture limit in a high-dimensional test with mean and variance shifts}  \label{appendix:MMD:mean:shift}

Consider the high-dimensional distribution test where $P = \cN(0, \Sigma_P)$ and $Q = \cN(\mu, \Sigma_Q)$. The two-sample distribution test becomes a test of the null hypothesis test $H_0: \mu = 0, \Sigma_P = \Sigma_Q$. Under the linear kernel $\kappa(x,y) = x^\top y$, the MMD test statistic is given by
\begin{align*}
    U_{\rm MMD} 
    \;\coloneqq&\; 
    \mfrac{1}{n_1(n_1-1)} \msum_{i \neq j}^{n_1} X_i^\top X_j
    -
    \mfrac{2}{n_1 n_2} \msum_{i \leq n_1, j \leq n_2} X_i^\top Y_j
    \\
    &\;
    +
    \mfrac{1}{n_2 (n_2-1)} \msum_{i \neq j}^{n_2} Y_i^\top Y_j
    \;.
\end{align*}
By \Cref{prop:MMD:imbalanced}, the asymptotic behaviour of $U_{\rm MMD}$ is captured by the variances 
\begin{align*}
    \sigma_1^2 
    \;=&\;
    \mfrac{4 \,
        \Var[  X_1^\top \mu ]
    }{n_1}
    +
    \mfrac{4 \,
        \Var [ 
            Y_1^\top \mu
        ]
    }{n_2}
    \\
    \;=&\;
    \mfrac{4 \mu^\top \Sigma_P \mu}{n_1}
    +
    \mfrac{4 \mu^\top \Sigma_Q \mu}{n_2}
    \;,
    \\
     \sigma_2^2 
    \;=&\;
        \mfrac{2 \, \Var[X_1^\top X_2]}{n_1 (n_1 -1 )}
        +
        \mfrac{2 \, \Var[Y_1^\top Y_2]}{n_2 (n_2 -1 )}
        +
        \mfrac{4 \, \Var[X_1^\top Y_1]}{n_1 n_2}
    \\
    \;=&\;
        \mfrac{2 \Tr(\Sigma_P^2)}{n_1 (n_1 -1 )}
        +
        \mfrac{2 \Tr(\Sigma_Q^2)}{n_2 (n_2 -1 )}
        +
        \mfrac{4 \Tr(\Sigma_P \Sigma_Q)}{n_1 n_2}
    \;.
\end{align*}
Suppose the data imbalance is such that 
\begin{align*}
    n_1 \;=\; o(n_2)\;,
\end{align*}
and consider the behaviour of $U_{\rm MMD}$ under the alternative $Q$ with 
\begin{align*}
    \mu \;=\; \mfrac{1}{\sqrt{d}} (1, \ldots, 1)^\top \;\in\; \R^d\;.
    \tagaligneq \label{eq:MMD:toy:mu}
\end{align*}
 Notice that $U_{\rm MMD}$ is a non-degenerate U-statistic in the classical sense, and the classical theory of U-statistics \cite{serfling1980approximation} predicts that $U_{\rm MMD}$ should be asymptotically normal. Our \Cref{prop:MMD:imbalanced}, which accommodates both the fixed-$d$ case and the growing-$d$ case, implies the following:
\begin{itemize}
    \item Classically for $d$ fixed, 
     $\sigma_1^2 = \Theta(n_1^{-1})$ whereas 
    $\sigma_2^2 = O(n_1^{-2})$. This implies $\sigma_1 / \sigma_2 = \omega(1)$ and therefore by \Cref{prop:MMD:imbalanced}(i), $U_{\rm MMD}$ is indeed asymptotically normal (upon an appropriate rescaling and centering). This asymptotic normality can be leveraged to study the test power of $U_{\rm MMD}$ under different test threshold choices \cite{gretton2012kernel}.
    \\[-.5em]
    \item If $d$ grows in $n$, however, $U_{\rm MMD}$ may become asymptotically degenerate. This can be observed in the case where $\mu$ is given as \eqref{eq:MMD:toy:mu}, $\Sigma_P = \Sigma_Q = I_d$ and $d = \omega(n_1^2 / n_2)$. This is analogous to the observations on the simple U-statistic in \Cref{example:simple:U}.
\end{itemize}
Both cases above have already been observed for MMD with balanced data, as elaborated in the main text. For imbalanced data, however, another possible case is for the following mixture of non-degenerate and degenerate components to dominate in $U_{\rm MMD}$:
\begin{align*}
    U^* 
    \;\coloneqq\; 
    \mfrac{1}{n_1(n_1-1)} \msum_{i \neq j}^{n_1} X_i^\top X_j
    +
    \mfrac{2}{n_2} \msum_{i \leq n_2} Y_i^\top \mu 
    \;.
\end{align*}
To see how this may arise, consider the case when $d$ is even, $\mu$ is given as \eqref{eq:MMD:toy:mu}, 
\begin{align*}
    \Sigma_P \;=&\; \begin{psmallmatrix}
        \Sigma_2 & & \\
            & \ddots & \\
            & & \Sigma_2
    \end{psmallmatrix}
    \;\;\text{ with }\;\;
    \Sigma_2 \;=\; 
    \mfrac{1}{\sqrt{2}}
    \begin{psmallmatrix}
        1 & -1 \\
        -1 & 1
    \end{psmallmatrix}
    &\text{ and }&&
    \Sigma_Q \;=&\; I_d
    \;.
\end{align*}
In this case,
\begin{align*}
    \sigma_1^2 
    \;=\; 
    \mfrac{4 \mu^\top \Sigma_Q \mu}{n_2} 
    \;=\;
    \mfrac{4}{n_2}
    \quad 
    \text{ whereas }
    \quad 
    \sigma_2^2 
    \;=\;
     \mfrac{2 d}{n_1 (n_1 -1 )}
     +
    \mfrac{2 d}{n_2 (n_2 -1 )}
    +
    \mfrac{4 d}{n_1 n_2}
    \;.
\end{align*}
Recall that we have assumed $n_1 = o(n_2)$. Classically for $d$ fixed, the mixture $U^*$ dominates precisely when the data imbalance is at the scale $n_1 = \Theta(n_2^{1/2})$. In the high-dimensional regime when $d$ is allowed to grow, however, $U^*$ may dominate outside of the regime $n_1 = \Theta(n_2^{1/2})$. For instance, $U^*$ dominates if
\begin{align*}
    n_1  \;=\; \Theta( n_2^{1/2} d^{1/2}  )\;.
    \tagaligneq \label{eq:MMD:imbalance:mixture:cond}
\end{align*}
This shows that in the high-dimensional regime, the dimensionality $d$ and the imbalanced sample sizes $n_1$ and $n_2$ can jointly affect the limiting distribution of the MMD statistic.

\begin{remark}[Beyond linear kernel] While the discussion above focuses on the simple case of a linear kernel, we expect this to extend to more general kernels under the functional decomposition of \Cref{appendix:assumption:L_nu}. The role of $X_i$ above would be replaced by  feature maps of the form $(\phi_1(X_i), \ldots, \phi_K(X_i))$ and re-weighted by the weights $\lambda_{k,k'}$, and we expect the role of $d$ in \eqref{eq:MMD:imbalance:mixture:cond} to be replaced by quantities that depend more delicately on how the choice of kernel $\kappa$ interacts with the data distribution. In the balanced sample size case, \cite{huang2023high} has demonstrated how the phase transition effect in linear kernel can be extended to Gaussian radial-basis-function kernel. However as shown in \cite{huang2023high}, a rigorous derivation requires tedious and involved moment computations on the decomposition of the Gaussian kernel, and we do not pursue this in this paper.
\end{remark}

\section{Proof for Corollary \ref{cor:highd:CLT}: Normal approximation of high-dimensional averages} \label{appendix:high:d:average}

Since $S_n$ and $S_n^Z$ are both averages of centred variables, we can WLOG let $\mean X_1 = 0$. Let $S_{nj}$ and $S^Z_{nj}$ denote the $j$-th coordinate of $S_n$ and $S^Z_n$ respectively. By the definition of $\cB_m$ and exploiting the invariance of the supremum over $t \geq 0$ under positive rescaling, we can write 
\begin{align*}
    &\;
    \msup_{B \in \cB_m}
        \big| \P \big( S_n \in B\big) - \P\big( S^Z_n \in B \big) \big| 
    \\
    &\;=\;
    \msup_{t \geq 0} 
    \,
    \Big| \P \big( \, \| S_n \|_{l_m}^m  \leq t  \big) 
    - 
    \P\big( \| S^Z_n  \|_{l_m}^m  \leq t  \big) \Big|
    \\
    &\;=\;
    \msup_{t \geq 0} 
    \,
    \Big| \P \Big( \, \msum_{j=1}^d (S_{nj})^m \leq t  \Big) - \P\Big( \msum_{j=1}^d (S^Z_{nj})^m \leq t  \Big) \Big|
    \\
    &\;=\;
    \msup_{t \geq 0} 
    \,
    \big| \P( v_m (X) \leq  t ) 
            - 
    \P( v_m (Z) \leq t ) 
    \big|
    \;,
\end{align*}
where we have denoted the polynomial function $p_m$
\begin{align*}
    v_m(x_1, \ldots, x_n) 
    \;\coloneqq&\;  
    \mfrac{1}{d}
    \msum_{j=1}^d \Big(
        \mfrac{1}{n} \msum_{i \leq n} x_{ij} 
    \Big)^m 
    \\
    \;=&\;
    \mfrac{1}{n^m} 
    \msum_{i_1, \ldots, i_m \leq n} 
    \,
    \mfrac{1}{d}
    \msum_{j=1}^d
    \mprod_{l=1}^m 
    x_{i_l j} 
    \\
    \;=&\;
    \mfrac{1}{n^m} 
    \msum_{i_1, \ldots, i_m \leq n} 
    \,
    \big\langle 
        S
        \,,\,
        x_{i_1} \otimes \ldots \otimes x_{i_m}
    \big\rangle
    \;,
\end{align*}
and defined the $\R^{d^m}$ tensor $S$ coordinate-wise by  
\begin{align*}
    S_{j_1 \ldots j_m} 
    \;\coloneqq\;
    \begin{cases}
        d^{-1}  &  \text{ if } j_1 = \ldots = j_m \\
        0 & \text{ otherwise }.
    \end{cases}
\end{align*} 
Observe that $v_m(X)$ is exactly a V-statistic of the form considered in \Cref{prop:simpler:V}. Moreover, \Cref{assumption:V:stat} is satisfied, whereas 
\begin{align*}
    \| S \|_{l_1}  \;=\; \msum_{j=1}^d  \mfrac{1}{d}  \;=\; 1 \;.
\end{align*}
Therefore under \Cref{assumption:sub:Weibull} and the assumption that $2 m^2 < n$ and $m \log m = o(n)$, we can apply \Cref{prop:simpler:V} to obtain
\begin{align*}
    &\;
    \msup_{B \in \cB_m}
        \big| \P \big( S_n \in B\big) - \P\big( S^Z_n \in B \big) \big| 
    \\
    & \,\;\leq\; \, 
    C m^{3 + \theta} n^{- \frac{\nu-2}{2\nu m + 2}}
    \,
    \bigg( 
            \mfrac{
                1
                +
                O(n^{-1})
            }
            {
                \Var\big[ \frac{1}{d} \sum_{j=1}^d X_{1j} \ldots X_{mj} \big] 
                + O(m^{m(4+2\theta)} n^{-1} )
            }
        \bigg)^{\frac{\nu}{2 \nu m + 2}}
    \\
    & \,\;=\; \,
    C m^{3 + \theta} n^{- \frac{\nu-2}{2\nu m + 2}}
    \,
    \bigg( 
            \mfrac{
                1
                +
                O(n^{-1})
            }
            {
                \frac{1}{d^2} \sum_{1 \leq j,j' \leq d}  \Sigma_{jj'}^m  
                + 
                O(m^{m(4+2\theta)} n^{-1} )
            }
        \bigg)^{\frac{\nu}{2 \nu m + 2}}
        \;.
\end{align*}
In the last line, we have used that since $X_1, \ldots, X_n$ are i.i.d.~with variance $\Sigma$ and we have WLOG supposed $\mean X_1 = 0$, the variance term in the denominator can be computed  as 
\begin{align*}
    \Var\Big[ \mfrac{1}{d} \msum_{j=1}^d X_{1j} \ldots X_{mj} \Big] 
    \;=&\;
    \mean\Big[ \mfrac{1}{d^2} \msum_{j,j'=1}^d 
        X_{1j} \ldots X_{mj}
        X_{1j'} \ldots X_{mj'}
     \Big]
    \\
    \;=&\;
    \mfrac{1}{d^2} \msum_{1 \leq j,j' \leq d}  \Sigma_{jj'}^m 
    \;.
\end{align*}
This gives the desired bound.

\section{Proof of Proposition \ref{prop:delta}: Non-classical delta method} 
\label{sec:proof:delta}

Since $g$ is $(m+1)$-times continuously differentiable, by an $m$-th order Taylor expansion around $\mean X_1$, we have that almost surely
    \begin{align*}
        \big|
        \hat g(X) 
        &- 
        \msum_{l=0}^{m} \mu_l
        -
        \big( \hat g^{(\mean X_1)}_m(X) - \mu_m \big)
        \big|
        \\
        \;=&\; 
        \Big|
        \mu_0
        +
        \msum_{j=1}^m \hat g_j^{(\mean X_1)}(X)
       \,+\,
       (m+1)
       \mean\Big[ 
       (1-\Theta)^{m+1}
       \,
       \hat g_{m+1}^{\big(\mean X_1 + \Theta n^{-1} \sum_{i \leq n} \bar X_i \big)}(X)
       \,\Big|\, X \Big]
       \\
       &\quad
        - 
        \msum_{l=0}^{m} \mu_l
        -
        \big( \hat g^{(\mean X_1)}_m(X) - \mu_m \big)
        \Big|
        \\
        \;=&\; 
        \Big| 
        \msum_{l=1}^{m-1} \big( \hat g^{(\mean X_1)}_{l}(X) - \mu_l \big)
        \,+\,
        (m+1)
       \mean\Big[ 
       (1-\Theta)^{m+1}
       \,
       \hat g_{m+1}^{\big(\mean X_1 + \Theta n^{-1} \sum_{i \leq n} \bar X_i \big)}(X)
       \,\Big|\, X \Big]
        \Big|
        \\
        \;\eqqcolon&\;
        R\;.
    \end{align*}
Now let $\bX_i = (\bar X_i, \ldots, \overline X_i^{\otimes m})$ and $\xi_i$ be defined as in \Cref{cor:non:multilinear}, $f_1$ be the multilinear representation of $\hat g^{(\mean X_1)}_m$ and denote 
$\sigma_m \;\coloneqq\; \Var\big[ \hat g^{(\mean X_1)}_m(X)\big] \;=\; \Var\big[ f_1(\bX)\big]\;.$ 
Then by \Cref{cor:non:multilinear}, 
there is some absolute constant $C' > 0$,
\begin{align*}
    &\msup_{t \in \R} 
    \big| \P\big(  \hat g(X) - \msum_{l=0}^{m} \mu_l  \leq t\big)
    - 
    \P\big(   f_1(\Xi) - \mean\big[ \hat g^{(\mean X_1)}_m(X) \big] \leq t\big) 
    \big| 
    \\
    &\leq
    C' m \Bigg(
    \Big(  \mfrac{\| R \|_{L_2}^2}{\sigma_m^2} \Big)^{\frac{1}{2 m+1}}
    \\
    &\hspace{5em}
    +
     \Bigg( 
        \mfrac{ \sum_{i=1}^n
        \big\| \partial_i f_1( \bX_1, \ldots, \bX_{i-1}, \mean[\bX_1], \xi_{i+1}, \ldots, \xi_n) (\bX_i - \mean[\bX_1]) \big\|_{L_\nu}^\nu }
    {  \sigma_m^\nu 
    }
    \Bigg)^{\frac{1}{\nu m +1}}
    \Bigg)\;.  
\end{align*}
The first term can be controlled by applying the triangle inequality twice, using a Jensen's inequality and noting the definition of $\epsilon_m$:
\begin{align*}
    &\mfrac{\| R \|_{L_2}^2}{\sigma_m^2}
    \\
    &\,\;\leq\;\,
    \mfrac{
    \Big( \Big\| 
    \sum_{l=1}^{m-1} \big( \hat g^{(\mean X_1)}_{l}(X) - \mu_l\big)
    \Big\|_{L_2}
    +
    (m+1)
    \Big\|
       \mean\Big[ 
       (1-\Theta)^{m+1}
       \,
       \hat g_{m+1}^{\big(\mean X_1 + \Theta n^{-1} \sum_{i \leq n} \bar X_i \big)}(X)
       \,\Big|\, X \Big]
    \Big\|_{L_2}\Big)^2}
    {\sigma_m^2}
    \\
    &\;\leq\;
    \mfrac{
    2
    \Var\Big[ 
    \sum_{l=1}^{m-1} \hat g^{(\mean X_1)}_{l}(X)
    \Big]
    }
    {\sigma_m^2}
    +
    \mfrac{
    2 
    (m+1)^2
    \Big\|
     \mean\Big[ 
    (1-\Theta)^{m+1}
       \,
       \hat g_{m+1}^{\big(\mean X_1 + \Theta n^{-1} \sum_{i \leq n} \bar X_i \big)}(X)
    \Big| X \Big]
    \Big\|_{L_2}^2
    }
    {\sigma_m^2}
    \\
    &\;=\; 
    2 \epsilon_m \;.
\end{align*}
By noting that \Cref{lem:universality:v:stat:aug} 
applies to $\hat g_m^{(\mean X_1)}(X)$, the second term can be bounded by
\begin{align*}
    C'' 
    \,
    n^{-\frac{\nu-2}{2\nu m + 2}}
    \,
    \Bigg(
    \mfrac{
            \sum_{k=1}^m
            \binom{n}{k}
            (\alpha_{m,\nu}(S))^2
    }
    {
    \sum_{k=1}^m
    \binom{n}{k}
    (\alpha_{m,2}(k))^2
    }
    \Bigg)^{\frac{\nu}{2 \nu m + 2}}
    \;\leq\; 
    C'' 
    n^{-\frac{\nu-2}{2\nu m + 2}}
    \beta_{m,\nu}^{\frac{\nu}{2 \nu m + 2}}  \;.
\end{align*}
for some absolute constants $C'' > 0$. Moreover, the first bound of \cref{prop:universality:v:stat} 
implies that for some absolute constant $C''' > 0$,
\begin{align*}
    &\sup_{t \in \R} 
    \big| \,
        \P\big(  f_1(\Xi) - \mean\big[ \hat g^{(\mean X_1)}_m(X)\big] \leq t\big) 
        -
        \P\big(   \hat g_m^{(\mean X_1)}(Z) - \mean[ \hat g_m^{(\mean X_1)}(Z)  ] \leq t\big) 
    \, \big| 
    \\
    &\;\leq\;
    C''' m\Big( \mfrac{R_n}{n} \Big)^{-\frac{1}{2m+1}}\;.
\end{align*}
Combining all three bounds gives that, for some absolute constant $C_* > 0$,
\begin{align*}
    \msup_{t \in \R} 
    \big| \P\big(  \hat g(X) -  \msum_{l=0}^m \mu_l \leq t\big)
    -  
    \P\big( &  \hat g_m^{(\mean X_1)}(Z) - \mean[ \hat g_m^{(\mean X_1)}(Z)  ]  \leq t\big) 
    \big| 
    \\
    \;\leq&\;
    C_*  m \big( \epsilon_m^{\frac{1}{2 m+1}} + \Big( \mfrac{R_n}{n} \Big)^{-\frac{1}{2m+1}} +  n^{-\frac{\nu-2}{2\nu m + 2}}
    \,
    \beta_{m,\nu}^{\frac{\nu}{2 \nu m + 2}}  
    \big)
    \;.
\end{align*}
The last two terms are exactly the terms controlled in the proof of \Cref{prop:simpler:V} (\Cref{appendix:proof:simpler:V}) under the sub-Weibull condition of \Cref{assumption:sub:Weibull} and the conditions $m \log m  = o(n)$ and $2m^2 < n$. Recycling the bounds from the proof of \Cref{prop:simpler:V} in \Cref{appendix:proof:simpler:V} gives that for some absolute constant $C > 0$,
\begin{align*}
    &\;\msup_{t \in \R} 
    \big| \P\big(  \hat g(X) -  \msum_{l=0}^m \mu_l \leq t\big)
    -  
    \P\big( \hat g_m^{(\mean X_1)}(Z) - \mean[ \hat g_m^{(\mean X_1)}(Z)  ]  \leq t\big) 
    \big| 
    \\
    &\;\leq\;
    C  \bigg( m \epsilon_m^{\frac{1}{2 m+1}}
    \\
    &\hspace{4em}
        +
        m^{3+\theta} n^{-\frac{\nu-2}{2\nu m +2} } 
         \bigg( 
            \mfrac{
                \| S \|_{l_1}^2 
                +
                \frac{K_3}{n} \,
                \| S \|_{l_1}^2 
            }
            {
                \max\Big\{ 
                    \big\| \langle S, X_1 \otimes \ldots \otimes X_m \rangle \big\|_{L_2}^2 
                -
                \frac{K_3}{n} \,
                m^{m (4 + 2 \theta)} \, \| S \|_{l_1}^2
                \,,\, 
                0
                \Big\} 
            }
        \bigg)^{\frac{\nu}{2 \nu m + 2}}
    \bigg)
    \\
    &\;\eqqcolon\; \Delta_{\hat g, m}
    \;.
\end{align*}
Therefore, under the stated conditions that $ m \epsilon_m^{\frac{1}{2 m+1}} = o(1)$ and 
\begin{align*}
    \max\Big\{ \mfrac{m^{m (4 + 2 \theta)}}{n}, 1\Big\} \times \| \partial_m g(\mean[X_1]) \|_{l_1}^2
    \;=\; 
    o\big( \big\| \langle \partial_m g(\mean[X_1]), \bar X_1 \otimes \ldots \otimes \bar X_m \rangle \big\|_{L_2}^2 \big)
    \;,
\end{align*}
then the universality approximation error converges to zero.

\vspace{.5em}

To prove the fourth moment result, applying \Cref{prop:gaussian} implies that 
\begin{align*}
    \sup_{t \in \R} 
    \Big| \P\big(  \hat g_{m}^{(\mean X_1)}(Z) - \mean[\hat g_{m}^{(\mean X_1)}(Z)]  \leq t\big)
    - 
   \Phi(\sigma_Z^{-1} t)
    \Big|
    \;\leq&\;
    \Big(\mfrac{4m-4}{3m} \, \big| \Kurt \big[\hat g_{m}^{(\mean X_1)}(Z)\big] \big| \Big)^{1/2}\;,
\end{align*}
where $\sigma_Z^2 \coloneqq \Var\big[ \hat g_{m}^{(\mean X_1)}(Z)\big]$. Write $\sigma_X^2 \coloneqq \Var\big[ \hat g_{m}^{(\mean X_1)}(X)\big]$. Note that if $t=0$,
\begin{align*}
    \big| \Phi(  \sigma_Z^{-1} t )  - \Phi(  \sigma_X^{-1} t ) \big| \;=\; 0\;.
\end{align*}
Suppose $t \neq 0$ and write $\eta \sim \cN(0,1)$. By \cref{lem:carbery:wright}, we have that
\begin{align*}
    \big| \Phi(  \sigma_Z^{-1} t )  - \Phi(  \sigma_X^{-1} t ) \big|
    \;\leq&\;
    C' \, \mfrac{\big| \sigma_Z^{-1} t - \sigma_X^{-1} t \big| \, / 2 }{ \big( \mean \big[ \big(\eta -  (\sigma_Z^{-1} t + \sigma_X^{-1}t) / 2 \big)^2  \big] \, \big)^{1/2} }
    \;\leq\;
    C' \, \mfrac{\big| \sigma_Z^{-1} t - \sigma_X^{-1} t \big| }
    {\big| \sigma_Z^{-1} t + \sigma_X^{-1}t \big| }
\end{align*}
for some absolute constant $C' > 0$. Rearranging and applying the second bound of \cref{lem:v:var:approx}, 
\begin{align*}
    \big| \Phi(  \sigma_Z^{-1} t )  - \Phi(  \sigma_X^{-1} t ) \big|
    \;\leq&\;
    C' \, \mfrac{\big| 1 - \sigma_X^{-1} \sigma_Z \big| }
    { 1 + \sigma_X^{-1} \sigma_Z  }
    \;\leq\; 
    C' \mfrac{(C'')^m \,  R_n^{1/2} n^{-1/2} }{ 2 - (C'')^m \, R_n^{1/2} n^{-1/2} }
    \;=\; o(1)
    \;,
\end{align*}
where $C'' > 0$ is  some absolute constant. Applying a triangle inequality and replacing $t$ with $\sigma_X t$, we obtain the final bound that 
\begin{align*}
    \msup_{t \in \R} 
    \Big| \P\Big( \sigma_X^{-1} & \big(\hat g(X) -  \msum_{l=0}^{m} \mu_l \big) \leq t\Big)
    - 
    \Phi(t)
    \Big|
    \\
    \;\leq&\; 
    \Delta_{\hat g, m}
    + 
    C' \mfrac{(C'')^m \,  R_n^{1/2} n^{-1/2} }{ 2 - (C'')^m \, R_n^{1/2} n^{-1/2} }
    +
    \Big(\mfrac{4m-4}{3m} \, \big|\Kurt \big[\hat g_{m}^{(\mean X_1)}(Z)\big] \big| \Big)^{1/2}. 
\end{align*}
Therefore the additional condition $\Kurt \big[\hat g_{m}^{(\mean X_1)}(Z)\big] \rightarrow 0$  implies  asymptotic normality. \qed

\section{Proofs for Appendix \ref{appendix:subgraph:u}} 
% \ref{sec:subgraph:u}} 
 \label{sec:proof:subgraph:u}

The results on vertex-level fluctuations concern a U-statistic of i.i.d.~variables, which has already been studied in \Cref{prop:higher:u}. % \cref{prop:higher:u}.

\begin{proof}[Proof of \Cref{cor:subgraph:vertex}] 
The first bound follows directly from the proof of \Cref{prop:higher:u} 
% \cref{prop:higher:u} 
in \cref{sec:proof:higher:u} by replacing $u_m(Y)$ by $\kappa_1(Y)$.
The variance computation follows from substituting $u_1$ into the definition of $\sigma^2_{m,n;j}$ in \Cref{prop:higher:u}, 
% \cref{prop:higher:u}, 
and computing 
    \begin{align*}
        \mbinom{m}{r}^2 
        \mbinom{n}{m}^2
        \mbinom{n}{r}^{-1} 
        \;=&\;
        \mbinom{m}{r}
        \mbinom{n}{m}
        \mfrac{(m!) (n!) (r!) (n-r)!}
        {(r!) ((m-r)!) (m!) ((n-m)!) (n!)}
        \;=\;
        \mbinom{m}{r}
        \mbinom{n}{m}
        \mbinom{n-r}{m-r}
        \;.
    \end{align*}
\end{proof}

\begin{proof}[Proof of \Cref{lem:subgraph:vertex:conditions}
] By the variance formula in \Cref{cor:subgraph:vertex} 
and the fact that $w$ is bounded, $\sigma^2_{m,n;r} = O(n^{2m-r})$ for all $r \in [m]$, which implies (i) $\Leftrightarrow$ (ii). We also have (ii) $\Leftrightarrow$ (iii) in view of the formula in \Cref{cor:subgraph:vertex}. 
To prove (iii) $\Leftrightarrow$ (iv), recall that $|\textrm{Aut}(H)|$ is the number of automorphisms of $H$, and let $P_m$ be the set of permutations of $\{1, \ldots, m\}$. Also denote $P_m^{i \rightarrow 1} \subseteq P_m$ as the set of permutations that sends $\{i\}$ to $\{1\}$. Then we can write
\begin{align*}
    &
        \;
        \msum_{H' \subseteq \cG_H(\{1, \ldots, m\})}
        \; 
        \mean\Big[
        \mprod_{(i_s,i_t) \in E(H')} w(U_{i_s}, U_{i_t})
        \, \Big|\, 
        U_1 = x \Big]
    \\
    \;=&\;
        \mfrac{1}{|\textrm{Aut}(H)|}
        \;
        \msum_{\sigma \in P_m} 
        \; 
        \mean\Big[
        \mprod_{(i_s,i_t) \in E(H)} w\big(U_{ \sigma(i_s)}, U_{\sigma(i_t)}\big)
        \, \Big|\, 
        U_1 = x \Big]
    \\
    \;=&\;
        \mfrac{1}{|\textrm{Aut}(H)|}
        \;
        \msum_{\sigma \in P_m} 
        \; 
        \mean\Big[
        \mprod_{(i_s,i_t) \in E(H)} w\big(U_{i_s}, U_{i_t}\big)
        \, \Big|\, 
        U_{\sigma^{-1}(1)} = x \Big]
    \\
    \;=&\;
        \mfrac{1}{|\textrm{Aut}(H)|}
        \;
            \msum_{i=1}^m 
            \msum_{\sigma \in P_m^{i \rightarrow 1}} 
            \; 
            \mean\Big[
            \mprod_{(i_s,i_t) \in E(H)} w\big(U_{i_s}, U_{i_t}\big)
            \, \Big|\, 
            U_i = x \Big]
    \\
    \;=&\;
        \mfrac{(m-1)!}{|\textrm{Aut}(H)|}
        \;
            \msum_{i=1}^m 
            \mean\Big[
            \mprod_{(i_s,i_t) \in E(H)} w\big(U_{i_s}, U_{i_t}\big)
            \, \Big|\, 
            U_i = x \Big]
        \;.
\end{align*}
(iii) says the above is constant for almost every $x \in [0,1]$, whereas (iv) is equivalent to requiring that 
\begin{align*}
    \mfrac{1}{m}
    \msum_{i=1}^m
    \; 
    \mean\Big[
    \mprod_{(i_s,i_t) \in E(H)} w(U_{i_s}, U_{i_t})
    \, \Big|\, 
    U_i = x\Big]
\end{align*}
is constant for almost every $x \in [0,1]$. This proves that (iii) $\Leftrightarrow$ (iv).
\end{proof}

To prove the results on the edge-level fluctuations, we first present a stronger lemma, which will imply \Cref{lem:subgraph:u:var:size}
and greatly simplify the proof of \Cref{prop:subgraph:u:two}. 
For $i \in [n_*]$,  write $W^{(i)}_j \coloneqq \bar Y_j$ for $j \leq i$ and $W^{(i)}_j \coloneqq Z_j$ for $j > i$. Denote $e_j$ as the edge in $K_n$ indexed by $j \in [n_*]$. Let $I$ and $E$ be edge sets. Recall that $\delta_H(I)$ is the indicator of whether the graph formed by $I$ is isomorphic to $H$. Also define $n_H(k, I, E )$ as the number of subgraphs of $K_n$ that are isomorphic to $H$ and can be formed with $k$ edges from $I$ and all edges from $E$.

\vspace{.5em}

 Given $i \in [n_*]$, $0 \leq k' \leq k$, the edge subsets $I \subseteq [n_*]$ with $|I| \geq k'$ and $E_{k-k'} \subseteq E(K_{n_*})$ with $|E_{k-k'}| = k-k'$, as well as some random variable $W$ that are either zero-mean or constant almost surely and are independent of $\{W^{(i)}_j \,|\, j \in I\}$, we define 
\begin{align*}
    S^{(i)}(k', I, E_{k-k'}, W) 
    \;\coloneqq&\; 
    \msum_{\substack{j_1, \ldots, j_{k'} \,\in\, I \\ j_1 < \ldots < \, j_{k'} } }
    \; 
    \delta_H( \{ e_{j_l} \}_{l \in [k']} \cup E_{k-k'} )
    \, W
    \mprod_{l \in [k'] } 
    W^{(i)}_{j_l}
    \;\;\; \text{ for } k' \geq 1\;,
    \\
    S^{(i)}(0, I, E_k, W) 
    \;\coloneqq&\; 
    \,
    \delta_H( E_k )
    \, W\;.
\end{align*}
Notice that $\kappa_2(\bar Y) = S^{(n_*)}(k, [n_*], \emptyset, 1)$. The next lemma controls the conditional moments of these quantities given $U$, by exploiting that $\bar Y_j$'s and $Z_j$'s are conditionally independent and zero-mean given $U$:

\vspace{.5em}

\begin{lemma} \label{lem:subgraph:moment:auxiliary} There are some absolute constants $c, C > 0$ and $\nu \in (2,3]$ such that almost surely
\begin{align*}
    \mean \big[ \big| S^{(i)}(k', I, E_{k-k'}, W) \big|^\nu \,\big|\, U \big]
        \;\leq&\; 
        C^{k'} \, 
        \big( n_H(k', I, E_{k-k'})\big)^{\nu/2} \;
        \mean[ | W |^\nu \,|\, U\,] \, 
    \\
    & \qquad \qquad \times \mmax_{\substack{i_1, \ldots, i_{k'} \in I \\ \text{ all distinct}}} \; \mprod_{l=1}^{k'}  
        \,  \mean\big[ | X_{i_l} |^\nu \,\big|\, U\,\big]\;, 
    \\
    \mean \big[ \big| S^{(n_*)}(k', I, E_{k-k'}, W) \big|^2 \,\big|\, U \big]
    \;\geq&\; 
    c^{k'} \, n_H(k', I, E_{k-k'})  
    \;
    \mean[ | W |^2 \,|\, U \,]
    \\
    & \qquad \qquad \times \mmax_{\substack{i_1, \ldots, i_{k'} \in I \\ \textrm{all distinct}}} \; \mprod_{l=1}^{k'}  
    \mean \big[ | X_{i_l} |^2 \,\big|\, U \big] \;.
\end{align*}
\end{lemma}

\begin{proof}[Proof of \cref{lem:subgraph:moment:auxiliary}] The proof proceeds by induction on $k \geq 0$. We claim that the constant in the upper bound is specified by $C = C_1 C_2$, where $C_1$ is the maximum of the absolute constant $C'_\nu$ in the second upper bound of \cref{lem:martingale:bound} over $\nu \in [2,3]$, and $C_2$ is the absolute constant in \cref{lem:gaussian:linear:moment}. The constant $c$ in the lower bound is the absolute constant in the lower bound of \cref{lem:martingale:bound} with $\nu=2$. 

\vspace{.5em}

The base case $k'=0$ is straightforward by noting that $ \delta_H( E_k ) = n_H(0, I, E_k)$ for all $I \subseteq [n_*]$. Suppose that the inductive statement holds for $k'-1$. To prove the statement for $k'$, we enumerate the elements of $I$ as $i_1 < \ldots < i_{n'}$ and consider a martingale difference sequence $(S_{l'})_{l'=1}^{n'}$ conditioning on $U$: For $l' \in [n']$, define 
\begin{align*}
    S^{(i)}_{l'} \;\coloneqq&\; 
    \mean\big[ S^{(i)}(k', I, E_{k-k'}, W)   \,\big|\, U, W, W^{(i)}_{i_1}, \ldots, W^{(i)}_{i_{l'}} \big] 
    \\
    &\; - 
    \mean\big[ S^{(i)}(k', I, E_{k-k'}, W)  \,\big|\, U, W, W^{(i)}_{i_1}, \ldots, W^{(i)}_{i_{l'-1}} \big] \;.
\end{align*}
Since $(W^{(i)}_{i_j})_{j \leq l'}$ are zero-mean and $(W^{(i)}_{i_j})_{j \leq l'} \cup \{ W \}$ are independent conditioning on $U$, for $k' \geq 1$, we have 
\begin{align*}
    \mean\big[  S^{(i)}(k', I, E_{k-k'}, W)  \,\big|\, U, W \big] 
    \;=\; 
    0 \qquad \text{ almost surely }.
\end{align*}
This allows us to express, almost surely,
\begin{align*}
     S^{(i)}(k', I, E_{k-k'}, W) 
    \;=\; 
    \msum_{l'=1}^{n'}  S^{(i)}_{l'} \;.
\end{align*}
Now notice that for $l' < k'$, $S^{(i)}_{l'}=0$ almost surely. Since the non-zero terms in each difference $S^{(i)}_{l'}$ must involve $W^{(i)}_{i_{l'}}$, for $l' \geq k' > 1$, we have that almost surely
\begin{align*}
     S^{(i)}_{l'} 
    \;=&\;  
        \, \sum_{\substack{j_1, \ldots, j_{k'-1} \,\in\, \{i_1, \ldots, i_{l'-1}\} \\ j_1 < \ldots < \, j_{k'-1} } }
    \; 
    \delta_H( \{ e_{j_l} \}_{l \in [k'-1]} \cup E_{k-k'} \cup \{ e_{i_{l'}} \} ) 
    \, W \, W^{(i)}_{i_{l'}}
    \mprod_{l \in [k'-1] } 
    W^{(i)}_{j_l}
     \;,
\end{align*}
and for $l' \geq k' = 1$, we have that almost surely
\begin{align*}
    S^{(i)}_{l'} 
    \;=&\; 
        \delta_H( E_{k-k'} \cup \{ e_{i_{l'}} \} )
          \, W \,  W^{(i)}_{i_{l'}}
    \;.
\end{align*}
This in particular implies that for $l' \geq k'$, almost surely
\begin{align*}
    S^{(i)}_{l'} 
    \;=&\; 
    S^{(i)}
    \big(k'-1 \,,\, \{i_1, \ldots, i_{l'-1}\} \,,\,  E_{k-k'} \cup \{e_{i_{l'}}\} \,,\, W \, W^{(i)}_{i_{l'}} \big)\;, \tagaligneq \label{eq:subgraph:u:induction:Lnu:step}
\end{align*}
Now by the martingale moment bound in \cref{lem:martingale:bound} with $\nu \in [2,3]$, we get the almost sure bound
\begin{align*}
    c \, \msum_{l'=k'}^{n'} \mean[ |S^{(i)}_{l'} |^\nu \, |\, U \, ]
    \;\leq&\; 
    \mean\big[ \big| S^{(i)}(k', I, E_{k-k'}, W)  \big|^\nu \, \big|\, U \, \big]
    \\
    \;\leq&\; 
    C_1 \, 
    \mean \Big[ \big( \msum_{l'=k'}^{n'} |S^{(i)}_{l'} |^2 \big)^{\nu/2} 
    \,\Big|\, U
    \Big]\;,
    \tagaligneq \label{eq:subgraph:u:induction:martingale:nu}
\end{align*}
and since $W \times W^{(i)}_{i_{l'}}$ is independent of $W^{(i)}_{i_1}, \ldots, W^{(i)}_{i_{l'-1}}$, we may apply the inductive statement to control the individual $S^{(i)}_{l'}$ term. Now notice that the moment bound on each $S^{(i)}_{l'} $ will introduce a constant  $w_{l'} \coloneqq n_H(k'-1, \{i_1, \ldots, i_{l'-1}\}, E_{k-k'} \cup \{e_{i_l'}\} )$. Denote their sum as $w_H 
\coloneqq \sum_{l'=k'}^{n'} w_{l'}$, which satisfies
\begin{align*}
    w_H 
    \;=\; 
    n_H(k', I, E_{k-k'})
    \;. \tagaligneq \label{eq:subgraph:u:induction:weights}
\end{align*} 
By reweighting the sum in the upper bound in \eqref{eq:subgraph:u:induction:martingale:nu} and a Jensen's inequality,  almost surely
\begin{align*}
    \mean\big[ \big| S^{(i)}(k', I, E_{k-k'}, W)  \big|^\nu \,\big|\, U \,\big]
    \;\leq&\; 
    C_1 \, 
    (w_H)^{\nu/2} \,
    \mean \Big[ \Big( \msum_{l'=k'}^{n'} \mfrac{w_{l'}}{w_H} \, \mfrac{1}{w_{l'}} |S^{(i)}_{l'} |^2 \Big)^{\nu/2} 
    \,\Big|\, U \,\Big]
    \\
    \;\leq&\;
    C_1 \, 
    (w_H)^{\nu/2} \,
    \msum_{l'=k'}^{n'} \mfrac{w_{l'}}{w_H}
    \,
    \mean \Big[ \Big( \, \mfrac{1}{w_{l'}} |S^{(i)}_{l'} |^2 \Big)^{\nu/2} 
    \,\Big|\, U \Big]
    \\
    \;=&\;
    C_1 \, 
    (w_H)^{\nu/2} \,
    \msum_{l'=k'}^{n'} \mfrac{w_{l'}}{w_H}
    \, 
    \mfrac{\mean\big[ \big| S^{(i)}_{l'} \big|^\nu \,\big|\, U \big]}{w_{l'}^{\nu/2}}
    \;.
\end{align*}
By applying the upper bound in the inductive statement to \eqref{eq:subgraph:u:induction:Lnu:step}, we have that almost surely
\begin{align*}
    &\mean\big[ \big| S^{(i)}_{l'} \big|^\nu \,\big|\, U \big]
    \\
    \;&\leq\; 
    C^{k'-1} \, 
     w_{l'}^{\nu/2} \,
     \mean \big[ \big| W \times W^{(i)}_{i_{l'}} \big|^\nu \,\big|\, U\big] 
     \,  
     \\
     &\hspace{5em} \times \mmax_{j_1, \ldots, j_{k'-1} \in \{i_1, \ldots, i_{l'-1}\} \text{ distinct}} \, \mprod_{l=1}^{k'-1}  
     \mean\big[ | X_{j_l} |^\nu \,\big|\, U \,\big]
     \\
     \;&\overset{(a)}{\leq}\; 
     C_2 C^{k'-1} \, 
     w_{l'}^{\nu/2} \, 
     \mean[ | W |^\nu | U]  \;
     \mean[ | X_{i_{l'}} |^\nu | U ] \, 
     \\
     &\hspace{5em} \times
     \max_{j_1, \ldots, j_{k'-1} \in \{i_1, \ldots, i_{l'-1}\} \text{ distinct}} \, \mprod_{l=1}^{k'-1}  
     \mean\big[ | X_{j_l} |^\nu \,\big|\, U\big]
     \\
     \;&\leq\; 
     C_2 C^{k'-1} \, 
     w_{l'}^{\nu/2} \,
     \mean[ | W |^\nu \,|\, U]  
     \, 
     \mmax_{i_1, \ldots, i_{k'} \in I \text{ distinct}} \, \mprod_{l=1}^{k'}  
     \, 
     \mean\big[ | X_{i_l} |^\nu \,|\, U\big] 
     \;,
\end{align*}
where we have noted that $W$ is independent of $W^{(i)}_{i_{l'}}$ and used \cref{lem:gaussian:linear:moment} in $(a)$ to replace $W^{(i)}_{i_{l'}}$ by $X_{i_{l'}}$. By noting that $C_* = C_1 C_2$ and using \eqref{eq:subgraph:u:induction:weights}, we get that almost surely
\begin{align*}
    \mean \big[ \big| S^{(i)}(k', I, & E_{k-k'}, W)   \big|^\nu \,\big|\, U \big]
    \\
    \;\leq&\; 
    C^{k'}
    \, (w_H)^{\nu/2} \,
    \msum_{l'=k'}^{n'} \mfrac{w_{l'}}{w_H}
    \,
    \mean[ | W |^\nu \,|\, U]  \, 
    \max_{i_1, \ldots, i_{k'} \in I \text{ distinct}} \, \mprod_{l=1}^{k'}  
    \mean [ | X_{i_l} |^\nu \,|\, U] 
    \\
    \;=&\;
    C^{k'}
    \, (n_H(k', I, E_{k-k'}))^{\nu/2}
    \, \mean[ | W |^\nu \,|\, U ]  \, 
    \mmax_{i_1, \ldots, i_{k'} \in I \text{ distinct}} \, \mprod_{l=1}^{k'}  \mean [ | X_{i_l} |^\nu ]
    \;,
\end{align*}
which proves the upper bound. Similarly by the lower bound in the inductive statement and noting that we do not need to use \cref{lem:gaussian:linear:moment}, we get that 
\begin{align*}
    \mean[ | S^{(n_*)}(k', I, & E_{k-k'},  W) |^2 \,|\, U]
    \\
    \;\geq&\;
    c^{k'} \, n_H(k', I, E_{k-k'})  \, 
    \mean[ | W |^2 \,|\, U]  \,
    \mmin_{i_1, \ldots, i_{k'} \in I \text{ distinct}} \, \mprod_{l=1}^{k'}  
    \mean[ | X_{i_l} |^2 \,|\, U] \;.
\end{align*}
which proves the lower bound.
\end{proof}

\vspace{.5em}

\begin{proof}[Proof of \Cref{lem:subgraph:u:var:size}]
Since $\mean[ \kappa_2(\bar Y) | U] = 0$ almost surely, by the law of total variance, 
\begin{align*}
    \Var[ 
        \kappa_2(\bar Y)
    ]
    \;=\; 
    \mean \, \Var[  \kappa_2(\bar Y)  | U ]
    \;=\; 
    \mean\, \Big\|
    \msum_{\substack{j_1, \ldots, j_k \in [n_*]  \\ j_1 < \ldots < \, j_k} }
    \; 
    \delta_H\big( \{e_{j_l}\}_{l \in [k]} \big)
    \,
    \mprod_{l \in [k]} 
    \bar Y_l
    \Big\|_{L_2 | U}^2 \;,
\end{align*}
where we have denoted  $\| \argdot \|_{L_2 | U}^2 \coloneqq \mean[|\argdot|^2 |\, U] $. Since $(\bar Y_{ij})_{(i,j) \in E(K_n)}$ is a collection of $n_*$ random variables that are conditionally independent and zero-mean given $U$, the quantity inside the conditional norm can be identified with $S^{(n_*)}(k, [n_*], \emptyset, 1)$. Applying \cref{lem:subgraph:moment:auxiliary} conditionally on $U$, we get that for some absolute constants $c, C > 0$, almost surely,
\begin{align*}
    \Var[  \kappa_2(\bar Y)  | U ]
    \;\leq&\;
    C^k 
    \, n_H(k, [n_*], \emptyset) \,
    \mmax_{i_1, \ldots, i_k \in [n_*] \text{ distinct}} \, \mprod_{l=1}^{k}  \Var[ \bar Y_{i_l} \, |\, U ]  \;, 
    \\
    \Var[  \kappa_2(\bar Y)  | U ]
    \;\geq&\;
    c^k 
    \, n_H(k, [n_*], \emptyset) \,
    \mmin_{i_1, \ldots, i_k \in [n_*] \text{ distinct}} \, \mprod_{l=1}^k \ \Var[ \bar Y_{i_l} \, |\, U ]  \;.
\end{align*}
Noting that $n(k, [n_*], \emptyset) = | \cG_H([n])|$ and taking an expectation yield the desired bounds.
\end{proof} 

\vspace{.5em}

\begin{proof}[Proof of \Cref{prop:subgraph:u:two}] 
Observe that $\kappa_2(\bar Y) $ is multilinear in $(\bar Y_i)_{i \in [n_*]}$ and a degree-$k$ polynomial. Moreover, as $\bar Y_{ij}$'s are conditionally centred given $U$, $\mean[ \kappa_2(\bar Y) \,|\, U] = 0$. Since $\Var[\kappa_2(Z) | U] > 0$ almost surely by assumption, by applying \Cref{thm:main} 
% \cref{thm:main} 
and noting that the bounding constant is absolute, we get that for some absolute constant $C' >0$, almost surely
\begin{align*}
    \msup_{t \in \R} \Big| \P\big(  \kappa_2(\bar Y) & \leq  \; t \, \big| \, U \big) 
    - 
    \P\big( \kappa_2(Z) \leq t \,\big| \, U \big)  \Big|
    \\
    \;\leq&\;
    C' m
    \Bigg( 
    \mfrac{ 
        \sum_{i=1}^{n_*} 
        \mean\big[ 
        \partial_i \kappa_2(\bar Y_1, \ldots, \bar Y_{i-1}, 0, Z_{i+1}, \ldots, Z_{n_*})
        \, \bar Y_i
        \big|^\nu \,\big|\, U \,\big]
        }
    {
        \Var[ \kappa_2(\bar Y) \,|\, U \,]^{\nu/2}
    }
    \Bigg)^{\frac{1}{\nu k +1}}
        \;. \tagaligneq \label{eq:subgraph:u:two:intermediate}
\end{align*}
\cref{lem:subgraph:moment:auxiliary} implies that, for some absolute constants $C_*, c_* > 0$, we have that almost surely
\begin{align*}
    \mean\big[ 
    \partial_i &\kappa_2(\bar Y_1,  \ldots, \bar Y_{i-1}, 0, Z_{i+1}, \ldots, Z_{n_*})
    \, \bar Y_i
    \big|^\nu \,\big|\, U \,\big]
    \\
    \;=&\; 
    \mean\Big[ \Big|
    \msum_{\substack{j_1, \ldots, j_{k-1} \in [n_*] \setminus \{i\} \\ j_1 < \ldots < \, j_{k-1} } }
    \; 
    \delta_H\big( \{e_{j_l}\}_{l \in [k-1]} \cup \{ e_i\} \big)
    \, W^{(i)}_i
    \mprod_{l \in [k-1] } 
    W^{(i)}_{j_l}
    \Big|^\nu
    \,\Big|\, U \,\Big] 
    \\
    \;=&\;
    \mean [ \, | S^{(i)}(k-1, [n_*] \setminus \{i\}, \{e_i\}, \bar Y_i)  |^\nu \,|\, U ]
    \\
    \;\leq&\;
    (C_*)^{k-1} \, 
     \big( n_H(k-1, [n_*] \setminus \{i\},  \{e_i\} ) \big)^{\nu/2} 
     \, \mean \big[ \,\big| \bar Y_i \big|^\nu \,\big|\, U \big]\, 
     \\
     &\qquad\qquad\qquad
     \times 
     \mmax_{i_1, \ldots, i_{k-1} \in [n_*] \setminus \{i\} \text{ distinct}} \, \mprod_{l=1}^{k-1}  
     \mean\big[ \big| \bar Y_{i_l} \big|^\nu \,\big|\, U \big]
    \\
    \;=&\; 
    (C_*)^{k-1} \, 
    \big( n_H(k-1, [n_*] \setminus \{i\},  \{e_i\} ) \big)^{\nu/2} \, \mmax_{i_1, \ldots, i_k \in [n_*] \text{ distinct}} \, \mprod_{l=1}^{k}  
    \mean\big[ \big| \bar Y_{i_l} \big|^\nu \,\big|\, U \big]
    \;, 
    \end{align*}
    and
    \begin{align*}
    \Var[
        \kappa_2(\bar Y)  \,|\, U
    ]
    \;=&\;
    \mean\Big[ \Big|
    \msum_{\substack{j_1, \ldots, j_k \in [n_*]  \\ j_1 < \ldots < \, j_k} }
    \; 
    \delta_H\big( \{e_{j_l}\}_{l \in [k]} \big)
    \,
    \mprod_{l \in [k]} 
    W^{(n_*)}_{j_l}
    \Big|^2 \, \Big| \, U \, \Big]
    \\
    \;=&\; 
    \mean\big[ \big| S^{(n_*)}(k, [n_*], \emptyset, 1) \big|^2 \big| \, U \,\big]
    \\
    \;\geq&\; 
    (c_*)^k
    \, n_H(k, [n_*], \emptyset) \, 
    \mmin_{i_1, \ldots, i_{k'} \in [n_*] \text{ distinct}} \, \mprod_{l=1}^{k'}  
    \mean[ | \bar Y_{i_l} |^2 \,|\, U ] \;.
\end{align*}
Now recall that $[n_*]$ indexes all edges of an complete graph $K_n$, whereas $n_H(k-1, [n_*] \setminus \{i\},  \{e_i\} )$ counts the number of subgraphs in $K_n$ that contains $e_i$ and is isomorphic to $H$. By symmetry, $n_H(k-1, [n_*] \setminus \{i\},  \{e_i\} )$ is the same for all $i \in [n_*]$, and 
    \begin{align*}
        \msum_{i=1}^{n_*} n_H(k-1, [n_*] \setminus \{i\},  \{e_i\} )
        \;=\; 
        k \, | \cG_H([n]) |
        \;=\; 
        k \, \big(  n_H(k, [n_*], \emptyset) \big)
        \;,
    \end{align*}
    since each subgraph of $K_n$ that is isomorphic to $H$ has been counted exactly $k=|H|$ times. Thus
    \begin{align*}
        \msum_{i=1}^{n_*} \big( n_H(k-1, [n_*] \setminus \{i\},  \{e_i\} ) \big)^{\nu/2}
        \;=&\;
        \msum_{i=1}^{n_*} \big( n_H(k-1, [n_*] \setminus \{1\},  \{e_1\} ) \big)^{\nu/2}
        \\
        \;=&\; n_* \big( \mfrac{k}{n_*} \, | \cG_H([n]) | )^{\nu/2}\;.
    \end{align*}
    Combining the above, we get that for some absolute constants $C'', C > 0$, almost surely
    \begin{align*}
        \msup_{t \in \R} &\Big| \P\big(  \kappa_2(\bar Y)  \leq  t\big) 
        - 
        \P\big( \kappa_2(Z)  \leq t \,\big|\, U \big)  \Big|
        \\
        \;\leq&\; 
        C'' m
        \Bigg( 
            \mfrac{ 
                \sum_{i=1}^{n_*} 
                \big( n_H(k-1, [n_*] \setminus \{i\},  \{e_i\} ) \big)^{\nu/2}
                }
            {
                | \cG_H([n]) |^{\nu /2 }  \, 
            }
            \, 
            \mfrac
            {\mmax_{i_1, \ldots, i_k \in [n_*] \text{ distinct}} \, \mprod_{l=1}^{k} 
            \mean[ | \bar Y_{i_l} |^\nu \,|\, U ]}
            {\mmin_{i_1, \ldots, i_{k'} \in [n_*] \text{ distinct}} \, \mprod_{l=1}^{k'} 
            \mean[ | \bar Y_{i_l} |^2 \,|\, U] }
            \Bigg)^{\frac{1}{\nu k +1}}
        \\
        \;\leq&\; 
        C'' m
        \bigg( \mfrac{ 
            \big( \sum_{i=1}^{n_*} 
             n_H(k-1, [n_*] \setminus \{i\},  \{e_i\} ) \big)^{\nu/2}
            }
        {
            | \cG_H([n]) |^{\nu /2 }  \, 
        } \bigg)^{\frac{1}{\nu k +1}}
        \rho_{\bar Y}(U)^{\frac{1}{\nu k +1}}
        \\
        \;=&\; 
        C'' m
        \, 
        k^{\frac{\nu}{2\nu k +2}}
        \,
        n_*^{ - \frac{\nu-2}{2\nu k +2}}
        \,
        \rho_{\bar Y}(U)^{\frac{1}{\nu k +1}}
        \;\leq\; 
        C  m
        \,
        n^{ - \frac{\nu-2}{\nu k + 1}}
        \,
        \rho_{\bar Y}(U)^{\frac{1}{\nu k +1}}\;.
    \end{align*}
    In the last line, we have noted that $n_* = \binom{n}{2} \leq n^2$. This proves the first bound. The second bound follows by applying \Cref{prop:gaussian} 
    % \cref{prop:gaussian} 
    while conditioning on $U$.
\end{proof}

% \bibliographystyle{imsart-number}

% \bibliography{ref}

% \end{document}

\end{document}